\documentclass[11pt,leqno]{amsart}

\usepackage[all,cmtip]{xy}
\usepackage{a4,latexsym}
\usepackage[utf8]{inputenc}
\usepackage{amsfonts}
\usepackage[hidelinks]{hyperref}
\usepackage{amsxtra,amsmath,amsfonts,amscd, amssymb, mathrsfs, amsthm}
\usepackage{color, mathtools}
\usepackage{bbm}
\usepackage{enumitem}
\usepackage{xcolor}
\usepackage{graphicx}
\usepackage{tikz-cd}
\usepackage{braket}% \Set command
\usepackage[all]{xypic}
\usepackage[toc,page,title,titletoc,header]{appendix}
\usepackage[capitalize]{cleveref}

\numberwithin{paragraph}{section}

\setlist[enumerate]{label=\it{(\roman*)},
	ref=\it{(\roman*)}}
	
\newlist{enumerate1}{enumerate}{1}
\setlist[enumerate1]{resume,leftmargin=*,label=\it(\arabic*),ref=\it{(\arabic*)}}

\newlist{enumeratea}{enumerate}{1}
\setlist[enumeratea]{resume,leftmargin=*,label=\it(\alph*),ref=\it{(\alph*)}}
\newcommand{\Bcal}{{\mathscr B}}

\newcommand{\Dcal}{{\mathscr D}}
\newcommand{\Hcal}{{\mathscr H}}
\newcommand{\Lcal}{{\mathscr L}}

\newcommand{\Xcal}{{\mathscr X}}

\newcommand{\C}{{\mathbb C}}
\newcommand{\R}{{\mathbb R}}

\newcommand{\Q}{{\mathbb Q}}
\newcommand{\Z}{{\mathbb Z}}
\newcommand{\N}{{\mathbb N}}

\newcommand{\K}{{\mathbb K}}

\renewcommand{\Im}{\mathrm{Im}}

\newcommand{\metr}{{\|\hspace{1ex}\|}}
\newcommand{\val}{{|\hspace{1ex}|}}

\newcommand{\Spec}{\mathrm{Spec}}

\newcommand{\Supp}{\mathrm{Supp}}

\newcommand{\Div}{\mathrm{Div}}

\newcommand{\Pic}{\mathrm{Pic}}

\newcommand{\arnef}{\mathrm{ar\text{-}nef}}
\newcommand{\arint}{\mathrm{ar\text{-}int}}
\newcommand{\relint}{\mathrm{rel\text{-}int}}
\newcommand{\relnef}{\mathrm{rel\text{-}nef}}
\newcommand{\relsnef}{\mathrm{rel\text{-}snef}}
\newcommand{\arsnef}{\mathrm{ar\text{-}snef}}

\newcommand{\nef}{\mathrm{nef}}
\newcommand{\snef}{\mathrm{snef}}
\newcommand{\integrable}{\mathrm{int}}
\newcommand{\FS}{\mathrm{FS}}

\newcommand{\tFS}{\mathrm{tFS}}

\newcommand{\YZ}{\mathrm{YZ}}
\newcommand{\CM}{\mathrm{CM}}
\newcommand{\gm}{\mathrm{gm}}
\newcommand{\SP}{\mathrm{SP}}
\newcommand{\DSP}{\mathrm{DSP}}
\newcommand{\mo}{\mathrm{mo}}
\newcommand{\cpt}{\mathrm{cpt}}
\newcommand{\OO}{\mathcal{O}}

\newcommand{\ord}{\mathrm{ord}}

\renewcommand{\Im}{\operatorname{Im}\nolimits}
\newcommand{\Ker}{\operatorname{Ker}\nolimits}

\renewcommand{\dim}{\operatorname{dim}\nolimits}

\newcommand{\kcirc}{K^\circ}
\newcommand{\ktilde}{\tilde{K}}

\newcommand{\trop}{\mathrm{trop}}
\newcommand{\an}{\mathrm{an}}

\newtheorem{theorem}{Theorem}[section]

\newtheorem{corollary}[theorem]{Corollary}

\newtheorem{lemma}[theorem]{Lemma}
\newtheorem{lemma*}{Lemma}
\newtheorem{proposition}[theorem]{Proposition}
\newtheorem{prop}[theorem]{Proposition}

\newtheorem{theointro}{Theorem}

\theoremstyle{definition}
\newtheorem{definition}[theorem]{Definition}
\newtheorem{proposition&definition}[theorem]{Proposition\&Definition}
\newtheorem{lemma&definition}[theorem]{Lemma\&Definition}
\newtheorem{theorem&definition}[theorem]{Theorem\&Definition}
\newtheorem{example}[theorem]{Example}
\newtheorem{ex}[theorem]{Example}
\newtheorem{example*}{Example}
\newtheorem{remark}[theorem]{Remark}
\newtheorem{remark*}{Remark}

\newtheorem{question*}{Question}

\newtheorem{art}[theorem]{}

 \thanks{The authors were supported by the collaborative research 
	center SFB 1085 \emph{Higher Invariants - Interactions between Arithmetic Geometry and Global Analysis} funded by the Deutsche Forschungsgemeinschaft.}

\begin{document}
	
	\title[Abstract divisorial spaces and arithmetic intersection numbers]{abstract divisorial spaces and arithmetic intersection numbers}
	
	\author[Y.~Cai]{Yulin Cai}
	\address{Y. Cai, Hangzhou International Innovation Institute, Beihang University, Hangzhou 311115, China}
	\email{ylcai5388339@gmail.com}
	
	\author[W.~Gubler]{Walter Gubler}
	\address{W. Gubler, Mathematik, Universit{\"a}t 
		Regensburg, 93040 Regensburg, Germany}
	\email{walter.gubler@mathematik.uni-regensburg.de}

	\begin{abstract}
		Yuan and Zhang introduced arithmetic intersection numbers for adelic line bundles on quasi-projective varieties over a number field. Burgos and Kramer generalized this approach allowing more singular metrics at archimedean places. We introduce abstract divisorial spaces as a tool to generalize these arithmetic intersection numbers to the setting of a proper adelic base curve in the sense of Chen and Moriwaki. We also allow more singular metrics at non-archimedean places using relative mixed energy there as well.
	\end{abstract}


\maketitle

\setcounter{tocdepth}{1}

\tableofcontents

%----------------------------------------------------------------------------------------
%	SECTION 1
%----------------------------------------------------------------------------------------	
\section{Introduction}

Arakelov theory is an arithmetic intersection theory which is used in arithmetic geometry to solve diophantine problems. First, Arakelov theory was developed for arithmetic surfaces which helped Faltings \cite{faltings1983end} to give a proof of the Mordell, Shafarevich and Tate conjecture. Gillet and Soul\'e \cite{gillet1990arithmetic} generalized Arakelov theory to higher dimensions which was used by Faltings \cite{Faltings91} to prove the Mordell--Lang conjecture for subvarieties of abelian varieties.

\subsection{Classical Arakelov theory} \label{subsec: classical Arakelov theory}
We give here a brief description of the classical Arakelov theory over a number field $K$. A similar description is possible also for function fields. The applications in diophantine geometry mainly use heights of the ambient variety or of subvarieties. These heights are defined as arithmetic intersection numbers of arithmetic divisors, so we stick to this case. Let $\Xcal$ be a projective arithmetic variety over the ring of algebraic integers $\OO_K$ which means a  flat projective variety over $\OO_K$ {with regular generic fiber}. We assume that the relative dimension of $\Xcal$ over $\OO_K$ is $d$. An \emph{arithmetic model divisor} on $\Xcal$ is a pair $(D,g_D)$ where $D$ is a Cartier divisor on $\Xcal$ and $g_D$ is a family of smooth Green functions for $D$ at the archimedean places. These Green functions correspond to smooth metrics $\metr$ of $\OO_{\mathscr X}(D)$ at the archimedean places given by $g_D=-\log\|s_D\|$ where $s_D$ is the canonical rational section of $\OO_{\mathscr X}(D)$. 
The arithmetic intersection product in the theory of Gillet and Soul\'e \cite{gillet1990arithmetic} gives arithmetic intersection numbers $(D_0,g_0) \cdots (D_d,g_d) \in \R$ for any arithmetic model divisors $(D_0,g_0), \dots, (D_d,g_d)$ on $\Xcal$. 
The product formula of $K$ yields that these arithmetic intersection numbers depend only on the isometry classes of the underlying metrized line bundles $(\OO_{\mathscr X}(D_j), \metr_j)$. 

The following positivity notions are important for the following. An arithmetic divisor is \emph{relatively nef} if the metrics have semipositive curvature forms at the archimedean places and if the restriction of the model $\OO_{\mathscr X}(D)$ to every fiber $\Xcal_{\mathfrak p}$ is a nef line bundle in the sense of algebraic geometry for every maximal ideal $\mathfrak p $ of $\OO_K$. We call an arithmetic divisor $(D,g_D)$ on $\Xcal$ \emph{arithmetically nef} if it is relatively nef and if for every horizontal curve $Y$ of $\Xcal$, the height of $Y$ with respect to the  metrized line bundle $(\OO_{\mathscr X}(D),\metr)$ corresponding to $(D,g_D)$ is non-negative where the height is the classical height of the generic point of $Y$ viewed as a closed point of the generic fiber of $\Xcal$ over $\OO_K$. Note that for vertical curves $Y$ (i.e.~curves contained in a  fiber $\Xcal_{\mathfrak p}$ for some maximal ideal $\mathfrak p$ of $\OO_K$), relative nefness yields that the degree of $Y$ with respect to $\OO_{\mathscr X}(D)$ is non-negative.

\subsection{Semipositive metrics in the sense of Zhang} \label{subsec: semipositive metrics in the sense of Zhang}
Classical Arakelov theory handles the non-archimedean places by using models over $\OO_K$. It was realized by Zhang \cite{zhang1995smallpoints} that the contributions of the non-archimedean places to the arithmetic intersection numbers is also determined by metrics. Given a projective arithmetic variety $\Xcal$ over $\OO_K$ with generic fiber $X=\Xcal_K$ and a line bundle $\Lcal$ on $\Xcal$ with $L=\Lcal_K$, then for every non-archimedean place $v$ of $K$, there is a unique metric $\metr_{\Lcal,v}$ of $L$ at $v$ such that $\|s(x)\|_{\Lcal,v}=1$ for every local frame $s$ of  $\Lcal$ at the reduction of $x$ in the  fiber of $\Xcal_{\mathfrak p_v}$ over the maximal ideal $\mathfrak p_v$ of $\OO_K$ corresponding to $v$. We call $(\Xcal,\Lcal)$ a \emph{model} of $(X,L)$. Then a \emph{model metric} $\metr$ of $L$ is a  metric $\metr_v$ of $L$ at every place $v$ of $K$ which is smooth if $v$ is archimedean and such that there is $n \in \N_{>0}$ and a global model $(\Xcal,\Lcal)$ of $(X,L^{\otimes n})$ such that $\metr_v^{\otimes n}=\metr_{\Lcal,v}$ if $v$ is non-archimedean.
The upshot is that arithmetic intersection numbers are now defined for line bundles $(L_0,\metr_0), \dots, (L_d,\metr_d)$ endowed with model metrics on a $d$-dimensional (regular) projective variety $X$ over $K$. The \emph{height} of a $e$-dimensional closed subvariety $Y$ of $X$ with respect to a metrized line bundle $\overline L$ endowed with a model metric is defined as the $(e+1)$-fold self-intersection number of $\overline L|_Y$ on $Y$. We call $\overline L=(L,\metr)$ \emph{relatively nef} (resp.~\emph{arithmetically nef}) if the arithmetic divisor $({\rm div}(s), -\log\|s\|)$ is {relatively nef} (resp.~{arithmetically nef}) for a non-zero rational section $s$ of $L$. This does not depend on the choice of $s$.

In case of toric or abelian varieties, very ample line bundles have canonical metrics. However, they are either not smooth at the archimedean places (in the toric case) or cannot be given by a model at non-archimedean places of bad reduction (in the case of abelian varieties).  Using Tate's limit argument, the canonical metrics are uniform limits of semipositive model metrics. Zhang defined \emph{semipositive metrics} as uniform limits of semipositive model metrics. Such a uniform convergence is allowed at finitely many places, at the other {places} we choose the model metrics induced by a single model. A metric of $L$ is called \emph{integrable} if it is the quotient of two semipositive  metrics with respect to tensor multiplication. Zhang then showed that the arithmetic intersection numbers extend to line bundles endowed with integrable metrics. This allows to deal with canonical metrics on toric and abelian varieties. Ullmo \cite{ullmo1998} used that to prove the Bogomolov conjecture about small geometric points of a curve inside its Jacobian and Zhang \cite{zhang1998} generalized this shortly afterwards for closed subvarieties in an abelian variety. 

\subsection{Arakelov theory over adelic curves} \label{subsec: Arakelov theory over adelic curves}
The above descriptions of arithmetic intersection products work well in the case of a number field and in case of function fields. But there are other fields of interest which have a product formula with respect to a non-necessarily discrete set of absolute values. An example arising from relative Arakelov theory was studied by Moriwaki \cite{moriwaki2000}. Let $\Bcal$ be a projective  arithmetic variety over $\OO_F$ {of relative dimension $b$} for a number field $F$, let $K$ be the function field of $\Bcal$ and let $B$ the generic fiber of $\Bcal$. We fix an arithmetic polarization $\overline\Hcal$ on $\Bcal$ which is an ample line bundle $\Hcal$ on $\Bcal$ endowed with an arithmetically nef metric which yields that the height $h_{\overline{\mathscr H}}(V)$ of any closed subvariety $V$ of $\Bcal$ with respect to $\Hcal$ is non-negative.  Then Moriwaki considers the set $\Omega$ of absolute values on $K$ where the non-archimedean ones are parametrized by the prime divisors $V$ of $\Bcal$ endowed with the weight $h_{\overline{\mathscr H}}(V)$ and where the archimedean places are given by the generic points (i.e. the points of the complex manifold  $B_v^\an$ associated to an archimedean place $v$ which do not lie in any proper closed algebraic subvariety). Then the corresponding absolute value is given by $|f(p)|_v$ for any $f \in K$. These archimedean absolute values are not discrete, they are parametrized by $B_v^\an$ for any archimedean place $v$ and are endowed with the weight 
$$\mu_v = \frac{[F_v:\Q_v]}{[F:\Q] }\, c_1(H_v,\metr_v)^b$$
where $\Q_v$ and $F_v$ are the completions of $\Q$ and $F$ with respect to the archimedean place $v$ of $F$. Then the induction formula for heights yields that $K$ satisfies the \emph{product formula}  
\begin{equation} \label{intro: product formula}
	\int_\Omega \log|f|_v \,\nu(dv) = 0
\end{equation}
for any non-zero element $f$ of $K$ where the counting measure is used at the non-archimedean places of $\Omega$. Moriwaki \cite{moriwaki2000} showed that Arakelov theory can be done for projective varieties over such finitely generated fields giving arithmetic intersection numbers and proving the Bogomolov conjecture as a natural generalization of the number field case.

This was the prototype for Chen and Moriwaki \cite{chen2020arakelov} to introduce adelic curves and developing Arakelov theory over such adelic curves. The adelic curves capture all situations where global heights in diophantine geometry occur. An \emph{adelic curve} consists of a field $K$ and a  space $\Omega$ parametrizing absolute values on $K$ by {a map $\phi$}. The space $\Omega$ is endowed with a positive measure $\nu$. We assume throughout that the adelic curve is \emph{proper} which means that the product formula \eqref{intro: product formula} holds. We assume that $K$ is countable or that the underlying $\sigma$-algebra of $\Omega$ is discrete. Let $X$ be a projective variety over $K$. 
A \emph{relatively nef adelic metric} of a line bundle $L$ over $X$ is a family $\metr=(\metr_\omega)_{\omega \in \Omega}$ of metrics which  is locally bounded and measurable such that $\metr_\omega$ is a relatively nef metric of $L$ in the sense of Zhang for every $\omega \in \Omega$, see Section \ref{section: global theory} for details.
Then Chen and Moriwaki \cite{chen2021arithmetic} introduce arithmetic intersection numbers for integrable adelic line bundles as a generalization of the numbers from Zhang's theory in \S \ref{subsec: semipositive metrics in the sense of Zhang}. 
Here, an \emph{integrable adelic line bundle} $(L,\metr)$ on $X$ has the form $(L,\metr)=(L',\metr') \otimes (L'',\metr'')^{-1}$ for {relatively nef adelic line bundles} $(L',\metr'),(L'',\metr'')$ on $X$.  Chen and Moriwaki   \cite{chen2024positivity} prove the arithmetic Hilbert--Samuel formula in this setting and give a generalization of Yuan's equidistribution theorem.

\subsection{Arakelov theory for line bundles with singular metrics} \label{subsec: Arakelov theory for line bundles with singular metrics}
Let us come back to a number field $K$. Some natural metrics occurring in arithmetic geometry are not continuous and hence cannot be handled by classical Arakelov theory or by Zhang's extension. An example is the Petersson metric of the Hodge bundle on a compactification of the moduli space $A_g$ of principally polarized abelian varieties of dimension $g$. This played a major role in the proof of Faltings \cite{faltings1983end} showing the three finiteness results mentioned at the beginning. A systematic approach to Arakelov theory for singular metrics was done by Burgos, Kramer and K\"uhn \cite{burgos2005arithmetic}, \cite{burgos2007cohomological} leading to arithmetic intersection products for line bundles with  $\log$-$\log$ singularities along the boundary. 

A new approach is given by Yuan and Zhang \cite{yuan2021adelic}. Let $U$ be a quasi-projective variety over the number field $K$ with projective compactification $X$ such that the boundary $B=X \setminus U$ is a Cartier divisor. Then they develop an Arakelov theory for \emph{compactified metrized} line bundles. Let $L_U$ be a line bundle over $U$. The metric $\metr$ of a compactified metrized line bundle $\overline L$ is given by a Cauchy sequence of model metrics of $L_U$  of $U$ with respect to a boundary topology induced by the arithmetic boundary divisor $(B,g_B)$. This means more precisely that there is a sequence of line bundles $L_n$ on projective compactifications $X_n$ of $U$ endowed with model metrics $\metr_n$ with $L_n|_U=L_U$ satisfying the following Cauchy condition: Let $s$ be any non-zero rational section of $L_U$, then it extends uniquely to a rational section $s_n$ of $L_n$ and we require  that for all $\varepsilon\in \Q_{>0}$, there is $n_0 \in \N$ with 
$$-\varepsilon B \leq {\rm div}(s_n) - {\rm div}(s_m) \leq \varepsilon B \quad \text{and} \quad -\varepsilon g_B \leq \log\|s_m\|_m - \log\|s_n\|_n \leq \varepsilon g_B$$ 
for all $n,m \geq n_0$, where the inequality on  the left takes place after pulling back to a joint projective compactification. This Cauchy condition does not depend on the choice of $s$. We identify two Cauchy sequences if they are zero-sequences with respect to the boundary topology (same conditions as above with no $m$). If we skip the metrics, we get \emph{compactified geometric line bundles}, so the compactified metrized line bundle $\overline L$ has an underlying compactified geometric line bundle $L$ given by the Cauchy sequence of line bundles $L_n$ on the projective compactifications $X_n$ of $U$. 

For a fixed place $v$ of $K$, compactified metrics of $L_U$ with the same underlying compactified geometric line bundle $L$ are precisely those whose Green functions $g,g'$ satisfy $g'-g=o(g_B)$ along the boundary $B$ for a Green function $g_B$ of $B$. This is a weaker assumption than in the approach of Burgos--Kramer--K\"uhn \cite{burgos2007cohomological} as $\log (\log(t))=o(\log(t))$ for $t \to 0$ in $\R_{>0}$ and as the singularities in \cite{burgos2007cohomological} are allowed only at archimedean places $v$. A compactified metrized line bundle is called \emph{strongly relatively nef} (resp.~\emph{strongly arithmetically nef}) if it is given by a Cauchy sequence of line bundles endowed with {relatively nef} (resp.~{arithmetically nef})  model metrics. If the metric of a compactified  line bundle is the quotient of two strongly arithmetically nef compactified metrics, then we call it \emph{arithmetically integrable}. It is clear that the arithmetically integrable compactified line bundles of $U$ include the line bundles on the projective compactification $X$ endowed with integrable metrics in the sense of Zhang. Then Yuan and Zhang \cite[Theorem~4.1.3]{yuan2021adelic} extend the arithmetic intersection product to arithmetically integrable compactified line bundles. They extend Yuan's equidistribution theorem to this setting \cite[Theorem~5.4.3]{yuan2021adelic} and show that canonical metrics on symmetric relatively ample line bundles of abelian schemes over $U$ are arithmetically {nef and hence arithmetically integrable} \cite[Theorem~6.1.1]{yuan2021adelic}.

The uniform Mordell conjecture asks to bound the number of $K$-rational number of a curve of genus $g>1$ solely in terms of $g$ and the rank of $J(K)$ for the Jacobian $J$. This was recently solved by Dimitrov, Gao, Habegger \cite{DGH21} subject to a condition which was later removed by K\"uhne \cite{kühne2021equidistribution}. Yuan \cite{yuan2021arithmetic} used the theory of compactified line bundles to show that Hodge line bundle on a relative curve endowed with its canonical metric is a compactified line bundle which is arithmetically nef and big. This allowed him to give a new proof of the uniform Mordell conjecture with sharper bounds and which also applies to the function field setting.   

\subsection{The energy approach of Burgos and Kramer} \label{energy approach of Burgos and Kramer}
While compactified metrized line bundles are allowed to have rather strong singularities along the boundary, the assumption that the metrics are arithmetically nef is quite restrictive. There are metrics with $\log$-$\log$ singularities which are semipositive in the approach of Burgos--Kramer--K\"uhn \cite{burgos2007cohomological}, but which are not arithmetically integrable. This is illustrated in the running example below. It is suggested by an example of Burgos and Kramer that this is also the case for the Petersson metric of the Hodge bundle on $A_g$. To combine the two approaches, Burgos and Kramer \cite{burgos2023on} gave the following extension of the arithmetic intersection numbers of compactified metrized line bundles. We consider a compactified metrized line bundle $(L,\metr)$ of the $d$-dimensional quasi-projective variety $U$ over the number field $K$ endowed with a (singular) strongly arithmetically nef metric $\metr$. Then we denote by $h_{(L,\metr)}(U)$ the height of $U$ with respect to $(L,\metr)$ given by the arithmetic intersection numbers of Yuan and Zhang. Burgos and Kramer consider now a more singular relatively nef metric $\metr'$ of $L$ where they assume that $\metr_v'=\metr_v$ for all non-archimedean places $v$ of $K$. For an archimedean place $v$ of $K$, they use the theory for the \emph{relative energy}  
$$E(\metr_v,\metr_v') \coloneqq \sum_{^j=0}^d \int_{U_v^\an} \log(\metr_v/\metr_v') \, c_1(L_v,\metr_v)^{j}  \wedge c_1(L_v,\metr_v')^{n-j} \in \R \cup \{-\infty\}$$
of singular metrics given by Darvas, Di Nezza and Lu \cite{darvas2023relative} to define the height of $X$ with respect to $(L,\metr')$ by
$$h_{(L,\metr')}(U)=h_{(L,\metr)}(U)+ \sum_{v | \infty} \frac{[K_v:\Q_v]}{[K:\Q]} \, E(\metr_v,\metr_v').$$
Note that the height is finite if and only if the relative energy is finite for all archimedean places $v$ of $K$. Of course, there is also a mixed version of this which can be used to extend the arithmetic intersection numbers of Yuan and Zhang. Burgos and Kramer also show that this extension covers all arithmetic intersection numbers from the Burgos--Kramer--K\"uhn approach described in \ref{subsec: Arakelov theory for line bundles with singular metrics}. In particular, arithmetic intersection numbers on $A_g$ using the Hodge line bundle endowed with the canonical Petersson metric are well-defined and finite. However, the previous approach of Burgos--Kramer--K\"uhn could not be used to deal with the line bundle $J_{k,m}$ of Siegel--Jacobi forms of weight $k$ and index $m$ over the universal abelian scheme $B_g$ as a coarse moduli space as the singularity of the canonical metric of $J_{k,m}$ is not of $\log$-$\log$ type. Note that $B_g$ is an abelian scheme over the quasi-projective variety $A_g$, we neglect the choice of a level structure for $N \geq 3$ for simplicity of the exposition. Now Burgos and Kramer  show that {the relative energy of} the canonical metric of $J_{k,m}$ with respect to any fixed choice of a strongly arithmetically nef metric of $J_{k,m}$ (which  exists in this case)  is finite. This shows that arithmetic intersection numbers involving $J_{k,m}$ endowed with the canonical  metrics at the archimedean places is well-defined which allowed them to compute the height of $B_g$ with respect to $J_{k,m}$ endowed with the canonical metric.

\subsection{A running example} \label{subsec: a running example}
Very often in algebraic or arithmetic geometry, toric geometry gives enlightening examples. In Gari Peralta's thesis \cite{peralta-thesis24}, he investigated compactified toric metrics on toric varieties over a number field. We pick from the toric world just a simple example to illustrate the rather abstract results. We do not claim any originality and we refer to his thesis for generalizations. We also thank Jose Burgos for discussions about the example.

We choose  $K=\Q$ for the base field. We denote by $M_\Q$ the set of places $v$ of $\Q$ given either by $v=\infty$ or by a prime number $p$. For $v \in M_\Q$, we choose the standard $v$-adic absolute value on $\Q$. We endow $\Omega=M_\Q$ with the counting measure to get a proper adelic curve as described in \ref{subsec: Arakelov theory over adelic curves}. 

We choose $U=\mathbb A_\Q^1$ with the projective compactification $X=\mathbb P_\Q^1$. We fix homogeneous coordinates $x_0,x_1$ of $X$ viewed as global sections of $\OO_X(1)$ such that $U=X \setminus {\rm div}(x_0)=X \setminus \{\infty\}$.  Then $U$ and $X$ are toric varietes and many things could be generalized to toric varieties, but for simplicity we stick to this simple example. 

Let us write $t\coloneqq x_1/x_0$ for the affine coordinate on $U=\mathbb A_\Q^1$. 
For a place $v$ of $\Q$, let $U_v^\an$ be the analytification of $U\times_\Q \Q_v$ over the completion $\Q_v$. In the archimedean case, it is a manifold while in the non-archimedan case it is the Berkovich space given by the multiplicative seminorms on $\Q_v[t]$ extending the absolute value of $\Q_v$. Then we have the \emph{tropicalization map}
$$\trop_v\colon U_v^\an \longrightarrow \overline \R \, \, \quad x \mapsto -\log|t(x)|_v$$
where $ \overline \R\coloneqq \R \cup \{\pm \infty\}$. We note that $\trop_v(0)=\infty$ and we extend $\trop_v$ to $X_v^\an$ by setting $\trop(\infty)=-\infty$. 

As usual in toric geometry, we associate to the toric divisor $D=[\infty]={\rm div}(x_0)$ the piecewise linear function $\Psi(u)=\min(0,u)$ on $\R$ where the correspondence is given by noting that the slope on $[-\infty,0]$ is the multiplicity $1$ of $D$ at $\infty$ while the slope on $[0,\infty]$ is the multiplicity $0$ of $D$ at $0$ (toric divisors are the divisors supported on $X \setminus \mathbb G_{\rm m}^1$). Then $g_{D}\coloneqq (-\Psi\circ \trop_v)_{v \in M_\Q}$ is a family of Green functions for $D$. The corresponding metric $\metr$ of $\OO_X(1)$ is given by $\|x_0\|_v = |x_0|_v/\max(|x_0|_v,|x_1|_v)$ for any $v \in M_\Q$. In the non-archimedean case, it is the model metric associated to the model $(\Xcal=\mathbb P_\Z^1,\OO_\Xcal(1))$ of $(X,L)$ and in the archimedean case it is a  semipositive metric which is continuous but not smooth. We call $\metr$ the \emph{canonical metric} of $\OO_X(1)$. 

A continuous metric $\metr_v$ of the base change $\OO_X(1)$ to $X_v^\an$ is called \emph{toric} if there	is a continuous function $\psi_v\colon \R \to  \R$ which extends continuously to a self-map of $\overline \R$ such that 
$$\psi_v\circ \trop_v(x)=\log\|x_0(x)\|_v$$
for all $x \in X_v^\an \setminus \{0,\infty\}$. 
The canonical  metric of $\OO_X(1)$  is the toric metric corresponding to $\psi_v=\Psi$. 

In the following, we illustrate the above notions in the case of toric metrics. 
In the non-archimedean case, i.e.~$p$ is a prime number,  the toric metric $\metr_v$ is a model metric if and only if $\psi_v$ is a rational piecewise affine function. 

A toric metric of $\metr_v$ of $\OO_X(1)$ is semipositive in the sense of Zhang if and only if $\psi_v$ is concave and $\psi_v-\Psi$ is bounded, see  \cite[Theorem~4.8.1]{BPS}, or \cite[Theorem~II]{gubler-hertel}. This is due to the fact that the corresponding concave function $\psi_v$  can be uniformly approximated by smooth (resp.~rational piecewise affine)  functions for $v=\infty$ (resp.~for $v$ non-archimedean). 

Now let us consider the trivial line bundle $L_U= \OO_U$ on $U=\mathbb A_\Q^1$. We consider now compactified metrics with underlying compactified geometric line bundle $L=\OO_X(1)$ which means that in defining Cauchy sequence, we can choose $L_n=\OO_X(1)$ and $L_n|_{U}=L_U$ by identifying the section $x_0$ with $1$ for all $n \in \N$.  A singular relatively nef toric metric $\metr'$ of $L$  in the approach of Yuan--Zhang is given  by a continuous concave functions $\psi_v'$ with $\psi_v'(u)-\Psi(u)= o(-u)$ for $u \to -\infty$ and with $\psi_v'(u)-\Psi(u)=O(1)$ for $u \to \infty$ requiring  $\psi_v'=\Psi$ up to finitely many places $v \in M_K$. 

We consider here the following example for such a singular metric $\metr'$ of $L$. We take the canonical metric $\metr$ of $L$ corresponding to $\psi_v=\Psi$ for all $v \in M_\Q$ as a reference metric. We fix $\alpha \in ]0,1[$ and a place $v$ of $M_K$ where the singular toric metric $\metr_v'$ of $L=\OO_X(1)$ is given by the concave piecewise smooth function $\psi_v' \coloneqq \Psi + \rho$ with 
$\rho(u)\coloneqq \frac{1}{\alpha}$ for $u \geq 0$ and $\rho(u)\coloneqq \frac{1}{\alpha}(1-u)^\alpha$ for $u \leq 0$. We take $\metr_w'=\metr_w$ for all other places $w \in M_\Q \setminus \{v\}$. In any case, it is clear that the singular metric $\metr'$ of $L$  is relatively nef. We will see in \cref{rareness of arithmetic nef} that the metric $\metr'$ is never arithmetically integrable. Moreover, we will see in \cref{specific example for height function} that the relative energy is finite if and only if $\alpha<1/2$, and in this case we can compute the height as
$$h_{(L,\metr')}(U) = \frac{2-3\alpha}{\alpha(2\alpha-1)}.$$
Note that this height does neither make sense in the approach of Burgos--Kramer--K\"uhn \cite{burgos2007cohomological} as the singularity of $\metr$ is not of $\log$-$\log$-type, nor does it make sense in the Yuan--Zhang approach \cite{yuan2021adelic} as $\metr$ is not an arithmetically integrable compactified metric. If $v$ is an archimedean place, we can use the energy approach of Burgos and Kramer \cite{burgos2023on} to define the height. As we extend their approach to non-archimedean places in this paper, the formula makes sense and is true in any case.

\subsection{Goals} \label{subsec: goals}
The approach of Yuan and Zhang, giving arithmetic intersection numbers for arithmetically integrable compactified metrized line bundles over quasi-projective varieties over $K$, is written in case of a number field $K$. It is hinted in \cite[\S 2.7]{yuan2021adelic} how to generalize the approach in case when $K$ is a function field. 

The goal of this paper is to generalize the arithmetic intersection numbers of compactified metrized line bundles in case where $(K,\Omega)$ is any adelic curve in the sense of Chen and Moriwaki. This will allow to use the approach of Yuan and Zhang in all possible situations considered in Arakelov theory. Replacing projective models by proper models, which leads only to minor technical complications, we will get arithmetic intersection numbers for compactified metrized line bundles on all algebraic varieties, not only quasi-projective varieties. This is only a slight generalization as the arithmetic intersection numbers are birational invariants, so we can always reduce even to the affine case by passing to a dense open subset. 

We will then extend the arithmetic intersection numbers to more singular metrics using the energy approach of Burgos and Kramer. We will allow also that such more singular metrics occur at non-archimedean places of the parameter space $\Omega$. Instead of using the advanced complex pluri-potential theory for the relative energy of  Darvas, Di Nezza and Lu, we use an approach based on the study of Zhang's continuous semipositive metrics from \cite{gubler2019on} at non-archimedean places.

\subsection{Abstract divisorial spaces} \label{subsec: intro abstract divisorial spaces}

It is clear that the construction of Yuan and Zhang \cite{yuan2021adelic} of compactified metrics and line bundles  is through a completion process with respect to the boundary topology. We introduce here abstract divisorial spaces which allow to see this completion process in an abstract setting. Such abstract divisorial spaces occur in various applications. They are our main tool in this paper allowing us to generalize compactified metrized line bundles to the case of a proper adelic base curve in the sense of Chen and Moriwaki \cite{chen2020arakelov}.

We define an \emph{abstract divisorial space} as a pair $(M,N)$ where $M$ is an ordered $\Q$-vector space and where $N$ is a cone in $M$ with $M=N-N$. For $b \in M_{\geq 0} \coloneqq \{x \in M \mid x \geq 0\}$, the space $M$ is endowed with the \emph{$b$-topology} which is the unique topology of $M$ such that $M$ is a topological group with a basis of neighborhoods of $0$ given by the (non-open) sets
$$\{x \in M \mid -\varepsilon b \leq x \leq \varepsilon b \} \quad (\varepsilon \in \Q_{>0}).$$ 
The completion of $M$ with respect to the $b$-topology is denoted by $\widehat{M}^{d_b}$ as it can be also defined with respect to a natural pseudo-metric $d_b$. Let $\widehat{N}^b$ be the closure of the image of $N$ in $\widehat{M}^{d_b}$ and let $\widehat{M}^b\coloneqq \widehat{N}^b-\widehat{N}^b$, then we get an abstract divisorial space $(\widehat{M}^b,\widehat{N}^b)$ which we can characterize by a universal property in the category of abstract divisorial spaces.

The basic example for an abstract divisorial space is the space of $\Q$-Cartier divisors on a projective variety $X$ endowed with nef cone. The partial order is induced by the cone of effective $\Q$-Cartier divisors. There are many variations of this geometric example as $\Pic(X) \otimes_\Z \Q$ endowed with the cone of isomorphism classes of nef $\Q$-line bundles. For a quasi-projective variety $U$ over the base field $K$, we consider a projective compactification $X$ such that $X\setminus U$ is the support of an effective Cartier divisor $B$. Using the direct limit of the associated divisorial spaces and its completion with respect to the $B$-topology, we get the space of compactified geometric line bundles of Yuan-Zhang \cite{yuan2021adelic} described in \S \ref{subsec: Arakelov theory for line bundles with singular metrics}. We can also define abstract divisorial spaces over $\R$ where $M$ is a vector space over $\R$ instead of $\Q$. This plays only a minor role in this paper, but a natural example is the space of signed Borel measures on a compact space endowed with the cone of positive Borel measures.

In the arithmetic setting with $\Xcal$ a projective arithmetic variety over $\OO_K$ for a number field $K$,  the arithmetic model $\Q$-divisors form an abstract divisorial space with the cone of arithmetically nef $\Q$-divisors. For $U$ a quasi-projective variety over $K$, a similar completion process of abstract divisorial spaces for projective models leads to the compactified metrized line bundles of Yuan-Zhang from \S~\ref{subsec: Arakelov theory for line bundles with singular metrics}. This generalizes to the case of an adelic base curve and will allow us to define compactified metrized line bundles in this setting. 

Let $n \in \N$ and let $(M,N)$ be an abstract divisorial space. Then an \emph{$(n+1)$-intersection map} is a multilinear symmetric map $h\colon M^{n+1}\to \R$ which is non-negative on $N^{n+1} \cup (M_{\geq 0}\times N^n)$. If $b \in M_{\geq 0}$ {has an upper bound} in $N$, then a crucial result in Section~\ref{section: abstract divisorial spaces} shows that $h$ has a unique extension to an $(n+1)$-intersection map $h\colon (\widehat{M}^b)^{n+1}\to \R$. In the geometric example of an abstract divisorial space given by the $\Q$-Cartier divisors on a projective variety $X$ of dimension $n$, the algebraic intersection numbers of $\Q$-Cartier give an $n$-intersection map. Similarly, if $\Xcal$ is a projective arithmetic variety over $\OO_K$ of relative dimension $n$, then the arithmetic intersection numbers of arithmetic model $\Q$-divisors define an $(n+1)$-intersection map. The crucial result above explains the extension of geometric (resp.~arithmetic) intersection numbers to the quasi-projective setting and will allow to perform a similar construction over adelic curves. In Section \ref{section: abstract divisorial spaces}, we will study abstract divisorial spaces and we will define $(n+1)$-intersection maps more generally as multilinear maps between abstract divisorial spaces. 

\subsection{Notation and terminology} \label{subsec: Notation and terminology}

Before we state our main results, we fix our notation used in this paper. 

The natural numbers $\N$ include $0$. If we have $S \subset T$, this allows $S=T$. 
For a field $K$ with absolute value $\val_\omega$, we denote by $K_\omega$ the completion of $K$ with respect to $\val_\omega$. For a partially ordered set $M$, we use $M_{\geq 0} \coloneqq \{x \in M \mid x \geq 0\}$ and $M_{>0} \coloneqq  \{x \in M \mid x > 0\}$. 

For any scheme $X$, we denote by $\Div(X)$ the group of Cartier divisors on $X$ and by $\Pic(X)$ 
the group of isomorphism classes of line bundles on $X$. An \emph{algebraic variety} over the field $K$ is defined as a  geometrically integral  separated scheme of finite type over $K$. We say that a line bundle $L$ or a Cartier divisor $D$  is \emph{semiample} if there is $n \in N_{>0}$ such that $L^{\otimes n}$ is generated by global sections for $L = \OO_X(D)$ in the latter case.

\subsection*{Geometric setting}
Let $U$ be an {algebraic variety} over a field $K$. The geometric setting is considered in Section \ref{sec: geometric theory}. The geometric theory is based on \emph{proper $K$-models $X$ of $U$} which means proper varieties $X$ containing $U$ as a dense open subset. Replacing $X$ by a blow up, we may assume that $X\setminus U$ is the support of an effective Cartier divisor $B$. Such a divisor is called a \emph{boundary divisor} and gives rise to the \emph{boundary topology} on the spaces below, similarly as in \S \ref{subsec: intro abstract divisorial spaces}. The boundary topology is independent of the choice of $B$. {We consider the following spaces and cones:}
\begin{itemize}
	\item $\Div_\Q(U)$ group of $\Q$-Cartier divisors of $U$
	\item $\Div_\Q(U)_\mo = \varinjlim\limits_X \Div_\Q(X)$ with $X$ ranging over all proper $K$-models  of $U$, see \ref{geometric model divisors}  
	\item $\widetilde{\Div}_\Q(U)_\cpt$ group of \emph{compactified geometric divisors}, obtained as the completion of $\Div_\Q(U)_\mo$ with respect to the boundary topology, see \ref{def:geometric boundary topology}
	\item $\widetilde{\Div}_\Q(U)_\snef$ cone of \emph{strongly nef compactified geometric divisors},  closure in $\widetilde{\Div}_\Q(U)_\cpt$ of the nef cone of $\Div_\Q(U)_\mo$ with respect to the boundary topology, see  \ref{strongly nef and nef compactified geometric divisors}
	\item $\widetilde{\Div}_\Q(U)_\integrable \coloneqq \widetilde{\Div}_\Q(U)_\snef-\widetilde{\Div}_\Q(U)_\snef$ divisorial space of \emph{integrable compactified geometric divisors}, see \ref{strongly nef and nef compactified geometric divisors}
	\item $\widetilde{\Div}_\Q(U)_\nef$ cone of \emph{nef compactified geometric divisors}, given as the closure of $\widetilde{\Div}_\Q(U)_\snef$ with respect to the finite subspace topology of $\widetilde{\Div}_\Q(U)_\cpt$, see  \ref{strongly nef and nef compactified geometric divisors}
\end{itemize}

The difference to the original approach of Yuan and Zhang is that we allow arbitrary algebraic varieties $U$ as opposed to quasi-projective varieties in \cite{yuan2021adelic} and the compactified geometric divisors rely on proper $K$-models instead of projective $K$-models. 

\subsection*{Local setting}  We assume  that $K$ is endowed with a complete absolute value $\val_v$. In the non-archimedean case, we denote the valuation ring by $K^\circ$ and the residue field by $\widetilde K$. Let $U$ be an algebraic variety  over $K$, let {$U^\an$ be the analytification of $U$} (in the sense of Berkovich). For $D \in \Div_\Q(U)$, we consider Green functions $g_D$ on {$U^\an$}, see \ref{Green functions for Q-divisors}. We fix a \emph{metrized boundary divisor $(B,g_B)$} consisting of a boundary divisor $B$ and a Green function $g_B>0$ for $B$. Again, this leads to a \emph{boundary topology} on the spaces below independent of the choice of $B$. 
Similarly as in the geometric case, we will use the following spaces studied in  Section \ref{section: local theory}.:
\begin{itemize}
	\item $\widehat\Div_\Q(U) \coloneqq \{(D,g_D) \mid \text{$D \in \Div_\Q(U)$, $g_D$ Green function for $D$}\}$
	\item $\widehat\Div_\Q(U)_\mo$ subgroup where $g_D$ comes from a smooth (resp.~model) metric over a proper $K$-model $X$ of $U$ for $v$ archimedean (resp.~non-archimedean), see \ref{model Green function}
	\item  $\widehat\Div_\Q(U)_\tFS$ subgroup of $\widehat\Div_\Q(U)$ generated by Green functions coming from twisted Fubini--Study metrics, used only if $v$ is trivial instead of $\widehat\Div_\Q(U)_\mo$,  see \ref{Fubini-Study metric}
	\item $\widehat{\Div}_\Q(U)_\cpt$ subgroup of \emph{compactified metrized divisors} in  $\widehat\Div_\Q(U)$, obtained as the completion of $\widehat\Div_\Q(U)_{\mo/\tFS}$  with respect to the boundary topology, see \ref{def:boundarytopologylocal}
		\item $\widehat\Div_\Q(U)_\snef$ closure in $\widehat{\Div}_\Q(U)_\cpt$ of the semipositive/twisted Fubini--Study cone of $\widehat\Div_\Q(U)_{\mo/\tFS}$ with respect to the boundary topology, see  \ref{strongly nef and nef compactified divisors}
	\item $\widehat\Div_\Q(U)_\integrable \coloneqq \widehat\Div_\Q(U)_\snef-\widehat\Div_\Q(U)_\snef$ divisorial space of \emph{integrable compactified metrized divisors}, see \ref{strongly nef and nef compactified divisors}
	\item $\widehat{\Div}_\Q(U)_\nef$ the cone of \emph{nef compactified metrized divisors}, given as the closure of $\widehat\Div_\Q(U)_\snef$ with respect to the finite subspace topology of $\widehat{\Div}_\Q(U)_\cpt$, see  \ref{strongly nef and nef compactified divisors}
\end{itemize}
We have similar notions for \emph{compactified metrized line bundles} which we discuss in \S \ref{subsection:local adelic line bundles}. 
The difference to the approach Yuan and Zhang \cite[Section 3.6]{yuan2021adelic} is again that we allow arbitrary algebraic varieties $U$, that we allow proper $K$-models $X$ of $U$ and proper models of $X$ over $K^\circ$ as opposed to projective models and that we allow $v$ to be the trivial valuation when we always use twisted Fubini--Study metrics instead of model metrics. In the non-trivially valued case and for $U$ quasi-projective, the above notions agree with the corresponding notions in \cite[Section 3.6]{yuan2021adelic}.

\subsection*{Global setting} We consider now a field $K$ endowed with the structure of a proper adelic curve $S$ consisting of a parameter space $\Omega$ of absolute values of $K$ satisfying the product formula with respect to a fixed positive measure $\nu$ on $\Omega$. We always assume that the archimedean absolute values are normalized (see Remark \ref{assumeption on adelic curves}) and that the underlying $\sigma$-algebra $\mathcal{A}$ is either discrete or that $K$ is countable. The adelic curves were introduced by Chen and Moriwaki who make these hypotheses frequently for obtaining more advanced results. We will summarize adelic curves in Section \ref{sec: adelic curves}. 

Let $U$ be an algebraic variety over $K$. Let $K_\omega$ be the completion of $K$ with respect to $\omega \in \Omega$ and let $U_\omega$ be the base change of $U$ to $K_\omega$. An \emph{$S$-Green function} $g_D$ for $D \in \Div_\Q(U)$ is a  family $g_D=(g_{D,\omega})_{\omega \in \Omega}$ where $g_{D,\omega}$ is a Green function for the base change $D_\omega$ of $D$ to $U_\omega$ for each $\omega \in \Omega$, see \ref{global Green functions}. In the following, we  always require that the family $g_D$ is locally $S$-bounded and $S$-measurable, see \ref{global Green functions}. On the spaces below, we consider various boundary topologies, each of them given as the $b$-topology as in \ref{subsec: intro abstract divisorial spaces} for some $b=(B,g_B)$ where $B$ is an effective Cartier divisor on a proper $K$-model $X$ with support in $X \setminus U$ and where $g_B$ is an $S$-Green function for $B$ with $g_B\geq 0$. The reason is that we do not know the existence of a cofinal $b$ in contrast to the geometric and the local setting. 
 In the global case, we use  the  notions:
 	\begin{itemize}
 		\item $\widehat\Div_{S,\Q}(U) \coloneqq \{(D,g_D) \mid \text{$D \in \Div_\Q(U)$, $g_D$ $S$-Green function for $D$}\}$, see \ref{global Green functions}
 		\item $\widehat\Div_{S,\Q}(U)_\CM$ subgroup {of $\widehat\Div_{S,\Q}(U)$} %generated by divisors from $\widehat\Div_{S,\Q}(X)$ for some proper $K$-model $X$ of $U$} (walter: we have not introduced the notation $\widehat\Div_{S,\Q}(X)$ in the introduction)
 {given} by $(D,g_D) \in \widehat\Div_{S,\Q}(U)$ such that  $g_{D,\omega}$ is induced by a continuous  Green function  on $X_\omega$ for any $\omega\in\Omega$ for some proper $K$-model $X$ of $U$, see \cref{global Green functions} 
 		\item $\widehat\Div_{S,\Q}(U)_\cpt$ subgroup of \emph{compactified $S$-metrized divisors} in  $\widehat\Div_{S,\Q}(U)$, obtained from completions of $\widehat\Div_{S,\Q}(U)_\CM$  with respect to some boundary topologies, see \cref{def:boundarytopologyglobal} 
 		\item $\widehat\Div_{S,\Q}(U)_\relsnef$ cone of \emph{strongly relatively nef compactified $S$-metrized divisors}, induced by Cauchy sequences of elements from the  {semipositive} cone of $\widehat\Div_{S,\Q}(U)_\CM$  with respect to some boundary topology, see  \ref{def:boundarytopologyglobal}
 		\item $\widehat\Div_{S,\Q}(U)_\relint \coloneqq \widehat\Div_{S,\Q}(U)_\relsnef-\widehat\Div_{S,\Q}(U)_\relsnef$ divisorial space of \emph{relatively integrable compactified $S$-metrized divisors}, see  \ref{def:boundarytopologyglobal}
 		\item $\widehat\Div_{S,\Q}(U)_\relnef$ \emph{relatively nef compactified $S$-metrized divisors},  the closure of the 
 		cone $\widehat\Div_{S,\Q}(U)_\relsnef$ wrt the finite subspace topology of $\widehat\Div_{S,\Q}(U)_\cpt$, see  \ref{def:boundarytopologyglobal}
 		\item $\widehat\Div_{S,\Q}(U)_\arsnef$ cone of \emph{strongly arithmetically nef compactified $S$-metrized divisors}, induced by Cauchy sequences of elements from the arithmetically nef cone of $\widehat\Div_{S,\Q}(U)_\CM$  with respect to some boundary topology, see  \ref{divisorial space on U based on arithmetic nef}
 		\item $\widehat\Div_{S,\Q}(U)_\arint \coloneqq \widehat\Div_{S,\Q}(U)_\relsnef-\widehat\Div_{S,\Q}(U)_\relsnef$ divisorial space of \emph{arithmetically integrable compactified $S$-metrized divisors}, see  \ref{divisorial space on U based on arithmetic nef}
 			\item $\widehat\Div_{S,\Q}(U)_\arnef$ \emph{arithmetically nef compactified $S$-metrized divisors},  the closure of the
 		cone $\widehat\Div_{S,\Q}(U)_\arsnef$ wrt the finite subspace topology of $\widehat\Div_{S,\Q}(U)_\cpt$, 	  see  \ref{divisorial space on U based on arithmetic nef}
 	\end{itemize}
We have similar notions for \emph{compactified $S$-metrized line bundles} which we discuss in \S \ref{subsection:global adelic line bundles}. 
The  theory is built upon the  adelic line bundles on proper $K$-models $X$ of $U$ introduced and studied by Chen and Moriwaki \cite{chen2020arakelov} (in case of projective $X$, but the generalization to proper varieties is straightforward).  We note that $\widehat\Div_{S,\Q}(X)$ is the group of $S$-metrized divisors arising from adelic line bundles in the sense of Chen and Moriwaki. Then $\widehat\Div_{S,\Q}(U)_\CM$ is defined as the direct limit of the groups $\widehat\Div_{S,\Q}(X)$ with $X$ ranging over all proper $K$-models of $U$. 

We have seen in \S \ref{subsec: Arakelov theory over adelic curves} that integrable adelic line bundles allow arithmetic intersection numbers. In particular, the arithmetic intersection numbers are defined on the semipositive cone in $\widehat\Div_{S,\Q}(U)_\CM$ which we will use later to extend them to $\widehat\Div_{S,\Q}(U)_\arint$. 

The above definitions are strongly motivated by the approach of Yuan and Zhang \cite{yuan2021adelic}. There are the following differences. It is clear that work in the more general setting of an adelic base curve $S=(K,\Omega,\mathcal{A},\nu)$. As in the geometric and the local setting, we allow arbitrary algebraic varieties $U$ over $K$. Another difference is that the approach of Yuan and Zhang relies on models over the algebraic integers of $K$ which does not make sense over an adelic curve $S$. Instead, we rely on the adelic line bundles (in the sense of Chen and Moriwaki) on proper $K$-models of $U$ which naturally leads to the fact that all the spaces introduced above are larger than the corresponding spaces introduced by Yuan and Zhang in the number field case, see Section \ref{sec: Comparing with Yuan-Zhang's theory} for details. 

Burgos and Kramer \cite{burgos2023on} work in a similar setting as Yuan and Zhang, but use real Cartier divisors and assume that the quasi-projective variety $U$ is normal. Keeping all this in mind, we have the following comparison of notions:

\begin{itemize}
	\item  \emph{Compactified $S$-metrized divisors} are called in \cite{yuan2021adelic} either \emph{adelic divisors} or \emph{compactified divisors}. We don't call them adelic divisors as this notion  has a different meaning in the work of Chen and Moriwaki \cite[6.2.3]{chen2020arakelov}. 
	\item  \emph{Compactified $S$-metrized line bundles} are called in \cite{yuan2021adelic} either \emph{adelic line bundles} or \emph{compactified line bundles}. Again we omit the notion adelic line bundle here. 
	\item \emph{(Strongly) arithmetically nef compactified $S$-metrized line bundles} are called \emph{(strongly) nef adelic line bundles} in \cite{yuan2021adelic}.
	\item \emph{Strongly arithmetically nef compactified $S$-metrized divisors} are called \emph{nef adelic arithmetic divisors} in \cite{burgos2023on}.
\end{itemize}
Note that relative nefness does not play a role in \cite{yuan2021adelic}. We use the notion of \emph{arithmetically nef} instead of just \emph{nef} to stress the difference to the notion \emph{relatively nef} which is crucial for extending the arithmetic intersection numbers later.

\subsection{Main results} \label{subsec: main results}

We fix a proper adelic curve $S=(K,\Omega,\mathcal{A},\nu)$ as above. 
Using the notions introduced in \ref{subsec: Notation and terminology}, we can state our first extension result for arithmetic intersection numbers.

\begin{theointro} \label{theointro: first extension}
	For any algebraic variety $U$ over $K$, for any $\overline{D_0}, \dots, \overline{D_k} \in \widehat{\Div}_{S,\Q}(U)_{\arint}$ and any $k$-dimensional cycle $Z$ of $U$, there is a unique arithmetic intersection number $(\overline{D_0} \cdots \overline{D_k} \mid Z)_S \in \R$ with the following properties:
\begin{enumerate}
	\item \label{intro: global intersection number of compactified line bundles} The number $(\overline{D_0} \cdots \overline{D_k} \mid Z)_S \in \R$ depends only on the isometry classes of the underlying $S$-metrized $\Q$-line bundles $\overline{L_j}=(\mathcal O_U(D_j), \metr_j)$, $j=0,\dots, k$, so we set
	$$(\overline{L_0} \cdots \overline{L_k} \mid Z)_S \coloneqq (\overline{D_0} \cdots \overline{D_k} \mid Z)_S.$$
	\item \label{intro: global intersection number multilinear symmetric} The pairing $(\overline{L_0} \cdots \overline{L_k} \mid Z)_S \in \R$ is multilinear and symmetric in $\overline{L_0}, \dots, \overline{L_k}$ and linear in $Z$. 
	\item \label{intro: global intersection number on proper varieities} If $U=X$ is proper, then $(\overline{L_0}\cdots \overline{L_k}\mid Z)_S$ agrees with the arithmetic intersection numbers introduced by Chen and Moriwaki. 
	\item \label{intro: global intersection number factorial} If $\varphi \colon U' \to U$ is a morphism of algebraic varieties over $K$ and if $Z'$ is a $k$-dimensional cycle on $X'$, then the projection formula holds:
	$$ (\varphi^*\overline{L_0} \cdots \varphi^*\overline{L_k} \mid Z')_S = (\overline{L_0} \cdots \overline{L_k} \mid \varphi_*Z')_S.$$
	\item \label{intro: global intersection number limits} 
	The arithmetic intersection  numbers $(\overline{D_0} \cdots \overline{D_k} \mid Z)_S$ are continuous in  $\overline{D_0}, \dots, \overline{D_k} \in \widehat{\Div}_{S,\Q}(U)_{\arsnef}$, with respect to any boundary topology.
		\item \label{intro: global intersection number field extension} 
	The arithmetic intersection numbers $(\overline{D_0} \cdots \overline{D_k} \mid Z)_S$ are invariant under base change to an algebraic extension $K'/K$.
\end{enumerate}
\end{theointro}

This will be proved in Theorem \ref{global intersection number on algebraic varieties}. The idea is that the semipositive cone $N_{S,\Q}(U)$ in $\widehat\Div_{S,\Q}(U)_\CM$ gives rise to an abstract divisorial space $(M_{S,\Q}(U),N_{S,\Q}(U))$, but it is better to use the cone $N'_{S,\Q}(U)$ of arithmetically nef $S$-metrized divisors in $\widehat\Div_{S,\Q}(U)_\CM$ as then the arithmetic intersection numbers of Chen and Moriwaki give a $(d+1)$-intersection map on the induced abstract divisorial space $(M_{S,\Q}'(U),N'_{S,\Q}(U))$ where $d \coloneqq \dim(U)$. Then the crucial extension result from Section \ref{section: abstract divisorial spaces} yields the theorem using that $\widehat{\Div}_{S,\Q}(U)_{\arint}$ is defined as a direct limit of abstract divisorial spaces obtained from $(M_{S,\Q}'(U),N'_{S,\Q}(U))$ by completing with respect to boundary topologies. 

For $\overline D \in \widehat{\Div}_{S,\Q}(U)_{\arint}$, the \emph{height} $h_{\overline L}(U)$ of $U$ with respect to $\overline L = \OO_U(\overline D)$ is defined by 
$$h_{\overline L}(U) \coloneqq (\overline D^{d+1}\mid U)_S.$$

However, Burgos and Kramer \cite{burgos2023on} have shown that we can extend the relative energy to certain relatively nef compactified metrics  in case of an archimedean place. We will show below that this holds also in the non-archimedean case. The relative energy is a good replacement for the lacking local heights which can be used to extend heights in the global case in some situations. 

Let $K$ be any field endowed with a complete absolute value $\val_v$. Using the notation from the local setting introduced in \S \ref{subsec: intro abstract divisorial spaces}, we consider  $(D,g),(D,h) \in \widehat{\Div}_\Q(U)_\nef$ with the same underlying compactified geometric divisor. This is stronger than just assuming that the $\Q$-Cartier divisors on $U$ agree, it means that they have the same image in $\widetilde{\Div}_{\Q}(U)_\nef$. We assume that \emph{$h$ is more singular than $g$} which means $h \leq g +C$ for some constant $C \in \R$. Then we define the \emph{relative energy} by 
$$E(g,h) \coloneqq \sum_{j=0}^d \int_{U^\an} (h-g) \, c_1(\OO_U(D), \metr_g)^{j}\wedge c_1(\OO_U(D), \metr_h)^{d-j} \in \R \cup \{-\infty\}$$
where $\metr_g, \metr_h$ are the metrics of $\OO_U(D)$ corresponding to the Green functions $g$ and $h$, respectively. Here, we use the Monge-Amp\`ere measures on $U^\an$ given by complex pluri-potential theory in the archimedean case and by the theory of real $(p,q)$ forms introduced by Chambert-Loir and Ducros \cite{chambert2012formes} in the non-archimedean case, see \S \ref{subsection: mixed MA measures} for details. There is also a mixed version of the relative energy which we study in Section \ref{section: mixed relative energy in the local case}. 

Let us come back to the global setting with a proper adelic curve $S=(K,\Omega,\mathcal{A},\nu)$ as above. We consider $(D,g),(D,h)  \in \widehat{\Div}_{S,\Q}(U)_\relnef$  with the same underlying  compactified geometric divisor in $\widetilde{\Div}_\Q(U)_\cpt$. We assume that \emph{$h$ is more singular than $g$} which means that $h_\omega$ is more singular than $g_\omega$ for all $\omega \in \Omega$. Since the relative energy $E(g_\omega,h_\omega)$ is a measurable function in $\omega \in \Omega$,  we define the \emph{relative energy} by
$$E(g,h)= \int_{\Omega} E(g_\omega,h_\omega) \, \nu(d\omega) \in \R \cup \{-\infty\}.$$
The relative energy depends only on the underlying relatively nef compactified $S$-metrics $\metr,  \metr'$ of the $\Q$-line bundle $L=\OO_U(D)$, so we set $E(\metr,\metr') \coloneqq E(g,h)$. There is also a mixed version of the relative energy in the global setting which we study in Section \ref{section: mixed relative energy in the global case}.

Using this, we can extend the arithmetic intersection numbers to all compactified geometric line bundles $L$ of $U$ and all relatively nef compactified $S$-metrics of $L$ if $L$ has also an arithmetically nef compactified $S$-metric, see Theorem  \ref{extension of global intersection number on algebraic varieties}. These intersection numbers might be equal to $-\infty$. For simplicity, we stick below to the case of \emph{heights} in the unmixed case. 

\begin{theointro} \label{B:extension of global intersection number on algebraic varieties}
	For every compactified  geometric $\Q$-line bundle $L$ of $U$ which has an arithmetically nef compactified $S$-metric, and for every relatively nef compactified $S$-metric $\metr$ of $L$, there is a unique $h_{(L,\metr)}(U) \in \R \cup \{-\infty\}$ satisfying the following properties:
	\begin{enumerate}
		\item \label{intro main3}
	If $\metr$ is an arithmetically nef compactified $S$-metric, then $h_{(L,\metr)}(U)$ is the height introduced above based on the arithmetic intersection numbers from Theorem \ref{theointro: first extension}.

		\item \label{intro main4}
		If $\metr,\metr'$ are relatively nef compactified $S$-metrics of $L$, then 
		$$h_{(L,\metr')}(U)=h_{(L,\metr)}(U)+ E(\metr,\metr').$$

		\item \label{intro main6}
		If $\varphi \colon U' \to U$ is a morphism of algebraic varieties over $K$, then
		$$h_{(\varphi^*L,\varphi^*\metr)}(U')=\deg(\varphi) h_{(L,\metr)}(U),$$
		where $\deg(\varphi)\coloneqq [K(U'):K(U)]$ if $\varphi$ is dominant and $\deg(\varphi)\coloneqq 0$ otherwise.
		
				\item \label{intro main5}
		We have invariance of $h_{(L,\metr)}(U)$  under base change to an algebraic field extension $K'/K$.
	\end{enumerate}
\end{theointro}
This generalizes the result of Burgos and Kramer \cite[Theorem 4.4]{burgos2023on}, first from the number field case to the setting of adelic curves in the sense of Chen and Moriwaki, and second allowing singular metrics also at non-archimedean places. In Theorem \ref{extension of global intersection number on algebraic varieties}, 
it is important that the metrics $\metr, \metr'$ and the arithmetically nef reference metric have the same underlying compactified metrized line bundle which is stronger than assuming that they are metrics on the same line bundle over $U$.

\subsection*{Acknowledgements}

We thank Clara Otte, Debam Biswas and Klaus K\"unnemann for their comments on earlier draft of this paper.
We are very grateful to Jose Burgos Gil and J\"urg Kramer for sharing a preliminary draft of \cite{burgos2023on} and many helpful discussions around the topic of singular metrics.  We thank also Ruoyi Guo for sending us a preliminary draft of \cite{guo2025an}. 
{We are very grateful to an anonymous referee for the careful reading and the precious comments.}

%----------------------------------------------------------------------------------------
%	SECTION 2
%----------------------------------------------------------------------------------------

\section{Abstract divisorial spaces and completions} \label{section: abstract divisorial spaces}

In the following, we denote by $\K$ either  $\Q$ or $\R$. {In fact, any field $\K$ with $\Q \subset \K \subset \R$ would be fine, but the two are the most important for applications.}

\subsection{Abstract divisorial spaces}

{Recall that an \emph{ordered abelian group} is an abelian group $G$ endowed with a partial order  "$\geq$" such that for any $x,y,z\in G$ we have that $x \geq y$ implies $x+z\geq y+z$. It is obvious that the partial order is determined by the submonoid  $G_{\geq 0}\coloneq\{x\in G\mid x\geq 0\}$. We will also use $G_{>0}\coloneq G_{\geq0}\setminus\{0\}$.}

{An \emph{ordered $\K$-vector space} is a $\K$-vector space $M$ equipped with a partial order "$\geq$" such that the underlying abelian group is an ordered abelian group and such that for any $x,y \in M$ and $r \in \K_{\geq 0}$, we have that $x\geq y$ implies $rx\geq ry$.}

{A subset $C$ of a $\K$-vector space $M$ is called a \emph{cone} if  $C+C \subset C$ and $\K_{\geq0}\cdot C\subset C$. A cone $C$ is called \emph{pointed} if $C \cap (-C)=\{0\}$. It is easy to see that if $C$ is a pointed cone in $M$, then there is a unique structure on $M$ as an ordered $\K$-vector space such that $C=M_{\geq 0}$. Conversely, in an ordered  $\K$-vector space $M$, the set $M_{\geq 0}$ is a pointed cone.}

\begin{definition}
	An \emph{abstract divisorial space} is a pairing $(M,N)$, where
	\begin{enumeratea}
		\item $M$ is a  ordered $\K$-vector space;
		\item $N$ is a cone in $M$ such that $M=N-N$. 
	\end{enumeratea}
\end{definition}

\begin{ex} \label{two examples for divisorial spaces}
	\begin{enumerate1}
		\item  The field of real numbers $\R$ is naturally an ordered $\R$-vector space and then $(\R,\R_{\geq 0})$ is an abstract divisorial space.
		\item The pair $(\Q^2, \Q^2_{\geq 0})$ is an abstract divisorial space with partial order defined by the pointed cone $\Q^2_{>0}\cup \{0\}$.
	\end{enumerate1}
\end{ex}

\begin{definition} \label{def: intersection map}
	An \emph{$(n+1)$-intersection map} {$h\colon (M,N)^{n+1}\to(M',N')$} between abstract divisorial spaces is a multilinear symmetric map $h\colon M^{n+1}\rightarrow M', \ \ (x_0,\dots, x_n)\mapsto x_0\cdots x_n$ satisfying the following axioms:
	\begin{itemize}
		\item[(\rm{NEF})] $h(N^{n+1})\subset N'$;
		\item[(\rm{EFF})] $h(M_{\geq 0}\times N^n)\subset M'_{\geq0}$;
	\end{itemize}
	We say that $h$ is \emph{positively non-degenerate} if, moreover,
	\begin{enumerate}
		\item [(\rm{AMP})] for any $x\in M_{>0}$, there is $a\in N$ such that $x a^n\in M_{>0}'$.
	\end{enumerate}
	Note that the notation with the products is just formal having the applications towards arithmetic intersection numbers in mind. 
\end{definition}

\begin{definition} \label{def: morphism of ADS}
	A \emph{morphism} between abstract divisorial spaces $(M,N)$ and $(M',N')$ over $\K$ is a linear map $\varphi\colon M \to M'$ over $\K$ which respects the order and with $\varphi(N) \subset N'$. In other words, a morphism is just  a $1$-intersection map. The abstract divisorial spaces with morphisms between them form a category, denoted by $\mathcal{ADS}_\K$.

	Similarly, a \emph{morphism} between $(n+1)$-intersection maps $h_1\colon (M_1,N_1)^{n+1}\rightarrow (M_1',N_1')$, $h_2\colon (M_2,N_2)^{n+1}\rightarrow (M_2',N_2')$ is a pair $(\varphi, \varphi')$ with $\varphi\colon  M_1\rightarrow M_2$ and $\varphi'\colon  M_1'\rightarrow M_2'$ morphisms in $\mathcal{ADS}_\K$  such that the following diagram commutes
	\[\xymatrix{M_1^{n+1}\ar[r]^-{h_1}\ar[d]_{\varphi^{n+1}}& M_1'\ar[d]^{\varphi'}\\
		M_2^{n+1}\ar[r]^-{h_2}& M_2'}.\] 
	The category of $(n+1)$-intersection maps is denoted by $\mathrm{Int}_\K^{n+1}$. 
\end{definition}

\begin{remark} \label{simultaneous element in AMP}
	If an $(n+1)$-intersection map $h\colon  (M,N)^{n+1} \to (M',N')$ is positively non-degenerate, then for  $x_1,\dots, x_r\in M_{>0}$, there is $a\in N$ such that $x_1a^n, \dots, x_ra^n\in M_{>0}'$.
	
	\begin{proof}
		Let $a_i$ be the element from axiom (\rm{AMP}) working for $x_i$, then $a=a_1+\dots+a_r$ works for $x_1,\dots,x_r$.
	\end{proof}
\end{remark}

\begin{remark} \label{absolute intersection maps}
	Let $(M,N)$ be an abstract divisorial space over $\Q$ and let $h\colon  M^{n+1}\rightarrow \R, \ \ (x_0,\dots, x_n)\mapsto x_0\cdots x_n$ be a multilinear symmetric map satisfying		
	$h(N^{n+1})\subset \R_{\geq 0}$ and $h(M_{\geq 0}\times N^n)\subset \R_{\geq 0}$. We call such a map an \emph{absolute $(n+1)$-intersection map}. It is a notion important for the applications to heights in diophantine geometry. Note that $h \in \mathrm{Int}_\Q^n$ if we view the divisorial space $(\R,\R_{\geq 0})$ from \cref{two examples for divisorial spaces} as an object of $\mathcal{ADS}_\Q$. 
\end{remark}

\begin{art} \label{base extension from integers to rationals}
	Let $M$ be an ordered abelian group and let $N$ be a submonoid of $M$. We define the \emph{divisible hull of $N$ in $M$} by $\sqrt{N}\coloneqq \{x \in M \mid \exists m \in \N_{>0}, \, mx \in N\}$. We assume that $M=\sqrt{N}-\sqrt{N}$. Let $M_\Q=M\otimes_\Z \Q$, then the natural map $\iota\colon M \to M_\Q$ has the subgroup of torsion elements as a kernel. There is a natural partial order on $M_\Q$ induced by the cone generated by $\iota(M_{\geq 0})$. We define $N_\Q$ as the cone in $M_\Q$ generated by $\iota(N)$. Then $(M_\Q,N_\Q)$ is an abstract divisorial space. 
			
	Let $M'$ be another abelian group with a submonoid $N'$ satisfying $M'=\sqrt{N'}-\sqrt{N'}$. We consider a multilinear symmetric map $h\colon M^{n+1} \to M'$ satisfying the axioms $({\rm NEF})$ and $({\rm EFF})$. By multilinearity, the map $h$ induces a unique multilinear map $h_\Q\colon M_\Q^{n+1} \to M_\Q'$. It is obvious that $h_\Q$ is an $(n+1)$-intersection map between abstract divisorial spaces over $\Q$. If for any non-torsion element $x \in M_{\geq 0}$, there is $a \in N$ such that $a x^n$ is non-torsion, then $h_\Q$ is positively non-degenerate.
\end{art}

\begin{art} \label{base extension from rational to reals}
	Let $(M,N)$ be an abstract divisorial space over $\Q$. We would like to define a natural base change $(M_\R,N_\R)$ as an abstract divisorial space. The natural candidates are $M_\R  \coloneq M \otimes_\Q \R$ and $N_\R$ the cone generated by $N$ in $M_\R$. The partial order should be induced by the cone $(M_\R)_{\geq0}$ in $M_\R$ generated by $M_{\geq 0}$. The problem is that the cone $(M_\R)_{\geq0}$ has not to be pointed which means that we do not necessarily have $(M_\R)_{\geq 0} \cap (-(M_\R)_{\geq 0})=\{0\}$ and hence this does not define a partial order by lack of antisymmetry. 
			
	To have such a natural base change, we assume that $(M,N)$ has an $(n+1)$-intersection map $h \colon (M,N)^{n+1}\to (\R,\R_{\geq 0})$ which is positively non-degenerate. Then we claim that $(M_\R)_{\geq 0}$ induces a partial order on $M_\R$ and that $(M_\R,N_\R)$ is an abstract divisorial space over $\R$. As we have seen above, the crucial point is to show that  $$(M_\R)_{\geq 0} \cap (-(M_\R)_{\geq 0})=\{0\},$$ so let $x=\sum_{i=1}^r\lambda_ix_i=-\sum_{j=1}^s\mu_jy_j\in (M_\R)_{\geq 0} \cap (-(M_\R)_{\geq 0})$ 
	with $x_i, y_j\in M_{>0}$ and $\lambda_i, \mu_j \in \R_{>0}$. By positive non-degeneracy of $h$ and Remark \ref{simultaneous element in AMP}, there is $a\in M_{>0}$ such that $x_1a^n,\dots, x_ra^n, y_1a^n,\dots, y_sa^n\in \R_{>0}$, so
	\[0=xa^n-xa^n=\sum\limits_{i=1}^r\lambda_ix_ia^n+\sum\limits_{j=1}^s\mu_jy_ja^n.\]
	This implies that $r=s=0$, so $x=0$.
			
	By multilinearity, the $(n+1)$-intersection map $h$ extends uniquely to an $(n+1)$-intersection map $h_\R \colon (M_\R,N_\R)^{n+1} \to (\R,\R_{\geq0})$ and using Remark \ref{simultaneous element in AMP} similarly as above, it is clear that $h_\R$ is positively non-degenerate. 
			
	If $h'$ is any $(n'+1)$-intersection map from $(M,N)$ to an abstract divisorial space $(M',N')$ over $\R$, then by multilinearity, there is a unique extension to an $(n'+1)$-intersection map $h_\R'$ from $M_\R$ to $M'$.
\end{art}

\begin{ex} \label{geometric intersection numbers}
	Let $X$ be a projective variety of dimension $n$ over a field $k$. We set $M$ as the group of all Cartier divisors on $X$. Then $M$ is an ordered abelian group using $M_{\geq 0}$ as the submonoid of effective Cartier divisors. We denote by $N$ the submonoid of semiample Cartier divisors. Then  $M=N-N$ and hence $(M_\Q,N_\Q)$ is an abstract divisorial space. 
			
	The multi-degree is an $n$-intersection map $d\colon M^{n} \to \R$ given by the geometric intersection numbers $d(D_1,\dots,D_n)=D_1 \cdots D_n$.

	Instead of $N$, we could use the submonoid $\overline{N}$ of $M$ given by the nef divisors. Then everything works as above and we have
	$N \subset \overline{N}$ and $N_\Q   \subset \overline{N}_\Q$. 
	The multi-degree is still an $n$-intersection map.
\end{ex}
		
\begin{ex} \label{geometric with line bundles}
	We assume now that $X$ is a 
	projective variety over a field $k$. Then we set $M=\Pic(X)$ as the Picard group. Every line bundle has a regular meromorphic section and hence we may view $\Pic(X)$ as a quotient of the group of Cartier divisors considered in Example \ref{geometric intersection numbers}. Everything goes as above, the only problem is to show that the Picard group $\Pic(X)$ is ordered. We have to show that $M_{\geq 0} \cap (-M_{\geq 0})=\{0\}$. An element on the left is induced by a line bundle which has a non-trivial global section $s$ and for which $L^{-1}$ has a non-trivial global section $t$. Then $s \otimes t$ is a non-trivial global section of $\mathcal O_X$ and hence is a constant $\alpha \in k^\times$ using that $X$ is geometrically reduced and geometrically connected \cite[Corollary~3.3.21]{liu2006algebraic}. It follows that $0 \leq \mathrm{div}(s) \leq \mathrm{div}(\alpha)=0$ and hence $s$ is a trivializing section of $L$ proving the claim. Using ample line bundles for $a$ in axiom $({\rm AMP})$, it is clear that {if $X$ is normal, then} the multi-degree is positively non-degenerate and hence induces a positively non-degenerate $n$-intersection map $d_\R\colon (M_\R,N_\R)^n \to (\R,\R_{\geq0})$ by \ref{base extension from rational to reals}. 
\end{ex}

\begin{ex} \label{example MA}
	Let $X$ be a projective complex manifold of dimension $n$. We set $M$ as the isometry classes of holomorphic line bundles endowed with smooth metrics modulo those metrized line bundles with $c_1(L,\metr)=0$. Then $N=M_{\geq 0}$ is the submonoid of classes of metrized line bundles $(L,\metr)$ with positive Chern form $c_1(L,\metr)$. This induces a partial order on the abelian group $M$. By Remark \ref{base extension from integers to rationals}, we get an abstract divisorial space $(M_\Q,N_\Q)$ over $\Q$. The multi-degree is an $n$-intersection map on $M_\Q$ using the fact that
	$$\deg_{L_1,\dots,L_n}(X)= \int_X c_1(L_1,\metr_1)\wedge \dots \wedge c_1(L_n,\metr_n).$$
	One can show that the $n$-intersection map is positively non-degenerate using basic properties of intersection products with ample classes. We conclude from \ref{base extension from rational to reals} that $(M_\R,N_\R)$ is a well-defined abstract divisorial space over $\R$. 
			
	The Monge--Amp\`ere operator $c_1(L_1,\metr_1)\wedge \dots \wedge c_1(L_n,\metr_n)$ induces an $n$-intersection map from $(M_\Q,N_\Q)^n$  to the abstract divisorial space over $\R$ of Radon measures $(M',N')$ on $X$. For the latter, we use $M_{\geq 0}'=N'$ the subcone of positive Radon measures. The Monge--Amp\`ere operator extends to an $n$-intersection map $(M_\R,N_\R)^n \to (M',N')$ as in Example \ref{base extension from rational to reals}. 
\end{ex}

\begin{lemma} \label{direct limits}
	The categories $\mathcal{ADS}_\K$ and $\mathrm{Int}_\K^{n+1}$ admit direct limits {and products of families}.
\end{lemma}
\begin{proof}
	Let $J$ be a directed set and let $(M_j, N_j) \in \mathcal{ADS}_\K$ for any $j\in J$ with given morphism $(M_i, N_i) \to (M_j, N_j)$  for $i \leq j$ satisfying the transitivity rule.   Let $M\coloneq \varinjlim_{j \in J}M_j$ and let  $N\coloneq \varinjlim_{j\in J}N_j$ in the category of $\K$-vector spaces. The partial order on $M$ is induced by $M_{\geq 0}\coloneq  \varinjlim_{j\in J}M_{j,\geq 0}$. Then $(M,N)$ is an abstract divisorial space and it is the desired direct limit. 
			
	Let $(M_i,N_i)_{i \in I}$ be a family of abstract divisorial spaces over $\K$. We define $M \coloneqq \prod_{i \in I} M_i$ as a $\K$-vector space and the cone $N \coloneqq \prod_{i \in I} N_i$. Using the cone $M_{\geq 0}\coloneqq \prod_{i \in I} (M_i)_{\geq 0}$ to define $M$ as an ordered $\K$-vector space, we easily deduce that $(M,N)$ is the product of the family $(M_i,N_i)_{i \in I}$ in  $\mathcal{ADS}_\K$.
			
	Using these two constructions, one readily proves the corresponding claims for  $\mathrm{Int}_\K^{n+1}$.
\end{proof}

\begin{art}[Induced divisorial subspace] \label{induced divisorial subspace}
	Let $(M,N)$ be an abstract divisorial space over $\K$ and let $E$ be a $\K$-subspace of $M$.	Then the \emph{induced divisorial subspace $(M_E,N_E)$ on $E$} is defined by the cone $N_E \coloneqq N \cap E$ in $E$ and by the $\K$-subspace $M_E \coloneqq N_E-N_E$ ordered by the induced partial order from $M$. Obviously, it is an abstract divisorial space over $\K$.
\end{art}
				
\begin{art}[Closure of the cone $N$] \label{closure for finite subspaces}
	Let $(M,N)$ be an abstract divisorial space over $\K$. The goal is to replace $N$ by some sort of closure in $M$.

	Let $E$ be any finite-dimensional subspace of $M$ over $\K$ and let $(M_E,N_E)$ be the induced divisorial subspace on $E$ from \ref{induced divisorial subspace}.  Then we have 
	$$(M,N) =  \varinjlim_{E}(M_E,N_E),$$
	with $E$ ranging over all such subspaces, using the direct limits from \cref{direct limits}. Let $\overline{N_E}$ be the closure of $N_E$ inside $E$ with respect to the canonical euclidean topology on the finite dimensional $\K$-vector space $E$. Then $\overline{N_E}$ is a cone in $M_E$ containing $N_E$ and hence $(M_E,\overline{N_E})$ is an abstract divisorial space over $\K$. We define then 
	$$\overline{N} = \varinjlim_{E}\overline{N_E}$$
	with $E$ ranging over all such subspaces. Then $\overline{N}$ is a cone in $M$ containing $N$, hence $(M,\overline{N})$ is an abstract divisorial space over $\K$. We call $\overline N$ the \emph{closure} of $N$ in $M$.

	If $h$ is an $(n+1)$-intersection map from $(M,N)$ to the abstract divisorial space $(M',N')$, then $h$ induces an $(n+1)$-intersection map from $(M,\overline N)$ to $(M',\overline{N'})$ using continuity on finite-dimensional subspaces. Note that if $(M',N')=(\R,\R_{\geq 0})$, then $\overline{N'}=N'$.
\end{art}
				
\begin{remark} \label{closure of semiample is nef}
	If $M$ is the  Picard group on a projective variety $X$ over $k$ and $N$ is the set of ample (or more generally the semiample) classes in $\Pic(X)$, then $\overline{N_\Q}$ is equal to the nef cone $\overline{N}_\Q$ from \cref{geometric with line bundles}.
\end{remark}
				
\begin{remark} \label{finite subspace topology}
	More generally, if $V$ is any $\K$-vector space, then we have the \emph{finite subspace topology} on $V$ which is the finest topology on $V$ such that the inclusions $E \to V$ are continuous for all finite dimensional $\K$-subspaces $E$ of $V$ endowed with the canonical euclidean topology. If $V=M$ for an abstract divisorial space $(M,N)$ over $\K$, then the closure $\overline N$ of the cone $N$ in $M$ from \ref{closure for finite subspaces} is indeed the closure of $N$ with respect to the finite subspace topology.
\end{remark}
				
\begin{lemma}\label{lemma:closure in finite subspace topology}
	Let $(M,N)$ be an abstract divisorial space over $\K$, and $\overline{N}$ the closure of $N$ in $M$ with respect to the finite subspace topology. Then $x\in \overline{N}$ if and only if there is $y\in N$ such that $x+\frac{1}{n}y\in N$ for any $n\in\N_{\geq 1}$.
\end{lemma}
\begin{proof}
	Let $x\in \overline{N}$, i.e. there is a finite-dimensional subspace $E$ such that $x\in \overline{E\cap N}$. Let $e_1,\dots, e_m\in E\cap N$ be a basis of $E$, and $\{y_j\}_{j{\geq 1}}\subset E\cap N$ a sequence converging to $x$. Write $x=a_1e_1+\cdots+a_me_m$ and $y_{j}=a_{j1}e_1+\cdots+a_{jm}e_m$ with $a_i, a_{ji}\in \K$, then $\lim_{j\to \infty}a_{ji}=a_i$ for any $i$. Set $y\coloneqq e_1+\cdots+e_m\in N$. For any $n\in\N_{\geq 1}$, there is $j$ large enough such that that $a_{i}+1/n \geq a_{ji}$, i.e. $x+1/n\cdot y= x_{j}+ z_j$ with $z_j\in \sum_{i=1}^m\K_{\geq 0}e_{i}\subset N$. Hence $x+\frac{1}{n}y\in N$. Conversely, for $x\in M$, if there is $y\in N$ such that $x+1/n\cdot y$ for any $n\in\N_{\geq1}$, since $x+1/n\cdot y\to x, n\to \infty$ with respect to finite subspace topology, then $x\in \overline N$.
\end{proof}

\subsection{The $b$-topology and completion}	\label{subsection: b-topology and completion}

\begin{definition} \label{def: b-metric}
	Let $(M,N)$ be an abstract divisorial space over $\K$ and {$b \in M_{\geq 0}$}. We have a pseudo-metric\footnote{Recall that a pseudo-metric $d$ satisfies the same axioms as a metric except that $d(x,y)=0$ is allowed for $x \neq y$. As a metric, it induces a uniform structure and hence we get a completion.}  on $M$ defined as follows: for any $x,y\in M$,
	\[d_b(x,y)\coloneq \min\{\inf\{\delta\in \K_{\geq 0} \mid -\delta b\leq x-y\leq \delta b\},1\}.\]
	The topology defined by $d_b(\cdot,\cdot)$ is called the \emph{$b$-topology} on $M$. 
\end{definition}
\begin{remark} \label{remarks pseudo-distance}
	\begin{enumerate1}
		\item For fixed $y \in M$ and for $0<\delta <\epsilon \leq 1$ in $\K$, we have 
		$$\{x \in M \mid d_b(x,y) \leq \delta \} \subset \{x \in M \mid -\epsilon b \leq x-y \leq \epsilon b \} \subset\{x \in M \mid d_b(x,y) \leq \epsilon \} $$
		and hence the sets $\{x \in M \mid -\epsilon b \leq x-y \leq \epsilon b \}_{\epsilon \in \K_{>0}}$ form a basis of neighborhoods for the $b$-topology.
		
		\item Notice that $M$ is not a topological vector space with respect to the $b$-topology, but it is a topological group, i.e. the scalar multiplication $\K \times M \to M$ might  not be continuous. However, for fixed $r \in \K$, the map $M \to M, \, x \mapsto rx$ is obviously continuous.
		
		\item The pseudo-metric is not necessarily a metric (equivalently, $M$ is not necessarily Hausdorff), i.e. the linear subspace $M_0^b\coloneq \{x\in M\mid d_b(x,0)=0\}$ is not $0$ in general, see the example below. 
	\end{enumerate1}
\end{remark}
				
\begin{ex}	
	\begin{enumerate1}
		\item Let $(M,N)=(\Q^2, \Q^2_{\geq 0})$ with partial order defined by the pointed cone $\Q^2_{>0}\cup \{0\}$, and $b=(1,1)$. 
			Then the $b$-topology is the Euclidean topology on $\Q^2$;
			
		\item Let $(M,N)=(\Q^2, \Q^2_{\geq 0})$ with partial order defined by the pointed cone $(\Q_{>0}\times\Q)\cup (\{0\}\times\Q_{\geq 0})$, and $b=(1,0)$. 
		Then the $b$-topology is not Hausdorff. Indeed, we have $d((0,1),(0,0))=0$.
	\end{enumerate1}
\end{ex}

\begin{remark} \label{continuity of morphisms}
	Let $\varphi\colon (M,N) \to (M',N')$ be a morphism of abstract divisorial spaces over $\K$, let $b \in M_{\geq 0}$ and let $b' \in M'_{\geq 0}$ with $\varphi(b)\leq b'$. Then $\varphi$ is continuous with respect  to the $b$-topology on $M$ and with respect to the $b'$-topology on $M'$. This crucial fact is obvious from the definitions.
\end{remark}

In the next proposition, we will define the \emph{completion} $(\widehat{M}^b,\widehat{N}^b)$ of an abstract divisorial space $(M,N)$ over $\K$ for a fixed {$b \in M_{\geq 0}$}. We will use the completion $\widehat{M}^{d_b}$ of $M$ with respect to the pseudo-metric $d_b$. Obviously, it is a vector space over $\K$. Note that the canonical map $M \to \widehat{M}^{d_b}$ is not necessarily  injective, the kernel is the space $M_0^b$ defined in Remark \ref{remarks pseudo-distance}. Similarly, we denote by $\widehat{N}^{d_b}$ the completion of $N$ with respect to $d_b$. Then $\widehat{N}^{d_b}$ is a closed cone of $\widehat{M}^{d_b}$. Similarly, we define $\widehat{M}_{\geq 0}^{d_b}$ as the completion of $M_{\geq 0}$ which is again a subset of $\widehat{M}^{d_b}$. In the next lemma, we will define the completion $(\widehat{M}^b,\widehat{N}^b)$ as an abstract divisorial space over $\K$ and we will characterize it by a universal property.
							
\begin{prop} \label{b-completion}
	We use the above assumptions and notation. Let $\widehat{N}^b\coloneqq \widehat{N}^{d_b}$ and  let us endow  $\widehat{M}^b\coloneqq \widehat{N}^b-\widehat{N}^b\subset\widehat{M}^{d_b}$ with the induced topology. 
	Then the cone $\widehat{N}^b$ is Hausdorff complete and the cone $\widehat{M}_{\geq 0}^b\coloneqq \widehat{M}_{\geq 0}^{d_b}\cap \widehat{M}^b$ is closed in $\widehat{M}^b$.  
	Moreover, $(\widehat{M}^b,\widehat{N}^b)$ is an abstract divisorial space over $\K$ with partial order defined by   $\widehat{M}_{\geq 0}^b$. The natural map $(M,N)\rightarrow(\widehat{M}^b,\widehat{N}^b)$ is a morphism which satisfies the following universal property: 
								
	For any abstract divisorial space $(M',N')$ over $\K$ endowed with the $b'$-topology for some $b' \in M_{\geq 0}'$ and any morphism $\varphi\colon (M,N)\rightarrow (M',N')$ such that $N'$ is Hausdorff complete,  $M_{\geq 0}'$ is closed in $M'$ and $\varphi(b)\leq b'$,  there is a unique morphism $\widehat\varphi$
    such that  the following diagram commutes:
	\[\xymatrix{(M,N)\ar[rr]^-{\varphi}\ar[dr]&& (M',N')\\
	&(\widehat{M}^b,\widehat{N}^b)\ar[ur]_{\widehat{\varphi}}&}.\]
\end{prop}
\begin{proof}
By completion, it is clear that $\widehat{N}^b$ is a closed cone spanning $\widehat{M}^b$. To see that $(\widehat{M}^b, \widehat{N}^b)$ is an abstract divisorial space, it suffices to show that  the cone $\widehat{M}_{\geq 0}^{d_b}$ is pointed, see the introductory remarks to Section \ref{section: abstract divisorial spaces}.
	Let $x\in \widehat{M}^{d_b}_{\geq 0}\cap (-\widehat{M}^{d_b}_{\geq 0})$ be represented by Cauchy sequences $(x_n)_{n\geq 1}\subset {M}_{\geq 0}$ and $(y_n)_{n\geq 1}\subset -{M}_{\geq 0}$. Then the Cauchy sequence $(x_n-y_n)_{n\geq 1}\subset M_{\geq0}$ represents $0\in \widehat{M}_{\geq 0}^{d_b}$. Since $-y_n\geq 0$, we have $0\leq x_n\leq x_n-y_n$. This implies that $(x_n)_{n\geq 1}$ converges to $0$ and hence $x=0$ proving that the cone is pointed.
								
	The natural map $M\rightarrow \widehat{M}^{d_b}$ is continuous and maps $N$, $M$, $M_{\geq 0}$ to $\widehat{N}^b, \widehat{M}^b,\widehat{M}_{\geq 0}^b$, respectively. So it is a morphism. Since $\widehat{M}_{\geq 0}^{d_b}$ is complete, it is closed in $\widehat{M}^{d_b}$ with respect to the $b$-topology and hence  $\widehat{M}_{\geq 0}^b$ is closed in $\widehat{M}^b$.
								
	For the universal property, let $\varphi\colon (M,N)\rightarrow (M',N')$ be as in the assumptions. Using $\varphi(b)\leq b'$, the map $\varphi$ is continuous by \cref{continuity of morphisms} and hence there is a unique continuous extension $\widehat{\varphi} \colon \widehat{M}^{d_b} \to  \widehat{M'}^{d_{b'}}$. Since  $N'$ is Hausdorff complete with respect to $d_{b'}$, it is closed in $\widehat{M'}^{d_{b'}}$. By density of the image of $N$ in $\widehat{N}^b$, we conclude that $\widehat{\varphi}(\widehat{N}^b) \subset N'$. 
	It remains to show that $\widehat{\varphi}(\widehat{M}^b_{\geq 0})\subset {M}'_{\geq 0}$. This is from the fact that the image of $M_{\geq 0}$ is dense in $\widehat{M}_{\geq 0}^b$ and from  $M'_{\geq 0}$ closed in $M'$ with respect to the $b'$-topology. Uniqueness of the map $\widehat{\varphi}\colon \widehat{M}^b \to M'$ 	is clear by density of the image of $M$ in $\widehat{M}^b$ and continuity. 
\end{proof}

\begin{definition} \label{S-completion}
	Let $(M,N)$ be an abstract divisorial space and let $S$ be a directed subset of $M_{\geq 0}$. For $b \leq b'$ in $S$, Proposition \ref{b-completion} gives a canonical morphism 
	$$(\widehat{M}^b,\widehat{N}^b) \to (\widehat{M}^{b'},\widehat{N}^{b'})$$
	of abstract divisorial spaces over $\K$. 
	Then the direct limit
	$$(\widehat{M}^S, \widehat{N}^S) \coloneqq \varinjlim_{b\in S}(\widehat{M}^b,\widehat{N}^b)$$ 
	in the category of abstract divisorial spaces is called the \emph{$S$-completion of $(M,N)$}. 
								
	In the applications, we will often take $S \coloneqq M_{\geq 0}$ and define
	$$(\widehat{M},\widehat{N}) \coloneqq \varinjlim_{b\in M_{\geq 0}}(\widehat{M}^b,\widehat{N}^b).$$
\end{definition}

\begin{remark} \label{final object}
	The universal property of direct limits shows that for directed subsets $S \subset S' \subset {M_{\geq 0}}$, we have natural morphisms 
	\begin{equation} \label{direct sets and completions}
		(\widehat{M}^S, \widehat{N}^S) \rightarrow (\widehat{M}^{S'}, \widehat{N}^{S'}) \rightarrow (\widehat{M},\widehat{N}).
	\end{equation}
	If now the directed set $S$ is cofinal in {$M_{\geq 0}$}, which means that for any $b \in {M_{\geq 0}}$, there is $b' \in S$ with $b \leq b'$, then all morphisms in \eqref{direct sets and completions} are obviously isomorphisms. 
\end{remark}
						
\subsection{Pull-backs} 
							
In this subsection, we fix an $(n+1)$-intersection map $h\colon  (M,N)^{n+1}\rightarrow (M',N')$ of divisorial spaces over $\K$ and we write as usual in a formal manner
$$h(x_0,\dots,x_n)= x_0\cdots x_n.$$
Our goal is to extend $h$ to the completions considered above. This requires the following notion of pull-back which is motivated by the Deligne pairing. 
							
\begin{definition} \label{definition pull-back map}
	A \emph{pull-back for $h$ at $x_0' \in M'_{\geq 0}\cap N'$} is  an element $\varphi(x_0') \in M_{\geq 0} \cap N$ such that for any $x_1,\dots, x_n\in M$, we have that $h(\varphi(x_0'),x_1,\dots, x_n)\in  \K x_0'$. 
\end{definition}
Such a pull-back is not unique. We always have  $\varphi(x_0')=0$ as a pull-back for $h$ at $x_0'$. To have further applications, we need the following properties of pull-backs.

\begin{lemma} \label{n-intersection map associated to pull-back}
	Let $\varphi(x_0')$ be a pull-back for $h$ at a non-zero element $x_0' \in N'\cap M_{\geq 0}'$. Then for $x_1,\dots,x_n \in M$, there is a unique $i(x_1,\dots,x_n)\in \K$ such that
	\[h(\varphi(x_0'),x_1,\dots, x_n)=i(x_1,\dots, x_n)\cdot x_0'.\]
	Moreover, the induced map $i\colon (M,N)^{n}\to(\K,\K_{\geq 0})$ is an $n$-intersection map.
\end{lemma}
\begin{proof}
Existence of $i(x_1,\dots,x_n)$ is immediate from the definition of pull-back and uniqueness follows from $x_0' \neq 0$. It remains to show that $i$ is an $n$-intersection map, i.e. $i(N^n) \cup i(M_{\geq 0}\times N^{n-1})\subset\K_{\geq 0}$.     
For $x_1,\dots, x_n\in N$, the axiom $({\rm EFF})$ gives
\begin{equation} \label{consequence of EFF and i} h(\varphi(x_0'),x_1,\dots, x_n)=i(x_1,\dots, x_n)\cdot x_0'\in M_{\geq 0}'.
\end{equation}
Using that $M_{\geq 0}'\cap (-M_{\geq 0}')=\{0\}$ and $x_0'\in M'_{\geq 0}$, we deduce $i(x_1,\dots, x_n) \geq 0$. It remains to see that $i(x_1,\dots,x_n) \geq 0$ if $x_1, \dots, x_{n-1} \in N$ and $x_n \in M_{\geq 0}$. Symmetry of $h$ and axiom (EFF) yield again \eqref{consequence of EFF and i} and  as above we get $i(x_1,\dots,x_n)\geq 0$ proving the claim.
\end{proof}
The following continuity results are the main reason for introducing  pull-backs.
							
\begin{lemma} \label{continuity of intersection maps} 
Let $b \in M_{\geq 0}$ and let $b' \in M_{\geq 0}' \cap N'$. We assume that there is a pull-back $\varphi(b')$ for $h$ at $b'$ with $b\leq \varphi(b')$. Then $h$ induces a continuous map  $N^{n+1}\rightarrow N'$  when $N$ is endowed with the $b$-topology, and $N'$ is endowed with the $b'$-topology.
\end{lemma}
\begin{proof}
  We may assume that $b=\varphi(b')$ and hence $b\in M_{\geq 0} \cap N$ as the $b$-topology then gets coarser. We may assume $b' \neq 0$, otherwise {$b\leq \varphi(b')=0$ which implies that $b\in M_{\geq 0}\cap M_{\leq 0}=\{0\}$, i.e. $b=0$,  hence} the $b$- and the $b'$-topology would be discrete which makes continuity obvious. Let 
  $i\colon  (M,N)^n\rightarrow (\K,\K_{\geq 0})$ be the $n$-intersection map from Lemma \ref{n-intersection map associated to pull-back} such that 
  \[h(\varphi(b'),x_1,\dots, x_n)=i(x_1,\dots, x_n)\cdot b'\]
  for $x_1,\dots, x_n\in M$.
								
  We fix $x_0,\dots, x_n\in N$. We claim that there is $C\coloneq C(b,x_0,\dots, x_n)\in \K_{>0}$ such that for any $y_0,\dots, y_n\in N$ with $-b\leq x_k-y_k\leq  b$, we have that
  \[0 \leq i(y_0,\dots, y_{j-1},x_{j+1},\dots, x_n) \leq C\] 
  for any $j=0,\dots, n$. Indeed, 
  since $i$ is an $n$-intersection map and $b=\varphi(b')\in N$, we have\begin{align*}
  	0 \leq i(y_0,\dots, y_{j-1},x_{j+1},\dots, x_n)&\leq i(x_0+ b,\dots, x_{j-1}+ b, x_{j+1},\dots, x_n).
  \end{align*}
Hence our claim holds. We fix such $C>0$.
								
  Let $\varepsilon \in \K_{>0}$. We take $\delta\coloneq \min\left\{1,\frac{\varepsilon}{(n+1)C}\right\}$. Then for any $y_0,\dots, y_n\in N$  with $-\delta b\leq x_k-y_k\leq \delta b$, we have that
  \begin{align*}
	y_0\cdots y_n-x_0\cdots x_n&=\sum\limits_{j=0}^ny_0\cdots y_{j-1}(y_j-x_j)x_{j+1}\cdots x_n\\
	& \leq \delta \sum\limits_{j=0}^ny_0\cdots y_{j-1}bx_{j+1}\cdots x_n\\
	& = \delta \sum\limits_{j=0}^ny_0\cdots y_{j-1}\varphi(b')x_{j+1}\cdots x_n\\
	& =\delta\sum\limits_{j=0}^ni(y_0,\dots, y_{j-1},x_{j+1},\dots, x_n)b'\\
	&\leq (n+1)\delta Cb'\\
	&\leq \varepsilon b'.
  \end{align*}
  Similarly, we can show that $y_0\cdots y_n-x_0\cdots x_n\geq -\varepsilon b'$. We have seen in Remark \ref{remarks pseudo-distance} that the sets $\{y_k \in M \mid -\delta b\leq y_k-x_k\leq \delta b\}_{\delta \in \K_{>0}}$ form a basis of neighborhoods of $x_k$. This completes the proof of the proposition.
\end{proof}

The continuity holds more generally when one factor is in $M$.
\begin{prop} \label{continuity of intersection maps for effective factor}
	Let $b\in {M_{\geq 0}}$, $b'\in N'\cap M_{\geq 0}'$. Assume that there is a pull-back  $\varphi(b')$ for $h$ at $b'$ with $b\leq \varphi(b')$. Then $h\colon  M \times N^{n}\rightarrow M'$ is continuous when $M,N$ are endowed with the $b$-topologies, and $M'$ is endowed with the $b'$-topology.
\end{prop}
\begin{proof}
	Again, we may assume $b' \neq 0$ and $b=\varphi(b') \in M_{\geq 0}\cap N$. We fix $x_0 \in M$ and $x_1,\dots, x_n \in N$. Let $\varepsilon \in \K_{>0}$. For $\delta \in \K$ with $0<\delta\leq 1$, we  assume that $y_0 \in M$ and $y_1, \dots, y_n \in N$ with $-\delta b \leq x_i - y_i \leq \delta b$ for $i=0,\dots, n$. For $\delta$ sufficiently small, we want to show that 
	\begin{equation} \label{continuity claim to show}
		-\varepsilon b' \leq y_0  \cdots y_n- x_0 \cdots x_n \leq \varepsilon b' .
	\end{equation} 
	Using axiom $({\rm EFF})$, we have 
	\begin{equation} \label{consequence of EFF}
		x_0 y_1 \cdots y_n - \delta b y_1 \cdots y_n\leq y_0  \cdots y_n \leq x_0 y_1 \cdots y_n + \delta b y_1 \cdots y_n.
	\end{equation}
	Using $b \leq \varphi(b')$, the $n$-intersection map $i\colon (M,N)^n \to (\K, \K_{\geq 0})$ and the constant $C=C(b,x_1,\dots,x_n)$ from the proof of  \cref{continuity of intersection maps} in \eqref{consequence of EFF}, we obtain
	\begin{equation} \label{consequence of EFF'}
		x_0 y_1 \cdots y_n - \delta C  b'\leq y_0  \cdots y_n \leq x_0 y_1 \cdots y_n + \delta C b' 
	\end{equation} 
	Now we use $M=N-N$ to find $v_0,w_0 \in N$ such that $x_0=v_0-w_0$. By the continuity for elements in $N$ obtained in  \cref{continuity of intersection maps}, we have for $\delta$ sufficiently small that 
	\begin{equation} \label{continuity consequence 1}
		-\frac{\varepsilon}{3} b' \leq v_0 y_1 \cdots y_n- v_0 x_1\cdots x_n \leq \frac{\varepsilon}{3} b' 
	\end{equation}
	and 
	\begin{equation} \label{continuity consequence 2}
	-\frac{\varepsilon}{3} b' \leq w_0  y_1\cdots y_n- w_0 x_1\cdots x_n \leq \frac{\varepsilon}{3} b' .
	\end{equation}
	We may assume $\delta C \leq \frac{\varepsilon}{3}$. Inserting \eqref{continuity consequence 1} and \eqref{continuity consequence 2} in \eqref{consequence of EFF'}, we deduce \eqref{continuity claim to show}.
\end{proof}

\begin{definition}
	Let $S\subset  {M_{\geq 0}}$ and $S'\subset N'\cap M_{\geq 0}'$ be directed subsets. We say $(S,S')$ is \emph{admissible with respect to $h$} if for any $b\in S$, there is an element $b'\in S'$ and a pull-back  $\varphi(b')$ with $b\leq \varphi(b')$. 
\end{definition}

\begin{corollary} \label{continuous extension of intersection maps in admissible case}
	Let $h\colon (M,N)^{n+1}\rightarrow (M',N')$ be an $(n+1)$-intersection map and let $S\subset  {M_{\geq 0}}$, $S'\subset N'\cap M'_{\geq 0}$ be directed subsets such that $(S,S')$ is admissible with respect to $h$.  Then there is a unique $(n+1)$-intersection map $\widehat{h}^S\colon (\widehat{M}^S,\widehat{N}^S)^{n+1}\rightarrow (\widehat{M'}^{S'},\widehat{N'}^{S'})$ such that
	\[\xymatrix{M^{n+1}\ar[r]^h\ar[d]& M'\ar[d]\\
	(\widehat{M}^S)^{n+1}\ar[r]& \widehat{M'}^{S'}}\]
	is a commutative diagram.
\end{corollary}

\begin{proof}
	Let $b\in S$. Since $(S,S')$ admissible with respect to $h$, there is an element $b'\in S'$ and a pull-back $\varphi(b')$ with $b\leq \varphi(b')$.  \cref{continuity of intersection maps} shows that $h\colon N^{n+1}\rightarrow N'$ is continuous with respect to the $b$- and $b'$-topologies and hence has a unique continuous extension
	\[\widehat{h}^b\colon  (\widehat{N}^b)^{n+1}\rightarrow \widehat{N'}^{b'}. \]
	There is a unique extension to a multilinear (symmetric) map
	\[\widehat{h}^b\colon  (\widehat{M}^b)^{n+1}\rightarrow \widehat{M'}^{b'}. \]
	These constructions depend only on continuity of $h|_{N^{n+1}}$ and not on the particular choice of the pull-backs $\varphi(b')$ for $b' \in S'$, hence they are functorial in the pairs $(b,b')$. Using the universal property of direct limits, we get 
	a multilinear (symmetric) map $\widehat{h}^S\colon (\widehat{M}^S)^{n+1}\rightarrow \widehat{M'}^{S'}$ fitting into the commutative diagram above. 
								
	It remains to see that $\widehat{h}^S$ is an $(n+1)$-intersection map. Using that direct limit exist in the category of abstract divisorial spaces by Lemma \ref{direct limits}, it is enough to show that $\widehat{h}^b$ is an $(n+1)$-intersection map. By  \cref{continuity of intersection maps for effective factor}, the map $h\colon M \times N^n \to M'$ is continuous with respect to the $b$- and $b'$-topologies and hence has a unique continuous extension
	$$\widehat{h}_0^{d_b}\colon \widehat{M}^{d_b} \times (\widehat{N}^b)^{n}\rightarrow \widehat{M'}^{d_{b'}}.$$
	By continuity, this extension is multilinear and symmetric. Obviously, it agrees with $\widehat{h}^b$ on $(\widehat{N}^b)^{n+1}$ and hence also on $\widehat{M}^{b}\times (\widehat{N}^b)^{n}$. By continuity and using the axiom  (\rm EFF) for $h$, the extension $\widehat{h}_0^{d_b}$ maps $\widehat{M}_{\geq 0}^{d_b}\times (\widehat{N}^b)^{n}$ to $\widehat{M'}_{\geq 0}^{d_{b'}}$. We conclude that restricting $\widehat{h}_0^{d_b}$ agrees with
	$$\widehat{h}^b \colon \widehat{M}^{b} \times (\widehat{N}^b)^{n}\rightarrow \widehat{M'}^{{b'}}$$
	and maps $\widehat{M}_{\geq 0}^{b}\times \widehat{N}^b=(\widehat{M}_{\geq 0}^{d_b} \cap \widehat{M}^b)\times \widehat{N}^b$ to $\widehat{M'}_{\geq 0}^{{b'}}=\widehat{M'}_{\geq 0}^{{d_b'}} \cap \widehat{M'}^{{b'}}$. This shows that $\widehat{h}^b$ is an $(n+1)$-intersection map.
\end{proof}

\begin{corollary}\label{continuous extension of intersection maps in absolute case}
	Let $h\colon  (M,N)^{n+1}\rightarrow (\R,\R_{\geq 0})$ be an $(n+1)$-intersection map, and let $S$ be a directed subset of {$M_{\geq 0}$ such that for every $b \in S$, there is $\widetilde{b}\in N$ such that $b\leq \widetilde{b}$}. Then the pair $(S, \{1\})$ is admissible with respect to $h$. In particular, we have a {unique} $(n+1)$-intersection map $(\widehat{M}^S,\widehat{N}^S)^{n+1}\rightarrow (\R,\R_{\geq 0})$ extending $h$.
\end{corollary}
\begin{proof}
	For $b \in S$, we set $\varphi(1)\coloneqq \widetilde b \in M_{\geq 0} \cap N$. Then $\varphi(1)$ is a pull-back for $h$ at $1$. Indeed, for any $x_1,\dots, x_n\in M$, we have that $h(\varphi(1), x_1,\dots, x_n)\in \R\cdot 1$. Hence $(S,\{1\})$ is admissible with respect to $h$ and we can apply \cref{continuous extension of intersection maps in admissible case}.
\end{proof}

%----------------------------------------------------------------------------------------
%	SECTION 3
%----------------------------------------------------------------------------------------

\section{Compactified geometric divisors} 
	\label{sec: geometric theory}	
In this section, we fix a field $K$ and an algebraic variety $U$ over $K$ of dimension $d$. We extend Yuan--Zhang's construction of geometric compactified divisors from  \cite{yuan2021adelic} to this setting. We recall that a proper variety $X$ over $K$ is called a \emph{proper $K$-model} of $U$ if {$U$ comes with an open immersion into $X$ over $K$. A \emph{morphism of proper $K$-models $X,X'$ of $U$} is a morphism $\varphi\colon X \to X'$ over $K$ compatible with the open immersions of $U$.}
		
\begin{lemma} 	\label{lemma:compactifications form an inverse system}
	The isomorphism classes of proper $K$-models of $U$ form an inverse system.
\end{lemma}
\begin{proof}
	Let $X, X'$ be two proper $K$-models of $U$, then we have a rational map $f\colon {X}\dashrightarrow {X}'$ compatible with the open embeddings. Then the graph $\Gamma_f\subset {X}\times {X}'$ of $f$, defined as the closure of $(x,f(x))$ with $x$ in the set of definition of $f$, is also a proper $K$-model of $U$, and we have morphisms $\Gamma_f\rightarrow {X}$ and $\Gamma_f\rightarrow {X}'$.
\end{proof}
			
\begin{art} \label{Q-Cartier divisors}
	We denote the group of Cartier divisors on $U$ by $\Div(U)$. Then $\Div_\Q(U) \coloneqq \Div(U) \otimes_\Z \Q$ is called the  \emph{group of $\Q$-Cartier divisors}  on $U$. Then a $\Q$-Cartier divisor is   given by \emph{local equations} $f_W \in K(U)^\times\otimes_\Z \Q$ for  suitable open sets $W$ covering $U$. 
			
	The \emph{support} $|D|$ of a $\Q$-Cartier divisor $D$ is the complement of the union of all open subsets $V$ where $D$ can be written by local equations $f_V \in \OO_U(V)^\times \otimes_\Z \Q$. 
			
	A $\Q$-Cartier divisor $D\in\Div_\Q(U)$ is \emph{effective}, denoted by $D\geq 0$, if {it is a non-negative} linear combination of some effective divisors in $\Div(U)$. It is \emph{strictly effective}, denoted by $D>0$, if is effective and non-zero.
			
	For a proper variety $X$, we say a $\Q$-Cartier divisor $D\in\Div_\Q(U)$ is \emph{nef} if there is $m\in\N_{>0}$ such that $mD\in \Div(U)$ is nef, i.e.~the degree of $\OO_X(mD)|_C$ is non-negative for any closed curve $C$ on $X$. 
\end{art}
\begin{remark}
	Let $f\colon X\rightarrow X'$ be a morphism of proper $K$-models over $U$. Then the pull-back of Cartier divisors $\Div(X)\rightarrow \Div(X')$ is well-defined as $f$ is surjective, see \cite[Lemma~7.1.33]{liu2006algebraic}.
\end{remark}
		
\begin{definition} \label{geometric model divisors}
	The space of \emph{geometric} \emph{model divisors} on $U$ is defined as the limit 
	\[\Div_\Q(U)_{\mathrm{mo}}\coloneq\varinjlim_{X}\Div_\Q(X),\]
	where $X$ runs through all proper $K$-models of $U$. We say an element ${D}\in\Div_\Q(U)_{\mathrm{mo}}$ on $U$ is \emph{effective} (resp. \emph{strictly effective}), denoted by ${D}\geq0$ (resp. ${D}>0$) if there is a proper $K$-model $X$ of $U$ such that ${D}\in {\Div}_\Q(X)$ and ${D}\geq0$ (resp. ${D}>0$).
\end{definition}
		
\begin{definition} 	\label{def:geometric boundary topology}
	A {\emph{weak}  \emph{boundary divisor}} (resp. \emph{boundary divisor}) of $U$ is a pair $({X}_0,B)$ consisting of a proper $K$-model $U\hookrightarrow X_0$ and an effective divisor $B\in {\Div}_\Q({X}_0)$ such that {$|B|\subset {X}_0\setminus U$} (resp. $|B|={X}_0\setminus U$).
			
	Let $({X}_0,B)$ be a {weak} boundary divisor. The \emph{$B$-boundary topology} on $\mathrm{Div}_\Q(U)_{\mathrm{mo}}$ is defined such that a basis of neighborhoods of a divisor $D$ of the topology is given by
   \begin{align*}
		B(r,D)\coloneq\{E\in {\Div}_\Q(U)_\mo\mid -rB\leq {E}-{D}\leq rB\},\ \ r\in \Q_{>0}.
	\end{align*} 
	If $B$ is a boundary divisor, the corresponding topology is called the \emph{boundary topology}. The space of \emph{compactified} \emph{geometric} \emph{divisors} on $U$ is defined as the completion of $\mathrm{Div}_\Q(U)_{\mo}$ with respect to the boundary topology, denoted by $\widetilde{\Div}_\Q(U)_\cpt$. Every compactified geometric divisor $D$ is the limit of a Cauchy sequence of geometric model divisors $D_i$ of $U$. As the sequence $D_i|_U$ becomes constant for large $i$, we have a canonical homomorphism $$\widetilde{\Div}_\Q(U)_{\cpt}\rightarrow {\Div}_\Q(U).$$			
\end{definition}
\begin{remark} \label{remark:independent of the choice of boundary divisors}
	By Nagata's compactification theorem and blowing up, there is always a boundary divisor. Every weak boundary divisor is dominated by a boundary divisor. Let $b\coloneq (X_0, {B})$ be a boundary divisor. Although the lemma in \cite[Lemma~2.4.1]{yuan2021adelic} is formulated for quasi-projective varieties with projective models, its argument remains valid for general algebraic varieties.
	Hence the subset $\{nb\mid n\in\N\}$ is cofinal in the set of weak boundary divisors.  In particular, the boundary topology is independent of the choice of the  boundary divisor. 
\end{remark}

Next, we will associate to $U$ an abstract divisorial space where the intersection map can be defined.
		
\begin{definition} \label{def: geometric Yuan-Zhang completion}
	For a proper variety $X$ over $K$, we set
	\[N_{\gm,\Q}(X)\coloneq\{D\in\Div_\Q(X)\mid D \text{ is nef}\},\]
	\[M_{\gm,\Q}(X)\coloneq N_{\gm,\Q}(X)-N_{\gm,\Q}(X).\]
	{Here the subscript ${\gm}$ stands for \emph{geometric}.}
	Then $(M_{\gm,\Q}(X), N_{\gm,\Q}(X))$ is an abstract divisorial space  over $\Q$ ordered by the cone of effective divisors in $M_{\gm,\Q}(X)$. 
			
	For the algebraic variety $U$ over $K$, we set
	\[(M_{\gm,\Q}(U), N_{\gm,\Q}(U))\coloneq\varinjlim_X (M_{\gm,\Q}(X), N_{\gm,\Q}(X))\]
	where $X$ ranges over all proper $K$-models of $U$. Let $S$ be the  directed subset of weak boundary divisors $(X_0,B)$ in $M_{\gm,\Q}(U)_{\geq 0}$. 
	We define the \emph{Yuan-Zhang completion} as
	\[(\widehat{M}_{\gm,\Q}^{\mathrm{YZ}}(U), \widehat{N}_{\gm,\Q}^{\mathrm{YZ}}(U))\coloneq (\widehat{M}_{\gm,\Q}^S(U), \widehat{N}_{\gm,\Q}^S(U))\]
	%\[(\widehat{M}_{\gm,\Q}^{\mathrm{YZ}}(U), \widehat{N}_{\gm,\Q}^{\mathrm{YZ}}(U))\coloneq (\widehat{M}_{\gm,\Q}^S(U), \widehat{N}_{\gm,\Q}^S(U))=(\widehat{M}_{\gm,\Q}^b(U), \widehat{N}_{\gm,\Q}^b(U))\]
	using the $S$-completion of abstract divisorial spaces from \cref{S-completion}. 
\end{definition}

{Note that there are proper non-projective varieties with no non-trivial nef line bundles. Fujino and Payne \cite{fujino2005smooth} have found such an example given by a smooth toric threefold. However, if $U$ is quasi-projective, then there is always a boundary divisor $b \in M_{\gm,\Q}(U)_{\geq 0}$ by using Remark \ref{remark:independent of the choice of boundary divisors} and passing to a projective model.  By \cite[Lemma~2.4.1]{yuan2021adelic}, the subset $\N b$ of $S$ is then  cofinal in $S$ and we have
		\[(\widehat{M}_{\gm,\Q}^{\mathrm{YZ}}(U), \widehat{N}_{\gm,\Q}^{\mathrm{YZ}}(U))\coloneq (\widehat{M}_{\gm,\Q}^b(U), \widehat{N}_{\gm,\Q}^b(U)).\]}

For any algebraic variety $U$, the cone $N_{\gm,\Q}(U)$ agrees with its closure in $M_{\gm,\Q}(U)$ in the sense of \cref{closure for finite subspaces}, but $\widehat{N}_{\gm,\Q}^\YZ(U)$ does not necessarily agree with its closure in $\widehat{M}_{\gm,\Q}^\YZ(U)$.

\begin{definition} \label{strongly nef and nef compactified geometric divisors}
	We call a compactified geometric divisor $D$ on $U$ \emph{strongly nef} if $D \in \widehat{N}_{\gm,\Q}^\YZ(U)$ leading to the  cone $\widetilde{\Div}_\Q(U)_\snef \coloneqq \widehat{N}_{\gm,\Q}^\YZ(U)$.
		
	We say that a compactified geometric divisor $D$ is \emph{nef} if it is in the closure of the strongly nef cone in $\widehat{M}_{\gm,\Q}^\YZ(U)$, where the closure is with respect to the finite subspace topology in the sense of \ref{closure for finite subspaces}. The nef compactified geometric divisors form a cone and we get an abstract divisorial space over $\Q$ denoted by
	\[\left(\widetilde{\Div}_\Q(U)_{\mathrm{int}},\widetilde{\Div}_\Q(U)_{\nef}\right)\coloneq\left(\widehat{M}_{\gm,\Q}^\YZ(U),\overline{\widetilde{\Div}_\Q(U)_\snef}\right).\]
\end{definition}
		
For a proper $d$-dimensional variety $X$ over $K$, it follows as in Example \ref{geometric intersection numbers} that  intersection numbers define an $d$-intersection map on the above abstract divisorial space $(M_{\gm,\Q}(X), N_{\gm, \Q}(X))$. This can be generalized as follows.

\begin{theorem}\label{thm:algebraic interesection number}
	The intersection pairing of nef divisors on proper varieties extends 
	to a unique $d$-intersection map $\widehat{M}_{\gm,\Q}^{\mathrm{YZ}}(U)^d \longrightarrow \R$ with respect to the abstract divisorial space $\left(\widetilde{\Div}_\Q(U)_{\mathrm{int}} ,\widetilde{\Div}_\Q(U)_{\nef}\right)$. 
\end{theorem}
	
For quasi-projective varieties, this was shown by Yuan and Zhang in \cite[Proposition~4.1.1]{yuan2021adelic}. In general, it follows from the theory of abstract divisorial spaces.
	
\begin{proof} 
	{Using the projection formula, we deduce easily that we might shrink $U$ to deduce the result. So we may assume $U$ quasi-projective. Then there is a projective $K$-model $X$ of $U$. After replacing $X$ by a suitable blowing up centered in $X \setminus U$, there is an effective ample Cartier divisor on $X$ with support containing $X \setminus U$.} 
Shrinking $U$ further, we may assume that $U$ has a nef boundary divisor $b$. Then the extension follows from  \cref{continuous extension of intersection maps in absolute case}.
\end{proof}
		
We describe the abstract divisorial space $	\widetilde{\Div}_\Q(U)_{\mathrm{int}}$ in a very simple example. In higher dimensions, more complicated examples can be studied with the theory of $b$-divisors \cite[Corollary 3.11]{burgos2023on}. In \cite[Chapter 3]{peralta-thesis24}, toric compactified geometric divisors are described in general. 
		
\begin{example} \label{example: projective line}
Let $U={\mathbb G}_{{\rm m},K}$ be the multiplicative torus of rank $1$ over the field  $K$ of characteristic $0$ and let $X=\mathbb P_K^1$ be the projective line as the smooth compactification of $U$. Then $X \setminus U=\{0,\infty\}$. Let $X'$ be any proper $K$-model of $U$. Replacing $X'$ by its normalization, we may assume that $X'$ is a normal variety and hence smooth as $K$ is of characteristic $0$, but then $X'=\mathbb P_K^1$. We conclude that 
$$\Div_\Q(U)_\mo=\Div_\Q(X)=\Div_\Q(U)\oplus \Q[0] \oplus \Q[\infty].$$ In this case, we can take $B=[0]+[\infty]$ as a boundary divisor. It follows that the boundary completion of $\Div_\Q(U)_\mo=\Div_\Q(X)$ is 
$$\widetilde{\Div}_\Q(U)_{\cpt}=\widetilde{\Div}_\Q(U)_{\mathrm{int}}=\Div_\Q(U) \oplus \R[0] \oplus \R[\infty].$$
\end{example}

Next, we recall the  compactified geometric line bundles defined in \cite[\S 2.5]{yuan2021adelic}.

\begin{art} \label{definition of compactified line bundles}
	For a proper variety $X$ over $K$, we denote $\Pic_\Q(X)\coloneq \Pic(X)\otimes_\Z\Q$ whose partial order is defined analogously to that in \cref{geometric with line bundles}, and $\Pic_\Q(X)_\nef$ the subcone generated by the nef line bundles in $\Pic_\Q(X)$. Set
	\[\text{$P_{\gm,\Q}(U) = \varinjlim_{X}\Pic_\Q(X)$  and  $\quad Q_{\gm,\Q}(U) = \varinjlim_{X}\Pic_\Q(X)_\nef$}\]
	where $X$ ranges over all proper $K$-models of $U$.
	
	We fix a boundary divisor $(X_0, {B})$ of $U$. 
	For  $L\in P_{\gm,\Q}(U)$, we have the restriction $L|_U \in \Pic_\Q(U)$. For $r \in \Q_{>0}$, we define $B(r,{L}) \subset P_{\gm,\Q}(U)$ as follows. An element $L' \in P_{\gm,\Q}(U)$ is in $B(r,L)$ if and only if we have an isomorphism $\sigma\colon L|_U\simeq L'|_U$ for some representatives $L,L'$ living on a joint proper $K$-model $X$ such that the induced $\Q$-Cartier divisor $\mathrm{div}(\sigma)$ of $L'\otimes L^{-1}$ on $X$ satisfies 
	$$-rB\leq \mathrm{div}(\sigma)\leq rB.$$
	Then the \emph{boundary topology}  of $P_{\gm,\Q}(U)$ is the unique topology such that for any $L\in P_{\gm,\Q}(U)$, the sets $(B(r,L))_{r \in \Q_{>0}}$ form a basis of neighborhoods.
	The boundary topology is independent of the choice of the  boundary divisor. %\green{Similarly as in \cref{def: b-metric}, we can define a pseudo-metric $d_{B}$ which defines the boundary topology.} 
	We define $\widetilde{\Pic}_\Q(U)_\cpt$ as the completion of $P_{\gm,\Q}(U)$ with respect to the boundary topology. An element in $\widetilde{\Pic}_\Q(U)_\cpt$ is called a \emph{compactified geometric line bundle} on $U$.
\end{art}

\begin{art} \label{abstract divisorial space of line bundles}
	We say that $\widetilde{L} \in \widetilde{\Pic}_\Q(U)_\cpt$ is \emph{strongly nef} (resp. \emph{effective}) if $\widetilde{L}$ is in the closure of $Q_{\gm,\Q}(U)$ (resp. $P_{\gm,\Q}(U)_{\geq 0}$) with respect to the boundary topology. The strongly nef elements form a cone $\widetilde{\Pic}_\Q(U)_\snef$ in $\widetilde{\Pic}_\Q(U)_\cpt$. We define the subspace
$$\widetilde{\Pic}_\Q(U)_\integrable \coloneqq \widetilde{\Pic}_\Q(U)_\snef - \widetilde{\Pic}_\Q(U)_\snef$$
of $\widetilde{\Pic}_\Q(U)_\cpt$. Finally, the \emph{nef cone} $\widetilde{\Pic}_\Q(U)_\nef$ is defined as the closure of $\widetilde{\Pic}_\Q(U)_\snef$ in $\widetilde{\Pic}_\Q(U)_\integrable$ with respect to the finite subspace topology from \cref{finite subspace topology}. 
We have a surjective map 
{${\Div}_\Q(U) \longrightarrow {\Pic}_\Q(U)$
which is continuous with respect to the boundary topologies. So it can be extended to a continuous map 
$$\widetilde{\Div}_\Q(U)_\cpt \longrightarrow \widetilde{\Pic}_\Q(U)_\cpt\, , \quad D \mapsto \OO_U(D)$$ 
which maps the sets  $\widetilde{\Div}_\Q(U)_\snef$, $\widetilde{\Div}_\Q(U)_\nef$ and $\widetilde{\Div}_\Q(U)_\integrable$ onto the sets $\widetilde{\Pic}_\Q(U)_\snef$, $\widetilde{\Pic}_\Q(U)_\nef$ and $\widetilde{\Pic}_\Q(U)_\integrable$, respectively.}

For $D_1, \dots, D_d \in \widetilde{\Div}_\Q(U)_\integrable$, the geometric intersection numbers $D_1 \cdots D_d$ from \cref{thm:algebraic interesection number} do depend only on $\OO_U(D_i) \in \widetilde{\Pic}_\Q(U)_\integrable$ for $i=0,\dots, d$, so we will use the notation
$$\OO_U(D_1) \cdots \OO_U(D_d)\coloneqq D_1 \cdots D_d $$
for the geometric intersection numbers. Indeed, this is obviously true for $D_1, \dots, D_d \in N_{\gm,\Q}(U)$ and yields the general case by continuity and multilinearity.
\end{art}

%----------------------------------------------------------------------------------------
%	SECTION 4
%----------------------------------------------------------------------------------------
	
\section{Local theory} \label{section: local theory}
	
In this section, we explain the local theory of metrized line bundles considered in Arakelov theory. 
The adic generalization introduced by Yuan and Zhang \cite{yuan2021adelic} can be seen as the completion in the formalism of abstract divisorial spaces.

We fix a field $K$ which is complete with respect to a given absolute value $\val_v$, where $v$ denotes the corresponding place of $K$. If $v$ is non-archimedean, we denote by $\kcirc$ the valuation ring of $K$ and by $\ktilde$ the residue field. We also fix an algebraic variety $U$ over $K$ of dimension $d$.
	
\subsection{Metrics and Green functions}

\begin{art} \label{analytification}
	We denote by $U^\an=U_v^\an$ the Berkovich analytification of $U$ with respect to the place $v$. From \cite[Example 1.5.4]{berkovich1990spectral}, we obtain the following cases.
	\begin{enumerate}
		\item If
		$v$ is a complex place, then $U^\an$ is a classical complex analytic space.
		\item If $v$ is real place, then $U^\an$ is the quotient of a complex space over the algebraic closure $\C_v$ by complex conjugation. We will work with complex spaces over $\C_v$ and we will assume that all considered objects are invariant under complex conjugation.
		\item If $v$ is a non-archimedean place, then $U^\an$ is a classical Berkovich space \cite[\S 3.4, \S 3.5]{berkovich1990spectral}.
	\end{enumerate}
\end{art}
	
\begin{art} \label{metrics}
	Let $L$ be a line bundle on an algebraic variety $U$ over $K$. The line bundle $L$ induces an analytic line bundle $L^\an$ on $U^\an$. A \emph{metric} $\metr$ of $L$ means a metric $\metr$ on $L^\an$, a concept well-known from complex analysis which is readily transformed to the Berkovich setting in non-archimedean geometry (see \cite[2.1]{chambert2006measures} or \cite[Definition~2.1.8]{chen2020arakelov}). We call $\metr$ \emph{continuous} if it is continuous on $L^\an$. The isometry classes of {continuously} metrized line bundles on $U^\an$ form a group, denoted by $\widehat{\Pic}(U)$.
\end{art} 
	
\begin{art} \label{Green functions}
	For a Cartier divisor $D$ of $U$, a \emph{Green function} is a continuous real function $g_D$   on $(U \setminus |D|)^\an$ such that for any local equation $f$ of $D$ on an open subset $V$ of $U$, we have that  $g+\log|f|$ is a continuous function on $V^\an$.
		
	The concept of Green functions is clearly equivalent to continuous metrics. Let $s_D$ be the canonical meromorphic section of $\OO_U(D)$ with $\mathrm{div}(s_D)=D$. If $g_D$ is a Green function, then there is a unique continuous metric $\metr$ of $\OO_U(D)$ such that for all $x  \in (U \setminus |D|)^\an$, we have
	\begin{equation} \label{metric and Green functions}
		-\log \|s_D(x)\| = g_D(x).
	\end{equation}
	Conversely, every continuous metric $\metr$ of $\OO_U(D)$ defines a unique Green function by \eqref{metric and Green functions}.
\end{art}

\begin{remark} \label{Green functions for Q-divisors}
	A $\Q$-Cartier divisor $D$ of $U$ can be written by local equations $f_V \in \OO_U(V)^\times \otimes_\Z \Q$.  Using the universal property of the tensor product, we get a well-defined continuous real function $\log|f_V|$ on $(V \setminus |D|)^\an$ for such open subsets $V$. Then we define  Green functions for $D$ as in \ref{Green functions}.
	
    We denote $\widehat{\Div}(U)$ (resp. $\widehat{\Div}_\Q(U)$) the group of pairs $(D,g)$ with $D\in \Div(U)$ (resp. $D\in\Div_\Q(U)$) and $g$ a Green function for $D$. Obviously, we have that $\widehat{\Div}_\Q(U)\simeq\widehat{\Div}(U)\otimes_\Z\Q$. {We say that $(D,g)$ in $\widehat{\Div}(U)$ (resp.~in $\widehat{\Div}_\Q(U)$) is \emph{effective} if $D$ is effective and $g \geq 0$.} 
\end{remark}
	
The following lemma will be useful to prove injectivity of the restriction map of divisors with Green functions.
\begin{lemma}\label{lemma: injectivity of restriction from integrally closed}
	Let $X$ be a proper $K$-model of $U$. If $X$ is integrally closed in $U$, then an element $(D,g)\in\widehat{\Div}_{\Q}(X)$ is effective if and only if its image in $\widehat{\Div}_{\Q}(U)$ is effective. 
\end{lemma}	
\begin{proof}
It suffices to show the "if" part. Let $(D,g)\in\widehat{\Div}_{\Q}(X)$ such that its image in $\widehat{\Div}_{\Q}(U)$ is effective.  Since $g$ is continuous with $g\geq 0$ on $U^\an$ and $U^\an$ is dense in $X^\an$, we have that $g\geq0$ on $X^\an$. 
{Let $w\in X\setminus U$ be a point with closure $Z$ of codimension $1$ in $X$. Using that $X$ is integrally closed in $U$, it is shown in the proof of \cite[Lemma~2.3.6]{yuan2021adelic} that the local ring $\OO_{X,w}$ is a discrete valuation ring. Since $g\geq0$, it follows that $\ord_w(D)\geq0$ as otherwise $g<0$ in a neighborhood of $Z$ in $X$.} 
Since $D|_U\geq 0$ and $\ord_w(D)\geq0$ for any $w\in X\setminus U$ of codimension $1$, by \cite[Lemma~2.3.6]{yuan2021adelic}, we have that $D\geq0$. 
\end{proof}

\begin{art} \label{classical Arakelov theory: definitions}
	In classical local Arakelov theory, one works with line bundles $L$ on a proper variety $X$ of dimension $n$ over $K$. The metrics are assumed to be \emph{smooth} if the place $v$ is archimedean and to be model metrics in the non-archimedean case. 
		
	We first deal with the archimedean case. Note that the analytic space $X^\an$ might be singular, so we use the smooth functions and differential forms on singular analytic spaces introduced by Bloom-Herrera \cite{bloom1969de}. The important property is that for every morphism $\varphi\colon X' \to X^\an$ from an analytic manifold $X'$, the pull-back of smooth forms on $X^\an$ leads to smooth forms on $X'$ in the usual sense. The metric $\metr$ of $L$ is called \emph{semipositive} if the $\varphi^*\metr$ is has semipositive curvature form $c_1(\varphi^*L,\varphi^*\metr)$ in the usual sense.

	In the non-archimedean case, smoothness is replaced by models. We call a flat proper scheme $\Xcal$ over the valuation ring $\kcirc$ a \emph{$\kcirc$-model} of $X$ if the generic fiber of $\Xcal$ is $X$. A line bundle $\Lcal$ on $\Xcal$ is a \emph{$\kcirc$-model of $L$} if $\Lcal|_X=L$. A metric $\metr$ of $L$ is called a \emph{model metric} if there is a non-zero $m \in \N$ and a $\kcirc$-model $\Lcal$ of $L^{\otimes m}$ such that $\metr^{\otimes m}$ is induced by the $\kcirc$-model $\Lcal$, see \cite[\S 2.3]{boucksom2021non} for details. The metric is called \emph{semipositive} if the line bundle $\Lcal$ is nef. This notion is independent of the choice of  $\Lcal$. 
\end{art}
	
\begin{definition} \label{model Green function} 
	A \emph{model Green function} for a Cartier divisor $D$  {on the proper variety $X$ over $K$}  is a Green function $g_D$ for $D$ which is induced by a smooth metric $\metr$ in the archimedean case and by a model metric $\metr$ in the non-archimedean case. We call $g_D$ \emph{semipositive} if $\metr$ is semipositive.
	We  consider the abelian group {of \emph{model metrized divisors}}
	$${\widehat{\Div}(X)_{\mathrm{mo}}} \coloneqq \{(D,g_D)\mid \text{$D$ Cartier divisor on $X$, \, $g_D$ model Green function for $D$}\},$$
	and set $\widehat{\Div}_\Q(X)_{\mathrm{mo}}\coloneqq \widehat{\Div}(X)_{\mathrm{mo}}\otimes_\Z\Q\subset \widehat{\Div}_\Q(X)$.

	We call a  model {metrized} divisor $\overline{D}\in \widehat{\Div}_\Q(X)_{\mathrm{mo}}$  \emph{effective} (resp. \emph{strictly effective})  if  $D\in \Div_\Q(X)$ is effective and $g \geq 0$ (resp. $g>0$). We will use the notation $\overline{D}\geq0$ for effective model metrized  divisors $\overline D$.
\end{definition}	
	
\begin{remark} \label{projective models}
	When $X$ is projective over $K$, note that for any $\kcirc$-model $\Xcal$ of $X$, there is always a projective $\kcirc$-model $\Xcal'$ of $X$ which dominates $\Xcal$ \cite[Proposition~10.5]{gubler2003canonical}. This means that the identity on $X$ extends to a morphism $\Xcal' \to \Xcal$. We conclude that in the projective case, the model metrics can be induced by line bundles $\Lcal$ on projective models.
\end{remark}

\begin{example} \label{Fubini-Study metric}
	Let $X$ be a proper variety over $K$ and let $L$ be a line bundle on $X$ generated by global sections $s_0, \dots, s_r$. This induces a metric $\metr_{\FS}$ of $L$ called a \emph{Fubini--Study metric} as follows. If $v$ is archimedean, then it is given in a neighborhood of $x \in X^\an$ for any local section $s$ of $L$ at $x$ by
	$$\|s(x)\|_{\FS}= \frac{|s(x)|_v}{\left(\sum_{j=0}^r|s_j(x)|_v^2\right)^{1/2}}$$
	using a fixed trivialization of $L$ at $x$ to make sense of $s(x)$ and $s_j(x)$. This is a smooth semipositive metric of $L$. 
		
	If $v$ is non-archimedean, then we define similarly
	$$\|s(x)\|_{\FS}= \frac{|s(x)|_v}{\max\{|s_j(x)|_v \mid j=0,\dots, r\}}.$$
	This is the model metric obtained by pull-back with respect to the morphism $\varphi \colon X \to \mathbb P_K^r$ induced by $s_0, \dots, s_r$ from the metric on $\OO_{\mathbb{P}_K^r}(1)$ induced by the canonical model $\OO_{\mathbb{P}_{\kcirc}^r}(1)$. Since the latter is very ample, we conclude that $\metr_\FS$ is a semipositive model metric. For $\lambda = (\lambda_0, \dots, \lambda_r) \in \R^{r+1}$, we define the \emph{twisted Fubini--Study metric} $\metr_{\tFS(\lambda)}$ of $L$ by 
	$$\|s(x)\|_{\tFS(\lambda)}= \frac{|s(x)|_v}{\max\{e^{-\lambda_j}|s_j(x)|_v \mid j=0,\dots, r\}}$$
	which is not always a model metric. 
    We can use similar twists in the archimedean case.				
\end{example}

\begin{prop} \label{classical local Arakelov theory as divisorial space}
	For a proper variety $X$ over $K$, we consider the submonoid $N_\mo(X)$ of $\widehat{\Div}(X)_\mo$ given by the model divisors $(D,g_D)$ with $g_D$ semipositive and the subgroup $M_\mo(X)\coloneqq{N_\mo}(X)-{N_\mo}(X)$ of $\widehat{\Div}(X)_\mo$. 
	Then the abelian group $M_\mo(X)$ is ordered by the submonoid $M_{\mo}(X)_{\geq 0} \coloneqq \{(D,g_D) \in M_\mo(X) \mid \text{$D$ effective and $g_D \geq 0$}\}$, and 
	$(M_\mo(X),N_\mo(X))$ satisfies the assumptions required in \ref{base extension from integers to rationals}, so  we obtain a pairing $(M_{\mo}(X)_{\Q},N_{\mo}(X)_{\Q})$ as an abstract divisorial space in the sense of \cref{section: abstract divisorial spaces}. Moreover, the cone $N_{\mo}(X)_{\Q}$ agrees with its closure in $M_{\mo}(X)_{\Q}$ from \cref{closure for finite subspaces}. If $X$ is projective, then we have $M_\mo(X)=\widehat{\Div}(X)_\mo$.
\end{prop}
{For similicity of notation, we will use in the following $N_\mo\coloneqq N_\mo(X)$ and $M_\mo\coloneqq M_\mo(X)$ when the dependence on $X$ is clear.}    
\begin{proof}   
	Obviously, the submonoid $M_{\mo,\geq 0}$ of $M_\mo$ satisfies $M_{\mo,\geq 0}\cap (-M_{\mo,\geq 0})=\{0\}$ and hence $M_\mo$ is an ordered abelian group. It follows from \cref{base extension from integers to rationals} that $(M_{\mo,\Q},N_{\mo,\Q})$ is an abstract divisorial space. Semipositivity is a numerical criterion as we can check in the archimedean case if the curvature current is positive and as we can check in the non-archimedean case if the intersection numbers with closed curves in the special fiber of a $K^\circ$-model are non-negative. This proves  that the cone $N_{\mo,\Q}$ is closed in $M_{\mo,\Q}$ in the sense of \ref{closure for finite subspaces}. {We explain this argument in the non-archimedean case. We pick a finite dimensional subspace $E$ of $N_{\mo,\Q}$. It is generated by a finite basis $(D_j,g_j)_{j=1,\dots,r}$. We may assume that every $D_j$ is a Cartier divisor on $X$ and that there is a  $K^\circ$-model $\Xcal$ of $X$ such for every $j=1,\dots,r$, there is a line bundle $\Lcal_j$ on $\Xcal$ inducing the Green function $g_j$. Now let $(D_n',g_n')_{n \in \N}$ be a sequence in $E \cap N_{\mo,\Q} $ which converges  to $(D,g_D)$ in $E$. This means $(D_n',g_n')=\sum_{i=1}^n x_{in} (D_i,g_i) \in  N_{\mo,\Q}$ with $x_{in} \in \Q$ such that $x_i = \lim_{n \to \infty} x_{in} \in \Q$ and $(D,g_D)= \sum_{i=1}^n x_i (D_i,g_i)$. Note that $g_D$ is induced by the $\Q$-line bundle $\Lcal=\sum_{i=1}^n x_i \Lcal_i$ on $\Xcal$ using additive notation. By a limit argument, the intersection number of $\Lcal$ with any closed curve in the special fiber is non-negative and hence $\Lcal$ is nef proving $(D,g_D) \in N_{\mo,\Q}$ which means that $N_{\mo,\Q}$ is closed.}
			
	It remains to prove the last claim, so we assume that $X$ is projective. We have to show that $\widehat{\Div}(X)_\mo={{N_\mo}-{N_\mo}}$. Let $(D,g_D)\in \widehat{\Div}(X)_\mo$. Since $X$ is projective, there is a very ample Cartier divisor $H$. In the non-archimedean case, there is $m \in \N_{>0}$ such that $mg_D$ is induced by a Cartier divisor $\Dcal$ on a $\kcirc$-model $\Xcal$ of $X$. By \cref{projective models}, we may assume that $\Xcal$ is projective over $\kcirc$. As we are free to choose the very ample Cartier divisor $H$, we may assume that $H$ is the restriction of a very ample Cartier divisor $\Hcal$ on $\Xcal$. For $k \in \N_{>0}$ sufficiently large, the Cartier divisor $\Dcal +k\Hcal$ is very ample. For the model Green function $g_H$ induced by $\Hcal$, we conclude that $(mD,mg_D)+k(H,g_H)$ and $k(H,g_H)$ are both in $N_\mo$. This proves $m(D,g_D) \in N_\mo-N_\mo$, so {$(D,g_D) \in N_\mo-N_\mo$} and hence $\widehat{\Div}(X)_\mo = {N_\mo}-{N_\mo}$ in the non-archimedean case. In the archimedean case, let $g_H$ be the smooth Green function induced by a Fubini--Study metric on $\OO_X(H)$ as in \cref{Fubini-Study metric}. For $k \in \N_{>0}$ sufficiently large, it is clear that the smooth metric corresponding to $g_D+kg_H$ has semipositive curvature form and hence we deduce $\widehat{\Div}(X)_\mo=N_\mo-N_\mo$. 
\end{proof}   
		
\begin{remark} \label{abstract divisorial space and FS}
	A Green function for a Cartier divisor $D$ on $X$ is called a \emph{$\FS$-Green function} if its associated metric on $\OO_X(D)$ is a  Fubini--Study metric as in Definition \ref{Fubini-Study metric}. Let 
	$$N_\FS \coloneqq \{(D,g_D) \in \widehat{\Div}(X)_\mo \mid \text{$D$ Cartier divisor on $X$, \, $g_D$ $\FS$-Green function for $D$}\}$$
	and let $M_\FS = \sqrt{N_\FS}-\sqrt{N_\FS}$. Then it follows as above that $(M_{\FS,\Q},N_{\FS,\Q})$ is an abstract divisorial space over $\Q$. Using twisted Fubini--Study metrics, we get in the same way an abstract divisorial space over $\Q$ denoted by $(M_{\tFS,\Q},N_{\tFS,\Q})$. 
	For $X$ projective, the identity $M_\FS=\widehat{\Div}(X)_\mo$ holds  in the non-archimedean case which follows from the argument in the proof of \cref{classical local Arakelov theory as divisorial space}. Twisted Fubini--Study metrics are used in the approaches of Boucksom--Jonsson \cite{boucksom2022global} and Chen--Moriwaki \cite{chen2020arakelov} for local Arakelov theory to get also interesting applications in the case of the trivial valuation.
	
	If there is a semipositive metric on a line bundle, then the line bundle is nef. There are proper non-projective varieties without non-trivial nef line bundles and hence the cones $N_\mo$ and $N_\FS$ might be $\{0\}$ for proper non-projective varieties, see the following paper of Fujino and Payne {\cite{fujino2005smooth}}.
\end{remark}  
		
In algebraic geometry, the closure of the (semi-)ample cone is the nef cone (see Remark \ref{closure of semiample is nef}). The next statement is similar in non-archimedean Arakelov theory.
		
\begin{prop} \label{closure of FS is nef}
	Assume that $v$ is non-archimedean and $X$ projective. Then we have $M_{\FS,\Q}=M_{\mo,\Q}$ and the closure of the cone $N_{\FS,\Q}$ in the sense of \cref{closure for finite subspaces} is equal to $N_{\mo,\Q}$.
\end{prop}
\begin{proof}
	It follows from Proposition \ref{classical local Arakelov theory as divisorial space} and Remark \ref{abstract divisorial space and FS} that $M_{\mo,\Q}$ and $M_{\FS,\Q}$ both agree with $\widehat{\Div}_\Q(X)_{\mo}$. 
	To prove the last claim, let $(D,g_D) \in N_{\mo,\Q}$. Again \cref{projective models} shows that  there is $m \in \N_{>0}$ such that $m(D,g_D)$ is induced by a nef  Cartier divisor $\Dcal$ on a projective $\kcirc$-model $\Xcal$. Let $\Hcal$ be an  ample  Cartier divisor on $\Xcal$ and let $g_H$ be the associated Green function of the restriction $H$ to the generic fiber $X$. Then we have
	\begin{equation}  \label{limit argument for nef and FS}
		(D,g_D)= \lim_{k \to \infty} \frac{1}{m}\left(mD+\frac{1}{k}H,mg_D+\frac{1}{k}g_H\right)
	\end{equation}
	as a limit in the subspace of $M_{\mo,\Q}$ generated by $(D,g_D)$ and $(H,g_H)$. Since {$\Dcal$} is a nef Cartier divisor on $\Xcal$ and since $\Hcal$ is an ample Cartier divisor on $\Xcal$, we conclude that $\Dcal+\frac{1}{k}\Hcal$ is an ample $\Q$-Cartier divisor on $\Xcal$. As it induces the elements on the right hand side of \eqref{limit argument for nef and FS} and as the corresponding Green functions are $\FS$ by \cref{Fubini-Study metric}, we conclude from \eqref{limit argument for nef and FS} that $(D,g_D)$ is in the closure of $N_{\FS,\Q}$. Since $N_{\mo,\Q}$ is closed in $M_{\mo,\Q}$ by \cref{classical local Arakelov theory as divisorial space}, this proves the claim.
\end{proof}

\begin{remark} \label{semipositive Green functions in sense of Zhang}
	Zhang \cite{zhang1995smallpoints} extended semipositivity of model metrics to continuous metrics as follows. Let $D$ be a Cartier divisor on a proper variety $X$ over $K$. We say that a Green function $g_D$ for $D$ is \emph{semipositive} if $g_D$ is the uniform limit of {semipositive} model Green functions $(g_{k,D})_{k \in \N}$ for $D$. Note that $g_{k,D}-g_D$ is a Green function for $0$ and hence a continuous real  function on $X^\an$, so uniform convergence means that the functions $g_{k,D}-g_D$ converge to $0$ with respect to the supremum norm on $X^\an$. A Green function for $D$ is called \emph{$\DSP$}, which stands for \emph{difference of semipositive}, if it is the difference of two semipositive Green functions. We set $$M_\SP \coloneqq \{(D,g_D) \mid \text{$D$ Cartier divisor on $X$, $g_D$ $\DSP$-Green function for $D$}\}.$$
	Then $M_\SP$ is an ordered abelian group with partial order $\geq$ given by 
	$$M_{\SP,{\geq 0}} \coloneqq \{(D,g_D)\in M_{\SP}\mid \text{$D$ effective and $g_D \geq 0$}\}.$$ 
	Let $N_\SP$ be given by $(D,g_D)\in M_\SP$ such that $g_D$ is a semipositive Green function. Then $M_\SP$ is an ordered abelian group, $N_{\SP}\subset M_{\SP}$ is a submonoid and $M_\SP=N_\SP-N_\SP$. Therefore, by performing base extension of scalars of $M_\SP$ to $\mathbb{Q}$, we obtain an abstract divisorial space $(M_{\SP,\Q},N_{\SP,\Q})$ over $\Q$, see \cref{base extension from integers to rationals}. We can see the elements of $M_{\SP,\Q}$ as pairs $(D,g_D)$ with $D$ a $\Q$-Cartier divisor and $g_D$ a Green function for $D$.
\end{remark}
		
We now relate the above construction to the completions of abstract divisorial spaces given in \cref{b-completion}. This is just an abstract reformulation of Zhang's construction of semipositive metrics and helps to understand later the generalizations of Yuan and Zhang to quasi-projective varieties.
		
\begin{prop} \label{Zhang metrics as divisorial space}
	Let $X$ be a proper variety over $K$ and let $S\subset M_{\mo,\Q,\geq 0} \cap N_{\mo,\Q}$ be given by those elements $(0,g)$ with constant non-negative Green functions $g$. Then the abstract divisorial space $(M_{\SP,\Q},N_{\SP,\Q})$ from \cref{semipositive Green functions in sense of Zhang} is the $S$-completion of the abstract divisorial space $(M_{\mo,\Q},N_{\mo,\Q})$ in the sense of \cref{S-completion}.
\end{prop}
\begin{proof}
	If $v$ is the trivial valuation on $K$, then the model Green function $g_D$ associated to a Cartier divisor $D$ is unique. If $\gamma$ is a local equation for $D$ on an open subset $U$ of $X$, then $g_D(x)=-\log |\gamma(x)|$ for $x \in U^\an$. This means that all semipositive Green functions are model Green functions. Since $S$ consists then only of $(0,0) \in M_{\mo,\Q}$, we conclude that $S$-completion of  $(M_{\mo,\Q},N_{\mo,\Q})$ is just $(M_{\mo,\Q},N_{\mo,\Q})=(M_{\SP,\Q},N_{\SP,\Q})$.
			
	Now we assume that $v$ is non-trivial. Then there is a positive $r \in v(K^\times)$ and hence $b=(0,r)$ is a non-trivial element in $S$. It is clear that the $b$-completion of $(M_{\mo,\Q},N_{\mo,\Q})$ does not depend on the choice of $r$ and hence it is enough to show that 
	\begin{equation} \label{b-completion of models}
    	\left(\widehat{M}_{\mo,\Q}^b,\widehat{N}_{\mo,\Q}^b \right)= (M_{\SP,\Q},N_{\SP,\Q}).
	\end{equation}
	Note that for $\varepsilon \in \Q_{>0}$ with $\varepsilon <1$ and for $\overline D=(D,g_D)\in M_{\mo,\Q}$, the condition $-\varepsilon b \leq \overline D \leq \varepsilon b$ is equivalent to $D=0$ as a $\Q$-Cartier divisor and $|g_D(x)| \leq \varepsilon r$ for all $x \in X^\an$. We conclude that the $b$-topology matches with the topology of uniform convergence of Green functions and hence \eqref{b-completion of models} holds.
\end{proof}

\begin{example} \label{running example for local Zhang}
We use the running example from \S \ref{subsec: a running example} to illustrate the above notions. Let $X=\mathbb P_K^1$  for $K=\Q_v$ where $v$ is either a prime number $p$ or  $\infty$ and where $\Q_v$ is the completion of $\Q$ with respect to the place $v$. We fix homogeneous coordinates $x_0,x_1$ of $X$ and use the affine coordinate $t=x_1/x_0$. Then the dense open torus $T=\mathbb G_{{\rm m},K}=X \setminus \{x_0\cdot x_1 \neq 0\}= X \setminus \{0,\infty\}$ acts on $X$ and we restrict our attention to  toric metrics on the line bundle $L=\OO_X(1)$ of $X$ and to the corresponding toric divisor $D=\mathrm{div}(x_0)=[\infty]$, see \S \ref{subsec: a running example}.

We have the \emph{tropicalization map}
$$\trop\colon X^\an \longrightarrow \overline \R, \, \, \quad x \mapsto -\log|t(x)|_v$$
where $ \overline \R\coloneqq \R \cup \{\pm \infty\}$. We note that $\trop(0)=\infty$ and $\trop(\infty)=-\infty$. The coordinate on $\R$ is denoted by $u$, then $D$ corresponds to the piecewise linear function $\Psi(u)=\min(0,u)$. Note that ampleness of $D$ is reflected by the fact that $\Psi$ is strictly concave.

Recall that a metric $\metr$ of $L$ is toric if and only if there is a continuous function $\psi\colon \R \to \R$ which extends continuously to a function $\overline \R \to \overline \R$ such that 
$$\psi\circ \trop(x)=\log\|x_0(x)\|$$
for all $x \in T^\an$. Note that $g_D = - \psi \circ \trop$ is then the Green function corresponding to the toric metric $\metr$.  The toric metric corresponding to $\Psi$ is called the \emph{canonical metric} of $L$. It is characterized by the property that $[m]^*\metr=\metr^{\otimes m}$ along the identification $[m]^*(L)=L^{\otimes m}$ for all $m \in \N$ (using the rigidification $x_0$ at $1 \in T(K)$). For $v$ non-archimedean, it follows directly from the definitions that the canonical metric is the Fubini--Study metric given by the global sections $s_0,s_1$. In the archimedean case, this is not true as the canonical metric is not smooth since $\Psi$ has a singularity at $u=0$. 

For a moment, we assume that $v$ is non-archimedean. Then the toric metric $\metr$ is a model metric if and only if the corresponding function $\psi$ is a piecewise affine function given by subdividing $\overline\R$ into finitely many intervals $[t_{j-1},t_j]$ with $-\infty=t_0<t_1< \dots < t_r=\infty$ such that 
$$\psi(u) \coloneqq m_j u + c_j$$
for $u \in [t_{j-1},t_j]$ with slopes $m_j \in \Q$ and constant terms $c_j \in \Q$; moreover $\metr$ is semipositive if and only if $\psi$ is concave (see \cite[\S 4.5]{BPS}). As $\metr$ is a toric metric of $L=\OO_X(1)$, we have $m_1=1$ and $m_r=0$ which means that the recession function of $\psi$ is $\Psi$.

Now we consider again the archimedean case and the non-archimedean case simultaneously. A toric metric of $\metr$ of $\OO_X(1)$ is semipositive in the sense of Zhang if and only if $\psi$ is concave and $\psi-\Psi$ is bounded, see  \cite[Theorem 4.8.1]{BPS}, or \cite[Theorem II]{gubler-hertel}.
\end{example}

\subsection{Compactified metrized divisors}

In this subsection, we recall the analytic description of adelic divisors (called \emph{compactified metrized divisors} in our paper) given by Yuan and Zhang in \cite[\S 3.6]{yuan2021adelic}. It is straightforward to generalize this description to any algebraic variety $U$ over a non-archimedean field $K$. The upshot of the construction is that we can allow Green functions for Cartier divisors on $U$ with prescribed singularities along the boundary.

\begin{definition}
	The space of \emph{model metrized divisors} (resp. \emph{{signed} twisted Fubini--Study divisors}) on $U$ is defined as the direct limit 
	\[\widehat{\Div}_\Q(U)_{\mo}\coloneq\lim\limits_{\underset{X}{\longrightarrow}}\widehat{\Div}_\Q(X)_{\mathrm{mo}},\]
    \[\text{(resp. {$\widehat{\Div}_\Q(U)_{\tFS}\coloneq\lim\limits_{\underset{X}{\longrightarrow}}M_{\tFS,\Q}(X)$})},\]
	where $X$ runs through all proper $K$-model of $U$ and $M_{\tFS,\Q}(X)$ was introduced in \cref{abstract divisorial space and FS}. We say an element $\overline{D}$ in $\widehat{\Div}_\Q(U)_{\mo}$ (resp. in  $\widehat{\Div}_\Q(U)_{\tFS}$) is \emph{effective}, denoted by $\overline{D}\geq0$, if there is a proper $K$-model $X$ such that $\overline{D}\in \widehat{\Div}_\Q(X)_{\mathrm{mo}}$ (resp. $M_{\tFS,\Q}(X)$) and $\overline{D}\geq0$. We call $\overline{D}=(D,g_D)$ \emph{strictly effective} if $D$ is effective and $g_D>0$.
\end{definition}
\begin{remark}
	A dominant morphism $\varphi\colon X_1\rightarrow X_2$ of proper varieties induces, via pullback, a homomorphism $\widehat{\Div}_\Q(X_2)_{\mathrm{mo}}\rightarrow \widehat{\Div}_\Q(X_1)_{\mathrm{mo}}$, so our definition of $\widehat{\Div}_\Q(U)_{\mathrm{mo}}$ makes sense.
\end{remark}
		
\begin{definition} 	\label{def:boundarytopologylocal}
	Assume that the valuation on $K$ is non-trivial. A \emph{{weak boundary divisor}} (resp. a \emph{boundary divisor}) of $U$ is a pair $({X}_0,\overline{B})$ consisting of a proper $K$-model $U\hookrightarrow {X}_0$ over $K$ and an effective  model divisor $\overline{B}\in \widehat{\Div}_\Q({X}_0)_{\mathrm{mo}}$ such that {$|B|\subset{X}_0\setminus U$} (resp. $|B|={X}_0\setminus U$). We say a boundary divisor $(X_0, \overline{B})$ is \emph{cofinal} if $\overline{B}$ is strictly effective. We will see in \cref{remarks about boundary topology} that there is always a cofinal boundary divisor.
			
	Let $({X}_0,\overline{B})$ be a {weak} boundary divisor with $\overline{B}=(B,g_B)$. The \emph{$\overline{B}$-boundary topology} on $\widehat{\mathrm{Div}}_\Q(U)_{\mathrm{mo}}$ is the unique topology such that a basis of neighborhoods of a model divisor $\overline{D}=(D,g)$ is given by
	\begin{align*}
		B(r,D)\coloneq&\{\overline{E}=(E,h)\in \widehat{\Div}_\Q(U)_{\mathrm{mo}}\mid -r\overline{B}\leq \overline{E}-\overline{D}\leq r\overline{B}\},%\\
		\ \ r\in \Q_{>0}.
	\end{align*}
	If $({X}_0,\overline{B})$ is a cofinal boundary divisor, we call the corresponding topology the \emph{boundary topology}. The space of \emph{compactified metrized} \emph{divisors} (with respect to $v$) on $U$ is defined as the completion of $\widehat{\Div}_\Q(U)_{\mo}$ with respect to the boundary topology, denoted by $\widehat{\Div}_\Q(U)_{\cpt}$.
			
	When the valuation is trivial, we can do a similar procedure on $\widehat{\Div}_\Q(U)_{\tFS}$, the corresponding completion is also denoted by $\widehat{\Div}_\Q(U)_{\cpt}$, called the space of \emph{compactified metrized} \emph{divisors} (with respect to $v$) on $U$. The space $\widehat{\Div}_\Q(U)_{\cpt}$ is the largest space of  divisors {considered in the local setting} when $v$ is either nontrivial or trivial. 
\end{definition}
\begin{remark} \label{remarks about boundary topology}
	We claim first that there is always a cofinal boundary divisor. By blowing up the boundary $\partial X_0=X_0 \setminus U$, we may assume that the boundary is the support of an effective Cartier divisor $B$ (in the non-archimedean case even the support of an effective Cartier divisor $\Bcal$ on $\Xcal_0$, where $\Xcal_0$ is a model of $X_0$). There is a model Green function $g_B$ for $B$ with $g_B>0$ on $U^\an$. This is automatic in the non-archimedean case by choosing the Green function induced by $\Bcal$, and can be obtained in the archimedean case by adding a sufficiently large constant.
		
	The boundary topology is independent of the choice of the cofinal boundary divisor $b \coloneqq (X_0,\overline B)$. The argument uses that $\{nb \mid n \in \N\}$ is cofinal in the set of  weak boundary divisors 
	as shown in \cite[Lemma~2.4.1]{yuan2021adelic}. Then for any weak boundary divisor $(X_0',\overline{B'})$, the $\overline{B}$-boundary topology is coarser than the $\overline{B'}$-boundary topology. Conversely, if $(X_0',\overline{B'})$ is cofinal, then the $\overline{B'}$-boundary topology is coarser than the $\overline{B}$-boundary topology. This proves the claim.
\end{remark}

In the remaining part of this subsection, we fix a cofinal boundary divisor $(X_0,\overline{B})$ for $\overline B=(B,g_B)$. 

\begin{art}  \label{justification for writing pairs}
We have a canonical injective homomorphism 
\begin{equation} \label{injection into arithmetic divisors}
	\widehat{\Div}_\Q(U)_{\cpt}\longrightarrow \widehat{\Div}_\Q(U).
\end{equation} 
Moreover, there is a canonical injective homomorphism
\begin{equation} \label{injection into direct sum}
\widehat{\Div}_\Q(U)_{\cpt}\longrightarrow \widetilde{\Div}_\Q(U)_\cpt \oplus G(U^\an)
\end{equation} 
where $\widetilde{\Div}_\Q(U)_\cpt$ denotes the space of compactified geometric divisors on $U$ from \cref{def:geometric boundary topology} and where $G(U^\an)$ is the space of Green functions on $U^\an$.
To see this, we write $\overline D \in \widehat{\Div}_\Q(U)_{\cpt}$ as a limit of a sequence %\green{or more generally of a net 
	$(D_i,g_i)_{i \in \N}$ 
from $\widehat{\Div}_\Q (U)_\mo$ in the non-trivially valued case (resp.~ from $\widehat{\Div}_\Q (U)_\tFS$ in the trivially valued case). Then the definition of the boundary topology shows that $D_i$ converges to a compactified geometric divisor $D \in \widetilde{\Div}_\Q(U)_\cpt$ with respect to the $B$-boundary topology. Moreover, we have $D|_U=D_i|_U$ for sufficiently large $i \in \N$, which is a well-defined Cartier divisor  on $U$ using that the support of $B$ is outside $U$. The definition of the boundary topology shows that $g_i$ converges locally uniformly to a Green function $g_D$ for $D|_U$ on $U^\an$. Clearly, the pair $(D,g_D)$ does not depend on the choice of $(D_i,g_i)_{i \in \N}$ and hence gives rise to the canonical homomorphism in \eqref{injection into direct sum} {and its restriction $(D|_U, g_D)$ on $U$ gives the canonical homomorphism in \eqref{injection into arithmetic divisors}}. In particular, the injectivity of \eqref{injection into arithmetic divisors} implies the injectivity of \eqref{injection into direct sum}. It remains to show the injectivity of \eqref{injection into arithmetic divisors}. {Let $\overline D=(D,g_D)\in  \widehat{\Div}_\Q(U)_{\cpt}$ with $D|_U=0$ and $g_D=0$. Then $\overline D$ is the limit of a sequence of elements $(D_i,g_i) \in \widehat{\Div}_\Q (X_i)_\mo$ in the non-trivially valued case (resp.~ from $\widehat{\Div}_\Q (X_i)_\tFS$ in the trivially valued case) for some proper $K$-model $X_i$ of $U$ for $i \in \N$. Replacing $X_i$ by its normalization in $U$,}  we can assume that $X_i$ is integrally closed in $U$. Since $D|_U=0$ and $g_D=0$, then for any $\varepsilon\in \Q_{>0}$, we have $B|_U=D_i|_U=0$ and
\[-\varepsilon g_B\leq g_i\leq \varepsilon g_B\]
on $U^\an$ for large $i$. By \cref{lemma: injectivity of restriction from integrally closed}, we have $({D}_i,g_i)+\varepsilon\overline{B}\geq 0$ and $\varepsilon \overline{B}-({D}_i,g_i)\geq 0$
for large $i$. Hence $\{(D_i,g_i)\}_{i\geq 1}$ is a Cauchy sequence equivalent to $0$ with respect to the $\overline{B}$-boundary topology. This proves injectivity of \eqref{injection into arithmetic divisors}.  
\end{art}

We use \cref{justification for writing pairs} to write an element $\overline D \in \widehat{\Div}_\Q(U)_{\cpt}$ from now on as $(D,g_D)$ with $D\in \widetilde{\Div}_\Q(U)_{\cpt}$.

\begin{definition} \label{Green functions for compactified divisors}
	Let $D\in \widetilde{\Div}_\Q(U)_{\cpt}$. A \emph{Green function} for $D$ is a  {Green function $g$ for $D|_U$} such that $(D,g)\in \widehat{\mathrm{Div}}_\Q(U)_{\cpt}$. 
\end{definition}

\begin{lemma}	\label{lemma:boundary convergence is locally uniformly convergence}
	Let $(D_i,g_i)_{i \in I}$ be a net in $\widehat{\Div}_\Q(U)_{\cpt}$ converging to a compactified metrized divisor $(D,g)$. Then $g_i$ converges locally uniformly to $g$ on $U^\an$. 
\end{lemma}
\begin{proof}
	It is the same argument as used in \cref{justification for writing pairs}.
\end{proof}

\begin{proposition}[\cite{yuan2021adelic}~Theorem~3.6.4] \label{proposition:green functions for adelic divisors}
	The space 
	\[C(U^\an)_0\coloneq\Im(\{h\in C(X_0^\an)\mid h|_{|B|^\an}\equiv0\}\rightarrow C(U^\an))\]
	is independent of the choice of the cofinal boundary divisor $(X_0,\overline{B})$, and we have a canonical exact sequence 
	\[\xymatrix{0\ar[r]& V\ar[r]& \widehat{\Div}_\Q(U)_{\cpt}\ar[r]& \widetilde{\Div}_\Q(U)_{\cpt}}\]
	where $V$ is a subspace of $g_B\cdot C(U^\an)_0$. If the valuation on $K$ is non-trivial or if $U$ is quasi-projective, then the above sequence is the short exact sequence
	\[\xymatrix{0\ar[r]& g_B\cdot C(U^\an)_0\ar[r]& \widehat{\Div}_\Q(U)_{\cpt}\ar[r]& \widetilde{\Div}_\Q(U)_{\cpt}\ar[r]&0}.\]
\end{proposition}

%\begin{proof}	This follows from the same arguments as in \cite[\S 3.6.4]{yuan2021adelic} adapted to our setting. We need quasi-projectivity in the trivially valued case to ensure density of {signed} twisted Fubini--Study functions on $X_0$ in the space of continuous functions on $X_0^\an$. In the non-trivially valued case, model functions are always dense by \cite[Theorem 7.12]{gubler1998local}. 
%\end{proof}

\begin{proof}
This follows from the same arguments as in \cite[\S 3.6.4]{yuan2021adelic} adapted to our setting. For the convenience of readers, we sketch the proof. %without % full details. }
	
	\smallskip
	\noindent\textbf{Step 1.} Let $(X_0,\overline{B})$ and $(X_0',\overline{B'})$ be two cofinal boundary divisors for $U$. By definition of the cofinal system, there exists a common refinement $(X_0'',\overline{B''})$ dominating both. So we may assume that $(X_0',\overline{B'})$ dominates $(X_0,\overline{B})$. We have inclusions 
	%induced pull-back maps
	\[
	\{h\in C(X_0^\an)\mid h|_{|B|^\an}\equiv0\}\rightarrow 	\{h\in C((X_0')^\an)\mid h|_{|B'|^\an}\equiv0\}\rightarrow C(U^\an)
	\]
	{which are induced by pull-back maps. The argument in \cite[\S 3.6.4]{yuan2021adelic} shows that first map is in fact surjective.} 
	%We can show that the first map is in fact surjective by constructing the preimage for any element in the second set. 
	This shows that $C(U^\an)_0$ is well-defined.
	
	\smallskip
	\noindent\textbf{Step 2.}
	{The proof of  \cite[Theorem 3.6.4(1)]{yuan2021adelic} carries over to a large extend to the trivially valued case. In any case, the same arguments show that the kernel $V$ of 
	$\widehat{\Div}_\Q(U)_{\cpt}\rightarrow\widetilde{\Div}_\Q(U)_{\cpt}$ is a subspace of $g_B\cdot C(U^\an)_0$.}
	
	\smallskip
	\noindent\textbf{Step 3.} It remains to show $V=g_B\cdot C(U^\an)_0$ and the surjectivity of $\widehat{\Div}_\Q(U)_{\cpt}\rightarrow\widetilde{\Div}_\Q(U)_{\cpt}$ in the given two cases. This requires density of 
	model functions in $C(U^\an)_0$:
	\begin{itemize}
		\item If the valuation is non-trivial, model functions are dense in $C(X_0^\an)$ by \cite[Theorem~7.12]{gubler1998local}.
		\item If the valuation is trivial, quasi-projectivity of $U$ ensures that $X_0$ carries an ample line bundle, so that signed twisted Fubini--Study functions are dense in $C(X_0^\an)$.
	\end{itemize}
	{In any case, the argument in \cite[\S 3.6.4]{yuan2021adelic} shows that $C(X_0^\an)$ is dense in $g_B \cdot C(U^\an)_0$ under the boundary topology. Using the above two density results, we see that $g_0 \cdot C(U^\an)$ is a subspace of $\widehat{\Div}_\Q(U)_{\cpt}$ and hence $V=g_B \cdot C(U^\an)_0$.} 
	%In either case, every $h\in C(U^\an)_0$ is a limit of admissible Green functions, yielding $V=g_B\cdot C(U^\an)_0$. 
	The density of model functions (or signed twisted Fubini--Study functions) ensures that each compactified divisor can be endowed with a Green function making it an element in  $\widehat{\Div}_\Q(U)_{\cpt}$, showing the surjectivity of $\widehat{\Div}_\Q(U)_{\cpt}\rightarrow\widetilde{\Div}_\Q(U)_{\cpt}$ (see the end of the proof in \cite[\S 3.6.4]{yuan2021adelic}).
\end{proof}

We use the above exact sequence to compute $\widehat{\Div}_\Q(U)_{\cpt}$ in our running example.

\begin{example} \label{compactified metrics in running example}
Let $K=\Q_v$ for a place $v$ of $\Q$ and let $U=T$ be the dense open torus of $X=\mathbb P_K^1$ as in Example \ref{running example for local Zhang}. We view the homogeneous coordinates $x_0, x_1$ as global sections of $L=\OO_X(1)$.

For $a,b \in \N$, we get a divisor $B(a,b)=a[\infty]+ b[0]$ with corresponding global section $x_0^ax_1^b$ of $\OO_X(a+b)$. We endow $\OO_X(a+b)$ with the Fubini--Study metric $\metr_\FS$ 
$$\|x_0^ax_1^b\|_\FS = \frac{|x_0^a x_1^b|_v}{\left(\sum_{i+j=a+b}|x_0^i x_1^j|_v^2\right)^{1/2}}$$
in the archimedean case and with
	$$\|x_0^ax_1^b(x)\|_{\FS}= \frac{|x_0^ax_1^b|_v}{\max\{|x_0^i x_1^j|_v \mid i+j=a+b\}}$$
in the non-archimedean case. These are the Fubini--Study metrics 
associated to the generating set $\{x_0^ix_1^j \mid i,j \in \N, \,i+j=a+b\}$ of global sections  by the definition in \cref{Fubini-Study metric}. For $r \in \R_{\geq 0}$, we get the Green function $g_{a,b,r}\coloneqq -\log \|x_0^ax_1^b\|_\FS+r$   for $B(a,b)$ and the corresponding twisted Fubini--Study metric. Then $\overline{B}=(B(a,b),g_{a,b,r})$ is a weak boundary divisor of $U$. It is clear that $\overline B$ is a boundary divisor if and only if $a,b>0$. In the non-archimedean case, it is obvious that $g_{a,b,r}\geq g_{a,b,r}(1)=r$ and so $\overline B$ is a cofinal boundary divisor of $U$ if and only if $a,b,r>0$. In the archimedean case, we have $g_{a,b,r}>r$ and hence $\overline B$ is a cofinal boundary divisor if and only if $a,b>0$.

We fix now such a cofinal boundary divisor $\overline B = (B,g_B)$. Usually, we take $a=b=1$ and $r$ a small positive number, but the choice does not matter as we have seen in \cref{remarks about boundary topology}. We obtain from \cref{proposition:green functions for adelic divisors} and   \cref{example: projective line} the short exact sequence
\[\xymatrix{0\ar[r]& g_B\cdot C(U^\an)_0\ar[r]& \widehat{\Div}_\Q(U)_{\cpt}\ar[r]& \Div_\Q(U) \oplus \R[0] \oplus \R[\infty]\ar[r]&0}.\]
Using that $U$ is the dense torus $T=\mathbb P_K^1 \setminus \{0,\infty\}$, all $\Q$-line bundles of $U$ are trivial. We have the following explicit description of a compactified metrized divisor $(D,g_D)$ of $U$. The restriction $D_U=D|_U$ is a $\Q$-Cartier divisor on $U=\mathbb P_K^1 \setminus \{0,\infty\}$ and hence is of the form $D_U=\frac{1}{k} \mathrm{div}(f(x_1/x_0))$ for some rational function $f(t) \in K(t)\setminus \{0\}$ with no zeros and poles at $0,\infty$ and some $k \in \N \setminus \{0\}$.

There are  uniquely determined $a,b \in \R$  with $D=D_U+a[0]+b[\infty]$. 
Using the Fubini--Study metrics, we get Green functions  $g_{[0]}, g_{[\infty]}$  for the divisors $[0]$ and $[\infty]$, respectively. Approximating $a,b$ with rational numbers, it is clear that $$\left(a[0]+b[\infty], ag_{[0]}+b g_{[\infty]}\right) \in \widehat{\Div}_\Q(U)_{\cpt}.$$
Now the Green function $g_D$ of the compactified metrized divisor $(D,g_D)$ of $U$ is 
 precisely a function of the form
\begin{equation} \label{compactified Green function in running example}
g_D=-\frac{1}{k}\log \left|f\left(\frac{x_1}{x_0}\right)\right|_v + a g_{[0]} + b g_{[\infty]} + \varphi
\end{equation}
where $\varphi =o(g_B)$. This follows from the short exact sequence above. Note that in this example the Green function $g_D$ completely determines the compactified metrized divisor $(D,g_D)$ as we can recover $D_U$ from the singularities of $g_D$ on $U$ and as we can recover $a,b$ from the singularities at $0$ and $\infty$. 
\end{example}

\subsection{The boundary completion introduced by Yuan and Zhang}

The space of compactified metrized divisors $\widehat{\Div}_\Q(U)_\cpt$ from the previous paragraph is too large to have an arithmetic intersection theory. We will use the abstract divisorial space associated to $X$ from Proposition \ref{classical local Arakelov theory as divisorial space} to have a subspace of $\widehat{\Div}_\Q(U)_\cpt$ where we can define a Radon measure for each element. 
In this subsection, we will explain how Yuan and Zhang complete the direct limit of these abstract divisorial spaces. The construction is similar to the geometric construction in Section \ref{sec: geometric theory}, where we obtained an abstract divisorial space $(\widehat{M}_{\gm,\Q}^\YZ(U), \widehat{N}_{\gm,\Q}^\YZ(U))$ in a similar way from the abstract divisorial spaces  $({M}_{\gm,\Q}(X), {N}_{\gm,\Q}(X))$.

\begin{art} \label{semipositive Green functions in sense of Yuan-Zhang}
	If $\varphi\colon X' \to X$ is a morphism of $K$-models of $U$, then pull-back induces a morphism $$\varphi^*\colon(M_{\mo,\Q}(X),N_{\mo,\Q}(X))  \longrightarrow (M_{\mo,\Q}(X'),N_{\mo,\Q}(X')), \quad (D,g_D) \mapsto (\varphi^*D,g_D \circ \varphi)$$
	of abstract divisorial spaces over $\Q$. Using the direct limit from \cref{direct limits}, we define an abstract divisorial space over $\Q$ by 
	$$(M_{\mo,\Q}(U),N_{\mo,\Q}(U)) \coloneqq \varinjlim_X (M_{\mo,\Q}(X),N_{\mo,\Q}(X))$$
	where $X$ ranges over all proper $K$-models of $U$. Let $S$ be the directed set of weak boundary divisors in $M_{\mo,\Q}(U)$. 
	Then the \emph{Yuan--Zhang completion} is
	$$(\widehat{M}_{\mo,\Q}^\YZ(U),\widehat{N}_{\mo,\Q}^\YZ(U)) \coloneqq (\widehat{M}_{\mo,\Q}^S(U),\widehat{N}_{\mo,\Q}^S(U))$$
	using the $S$-completion of abstract divisorial spaces from \cref{S-completion}. By construction, $\widehat{M}_{\mo,\Q}^\YZ(U)$ is a subspace of the $\Q$-vector space $\widehat{\Div}_\Q(U)_\cpt$ introduced in \ref{def:boundarytopologylocal}. 
\end{art}

\begin{remark} \label{interpretation of YZ-elements}
	We relate the above abstract divisorial space to the original construction of Yuan and Zhang as follows. As in their paper, we assume that $U$ is quasi-projective. {By \cite[Corollaire~5.7.14]{raynaud1971criteres}, for any proper $K$-model $Y$ of $U$, there exists a \emph{$U$-admissible blow-up} (i.e., the center of the blow-up is disjoint from $U$) $Y' \to Y$ such that $Y'$ is quasi-projective over $K$. Since $Y$ is proper over $K$, so is $Y'$, and hence $Y'$ is projective over $K$. Moreover, $Y'$ is a projective $K$-model of $U$ because the blow-up is $U$-admissible. Consequently, the system of projective $K$-models of $U$ is cofinal in the system of all proper $K$-models of $U$. So
	\begin{align}\label{proj model is cofinal}
(M_{\mo,\Q}(U), N_{\mo,\Q}(U)) = \varinjlim_X (M_{\mo,\Q}(X),N_{\mo,\Q}(X))
	\end{align} 
	where $X$ ranges over all projective $K$-models of $U$.}
	
We pick any cofinal boundary divisor $\overline B=(B,g_B)$ on a projective $K$-model $X_0$ of $U$, see Remark \ref{remarks about boundary topology}. 
This is indeed possible by using a projective $K$-model in the construction of the cofinal boundary divisor in Remark \ref{remarks about boundary topology} or by the fact that the system of projective $K$-models of $U$ is cofinal.
Then Proposition \ref{classical local Arakelov theory as divisorial space} shows that $\overline{B}$ represents an element $b$ in $M_{\mo,\Q}(X_0)$ and hence also in $M_{\mo,\Q}(U)$.
Let $b$ be the element in $M_{\mo,\Q}(U)_{> 0}$ represented by $\overline{B}$. By \cite[Lemma 2.4.1]{yuan2021adelic} (which also holds when we consider the system of all proper $K$-models of $U$), the subset $S'=\N b$ of $S$ is cofinal in $S$ and hence \cref{final object} shows
\begin{equation} \label{yuan-zhang as completion}
	(\widehat{M}_{\mo,\Q}^\YZ(U),\widehat{N}_{\mo,\Q}^\YZ(U))=(\widehat{M}_{\mo,\Q}^b(U), \widehat{N}_{\mo,\Q}^b(U)).
\end{equation}

For any proper $K$-model $X$ of $U$, we have a canonical morphism
$$({M}_{\mo,\Q}(X),{N}_{\mo,\Q}(X)) \longrightarrow ({M}_{\gm,\Q}(X), {N}_{\gm,\Q}(X)), \quad (D,g_D)\mapsto D$$
of abstract divisorial spaces over $\Q$. By passing to the direct limit and then using continuity with respect to the $b=\overline{B}$-topology on the source and with respect to the $B$-topology on the target, we get a canonical morphism
\begin{equation} \label{canonical geometric component}
	(\widehat{M}_{\mo,\Q}^\YZ(U),\widehat{N}_{\mo,\Q}^\YZ(U)) \longrightarrow (\widehat{M}_{\gm,\Q}^\YZ(U), \widehat{N}_{\gm,\Q}^\YZ(U))
\end{equation}
of abstract divisorial spaces over $\Q$. 

{Equalities \eqref{proj model is cofinal} \eqref{yuan-zhang as completion} show that} the abstract divisorial space $(\widehat{M}_{\mo,\Q}^\YZ(U),\widehat{N}_{\mo,\Q}^\YZ(U))$ agrees with the original construction $(\widehat{M}_{\mo,\Q}^b(U), \widehat{N}_{\mo,\Q}^b(U))$ of so called adelic divisors introduced in \cite[Chapter 2]{yuan2021adelic} (they only consider projective $K$-models of $U$ when forming the direct limit), see also \cite[Chapter 3]{yuan2021adelic} for a description in terms of Berkovich spaces.
Note that Yuan and Zhang introduced an integral structure over $(\widehat{M}_{\mo,\Q}^b(U), \widehat{N}_{\mo,\Q}^b(U))$ which can be seen as a fiber product of the map in \eqref{canonical geometric component} with the canonical map 
$$(\widehat{M}_{\gm}^\YZ(U), \widehat{N}_{\gm}^\YZ(U)) \longrightarrow (\widehat{M}_{\gm,\Q}^\YZ(U), \widehat{N}_{\gm,\Q}^\YZ(U)).$$
The integral structure will not be important for our paper as its base extension to $\Q$ is \eqref{yuan-zhang as completion}.
\end{remark}
		
\begin{art} \label{interpretation of elements in YZ-completion}
	Let {$U$ be quasi-projective and let}  
	$b\in M_{\mo,\Q}(U)_{>0}$ be represented by a cofinal boundary divisor $ \overline{B}=({B},g_B)$ as above. By definition, any $\widehat{D} \in \widehat{M}_{\mo,\Q}^\YZ(U)$ can be written as a limit $\widehat{D}=\lim\limits_{n\to\infty} (D_n,g_n)$ in the complete metric space $\widehat{M}_{\mo,\Q}^\YZ(U)=\widehat{M}_{\mo,\Q}^b(U)$ for a sequence $(D_n,g_n)_{n\geq 1}$ with $g_n$ a Green function for a $\Q$-Cartier divisor $D_n$ on a proper $K$-model $X_n$ of $U$. The $\Q$-Cartier divisors $D_n$ converge to an element $D \in \widehat{M}_{\gm,\Q}^\YZ(U)$ which is the image of $\widehat{D}$ under the map \eqref{canonical geometric component}. 
			
	By definition of the $B$-topology on $\widehat{M}_{\gm,\Q}^\YZ(U)=\widehat{M}_{\gm,\Q}^B(U)$, the restriction of $D_n$ to $U$ is a fixed $\Q$-Cartier divisor $D_U$ on $U$ for $n$ sufficiently large. Moreover, the definition of the $b$-topology on $\widehat{M}_{\mo,\Q}^\YZ(U)$ shows that the restrictions $g_n|_U$ are Green functions for $D_U$ on $U$ which converge locally uniformly. The limit is a Green function $g_D$ for $D_U$ which is obviously independent of the choice of the approximating sequence $(D_n,g_n)$. By \cref{justification for writing pairs}, $\widehat{D}$ is determined by $(D,g_D)$ and so we write $\widehat{D}=(D,g_D)$ from now on for  elements in $\widehat{M}_{\mo,\Q}^\YZ(U)$. 
\end{art}

\begin{remark}  \label{same Green function associated to YZ}
	{We use the above assumptions and notation.} By \cref{proposition:green functions for adelic divisors}, 
	if two elements $(D,g_D)$ and $(D,g_D')$ in $\widehat{M}_{\mo,\Q}^\YZ(U)$ have the same geometric part $D \in \widehat{M}_{\gm,\Q}^\YZ(U)$, then $g_D'-g_D= h $  for a continuous function $h$ on $U^\an$ with $h=o(g_B)$ along the boundary $\partial X_0 = X_0 \setminus U$ where $X_0$ is the proper $K$-model  for the Green function $g_B$ of the cofinal boundary divisor $\overline B$ from Remark \ref{interpretation of YZ-elements}. Together with the exact sequence in \cref{proposition:green functions for adelic divisors},  
	we see that the elements in $\widehat{M}_{\gm,\Q}^\YZ(U)$ allow Green functions on $U^\an$ with rather strong singularities along the boundary. 
\end{remark}

{In the following, $U$ is any algebraic variety over $K$.}  
		
\begin{remark} \label{trivial valuation}
	If $v$ is the trivial valuation, then for any proper $K$-model $X$ of $U$ and $(D,g_D)\in M_{\mo,\Q}(X)$, we note that the model Green function $g_D$ is determined by its $\Q$-Cartier divisor $D$ and hence we get
	$${(\widehat{M}_{\mo,\Q}^{\YZ}(U),\widehat{N}_{\mo,\Q}^{\YZ}(U))}=(\widehat{M}_{\gm,\Q}^\YZ(U), \widehat{N}_{\gm,\Q}^\YZ(U)).$$
	To get a richer theory for applications, it is then better to rely our constructions on twisted  Fubini--Study metrics as indicated in the following.
\end{remark}
		
\begin{art} \label{twisted Fubini-Study metrics and YZ-completion}
	For a proper $K$-model $X$ of $U$, let $(M_{\tFS,\Q}(X),N_{\tFS,\Q}(X))$ be the abstract divisorial space over $\Q$ based on twisted Fubini--Study metrics which we introduced in \cref{abstract divisorial space and FS}. 
	We define the abstract divisorial space over $\Q$
	\begin{equation} \label{direct limit for tFS}
		(M_{\tFS,\Q}(U), N_{{\tFS},\Q}(U)) \coloneqq \varinjlim_X (M_{\tFS,\Q}(X), {N_{\tFS,\Q}(X)})
	\end{equation}
	where $X$ ranges over all proper $K$-models of $U$. Let $S$ be the directed subset of $M_{\tFS,\Q}(U)$ given by the elements $(D,g_D)$ represented by effective $\Q$-Cartier divisors on a proper $K$-model $X$ of $U$ with support in $X \setminus U$ and by a 
	Green function $g_D \geq 0$. Then we define the abstract divisorial space over $\Q$ 
	\begin{equation} \label{YZ completion for tFS}
		(\widehat{M}_{\tFS,\Q}^\YZ(U),\widehat{N}_{\tFS,\Q}^\YZ(U)) \coloneqq (\widehat{M}_{\tFS,\Q}^S(U),\widehat{N}_{{\tFS},\Q}^S(U))
	\end{equation}
	using the $S$-completion of abstract divisorial spaces from \cref{S-completion}. 
	Note that for proper non-projective varieties $X$, the abstract divisorial space $(M_{\tFS,\Q}(X), {N_{\tFS,\Q}(X)})$  might have trivial geometric part as there are such $X$ with no non-trivial semiample (even nef) line bundles \cite{fujino2005smooth}, but even then we have $N_{\tFS,\Q}(X) = \{(0,r)\mid r\in{\R}\}$.
\end{art}

{The following result compares these notions with the model case. It will not be used in the sequel.}		
\begin{prop} \label{Yuan-Zhang's completion vs twisted FS}
	We assume that $v$ is non-trivial and that $U$ is quasi-projective. Then 
	\begin{equation} \label{inclusion for completions}
		\widehat{M}_{\tFS,\Q}^\YZ(U) \subset \widehat{M}_{\mo,\Q}^\YZ(U)  \quad{\text{and}} \quad
		\widehat{N}_{\tFS,\Q}^\YZ(U) \subset \widehat{N}_{\mo,\Q}^\YZ(U) .
	\end{equation}
\end{prop}
\begin{proof}
	The beginning of the proof works for all algebraic varieties $U$ over $K$. For a proper $K$-model $X$ of $U$, we have introduced the abstract divisorial space $(M_{\FS,\Q}(X), N_{{\FS},\Q}(X))$ over $\Q$ based on  $\FS$-Green functions on $X^\an$, see \cref{abstract divisorial space and FS}. Similarly as above, we define
	\begin{equation} \label{direct limit for FS}
		(M_{\FS,\Q}(U), N_{{\FS},\Q}(U)) \coloneqq \varinjlim_X (M_{\FS,\Q}(X), {N_{\FS,\Q}(X)})
	\end{equation}
	where $X$ ranges over all proper $K$-models of $U$. It follows from \cref{Fubini-Study metric} that every $\FS$-Green function $g_D$ for a $\Q$-Cartier divisor $D$ on a proper $K$-model $X$ of $U$ is a semipositive model Green function. We conclude that ${N}_{\FS,\Q}(X) \subset  {N}_{\mo,\Q}(X)$ and ${M}_{\FS,\Q}(X) \subset  {M}_{\mo,\Q}(X)$. 
	By passing to the direct limit, we get 
	\begin{equation} \label{inclusion for FS}
		{M}_{\FS,\Q}(U) \subset {M}_{\mo,\Q}(U)  \quad{\text{and}} \quad
		{N}_{{\FS},\Q}(U) \subset {N}_{\mo,\Q}(U).
	\end{equation} 
	For the abstract divisorial space   $(M_{\SP,\Q}(X),N_{\SP,\Q}(X))$ over $\Q$ given by Zhang's construction in Remark \ref{semipositive Green functions in sense of Zhang}, we set
	\begin{equation} \label{direct limit for Zhang}
		(M_{\SP,\Q}(U), N_{\SP,\Q}(U)) \coloneqq \varinjlim_X (M_{\SP,\Q}(X), {N_{\SP,\Q}(X)})
	\end{equation}
	where $X$ ranges over all proper $K$-models of $U$. Using that $v$ is non-trivial, we can approximate every $r \in \R$ by a sequence from the divisorial hull of $v(K^\times)$ and hence twisted $\FS$-Green functions on $X$ are semipositive Green functions in the sense of Zhang. Similarly as above, we get inclusions ${N}_{\tFS,\Q}(X) \subset  {N}_{\SP,\Q}(X)$ and ${M}_{\tFS,\Q}(X) \subset  {M}_{\SP,\Q}(X)$. By passing to the direct limit over all proper $K$-models $X$ of $U$, we get 
	\begin{equation} \label{inclusion for tFS}
		M_{\tFS,\Q}(U) \subset {M}_{\SP,\Q}(U)  \quad{\text{and}} \quad
		N_{{\tFS},\Q}(U) \subset {N}_{\SP,\Q}(U).
	\end{equation} 
	We recall from \cref{interpretation of YZ-elements} that the Yuan--Zhang completion $(\widehat{M}_{\mo,\Q}^\YZ(U),\widehat{N}_{\mo,\Q}^\YZ(U))$ is the $b$-completion of  $(M_{\mo,\Q}(U),N_{\mo,\Q}(U))$ with respect an element  $b=(B,g_B)$ induced by a $\Q$-Cartier divisor $B$ on a proper $K$-model $X_0$ of $U$ with $|B|=X_0\setminus U$ and by a Green function $g_B>0$ for $B$ on $X_0^\an$. We note that we can see $b$ also as an element of $M_{\SP,\Q}(U)_{\geq 0}$ using the canonical morphism 
	\begin{equation} \label{model to SP}
		(M_{\mo,\Q}(U),N_{\mo,\Q}(U)) \longrightarrow (M_{\SP,\Q}(U), N_{\SP,\Q}(U))
	\end{equation}
	induced by the corresponding morphisms for proper $K$-models $X$ of $U$ and then using direct limits. Then the universal property of $b$-completions from \cref{b-completion} gives a canonical morphism
	\begin{equation} \label{isomorphism of completions}
		(\widehat{M}_{\mo,\Q}^b(U),\widehat{N}_{\mo,\Q}^b(U)) \longrightarrow (\widehat{M}_{\SP,\Q}^b(U), \widehat{N}_{\SP,\Q}^b(U))
	\end{equation}
	of abstract divisorial spaces over $\Q$. We claim that this is an isomorphism. 
	By \cref{Zhang metrics as divisorial space}, Zhang's abstract divisorial space $(M_{\SP,\Q}(X),N_{\SP,\Q}(X))$ is the $b_r$-completion of $(M_{\mo,\Q}(X),N_{\mo,\Q}(X))$ with respect to $b_r=(0,r)$ for any $r \in \R_{>0}$. As the $b_r$-completion is independent of the choice of $r$, we may assume that $r \leq g_B$ on $X_0 \setminus |B|$. Then again the universal property of $b$-completions from \cref{b-completion} gives a canonical morphism
	$$(M_{\SP,\Q}(X),N_{\SP,\Q}(X))=(\widehat{M}_{\mo,\Q}^{b_r}(X),\widehat{N}_{\mo,\Q}^{b_r}(X)) \longrightarrow (\widehat{M}_{\mo,\Q}^b(U),\widehat{N}_{\mo,\Q}^b(U))$$
	of abstract divisorial spaces over $\Q$. Passing to the direct limit over all proper $K$-models of $X$,
	we get a canonical morphism 
	\begin{equation} \label{inverse for completions}
		(M_{\SP,\Q}(U),N_{\SP,\Q}(U))\longrightarrow (\widehat{M}_{\mo,\Q}^b(U),\widehat{N}_{\mo,\Q}^b(U))
	\end{equation}
of abstract divisorial spaces over $\Q$. Then the universal property of $b$-completions gives a morphism inverse to the morphism in \eqref{isomorphism of completions} proving that the latter is an isomorphism.
			
	Now we use that $U$ is quasi-projective over $K$.  Then we can choose a projective $K$-model $X_0$ of $U$ in \cref{interpretation of YZ-elements} and we may assume that the element  $b=(B,g_B)$ considered above is from $M_{\FS,\Q}(X_0)_{\geq 0}$ (using that every line bundle on a projective variety is the difference of two very ample line bundles; in the non-archimedean case, we apply this to the projective  $\kcirc$-model $\Xcal_0$). Similarly as in \cref{interpretation of YZ-elements}, we get that  $S'=\N b$  is cofinal in the directed subset of $M_{\tFS,\Q}(U)_{\geq 0}$ considered for completion in \eqref{YZ completion for tFS} and hence \cref{final object} shows
	\begin{equation} \label{tFS as b-completion}
		(\widehat{M}_{\tFS,\Q}^\YZ(U),\widehat{N}_{\tFS,\Q}^\YZ(U))=(\widehat{M}_{\tFS,\Q}^b(U), \widehat{N}_{{\tFS},\Q}^b(U)).
	\end{equation}
	Applying the $b$-completions in \eqref{inclusion for tFS}, we get \eqref{inclusion for completions}.			
\end{proof}

\begin{remark}
	Assume that $X$ is a projective variety over $K$. A Green function $g$ for a semiample divisor $D$ is plurisubharmonic in the sense of \cite[Definition~3.3.4]{chen2021arithmetic} if and only if $(D,g)\in \widehat{N}_{\mo,\Q}^\YZ(X)={N}_{\SP,\Q}(X)$ when $v$ is non-trivial, or $(D,g)\in \widehat{N}_{\tFS,\Q}^\YZ(X)$ when $v$ is trivial. Indeed, when $v$ is archimedean, this is from \cite[Theorem~2.3.7]{chen2020arakelov}; when $v$ is non-trivial and non-archimedean, this is from \cite[Theorem~3.2.19]{chen2021arithmetic}; when $v$ is trivial, this is from \cite[Theorem~1.19~(ii)]{boucksom2021spaces}.
\end{remark}

Now we will give the definition of the largest abstract divisorial spaces on $U$ which we will use in the local theory. 

\begin{definition} \label{strongly nef and nef compactified divisors}
We call a compactified metrized divisor $(D,g)$ on $U$ \emph{strongly nef} if $(D,g) \in \widehat{N}_{\mo,\Q}^\YZ(U)$ in case of $v$ a non-trivial valuation and if $(D,g) \in \widehat{N}_{\tFS,\Q}^\YZ(U)$ in the trivially valued case. 
{We denote the cone of strongly nef compactified metrized divisors by $\widehat{\Div}_\Q(U)_\snef$.}

We say that a compactified metrized divisor $(D,g)$ is \emph{nef} if it is in the closure of {$\widehat{\Div}_\Q(U)_\snef$} in {$\widehat{M}_{\mo,\Q}^\YZ(U)$ (resp. $\widehat{M}_{\tFS,\Q}^\YZ(U)$) if $v$ is non-trivial (resp. $v$ is trivial)}, where the closure is in the sense of \cref{closure for finite subspaces}. The nef compactified metrized divisors form a cone and we get an abstract divisorial space over $\Q$ denoted by
	{\[\left(\widehat{\Div}_\Q(U)_{\mathrm{int}},\widehat{\Div}_\Q(U)_{\nef}\right)\coloneq\begin{cases}
	\left(\widehat{M}_{\mo,\Q}^\YZ(U), \overline{\widehat{\Div}_\Q(U)_\snef}\right), & \text{ if $v$ is non-trivial,}\\
	\left(\widehat{M}_{\tFS,\Q}^\YZ(U), \overline{\widehat{\Div}_\Q(U)_\snef}\right), & \text{ if $v$ is trivial.}
\end{cases}\] }
\end{definition}

\begin{remark} \label{nef become strongly nef after shrinking U}
	If we have finitely many  nef compactified metrized divisors $\overline{D_j}=(D_j,g_j)$ for  $j=0,\dots, k$, then by \cref{lemma:closure in finite subspace topology}, there exists a strongly nef compactified metrized divisor $\overline{H_j} =(H_j,h_j)$ of $U$ such that $\overline{D_j}+\frac{1}{n}\overline{H_j}$ is also strongly nef for all $n \in \N_{\geq 0}$. Choosing an open subset $U'$ of $U$ sufficiently small so that it does not intersect the union of the supports of the divisors $H_j|_U$, it is easy to see that the restriction of each $\overline{D_j}$ to $U'$ is a strongly nef compactified metrized divisor of $U'$.
\end{remark}

\begin{art} \label{shrinking U}
Let $\varphi\colon U' \to U$ be a dominant morphism of algebraic varieties over $K$. By Nagata's compactification theorem, every proper $K$-model of $U$ is dominated by  a proper $K$-model of $U'$ using an extension of $\varphi$. Using pull-back of model metrics,  we get a canonical morphism
$$\varphi^*\colon (M_{\mo,\Q}(U),N_{\mo,\Q}(U)) \longrightarrow (M_{\mo,\Q}(U'),N_{\mo,\Q}(U'))$$
of abstract divisorial spaces over $\Q$. To pass to completions, we have to consider the directed sets $S$ for $U$ and $S'$ for $U'$ as in \ref{semipositive Green functions in sense of Yuan-Zhang}. Since $S$ maps to $S'$, the universal properties of completions in \cref{b-completion} and of direct limits in Lemma \ref{direct limits} give a canonical morphism 
\begin{equation} \label{restriction}
\varphi^* \colon	(\widehat{M}_{\mo,\Q}^\YZ(U), \widehat{N}_{\mo,\Q}^\YZ(U)) \longrightarrow (\widehat{M}_{\mo,\Q}^\YZ(U'), \widehat{N}_{\mo,\Q}^\YZ(U'))
\end{equation}
which we call the \emph{pull-back}. In the special case when $U'$ is an open subset of $U$ and $\varphi$ is the corresponding open immersion, then we call the pull-back just the \emph{restriction to $U'$}.
\end{art}
	
\begin{example}   \label{running example integrable local}
	We continue with our running example from \S \ref{subsec: a running example}, so let $K=\Q_v$ for a place $v$ of $\Q$ and let $U=\mathbb A_K^1$ embedded as usual in $X=\mathbb P_K^1$ by using homogeneous coordinates $x_0,x_1$ and $U= \mathbb P_K^1 \setminus D$ for the toric divisor $D= \mathrm{div}(x_0)$ of $X$. Similarly as in \cref{example: projective line}, we have then 
	$$\widetilde{\Div}_\Q(U)= \Div_\Q(U) \oplus \R[\infty]= \left(\bigoplus_{u}\Q[u]\right) \oplus \R  [\infty]$$
	where $u$ ranges over all closed points of $U=\mathbb A_K^1$. We want to describe elements  
	$(D,g_D) \in \widehat{\Div}_\Q(U)_{\snef}$. For simplicity, we assume that the underlying geometric compactified divisor in $\widetilde{\Div}_\Q(U)$ is $D=\mathrm{div}(x_0)=[\infty]$ and that the Green function $g_D$ of $D$ corresponds to a toric metric as in \S \ref{subsec: a running example} and as in \cref{running example for local Zhang}. This means that there is a continuous function $\psi:\R \to \R$ such that $g_D = -\psi \circ \trop$ on $T^\an=U^\an \setminus\{0\}$. Similarly as in \cref{compactified metrics in running example}, we choose the cofinal boundary divisor $\overline B$ with $B=D=[\infty]$ and with Green function $g_B$ corresponding to the global section $x_0$ of $\OO_X(1)$ endowed with the Fubini--Study metric twisted by some $r>0$. By definition, our strongly nef compactified metrized divisor $(D,g_D)$ is given as a limit of a sequence $(D_i,g_i)$ of semipositive model metrized divisors which means that for every $\varepsilon \in \Q_{>0}$, we have 
	\begin{equation} \label{Cauchy in running example}
		- \varepsilon (D,g_B) \leq (D_i,g_i)-(D,g_D) \leq \varepsilon (D,g_B)
	\end{equation}
for $i \gg 0$. Using this only for the first component, we deduce that $D_i|_U$ is the zero-divisor for $i \gg 0$. Using that $X \setminus U=D$, we see that by scaling we may assume $D_i=D$ for all $i \in \N$.

Recall that $T$ denotes the open dense torus of $U^\an =(\mathbb A_K^1)^\an$ and that $t=x_1/x_0$ is the affine coordinate. In a next step, we show that we also may assume that all Green functions $g_i$ are toric. This is due to a canonical process called \emph{torification} of a Green function $g$ for $D$ or what is equivalent of the corresponding continuous metric of $\OO_X(1)$. We note that the compact torus $\mathbb S=\{t \in (\mathbb A_K^1)^\an \mid |t|=1\}$ acts on the fiber $\trop^{-1}(u)$ for every $u \in \R$, so in the archimedean case we can average $g$ relative to this action with respect to the Haar measure of $\mathbb S$ which gives a toric Green function $g_{\mathbb S}$, see \cite[Definition 4.3.3]{BPS}. In the non-archimedean case, we use the distinguished weighted Gauss point $\eta_u$ in $\trop^{-1}(u)$ which is given by the multiplicative seminorm
$$|f(\eta_u)|= \max_{k=i,\dots,j}e^{-uk} |a_k|$$
for $f=\sum_{k=i}^j a_k t^k$.  
Then $S(T)\coloneqq \{\eta_u \mid u \in  \R\}$ is called the \emph{tropical skeleton} of $T$. The tropical skeleton $S(T)$ is a closed subset of $T^\an$ which can be identified with $ \R$ by using the inverse maps $\trop|_{S(T)}$ and $u \to \eta_u$. Hence we may see the restriction of $g$ to $S(T)$ as a continuous real function on $\R$ which we may compose with $\trop$ to get a toric Green function $g_{\mathbb S}$ of $D$, see \cite[Definition 4.3.3]{BPS}. In any case, the torification  leaves  toric Green functions invariant \cite[Proposition 4.3.4]{BPS}, maps model Green functions to toric model Green functions  and if they are additionally semipositive, then the torification is also semipositive \cite[Proposition 4.4.2]{BPS}. Applying torification to \eqref{Cauchy in running example}, we conclude that we may assume $g_i$ to be a semipositive toric model Green function. As we have seen in Example \ref{running example for local Zhang}, there is a corresponding  concave piecewise affine function $\psi_i\colon \R \to \R$  such that $g_i=-\psi_i \circ \trop$ on $T^\an$. Moreover, the asymptotic slope of $\psi_i$ are $1$ for $u\to -\infty$ and $0$ for $u \to \infty$, and $\psi_i$ is smooth in the archimedean case and piecewise affine with rational slopes and rational constant terms in the non-archimedean case. It follows from \eqref{Cauchy in running example} that the functions $\psi_i$ converge locally uniformly to the corresponding function $\psi$ of $g_D$. In particular, we get that $\psi$ is a concave function with asymptotic slope $1$ for $u \to -\infty$ and with $\lim_{u \to \infty}\psi(u)\in \R$. We leave it as an exercise in Analysis that conversely, every such concave function can be approximated in the above way. 
\end{example}

\begin{remark} \label{conclusion of running local example}
The above example shows that for $U=\mathbb A_K^1 \subset X=\mathbb P_K^1$ over $K=\Q_v$ and any toric geometric compactified divisor $D=a[0]+b[\infty]$ of $X$ with toric Green function $g_D$, we have $(D,g_D) \in \widehat{\Div}_\Q(U)_\snef$ if and only if the corresponding concave function $\psi$ has asymptotic slope $b$ for $u \to -\infty$ and asymptotic slope $-a$ for $u \to \infty$. Indeed, toric means that the support of $D$ is in $\{0,\infty\}$ and \cref{example: projective line} shows that $a \in \Q$ and $b \in \R$. Then $D$ is rationally equivalent to $(a+b)[\infty]$ and this rational equivalence means a linear change of the function $\Psi$, so we easily deduce from \cref{running example integrable local} the claim. Note that existence of such a $g_D$ or equivalently existence of such a concave function $\psi$ is only possible when $a+b \geq 0$. For the description of toric compactified metrized divisors in general, we refer to \cite[Chapter 4]{peralta-thesis24}.
\end{remark}

\begin{remark} \label{dc functions}
For the multiplicative torus $T=\mathbb G_m$ over $\Q_v$, one can show similarly as above that a toric Green function $g_D$ for a toric geometric compactified divisor $D$ of $X=\mathbb P^1$ induces  $(D,g_D) \in  \widehat{\Div}_\Q(U)_{\integrable}$ if and only if the function $\psi\colon \R \to \R$ with $g_D= -\psi \circ \trop$ is the difference of two concave functions with finite asymptotic slopes for $u \to \pm \infty$.  It is a well known fact from optimization theory (see e.g.~\cite{hartman}) that a function on $\R$ is the difference of two concave functions if and only if it has right and left derivatives and these derivatives are of bounded variation on every bounded closed interval. We see from \eqref{compactified Green function in running example} in  \cref{compactified metrics in running example} that the inclusion $\widehat{\Div}_\Q(U)_{\integrable} \subset \widehat{\Div}_\Q(U)_{\cpt}$ is strict as there are many functions $\varphi\colon \R \to \R$ with $\varphi=o(g_B)$ which do not satisfy these requirements. 
\end{remark}

\subsection{Compactified metrized line bundles}
\label{subsection:local adelic line bundles}

Recall that $U$ is an algebraic variety over the complete field $K$ and that $\widehat \Pic(U)$ denotes the group of isometry classes of continuously metrized line bundles on $U^\an$.

\begin{art} \label{boundary topology on metrized line bundles}
We fix a cofinal boundary divisor $(X_0, \overline{B})$ of $U$ with $\overline{B}=(B, g_B)$. Notice that $g_B$ is a continuous function on $U^\an$. The \emph{boundary topology} on $\widehat{\Pic}_\Q(U)$ is defined such that a {basis of neighborhoods} of an element $\overline{L}=(L, \metr)\in \widehat{\Pic}_\Q(U)$ of the topology is given by
\[B(r,\overline{L})\coloneq \left\{(L,\metr')\in \widehat\Pic_\Q(U) \,\middle\vert\, -rg_B\leq \log\frac{\metr'}{\metr}\leq rg_B\right\}, \quad r\in \Q_{>0}.\]
The boundary topology is independent of the choice of the cofinal boundary divisor. Similarly as in \cref{def: b-metric}, we can define a pseudo-metric $d_{\overline B}$ which defines the boundary topology.  Then $\widehat{\Pic}_\Q(U)$ is complete since the space of continuous functions on $U^\an$ is complete, see \cite[Lemma~3.6.3]{yuan2021adelic}.
\end{art}

\begin{art} \label{definition of model  line bundles}
Let $P_\Q(U)$ be the subspace of $\widehat{\Pic}_\Q(U)$ generated by all $(L,\metr)$ with $\metr$ a model metric if the valuation $v$ on $K$ is non-trivial (resp.~a twisted Fubini--Study metric if $v$ is trivial). We denote by $Q_\Q(U)$ the subcone of $P_\Q(U)$ induced by the  $(L,\metr)$ as above with $\metr$ semipositive (resp.~a twisted Fubini--Study metric). This means 
	\[\text{$P_\Q(U) = \varinjlim_{X}P_\Q(X)$  and  $\quad Q_\Q(U) = \varinjlim_{X}Q_\Q(X)$}\]
where $X$ ranges over all proper $K$-models of $U$. Indeed, we may restrict the direct limit to the proper $K$-models which are integrally closed and then the map $P_{\Q}(X) \to P_{\Q}(U)$ becomes injective by \cref{lemma: injectivity of restriction from integrally closed}, hence the above direct limits become unions. 
\end{art}

\begin{definition} \label{definition: compactified line bundles}
Using the above terminology, we define $\widehat{\Pic}_\Q(U)_\cpt$ as the closure of $P_\Q(U)$ in $\widehat{\Pic}_\Q(U)$ with respect to the boundary topology. By \ref{boundary topology on metrized line bundles}, this is the completion of $P_\Q(U)$ with respect to the boundary topology. 

We say that $(L,\metr) \in \widehat{\Pic}_\Q(U)$ is \emph{strongly nef} if $(L,\metr)$ is in the closure of $Q_{\Q}(U)$ with respect to the boundary topology. The strongly nef elements form a cone $\widehat{\Pic}_\Q(U)_\snef$ in $\widehat{\Pic}_\Q(U)$. We define the subspace
$$\widehat{\Pic}_\Q(U)_\integrable \coloneqq \widehat{\Pic}_\Q(U)_\snef - \widehat{\Pic}_\Q(U)_\snef$$
of $\widehat{\Pic}_\Q(U)$. Finally, the \emph{nef cone} $\widehat{\Pic}_\Q(U)_\nef$ is defined as the closure of $\widehat{\Pic}_\Q(U)_\snef$ in $\widehat{\Pic}_\Q(U)_\integrable$ with respect to the finite subspace topology from \cref{finite subspace topology}. 
\end{definition}

\begin{prop} \label{connecting metrized divisors and metrized line bundles}
For $(D,g_D) \in \widehat{\Div}_\Q(U)$, let $\metr_D$ be the metric on $\OO_U(D)$ associated to $g_D$ as in \eqref{metric and Green functions}. This gives a surjective map
$$\widehat{\Div}_\Q(U) \longrightarrow \widehat{\Pic}_\Q(U) \, , \quad (D,g_D) \mapsto (\OO(D),\metr_D)$$ 
which maps the sets $\widehat{\Div}_\Q(U)_\cpt$,  $\widehat{\Div}_\Q(U)_\snef$, $\widehat{\Div}_\Q(U)_\nef$ and $\widehat{\Div}_\Q(U)_\integrable$ onto the sets  $\widehat{\Pic}_\Q(U)_\cpt$, $\widehat{\Pic}_\Q(U)_\snef$, $\widehat{\Pic}_\Q(U)_\nef$ and $\widehat{\Pic}_\Q(U)_\integrable$, respectively.
\end{prop}
\begin{proof}
Since any line bundle $L$ of $U$ has a non-zero meromorphic section, it is clear that the correspondence between Green functions and continuous metrics from \ref{Green functions} yields that the map is surjective. We prove the remaining cases in the non-trivially valued case. The arguments in the trivially valued case are similar. Obviously, the space $\widehat{\Div}_\Q(U)_\mo$ of model divisors is mapped onto $P_\Q(U)$ and the cone $N_{\mo,\Q}(U)$ of nef model divisors is mapped onto $Q_{\Q}(U)$. The crucial point is to show that the map 
$\widehat{\Div}_\Q(U)_\mo \to P_\Q(U)$ is continuous with respect to the boundary topologies. Let $\overline B=(B,g_B)$ be a boundary divisor of $U$ and $r \in \Q_{>0}$. For $(D,g), (D',g') \in \widehat{\Div}_\Q(U)$ with 
$$ -r \overline B \leq (D',g') - (D,g) \leq r \overline B.$$
Then we have $D=D'$ as divisors on $U$ and $-rg_B \leq g'-g \leq r g_B$ on $U^\an$. For the metrics $\metr,\metr'$ of $L=\OO_U(D)$ corresponding to $g,g'$, we have $g'-g=\log(\metr/\metr')$ and hence $(L,\metr') \in B(r,(L,\metr))$ using the neighborhoods from \cref{boundary topology on metrized line bundles}. This proves continuity of the map $\widehat{\Div}_\Q(U)_\mo \to P_\Q(U)$ with respect to the boundary topology. Moreover, the argument shows that the map $B(r,\overline D) \to B(r,\overline L)$ is bijective, where $B(r,\overline D)=\{\overline{D'}\in {\widehat{\Div}_\Q(U)}\mid -r \overline B \leq \overline{D'} - \overline{D} \leq r \overline B\}$. 
 Hence the spaces $\widehat{\Div}_\Q(U)_\mo$ and $P_\Q(U)$ are locally isometric. The remaining claims then follow easily from the definitions as continuous maps extend to completions.
\end{proof}

\begin{prop} \label{pull-back in local case}
	Let $\varphi\colon U'\to U$ be a morphism of algebraic varieties over $K$. Then the pull-back $\varphi^*\colon \widehat{\Pic}_{\Q}(U)\to \widehat{\Pic}_{\Q}(U')$ maps the sets $\widehat{\Pic}_{\Q}(U)_\cpt$, $\widehat{\Pic}_\Q(U)_\snef$, $\widehat{\Pic}_\Q(U)_\nef$ and $\widehat{\Pic}_\Q(U)_\integrable$ to the sets $\widehat{\Pic}_{\Q}(U')_\cpt$, $\widehat{\Pic}_\Q(U')_\snef$, $\widehat{\Pic}_\Q(U')_\nef$ and $\widehat{\Pic}_\Q(U')_\integrable$, respectively.
\end{prop}
\begin{proof}
	We only consider the case where $v$ is non-trivial, the trivially valued case is similar. For a {proper $K$-model} $X$ of $U$, we can find a {proper $K$-model} $X'$ of $U'$ and a morphism $X'\to X$ extending $\varphi$, such a morphism is still denoted by $\varphi$. Then we have the pull-back $\varphi^*\colon P_\Q(X)\to  P_\Q(X')$. Hence  $\varphi^*(P_\Q(U))\subset P_\Q(U')$. It is similar for semipositive model metrized divisors. We have to show that $\varphi^*$ is continuous with respect to the boundary topologies given in \cref{boundary topology on metrized line bundles}.
	{Recall that we fix a cofinal boundary divisor $(X_0, \overline{B})$.} 
	We choose an extension $\varphi_0\colon X'_0 \to X_0$ to a suitable proper $K$-model $X_0'$ as above. Notice that $\overline{\varphi(U')}=\varphi(X')$ in $X$ and $\varphi(U')\cap |B|=\emptyset$, then $\varphi^*\overline{B}\in \widehat{\Div}_{\Q}(X')$ is always defined, see \cite[Lemma~7.1.29, Lemma~7.1.33~(2)]{liu2006algebraic}, and {$\varphi^*\overline{B}$ is a weak boundary divisor of $U'$}. By \cref{remarks about boundary topology}, we can find a cofinal boundary divisor $\overline{B'}$ of $U'$ such that $\varphi^*\overline{B}\leq \overline{B'}$.
This immediately implies that $\varphi^*\colon \widehat{\Pic}_{\Q}(U)\to \widehat{\Pic}_{\Q}(U')$ is continuous with respect to the boundary topologies which proves all the claims except for the nef cones. For the latter, we use also the obvious fact that pull-back is continuous with respect to the finite subspace topologies.
\end{proof}

\begin{art} \label{forgetting map local}
Recall from \cref{definition of compactified line bundles} that the groups $P_\Q(U)$ and $Q_\Q(U)$ have their geometric analogues $P_{\gm,\Q}(U)$ and $Q_{\gm,\Q}(U)$. Obviously, we have a canonical map $P_\Q(U)\to P_{\gm,\Q}(U)$ which maps $Q_\Q(U)$ into $Q_{\gm,\Q}(U)$. By continuity, we get a linear map
	\begin{align}\label{local forgetting map}
		\widehat{\Pic}_\Q(U)_\cpt\longrightarrow \widetilde{\Pic}_\Q(U)_\cpt
	\end{align}
which maps the sets $\widehat{\Pic}_\Q(U)_\snef$, $\widehat{\Pic}_\Q(U)_\nef$ and $\widehat{\Pic}_\Q(U)_\integrable$ to the sets  $\widetilde{\Pic}_{\gm,\Q}(U)_\snef$, $\widehat{\Pic}_{\gm,\Q}(U)_\nef$ and $\widehat{\Pic}_{\gm,\Q}(U)_\integrable$, respectively.
\end{art}

\subsection{Mixed Monge--Amp\`ere  measures} 
\label{subsection: mixed MA measures}

In the archimedean case, mixed Monge--Amp\`ere  measures are a classical tool in pluripotential theory. In the non-archimedean case, they were introduced by Chambert-Loir \cite{chambert2006measures}. The following summarizes their properties in the compactified setting.

\begin{proposition} 	\label{proposition:measures for nef adelic}
	There is a unique way to associate to any $d$-dimensional algebraic variety $U$ over $K$ and to any family $\overline{L_1}, \dots, \overline{L_d}\in \widehat{\Pic}_\Q(U)_{\nef}$ a positive Radon measure $c_1(\overline{L_1})\wedge\cdots\wedge c_1(\overline{L_d})$ on $U^\an$ such that the following properties hold:
	\begin{enumeratea}
		\item \label{MA measure is multilinear and symmetric} The measure $c_1(\overline{L_1})\wedge\cdots\wedge c_1(\overline{L_d})$ is multilinear and symmetric in $\overline{L_1}, \dots, \overline{L_d}$.
		\item \label{MA measure via dominant morphism} If $\varphi\colon U'\rightarrow U$ is a dominant morphism of $d$-dimensional algebraic varieties over $K$, then the following projection formula holds:
		\[\varphi_*(c_1(\varphi^*\overline{L_1})\wedge\cdots\wedge c_1(\varphi^*\overline{L_d})) = \deg(\varphi)c_1(\overline{L_1})\wedge\cdots\wedge c_1(\overline{L_d})\]
		where on the left we use the image measure of Borel measures.
		\item \label{MA measure of model metric} Assume $K$ algebraically closed, $U=X$  a  proper variety over $K$ and  $L_1,\dots, L_d$ are nef model line bundles. 
		\begin{itemize}
			\item If $v$ is archimedean, $c_1(\overline{L_1})\wedge\cdots\wedge c_1(\overline{L_d})$ is exactly given by the $\wedge$-product of first Chern forms $c_1(\overline{L_i})$ of $\overline{L_i}$.
			\item If $v$ is non-archimedean, $\mathcal{X}$ is a model of $X$ over $K^\circ$ with reduced special fiber $\mathcal{X}_s$, and the metric of $\overline{L_j}$ is induced by a  model $\mathcal{L}_j$ of $L_j$ on $\mathcal{X}$ for every $j=1,\dots, d$, then
			\[c_1(\overline{L_1})\wedge\cdots\wedge c_1(\overline{L_d})=\sum\limits_{Y}\mathcal{L}_1|_{Y}\cdots\mathcal{L}_d|_{Y}\delta_{\xi_Y},\]
			where $Y$ ranges over the irreducible components of $\mathcal{X}_s$ and $\delta_{\xi_Y}$ is the Dirac measure in the unique point $\xi_Y\in X^\an$ which reduces to the generic point of $Y$ in $\mathcal{X}_s$.
		\end{itemize}
		\item \label{MA measure via field extension} {Let $K'/K$ be an extension of valued fields of rank $\leq 1$, and let $\pi\colon U_{K'}\to U$ be the canonical morphism. Then 
		\[c_1(\overline{L_1})\wedge\cdots\wedge c_1(\overline{L_d})=\pi_*(c_1(\pi^*\overline{L_1})\wedge\cdots\wedge c_1(\pi^*\overline{L_d})).\]} 
		\item \label{MA measure via convergence in boundary topology} If $\overline{L_j}$ are limits of sequences of semipositive model metrized line bundles $(L_{j,n_j})_{n_j}$ with respect to the boundary topology, then
		\[c_1(\overline{L_{1,n_1}})\wedge\cdots\wedge c_1(\overline{L_{d,n_d}}) \overset{v}{\rightarrow } c_1(\overline{L_{1}})\wedge\cdots\wedge c_1(\overline{L_{d}}), \ \ {(n_1,\dots, n_d)\to\infty}\]
where the convergence is the vague convergence of measures (see \cref{def: vague convergence}).		
%(see \cref{def: vague convergence} for the vague convergence of measures).
		\item \label{MA measure via convergence in finite subspace topology} If $\overline{L}_j$ are limits of sequences of strongly nef line bundles $(\overline{L_{j,n_j}})_{n_j}$ with respect to the finite subspace topology of $\widehat{\Pic}_\Q(U)_\integrable$, then
		\[c_1(\overline{L_{1,n_1}})\wedge\cdots\wedge c_1(\overline{L_{d,n_d}}) \overset{w}{\rightarrow } c_1(\overline{L_{1}})\wedge\cdots\wedge c_1(\overline{L_{d}}), \ \ {(n_1,\dots, n_d)\to\infty}\]
where the convergence is the weak convergence of measures (see \cref{def: vague convergence}).
%		(see \cref{def: vague convergence} for the weak convergence of measures).
		\item \label{guo's theorem} (\cite[Theorem~1.2]{guo2025an}) The total mass is given by
		$$\int_{U^\an}c_1(\overline{L_1})\wedge\cdots\wedge c_1(\overline{L_d}) = \widetilde{L_1}\cdots \widetilde{L_d},$$
		where $\widetilde{L_i}\in \widetilde{\Pic}_\Q(U)_{\integrable}$ is the underlying compactified line bundle of $L_i$ defined in \cref{forgetting map local}, and $\widetilde{L_1}\cdots \widetilde{L_d}$ is the geometric intersection number defined in \cref{abstract divisorial space of line bundles}.
	\end{enumeratea}
\end{proposition}

Note that the notation $c_1(\overline{L_1})\wedge\cdots\wedge c_1(\overline{L_d})$ of the Radon measure is just formal and not meant as a product of first Chern forms, except for smooth metrics in the archimedean case. In the non-archimedean case, it would make sense as a product of delta-forms introduced in \cite{gubler2017a}, but we don't need that here.

\begin{proof} 
For nef model line bundles, this proposition is well-known, see \cite[Proposition~2.23]{boucksom2021non}. Uniqueness is clear  from \ref{MA measure is multilinear and symmetric},  \ref{MA measure via dominant morphism}, \ref{MA measure of model metric}, \ref{MA measure via field extension}, \ref{MA measure via convergence in boundary topology} and \cref{nef become strongly nef after shrinking U}. 

To show existence, we assume  first that $U$ is quasi-projective. Then Yuan and Zhang \cite[\S 3.6.7]{yuan2021adelic} constructed the mixed Monge--Amp\`ere measures using that semipositive compactified metrics are locally uniform limits of semipositive model metrics and then using the classical non-pluripolar mixed Monge--Amp\`ere measures in the archimedean case \cite{demailly2012complex} and the corresponding  theory of Chambert--Loir and Ducros \cite[\S 5.6]{chambert2012formes} in the non-archimedean case. Yuan and Zhang showed property \ref{MA measure via convergence in boundary topology} and hence \ref{MA measure is multilinear and symmetric}, \ref{MA measure via dominant morphism}, \ref{MA measure via field extension} follow from the model case by passing to {limits}. Property \ref{guo's theorem} is Guo's result \cite[Theorem 1.2]{guo2025an}. {For \ref{MA measure via convergence in finite subspace topology}, we take a finite-dimensional subspace $V$ of $\widehat{\Pic}_\Q(U)_{\mathrm{int}}$ containing $\overline{L_j}$, $\overline{L_{j,n_j}}$ and $\lim_{n_j\to\infty}\overline{L_{j,n_j}}=\overline{L_j}$ in $V$ for all $j$. We may assume that $V$ is generated by a basis $\overline{N_1},\dots, \overline{N_m} \in V$ of strongly nef line bundles. Write $\overline{L_j}= \sum_{k=1}^ma_{jk}\overline{N_k}$ and $\overline{L_{j,n_j}}=\sum_{k=1}^ma_{jn_jk}\overline{N_k}$. Then $\lim_{n_j\to\infty}a_{jn_jk}=a_{jk}$. By multilinearity and  \ref{guo's theorem}, it is easy to see that \ref{MA measure via convergence in finite subspace topology} holds.}

For an arbitrary algebraic variety $U$, we can use that $U$ is locally quasi-projective, so the same argument yields the non-pluripolar mixed Monge--Amp\`ere measures on $U^\an$. We can also use that the mixed Monge--Amp\`ere measures constructed in the quasi-projective case do not charge proper Zariski closed subsets as a consequence of Guo's result. It follows that we can construct the mixed Monge--Amp\`ere measures on an arbitrary algebraic variety $U$ by using a quasi-projective open covering $(U_i)_{i \in I}$ and then patching the mixed Monge--Amp\`ere measures for the $U_i$'s together. Properties \ref{MA measure is multilinear and symmetric}--\ref{guo's theorem} follow immediately from the quasi-projective case. 
\end{proof}

\begin{definition}
	Let $\overline{D_1}, \dots, \overline{D_d}\in \widehat{\Div}_\Q(U)_{\nef}$. We define a positive Radon measure associated to $\overline{D_1}, \dots, \overline{D_d}$ on $U^\an$ as
	\[c_1(\overline{D_1})\wedge\cdots\wedge c_1(\overline{D_d})\coloneq c_1(\OO_U(\overline{D_1}))\wedge\cdots\wedge c_1(\OO_U(\overline{D_d})).\]
	By linearity, we can assign a signed Radon measure $c_1(\overline{D_1})\wedge\cdots\wedge c_1(\overline{D_d})$ (resp. $c_1(\overline{L_1})\wedge\cdots\wedge c_1(\overline{L_d})$) to any family $\overline{D_1}, \dots, \overline{D_d} \in \widehat{\Div}_\Q(U)_{\mathrm{int}}$ (resp. $\overline{L_1},\dots, \overline{L_d}\in \widehat{\Pic}_\Q(U)_{\mathrm{int}}$).
\end{definition}

\begin{prop} 	\label{corollary:weakly convergence of full mass}
	Let  $\overline{D_{1}}, \dots, \overline{D_{k-1}} \in \widehat{\Div}_\Q(U)_{\nef}$ for some $k \in \{1,\dots,d\}$. Moreover, let $\overline{D_{k}}, \dots, \overline{D_{d}}\in \widehat{\Div}_\Q(U)_{\snef}$ with $\overline{D_{j}}$ given as the limit of a 
 a sequence $(\overline{D_{j,n_j}})_{n_j\in \N}$
	{of strongly nef compactified metrized divisors on $U$} for all $j \in \{k,\dots,d\}$. 
 Then
	\[c_1(\overline{D_{1}})\wedge\cdots\wedge c_1(\overline{D_{k-1}})\wedge c_1(\overline{D_{k,n_{k}}})\wedge\cdots\wedge c_1(\overline{D_{d,n_d}})\overset{w}{\longrightarrow}c_1(\overline{D_{1}})\wedge\cdots\wedge c_1(\overline{D_{d}})\]
	for $(n_{k},\dots, n_d)\to\infty$.
\end{prop}
\begin{proof} 
	By \cref{nef become strongly nef after shrinking U} and \cref{proposition:measures for nef adelic}~\ref{MA measure via dominant morphism}, we can assume that $\overline{D_1},\dots, \overline{D_d}$ are strongly nef after shrinking $U$ and hence we may assume that $k=1$.
	We write $\overline{D_j}=(D_j,g_{j})$, $\overline{D_{j,n_j}}=(D_{j,n_j},g_{j,n_j})$ for $n\in \N$, $1 \leq j\leq d$, and take a {cofinal} boundary divisor $(X_0, (B,g_B))$. Notice that ${D}_{j,n_j}$ converges to ${D}_j$ in $\widetilde{\Div}_\Q(U)_{\cpt}$ for any $1\leq j\leq d$. By \cref{thm:equivalence of vague convergence and weak convergence}, \cref{thm:algebraic interesection number} and \cref{proposition:measures for nef adelic}~\ref{guo's theorem}, it is sufficient to show that 
	\[c_1(\overline{D_{1,n_1}})\wedge\cdots\wedge c_1(\overline{D_{d,n_d}})\overset{v}{\longrightarrow}c_1(\overline{D_{1}})\wedge\cdots\wedge c_1(\overline{D_{d}}),  \ \ {(n_1,\dots, n_d)\to\infty}.\] 
	By \cref{proposition:measures for nef adelic}~\ref{MA measure via convergence in boundary topology}, this claim holds when $\overline{D_{j,n_j}}$ are semipositive model metrized divisors of $U$.
	In general, every $\overline{D_{j,n_j}}$ is a limit of elements of semipositive model metrized divisors of $U$
	and hence the claim follows easily from the special case.
\end{proof}

\begin{example} \label{MA measure in running example}
{We describe the Monge--Amp\`ere measures in our running example with $U=\mathbb A_K^1$ for $K=\Q_v$ the completion of $\Q$ with respect to a place $v$. We restrict our attention to the toric divisor $D=[\infty]$ and to toric Green functions $g_D$ for $D$ which means that there is a continuous function $\psi\colon \R \to \R$ with $g_D= -\psi \circ \trop$. Moreover, we assume that $(D,g_D)$ is strongly nef which means that $\psi$ is a concave function with asymptotic slope $1$ for $u \to -\infty$ and with $\lim_{u \to \infty} \psi(u) \in \R$, see Example \ref{running example integrable local}. Let $\mathrm{MA}(\psi)$ be the real Monge--Amp\`ere operator of the concave function $\psi$, see \cite[\S 2.7]{BPS}. In the archimedean case, we consider the compact torus $\mathbb S\coloneqq \{t \in U^\an \mid |t|=1\}$. Then the Monge--Amp\`ere measure $c_1(D,g_D)$ is the unique $\mathbb S$-invariant positive Radon measure $\mu$ on $T^\an$ with $\trop_*(\mu)=\mathrm{MA}(\psi)$. In the non-archimedean case, we may identify $\R$ with the tropical skeleton $S(T)$ in $T^\an$, see \cref{running example integrable local}, and then the Monge--Amp\`ere measure $c_1(D,g_D)$ is the unique positive Radon measure $\mu$ supported on $S(T)$ with $\trop_*(\mu)=\mathrm{MA}(\psi)$. For details and generalizations to toric varieties, we refer to \cite[Proposition~6.3.1,  6.3.2]{bgjk2021}.}
\end{example}

\begin{remark} \label{remark: no local heights}
{For a projective or more generally a proper $d$-dimensional variety $X$ over $K$, Monge--Amp\`ere integrals $\int_{X^\an} f \,d\mu$ are the special case of local heights $\lambda_{\overline{D_0}, \dots, \overline{D_d}}(X)$ where $D_0=0$ and where $\overline{D_0}=(0,f)$ for a continuous real function $f$ on $X^\an$. For arbitrary $\overline{D_0} \in \widehat{\Div}_\Q(X)_\nef$, the local height depends on the choice of the underlying $\Q$-Cartier divisors $D_i$ and is only defined if $|D_0| \cap \dots \cap |D_d|=\emptyset$ where $|D_i|$ denotes the support of $D_i$, see \cite{gubler2003canonical}. For a $d$-dimensional quasi-projective variety $U$ over $K$, an extension of local heights to $(\widehat{\Div}_\Q(U)_\snef)^{d+1}$ becomes messy as the condition  about the empty intersection of supports has to involve the boundary.  However, the completion with respect to the boundary topology allows to extend global heights as we will see in Section~\ref{section: The global boundary completion}.}
\end{remark}

%----------------------------------------------------------------------------------------
%	SECTION 5
%----------------------------------------------------------------------------------------

\section{Adelic curves} \label{sec: adelic curves}

In this section, we recall the definition of adelic curves in \cite{chen2020arakelov}.

\subsection{Adelic curves} \label{subsec: adelic curves}

\begin{definition}[\cite{chen2020arakelov}~\S~3.1] 	\label{def:adelic curve}
	Let $K$ be a field. We denote the set of absolute values on $K$ by $M_K$. An \emph{adelic structure on $K$} is a measure space $(\Omega,\mathcal{A},\nu)$ equipped with a map $\phi\colon \Omega\rightarrow M_K, \ \ \omega\mapsto \val_\omega$ such that for any $a\in K^\times$, the function $\log|a|\colon \Omega\rightarrow \R$ is $\mathcal{A}$-measurable, integrable with respect to $\nu$. We call $(K, (\Omega, \mathcal{A},\nu),\phi)$  an \emph{adelic curve}. Moreover, the space $\Omega$ and the map $\phi$ are called a \emph{parameter space of $M_K$} and a \emph{parameter map}. Further, if the equality 
	\begin{align}
		\int_{\Omega} \log|a|_\omega\ \nu(d\omega)=0
		\label{product formula}
	\end{align}
	holds for each $a\in K^\times$, then the adelic curve $(K, (\Omega, \mathcal{A},\nu),\phi)$ is said to be \emph{proper}. We call \eqref{product formula} the \emph{product formula}.
\end{definition}
\begin{remark} \label{assumeption on adelic curves}
	We set 
	\[\Omega_\infty\coloneq\{\omega\in \Omega\mid \text{$\phi(\omega)$ archimedean}\},\]
	\[\Omega_\mathrm{fin}\coloneq\{\omega\in \Omega\mid \text{$\phi(\omega)$ non-archimedean}\},\]
	\[\Omega_0\coloneq\{\omega\in \Omega\mid \text{$\phi(\omega)$ trivial}\}\subset \Omega_\mathrm{fin}.\]
	For any $\omega\in \Omega$, we set $K_\omega$ the completion of $K$ with respect to $\val_\omega$. By Ostrowski's theorem, for any $\omega\in\Omega_\infty$, we have $K_\omega= \R$ or $\C$. In this case, there is a function $\kappa\colon \Omega_\infty\to (0,1]$ such that $|a|_\omega=a^{\kappa(\omega)}$ for any $a\in\Q_{\geq0}$. 
	
	 {In this paper, we will always assume that adelic curves satisfy the following properties. 
	\begin{enumeratea}
		\item The function $\kappa\equiv 1$. As \cite{chen2024positivity} mentioned, this assumption is harmless since in general case we can replace the absolute values $\{\val_\omega\}_{\omega\in\Omega_\infty}$ by the usual ones and consider the measure $d\widetilde{\nu}=(\mathbbm{1}_{\Omega\setminus\Omega_\infty}+\kappa\mathbbm{1}_{\Omega_\infty})d\nu$.  Under this assumption, we have that $\nu(\Omega_\infty)<\infty$ by \cite[Proposition~3.1.2]{chen2020arakelov}.
	    \item The set $\nu(\mathcal{A})\not\subset\{0,+\infty\}$, otherwise, the theory is not interesting.
	\end{enumeratea}}	
\end{remark}

In this paper, for simplicity, we write $(K,\Omega,\mathcal{A},\nu)$ for an adelic curve. Given an adelic curve $S=(K,\Omega, \mathcal{A},\nu)$ and algebraic extension $L/K$, there is a canonical structure of adelic curve on $L$, denoted by $S_L=S\otimes_KL = (L,\Omega_L, \mathcal{A}_L,\nu_L)$, see \cite[\S 3.3, \S 3.4]{chen2020arakelov}. If $S$ is proper, then so is $S_L$.

\begin{definition} 	\label{def:heightfunction}
	Let $S=(K, \Omega, \mathcal{A},\nu)$ be a proper adelic curve, and $S_{\overline{K}}=(\overline{K}, \Omega_{\overline{K}}, \mathcal{A}_{\overline{K}},\nu_{\overline{K}})$. We define the \emph{$S$-height} $h_S\colon \overline{K} \rightarrow \R_{\geq 0}$ as follows: for $\alpha \in \overline{K}$, we set
	$$h_S(\alpha) =  \int_{\Omega_{\overline{K}}}\log\max\{1,|\alpha|_x\}\,\nu_{\overline{K}}(dx).$$
\end{definition}

\begin{example} \label{example:adelic structure of number fields}
	Let $K$ be a number field. Denote by $\Omega$ the set of all places of $K$, equipped with the discrete $\sigma$-algebra $\mathcal{A}$. For any $\omega\in\Omega$, let $\phi(\omega)=\val_\omega$ be the absolute value on $K$ in the equivalence class $\omega$, which extends either the usual absolute value when $\omega$ is archimedean, or the usual $p$-adic absolute values, i.e. $|p|_\omega=\frac{1}{p}$, when $\omega$ is non-archimedean. Let $\nu$ be the measure on $(\Omega,\mathcal{A})$ such that $\nu(\{\omega\})=[K_\omega\colon\Q_\omega]{/[K:\Q]}$. Then $(K,\Omega,\mathcal{A},\nu)$ is a proper adelic curve.
\end{example}

\subsection{Adelic vector bundles on adelic curves}

The arithmetic degree and the slopes are classical invariants of adelic vector bundles over a number field or a function field. These constructions were generalized by Chen and Moriwaki to the setting of adelic curves, see \cite[\S 2.5]{chen2024positivity}. For convenience of the reader, we recall their definitions here.

We fix a proper adelic curve $S=(K, \Omega, \mathcal{A},\nu)$. 

\begin{definition}
	An \emph{adelic vector bundle} $\overline{E}=(E,\metr)$ on $S$ is a finite dimensional vector space $E$ over $K$ with a family of norms $\metr=(\metr_\omega)_{\omega\in\Omega}$ on $E_\omega\coloneq E\otimes_KK_\omega$, which is assumed to be ultrametric over non-archimedean places (as in \cite[\S 2.4.5]{chen2024positivity}), satisfying the following conditions:
	\begin{enumeratea}
		\item $(E,\metr)$ is \emph{strongly dominated}, i.e. there exists an integrable function $C\colon\Omega\rightarrow\R_{\geq 0}$ and a basis $(e_1,\dots, e_r)$ of $E$ over $K$, such that, for any $\omega\in \Omega$ and $(\lambda_1,\dots, \lambda_r)\in K^r_\omega\setminus\{(0,\dots,0)\}$,
		\[\left|\log\|\lambda_1e_1+\cdots+\lambda_re_r\|_\omega - \log\max\limits_{1\leq i\leq r}|\lambda_i|_\omega\right|\leq C(\omega).\]
		\item $(E,\metr)$ is \emph{measurable}, i.e. for any $s\in E$, the function $\|s\|\colon \Omega\rightarrow\R$ is $\mathcal{A}$-measurable.
	\end{enumeratea}
\end{definition}
\begin{remark}
	If $\overline{E}=(E,\metr)$ is an adelic vector bundle over $S$, any subspace (resp. quotient space) of $V$ with the family of restricted norms (resp. quotient norms) is also an adelic vector bundle over $S$.
\end{remark}

\begin{definition}
	Let $\overline{E}=(E, \metr)$ be an adelic vector bundle over $S$. For any $0\not=s\in E$, we have $\log\|s\|_\omega\in \mathscr{L}^1(\Omega, \mathcal{A},\nu)$. We define the \emph{Arakelov degree} of $s$ with respect to $\metr$ as
	\[\widehat{\deg}_{\metr}(s)\coloneq -\int_{\Omega}\log\|s\|_\omega \,\nu(d\omega).\]
\end{definition}
\begin{remark}
	Since $S$ is proper, we have 
	\[\widehat{\deg}_{\metr}(as) = \widehat{\deg}_{\metr}(s)\]
	for any $a\in K^*$.
\end{remark}

\begin{definition}
	Let $\overline{L}=(L, \metr)$ be an adelic line bundle over $S$. The \emph{Arakelov degree} of $\overline{L}$ is defined as
	\[\widehat{\deg}(\overline{L})\coloneq\widehat{\deg}_{\metr}(s),\]
	where $s$ is a non-zero element in $L$. It doesn't depend on the choice of $s$. 
\end{definition}

\begin{definition} \label{def:minimal slop}
	Let $\overline{E}=(E, \metr)$ be an adelic vector bundle over $S$. We define the \emph{Arakelov degree} of $\overline{E}$ as 
	\[\widehat{\deg}(\overline{E})\coloneq\widehat{\deg}(\det\overline{E})=-\int_{\Omega}\log\|e_1\wedge\cdots\wedge e_r\|_{\omega,\det}\,\nu(d\omega),\]
	where $e_1,\dots, e_r$ is a basis of $E$ over $K$, and $\metr_{\omega,\det}$ is the determinant norm of $\metr_\omega$, i.e. for any $\eta\in\wedge^rE$,
	\[\|\eta\|_{\omega,\det}=\inf\limits_{\eta=s_1\wedge\cdots\wedge s_r}\|s_1\|_\omega\cdots\|s_r\|_\omega.\] In the case where $E$ is non-zero, we denote by 
	\[\widehat{\mu}(\overline{E})\coloneq \frac{\widehat{\deg}(\overline{E})}{\dim_KE},\]
	the \emph{slope} of $\overline{E}$. We define the \emph{minimal slope} of $\overline{E}$ as 
	\[\widehat{\mu}_{\min}(\overline{E})\coloneq \inf\limits_{E\twoheadrightarrow W\neq\{0\}}\widehat{\mu}(\overline{W}),\]
	where $\overline{W}$ runs over the non-zero quotient adelic vector bundles of $\overline{E}$. 
\end{definition}
\begin{remark}
	Let $\overline{E}=(E, \metr)$ be an adelic vector bundle over $S$, and $c\in\mathscr{L}^1(\Omega,\mathcal{A},\nu)$. Then  $\overline{E'}=(E, e^{-c}\metr)$ is an adelic vector bundle over $S$, and 
	\[\widehat{\mu}(\overline{E'})=\widehat{\mu}(\overline{E})+\int_{\Omega}c\,\nu(d\omega),\]
	\[\widehat{\mu}_{\min}(\overline{E'})=\widehat{\mu}_{\min}(\overline{E})+\int_{\Omega}c\,\nu(d\omega).\]
\end{remark}
%----------------------------------------------------------------------------------------
%	SECTION 6
%----------------------------------------------------------------------------------------

\section{Global theory on proper varieties} \label{section: global theory}

In this section, we recall first the approach of Chen and Moriwaki for arithmetic intersection theory on projective varieties over a proper adelic curve and then we give a rather straightforward generalization to proper varieties. We will formulate the theory in the language of abstract divisorial spaces introduced in Section \ref{section: abstract divisorial spaces}. This will be useful for the boundary completion considered in the next section.

We fix a proper adelic curve $S=(K,\Omega,\mathcal{A},\nu)$ satisfying the normalization assumption for archimedean valuations given in \cref{assumeption on adelic curves}. We also assume that either the $\sigma$-algebra $\mathcal{A}$ is discrete, or {that $K$ is countable.}

\subsection{Bounded and measurable $S$-metrics and $S$-Green functions}

Let $U$ be an algebraic variety over $K$ of dimension $d$. For $\omega \in \Omega$,  we  set 
\[U_\omega \coloneq U\times_{K}\Spec(K_\omega)\] 
which is an algebraic variety over the completion $K_\omega$ of $K$ and we denote its analytification by $U_\omega^\an$ as in   \ref{analytification}. We will also use 
$$U_S^\an \coloneqq \coprod_{\omega \in \Omega} U_\omega^\an.$$

\begin{art} \label{global metrics}
	Let $L$ be a line bundle on an algebraic variety $U$ over $K$. Then $L$ induces an analytic line bundle $L_\omega^\an$ on $U_\omega^\an$ for every $\omega \in \Omega$. An \emph{$S$-metric} $\metr$ of $L$ is a family of metrics $\metr_\omega$ on $L_\omega^\an$. We call $\metr$ \emph{continuous} if every $\metr_\omega$ on $L_\omega^\an$ is continuous on $L_\omega^\an$. 
\end{art}

\begin{art} \label{points and valuations}
For $\omega \in \Omega$, we have the kernel map $\mathrm{Ker}\colon U_\omega^\an\rightarrow U$, see \cite[Remarks~1.2.5]{berkovich1990spectral}. For any $x\in U_\omega^\an$, locally $x$ corresponds to a multiplicative seminorm $\val_x\colon A\rightarrow \R_{\geq 0}$, where $A=\OO_U(V)$ for some affine open subset $V$ of $U$. Then $\Ker(x)$ is the kernel of $\val_x$ and $\val_x$ induces an absolute value on the residue field $F\coloneq \kappa(\Ker(x))$. 

We assume now that $\omega$ leads to the trivial valuation. 
If the transcendence degree $\mathrm{trdeg}(F/K) = 1$, then we have a unique regular projective curve $C_F$ over $K$ such that the function field of $C_F$ is $F$. {Using that the valuation is trivial on $K$, we have} $\val_x= \exp(-q\ord_\xi(\cdot))$ for some closed point $\xi\in C_F$ and $q\in \R_{\geq0}$. We say that $q$ is the \emph{exponent} of the absolute value $\val_x$. We set $U_{1,\Q}^\an$ as the set of such points of $U_\omega^\an$ with rational exponents. 
By \cite[Lemma~6.1.21]{chen2020arakelov}, 
the set $U_{1,\Q}^\an$ together with the closed points of $U$ is dense in $U_\omega^\an$ with respect to the Berkovich topology.
\end{art}

Recall that $\Omega_0$ denotes the set of $\omega \in \Omega$ inducing the trivial valuation on $K$. We endow it with the $\sigma$-algebra $\mathcal A_0$ obtained from $\mathcal A$ by restriction. We have also seen in \S \ref{subsec: adelic curves} that for every algebraic field extension $L/K$, there is a canonical adelic curve $S_L=(L,\Omega_L, \mathcal{A}_L,\nu_L)$.

\begin{definition} \label{definition: measurable metrics}
	Let $\metr$ be a continuous $S$-metric of a line bundle $L$ on $U$. We say that $\metr$ is \emph{measurable} or more precisely \emph{$S$-measurable} if for any $x \in U$ and any local section $s$ of $L$ at $x$ the following hold:
	\begin{enumerate}
	\item If $x$ is a closed point of $U$, then the function $\Omega_{\kappa(x)} \to \R$, $\omega \to \|s(x)\|_\omega$ is measurable. 
	\item If $x\in U_{1, \Q}^\an$, then the function  $\Omega_0 \to \R$, $\omega \to \|s(x)\|_\omega$ is measurable. 
	\end{enumerate}
\end{definition}

\begin{remark} \label{properties of measurable metrics}
 It is shown in \cite[\S 6.1.4]{chen2020arakelov} that the concept of measurable metrics is compatible with tensor product and pull-back, with passing to the dual metric and with passing to pointwise limits. Note that the projectivity assumptions were never used in the proofs. It follows also from \cite[Lemma~4.1.11]{chen2021arithmetic} that for any algebraic extension $K'$ of $K$ endowed with the canonical structure as a proper adelic curve, the metric $\metr$ of $L$ is measurable if and only if the base change $\metr'$ to $L \otimes_K K'$ is measurable.
\end{remark}

We next need the concept of bounded subsets of $U_S^\an$. This is a classical notion from diophantine geometry which was generalized to $M$-fields (a similar concept as adelic curves) in \cite{gubler1997heights}.

\begin{definition} \label{definition of bounded subsets in affine case}
Let $W$ be an affine variety over $K$. Then a subset $E$ of $W_S^\an$ is called \emph{bounded}  or more precisely \emph{$S$-bounded} in $W$ if for any $f \in \mathcal{O}_U(W)$, there is an integrable function $c$ on $\Omega$ such that for all $(x,\omega) \in E$, we have 
${\log}|f(x)|_\omega \leq c(\omega)$.
\end{definition}

This is generalized to arbitrary varieties as follows:

\begin{definition} \label{definition of bounded subsets}
	A subset $E$ of $U_S^\an$ is called \emph{bounded}  or more precisely \emph{$S$-bounded} in $U$ if there is  a covering of $U$ by affine open subsets $W_1, \dots, W_r$ and a decomposition $E = \bigcup_{j=1}^r E_j$ with $E_j \subset (W_j)_S^\an$ such that $E_j$ is $S$-bounded in $W_j$ in the sense of Definition \ref{definition of bounded subsets in affine case}.
\end{definition}

\begin{remark} \label{properties of bounded sets}
Obviously, both definitions agree in the case of affine varieties. The image of an $S$-bounded set with respect to a morphism is $S$-bounded. The preimage of an $S$-bounded subset with respect to a proper $S$-morphism is $S$-bounded. In particular, for a projective or more generally a proper variety $X$ over $K$, the set $X_S^\an$ is bounded. For details, we refer to \cite[\S 2]{gubler1997heights}.
\end{remark}

\begin{definition} \label{definition of locally bounded functions}
	A function $\alpha\colon X_S^\an \to \R$ is called \emph{locally $S$-bounded} if for any $S$-bounded subset $E$ of $X_S^\an$, there is an integrable function $c\colon \Omega \to \R$ such that for all $(x,\omega) \in E$, we have
	$|\alpha(x,\omega)| \leq c(\omega)$. 
	
	We call $\alpha$ \emph{$S$-bounded} if there is an integrable function such that 	$|\alpha(x,\omega)| \leq c(\omega)$ holds for all $\omega \in \Omega$ and all $x \in U_\omega^\an$. 
\end{definition}

\begin{definition} \label{definition: bounded metrics}
	Let $\metr$ be a continuous metric of a line bundle $L$ on $U$. We say that $\metr$ is \emph{locally bounded} or more precisely \emph{locally $S$-bounded} if for any open subset $W$ of $U$ and any nowhere vanishing section $s \in H^0(W,L)$, the function $\log \|s\|_\omega$ is locally $S$-bounded on $W_S^\an$. 
\end{definition}

\begin{remark} \label{properties of bounded metrics}
Obviously, the trivial metric of $\mathcal O_U$ is a locally $S$-bounded metric.  The concept of locally $S$-bounded metrics is compatible with tensor product, with pull-back and with passing to the dual metric. It follows also from \cite[Lemma~4.1.11]{chen2021arithmetic} that for any algebraic extension $K'$ of $K$ endowed with the canonical structure as a proper adelic curve, the metric $\metr$ of $L$ is locally $S$-bounded if and only if the base change $\metr'$ to $L \otimes_K K'$ is locally $S$-bounded.
\end{remark}

\begin{remark} \label{boundedness on proper varieties}
 On a proper variety $X$ over $K$, a function $\alpha \colon X_S^\an \to \R$ is locally $S$-bounded if and only if it is $S$-bounded as $X_S^\an$ is bounded.
 Therefore, we call a locally $S$-bounded metric on a line bundle $L$ of $X$ just \emph{$S$-bounded}. If $\metr, \metr'$ are $S$-bounded metrics on the same line bundle $L$ of $X$, then it follows that $\log(\metr'/\metr)$ is an $S$-bounded function on $X_S^\an$. 
\end{remark}

\begin{example} \label{global Fubini-Study metric}
	Let $X$ be a proper variety over $K$ and 
	let $L$ be a line bundle on $X$ generated by global sections $s_0, \dots, s_r$. Using Example \ref{Fubini-Study metric} $\metr_{\FS,\omega}$ for every $\omega \in \Omega$, we get a \emph{(global) Fubini--Study metric} of $L$.
	
    More generally, for integrable real functions  $\lambda_0, \dots, \lambda_r$ on $\Omega_{\mathrm{fin}}$, the twisted Fubini--Study metrics $\metr_{\tFS(\lambda),\omega}\coloneq \metr_{\tFS(\lambda(\omega))}$ of $L_\omega$ with respect to the twists $\lambda_0(\omega), \dots, \lambda_r(\omega)$ define the \emph{(global) twisted Fubini--Study metric of $L$}, {where $\lambda=(\lambda_0,\dots, \lambda_r)$ and $\lambda(\omega)=(\lambda_0(\omega),\dots, \lambda_r(\omega))$}.
It follows easily from  the definitions that $\metr_{\tFS(\lambda)}$ is a (locally) $S$-bounded $S$-measurable metric.
\end{example}

\begin{remark} \label{projective case for locally bounded}
If $L$ is a line bundle on a projective variety $X$ over $K$ endowed with an $S$-metric $\metr$, then it follows from Example \ref{global Fubini-Study metric} that $\metr$ is (locally) $S$-bounded if and only if $\metr$ is dominated in the sense of Chen and Moriwaki \cite[Definition~6.1.9]{chen2020arakelov}. 
Indeed, we have seen in Example \ref{global Fubini-Study metric} that a Fubini--Study metric is $S$-bounded and it is also dominated in the sense of Chen and Moriwaki.
%Note here that in {\it loc. cit.}, we may use any $\nu$-dominated norm family on very ample bundles and so in particular the family underlying a Fubini--Study metric as shown in \cite[Proposition~6.1.5]{chen2020arakelov}.
Using additivity together with Serre's theorem, it suffices to assume that $L$ is very ample. We endow $L$ with a Fubini--Study metric $\metr_\FS$. It remains to show that $\alpha \coloneqq \log(\metr/\metr_\FS)$ is an $S$-bounded function if and only if it is dominated. Both properties mean by definition that there is an integrable function $c$ on $\Omega$ such that $|\alpha(x,\omega)| \leq c(\omega)$ for all $\omega \in \Omega$ and all $x \in X_\omega^\an$.
\end{remark}

\begin{remark} \label{group of adelic line bundles} 
The isometry classes of line bundles on $U_S^\an$ endowed with locally $S$-bounded and $S$-measurable metrics form a group, denoted by {$\widehat{\Pic}_S(U)$}. This is the ambient group where arithmetic intersection theory takes place. If $U=X$ is a projective variety over $K$, then it is precisely the group of isometry classes of adelic line bundles of $X$  considered by Chen and Moriwaki \cite[Definition 4.1.9]{chen2021arithmetic}. 
{Similarly, we define $\widehat{\Pic}_{S,\Q}(U)\coloneq \widehat{\Pic}_S(U)\otimes_\Z\Q$ for $\Q$-line bundles.} 
These groups are stable under pull-back with respect to morphisms of algebraic varieties.
\end{remark}

\begin{art}
 \label{global Green functions}
	For a Cartier divisor (resp.~a $\Q$-Cartier divisor) $D$ of $U$, an \emph{$S$-Green function} $g_D$ is a family of Green functions $g_{D,\omega}$ for the base change $D_\omega$ of $D$ to $U_\omega$ with $\omega$ running over $\Omega$.  We recall from \ref{Green functions} that there is an associated continuous metric $\metr_\omega$ of $L_\omega^\an$ for the line bundle (resp.~$\Q$-line bundle) $L=\OO_X(D)$. By abuse of notation, we say that the Green function $g_D$ is \emph{$S$-measurable} (resp.~\emph{locally $S$-bounded}) if the metric $\metr_\omega$ of $L_\omega$ is $S$-measurable (resp.~locally $S$-bounded) for all $\omega \in \Omega$. We denote by $\widehat{\Div}_S(U)$ (resp.~$\widehat{\Div}_{S,\Q}(U)$) the group of pairs $(D,g)$ with $D$ a Cartier divisor (resp.~$\Q$-Cartier divisor) on $U$  and $g$ an $S$-measurable locally $S$-bounded Green function for $D$. {We say that an element $(D,g)\in\widehat{\Div}_S(U)$ (resp. $(D,g)\in\widehat{\Div}_{S,\Q}(U)$) is \emph{effective} if the underlying Cartier divisor (resp.\ $\Q$-Cartier divisor) $D$ is effective and $g_\omega\geq 0$ on $U^\an_\omega$ for any $\omega\in\Omega$.}
\end{art}

We have the following lemma which is the global version of \cref{lemma: injectivity of restriction from integrally closed}.
	\begin{lemma}\label{lemma: global injectivity of restriction from integrally closed}
		Let $X$ be a proper $K$-model of $U$. If $X$ is integrally closed in $U$, then an element $(D,g)\in\widehat{\Div}_{S,\Q}(X)$ is effective if and only if its image in $\widehat{\Div}_{S,\Q}(U)$ is effective. 
	\end{lemma}	
	\begin{proof}
It suffices to show the "if" part. Let $(D,g)\in\widehat{\Div}_{S,\Q}(X)$ such that its image in $\widehat{\Div}_{S,\Q}(U)$ is effective.  For %$\nu$-almost all 
$\omega\in\Omega$, since $g_\omega$ is continuous with $g_\omega\geq 0$ on $U_\omega^\an$ and $U_\omega^\an$ is dense in $X^\an_\omega$, we have that $g_\omega\geq0$ on $X_\omega^\an$. Let $v\in X\setminus U$ be a point with closure $Z$ of codimension $1$ in $X$. Using that $X$ is integrally closed in $U$, it is shown in the proof of \cite[Lemma~2.3.6]{yuan2021adelic} that the local ring $\OO_{X,v}$ is a discrete valuation ring. Since $g_\omega\geq0$, it follows that $\ord_v(D)\geq0$ otherwise $g_\omega<0$ in a neighborhood of $Z_\omega$ in $X_\omega$.
Since $D|_U\geq 0$ and $\ord_v(D)\geq0$ for any $v\in X\setminus U$ of codimension $1$, by \cite[Lemma~2.3.6]{yuan2021adelic}, we have that $D\geq0$, so our claim holds.
\end{proof}

The following lemma is well-known. We state and prove it for convenience of the reader. 
\begin{lemma} \label{Green functions for effective divisors}
Let $g_D$ be a locally $S$-bounded Green function for an effective $\Q$-Cartier divisor $D$ of $U$ and let $E$ be a bounded subset of $U_S^\an$. Then there is an integrable function $c$ on $\Omega$ such that $g_{D,\omega}(x)\geq c(\omega)$ for all $\omega \in \Omega$ and all $x \in E\cap U_\omega^\an$. In particular, if $U=X$ is a proper variety over $K$, then $g_D$ is bounded from below by an integrable function on $\Omega$.
\end{lemma}
\begin{proof}
Multiplying with a suitable positive integer, we may assume that $D$ is an effective Cartier divisor of $U$. Then there is a finite open affine covering $(U_j)_{j \in J}$ of $U$ such that $D$ is given on $U_j$ by $\gamma_j \in \mathcal O_U(U_j)$. For the canonical global section $s_D$ of $L=\mathcal O_U(D)$, there is a nowhere vanishing section $s_j \in H^0(U_j,L)$ with
\begin{equation} \label{decomposition of Green function}
	-\log \|s_j\|_\omega = g_{D,\omega}+ \log |\gamma_j|_\omega
	\end{equation}
for all $\omega \in \Omega$. By definition of a Green function, this is a continuous function on $U_{j,\omega}^\an$. By definition of a locally $S$-bounded Green function, $\log \|s_j\|$ is locally $S$-bounded on $U_j$ for every $j \in J$. Since $E$ is $S$-bounded, we have a decomposition $E = \bigcup_j E_j$ with $E_j \subset U_{j,S}^\an$ such that $E_j$ is $S$-bounded in the affine variety $U_j$ in the sense of Definition \ref{definition of bounded subsets in affine case}, see \cite[Lemma 2.15]{gubler1997heights}. It follows that there are integrable functions $c_j',c_j''$ on $\Omega$ such that for all $\omega \in \Omega$ and all $x \in E_j \cap U_{j,\omega}^\an$, we have 
$$\log|\gamma_j(x)|_\omega \leq c_j'(\omega)$$
by boundedness of $E_j$ in $U_j$ and
$$|\log \|s_j(x)\|_\omega| \leq c_j''(\omega)$$
by using that $\log \|s_j\|$ is locally $S$-bounded. Using this in \eqref{decomposition of Green function} and that  $E = \bigcup_j E_j$ is a finite decomposition, we get the claim with the lower bound $c=-\max\{c_j'+c_j''\mid j \in J\}$. If $U=X$ is a proper variety, we can apply this to the $S$-bounded set $E \coloneqq X_S^\an$ from \cref{properties of bounded sets} to deduce the last claim.
\end{proof}

We recall from Remark \ref{assumeption on adelic curves} that $\Omega_\infty$, $\Omega_{\mathrm{fin}}$ and $\Omega_0$ denote the partition of $\Omega$ into the parts which induce archimedean, non-archimedean and trivial valuations, respectively.

\begin{definition} \label{global adelic Green functions for proper varieties}
	Let $X$ be a proper variety over $K$.
	We consider the submonoid
	$N_{S,\Q}(X)$  of $\widehat{\Div}_{S,\Q}(X)$ consisting of $(D,g_D)\in \widehat{\Div}_{S,\Q}(X)$ with $g_{D,\omega}$ a uniform limit of semipositive model (resp.~twisted 
	Fubini--Study) 
	Green functions for $D_\omega$ on $X_\omega^\an$ if $\omega \in \Omega \setminus \Omega_0$ (resp.~$\omega \in \Omega_0$). This uses the notation introduced in the local theory in \cref{section: local theory} which we recall in the following. 
	We assume that $g_{D,\omega}$ is the uniform limit of a sequence of Green functions $g_{D,\omega,k}$ for $D_\omega$ on $X_\omega^\an$ satisfying the following conditions:
	
	If $\omega \in \Omega_\infty$, then $g_{D,\omega,k}$ is a semipositive smooth Green function for $D_\omega$  or in other words $(D_\omega,g_{D,\omega})$ is in the cone $N_{\SP,\Q}(X_\omega)$ of the abstract divisorial space $(M_{\SP,\Q}(X_\omega),N_{\SP,\Q}(X_\omega))$  over $\Q$ from \cref{Zhang metrics as divisorial space}. 
	
	If $\omega \in \Omega_{\mathrm{fin}} \setminus \Omega_0$, then $g_{D,\omega,k}$ is a semipositive model Green function for $D_\omega$ or equivalently $(D_\omega,g_{D,\omega})$ is again in the cone $N_{\SP,\Q}(X_\omega)$. 
	
	If $\omega \in \Omega_0$, then $g_{D,\omega,k}$ is 
	{a twisted Fubini--Study Green function for $D_\omega$, see \cref{abstract divisorial space and FS}}. 
\end{definition}

\begin{prop} \label{classical global Arakelov theory as divisorial space}
    Under the above assumptions, let
	$M_{S,\Q}(X)\coloneqq {N_{S,\Q}(X)}-{N_{S,\Q}(X)}$ as a subgroup of $\widehat{\Div}_{S,\Q}(X)$ . Then the abelian group $M_{S,\Q}(X)$ is ordered by the submonoid 
	$$M_{S,\Q}(X)_{\geq 0} \coloneqq \{(D,g_D) \in M_{S,\Q}(X) \mid \text{$D$ effective and $g_{D,\omega} \geq 0$ for  
		all $\omega \in \Omega$}\},$$ and
	$(M_{S,\Q}(X),N_{S,\Q}(X))$ is an abstract divisorial space over $\Q$ in the sense of \cref{section: abstract divisorial spaces}. 
\end{prop}

\begin{proof}
This follows easily from the definitions as $M_{S,\Q}(X)_{\geq 0}$ is a pointed cone in $\widehat{\Div}_{S,\Q}(X)$.
%\green{For $\omega \in \Omega$, we will define the abstract divisorial space $(M_{\omega,\Q}(X),N_{\omega,\Q}(X))$ over $\Q$ as follows. If $\omega \not \in \Omega_0$, then we set
%$$(M_{\omega,\Q}(X),N_{\omega,\Q}(X))\coloneqq (M_{\SP,\Q}(X_\omega),N_{\SP,\Q}(X_\omega))$$
%using Proposition \ref{classical local Arakelov theory as divisorial space}.  
%If $\omega \in \Omega_0$, then we consider the cone $N_{\omega,\Q}(X)$ in $\widehat{\Div}_\Q(X_\omega)$ given by pairs $(D_\omega,g_{\omega})$ where $g_{\omega}$ is a uniform limit of twisted Fubini--Study Green functions for  $D_\omega$ on $X_\omega^\an$ using \cref{abstract divisorial space and FS}. Setting $M_{\omega,\Q}(X)\coloneqq N_{\omega,\Q}(X)-N_{\omega,\Q}(X)$, we get our abstract divisorial space $(M_{\omega,\Q}(X),N_{\omega,\Q}(X))$ using the induced partial order from $\widehat{\Div}_\Q(X_\omega)$.
%We consider the canonical inclusion 
%$$\widehat{\Div}_{S,\Q}(X) \longrightarrow \Div_\Q(X)\times\prod_{\omega \in \Omega}\widehat{\Div}_\Q(X_\omega) \quad, \quad (D,g_D) \longmapsto {(D, (D_\omega, g_{D,\omega})_{\omega \in \Omega})}.$$
%By construction and using \ref{induced divisorial subspace}, $(M_{S,\Q}(X),N_{S,\Q}(X))$ is the induced divisorial subspace on {$\widehat{\Div}_{S,\Q}(X) \cap \left(M_{\gm,\Q}(X)\times\prod\limits_{\omega \in \Omega} M_{\omega,\Q}(X)\right)$} inside the product {$(M_{\gm,\Q}(X), N_{\gm,\Q}(X))\times\prod\limits_{\omega \in \Omega} (M_{\omega,\Q}(X),N_{\omega,\Q}(X))$} of abstract divisorial spaces over $\Q$, see Lemma \ref{direct limits}.}
\end{proof}

\begin{remark} \label{global functoriality}
Let $\varphi\colon X' \to X$ be a surjective morphism of proper varieties over $K$. Then we get a morphism
$$\varphi^*\colon M_{S,\Q}(X) \longrightarrow M_{S,\Q}(X') \quad (D,g_D) \longmapsto (\varphi^*(D), g_D \circ \varphi)$$
of abstract divisorial spaces over $\Q$. This often will be used to reduce to the projective situation considered next.
\end{remark}

Chen and Moriwaki \cite{chen2021arithmetic} have extended arithmetic intersection theory to a projective variety $X$ over the proper adelic curve $S$. We will compare their notions with ours.

\begin{remark} \label{adelic Green functions of Chen-Moriwarki}
	Let $X$ be a projective variety over $K$. Chen and Moriwaki consider \emph{adelic line bundles} on $X$ which means a line bundle $L$ endowed with  a continuous $S$-metric $\metr$ which is locally $S$-bounded (dominated in their terminology, see Remark \ref{projective case for locally bounded}) and measurable.
	If $L$ is semiample, then they call such a metric $\metr$ \emph{semipositive} if $\metr_\omega$ is uniform limit of so called \emph{quotient metrics} of $L$ \cite[Definition~3.2.6]{chen2021arithmetic}. A \emph{semipositive adelic} {\emph{$\Q$-}}\emph{divisor} is then a pair $(D,g_D)$ with $D$ a $\Q$-Cartier divisor on $X$ such that $L=\OO_X(D)$ is semiample and such that the $S$-metric of $L$ corresponding to $g_D$ is semipositive in the sense of Chen and Moriwaki.

	We claim that, under the assumption that $D$ is semiample, $(D,g_D)\in N_{S,\Q}(X)$ if and only if $(D,g_D)$ is a semipositive adelic $\Q$-divisor in the above sense of Chen and Moriwaki.  
	To see this, we note first for $\omega \in \Omega_\infty$, that the continuous metric $\metr_\omega$ is plurisubharmonic in the sense of complex analysis if and and only if it is a (decreasing) limit of Fubini--Study metrics, see  \cite[Theorem~7.1]{boucksom2021spaces}. The latter are special quotient metrics giving rise to smooth metrics, so this proves the archimedean case. If $\omega \in \Omega_{\mathrm{fin}} \setminus \Omega_0$, it follows from \cite[Theorem~3.2.19]{chen2021arithmetic} that a continuous metric $\metr_\omega$ of $L_\omega$ is a uniform limit of semipositive model metrics if and only if it is semipositive in the sense of Chen and Moriwaki. If $\omega \in \Omega_0$, then we have used twisted Fubini--Study metrics. As the latter are also quotient metrics, it follows that a uniform limit of twisted Fubini--Study metrics of $L_\omega$ is semipositive in the sense of Chen and Moriwaki. Conversely, it remains to show that a quotient metric of $L$ is a uniform limit of twisted Fubini--Study metrics of $L$. This follows from density of diagonalizable ultra-metric norms  in the space of ultra-metric norms \cite[Theorem~1.19]{boucksom2021spaces}, from \cite[(3.1)]{chen2021arithmetic} and from the fact that every quotient metric of $L$ can be defined by an ultra-metric norm \cite[Remark 1.1.19, Proposition 1.2.14]{chen2020arakelov}. Putting these statements comparing semipositivity for 
	every $\omega \in \Omega$ together, we readily prove the above claim.
	
	Chen and Moriwaki call an adelic divisor $(D,g_D)$ of $X$ \emph{integrable} if it is the difference of two semipositive adelic divisors. Since any $\Q$-Cartier divisor on $X$ is the difference of two (semi-)ample $\Q$-Cartier divisors, the above shows that $(D,g_D)\in \widehat{\Div}_{S,\Q}(X)$ is in $M_{S,\Q}(X)$ if and only if $(D,g_D)$ is an integrable adelic divisor in the sense of Chen and Moriwaki. 
\end{remark}

In this paper, we generalize these terminologies to proper varieties. If $X$ is proper over $K$, an element in $\widehat{\Div}_{S,\Q}(X)$ is called an \emph{adelic divisor in the sense of Chen and Moriwaki}, and adelic divisors in $N_{S,\Q}(X)$ (resp. $M_{S,\Q}(X)$) are called \emph{semipositive} (resp. \emph{integrable}). Similarly, we proceed for $S$-metrized line bundles.
	{Note that if $\Omega_0 \neq \emptyset$ (i.e.~there are trivial valuations in $\Omega$), then semipositive $S$-metrized line bundles  {have underlying semiample line bundles.} It is certainly possible that a proper non-projective variety has non-trivial semiample line bundles (using e.g.~the product of a non-projective and a projective variety).}

\begin{example} \label{semipositive adelic metrics in running example}
Let $X=\mathbb P_K^1$ for $K=\mathbb Q$. We endow $K$ with the canonical structure of an adelic curve $S=(K,\Omega,\mathcal A,\nu)$ as in \cref{example:adelic structure of number fields}. For simplicity, we just consider the divisor $D=[\infty]$. Let $\Psi(u) \coloneqq \min(u,0)$ be the corresponding piecewise linear function on $\R$. We have seen in \cref{running example for local Zhang} that for a place $v$ of $\Q$, a semipositive toric Green function $g_{D,v}$ for the base change $D_v$ of $D$ to $\mathbb P_{\mathbb \Q_v}^1$ corresponds to a concave function $\psi_v\colon \R \to \R$ with $\psi_v-\Psi=O(1)$. We conclude that a toric $S$-Green function $g_D=(g_{D,v})_{v \in \Omega}$ for $D$ is given by a family $(\psi_v)_{v \in \Omega}$ of such concave functions. Then $g_D$ is 
locally $S$-bounded if and only if there is an integrable function $C \colon \R \to \R$ such that 
\begin{equation} \label{integrable bound}
|\psi_v- \Psi| \leq C(v)
\end{equation}
for all $v \in \Omega$. This condition is necessary and sufficient for $(D,g_D) \in N_{S,\Q}(X)$ as $g_D$ is always $S$-measurable since the $\sigma$-algebra $\mathcal A$ of $S$ is discrete.
\end{example}

Chen and Moriwaki defined \emph{arithmetic intersection numbers} for integrable adelic divisors in the above setting. This is easily generalized to proper varieties. We summarize their result in  the  following proposition.

\begin{prop} \label{global intersection number}
	Let $S=(K,\Omega,\mathcal{A},\nu)$ be the given  proper adelic curve. For any proper variety $X$ over $K$, for any $\overline{D_0}, \dots, \overline{D_k} \in M_{S,\Q}(X)$ and any $k$-dimensional cycle $Z$ of $X$, there is a unique $(\overline{D_0} \cdots \overline{D_k} \mid Z)_S \in \R$ with the following properties:
	\begin{enumerate}
		\item \label{classical global intersection number line bundles} The number $(\overline{D_0} \cdots \overline{D_k} \mid Z)_S \in \R$ depends only on the isometry classes of the underlying $S$-metrized $\Q$-line bundles $\overline{L_j}=(\mathcal O_X(D_j), \metr_j)$, $j=0,\dots, k$, and on the cycle $Z$, but not on the particular choices of $\overline{D_0}, \dots, \overline{D_k}$, so we set
		$$(\overline{L_0} \cdots \overline{L_k} \mid Z)_S \coloneqq (\overline{D_0} \cdots \overline{D_k} \mid Z)_S.$$
		\item The pairing  $(\overline{L_0} \cdots \overline{L_k} \mid Z)_S \in \R$ is multilinear and symmetric in $\overline{L_0}, \dots, \overline{L_k}$ and linear in $Z$.
		\item \label{classical global intersection number factorial} If $\varphi \colon X' \to X$ is a morphism of proper varieties over $K$ and if $Z'$ is a $k$-dimensional cycle on $X'$, then the projection formula holds:
		$$ (\varphi^*\overline{L_0} \cdots \varphi^*\overline{L_k} \mid Z')_S = (\overline{L_0} \cdots \overline{L_k} \mid \varphi_*Z')_S.$$
		\item \label{global intersection number by induction} If $Z=X$ and  $k=d=\dim(X)$, then $\int_{X_\omega^\an} g_{D_0,\omega} c_1(\overline{L_1}) \wedge \dots \wedge c_1( \overline{L_d})$ is an integrable function in $\omega \in \Omega$ and the induction formula holds:
		$$(\overline{L_0} \cdots \overline{L_d} \mid X)_S= (\overline{L_1} \cdots \overline{L_d} \mid  \mathrm{cyc}(D_0))_S+ \int_{\Omega} \left(\int_{X_\omega^\an} g_{D_0,\omega} c_1(\overline{L_1}) \wedge \dots \wedge c_1( \overline{L_d})\right)\,\nu(d\omega).$$
		\item \label{classical global intersection number field extension} Let $K'/K$ be an algebraic extension of $K$ and let $S'$ be the canonical adelic curve on $K'$ induced by $S$. Let $X',\overline{L_0'}, \dots, \overline{L_k'},Z'$ be obtained from $X,\overline{L_0}, \dots, \overline{L_k},Z$ by base change. Then we have
		$$(\overline{L_0'} \cdots \overline{L_k'} \mid Z')_{S'}=(\overline{L_0} \cdots \overline{L_k} \mid Z)_S.$$
	\end{enumerate}
\end{prop}
\begin{proof}
	The uniqueness is from \ref{global intersection number by induction}. The existence is from \cite[Theorem~4.2.11, Proposition~4.4.4, Theorem~4.4.9, Proposition~3.8.9]{chen2021arithmetic} for projective varieties, and the proper case follows easily from the projection formula and Chow's lemma.
\end{proof}

In the following, for any $\overline{D}\in M_{S,\Q}(X)$ (resp.~an $S$-metrized line bundle $\overline{L}$ corresponding to some divisor $\overline{D}\in M_{S,\Q}(X)$) and any $k$-dimensional cycle $Z$ of $X$, we denote by 
$h_{S,\overline{D}}(Z)\coloneq (\overline{D}^{k+1}\mid Z)_S$ (resp.~$h_{S,\overline{L}}(Z)\coloneq (\overline{L}^{k+1}\mid Z)_S$)  %$h_{S,\overline{D}}(X)\coloneq (\overline{D}^{\dim(X)+1}\mid X)_S$ (resp.~$h_{S,\overline{L}}(X)\coloneq (\overline{L}^{\dim(X)+1}\mid X)_S$) 
the {\emph{height of $Z$ with respect to $\overline{D}$ (resp.~$\overline{L}$)}.} We will omit $S$ when the adelic curve is clear.

\subsection{{$S$-ample} line bundles} \label{subsection: arithmetically nef line bundles}

We recall the following definition from Chen and Moriwaki \cite[Definition~9.1.1]{chen2024positivity}. 

\begin{definition} \label{definition: arithmetically ample}
	Let $\overline{L}=(L,\metr)$ be an adelic line bundle on a projective variety $X$ over $K$ in the sense of Chen and Moriwaki (see Remark \ref{adelic Green functions of Chen-Moriwarki}).  We say that $\overline{L}$ is 
	\begin{itemize}
		\item \emph{relatively $S$-ample} if $L$ is an ample line bundle and if $\metr$ is a semipositive $S$-metric in the sense of Chen and Moriwaki (see Remark \ref{adelic Green functions of Chen-Moriwarki});
		\item \emph{$S$-ample} if $\overline{L}$ is relatively $S$-ample and if there is $\varepsilon >0$ such that for every integral closed subscheme $Y$ of $X$, we have
		\begin{equation} \label{ample height inequality}
			h_{\overline L}(Y) \geq \varepsilon \deg_L(Y) (\dim(Y)+1).
		\end{equation} 
	\end{itemize}%\emph{$S$-ample} if $L$ is an ample line bundle, if $\metr$ is a semipositive $S$-metric in the sense of Chen and Moriwaki (see Remark \ref{adelic Green functions of Chen-Moriwarki}, we say that $\overline{L}$ is \emph{relatively $S$-ample} if these two conditions hold) and if there is $\varepsilon >0$ such that for every integral closed subscheme $Y$ of $X$, we have
	%\begin{equation} \label{ample height inequality}
	%	h_{\overline L}(Y) \geq \varepsilon \deg_L(Y) (\dim(Y)+1).
	%\end{equation}
	{For $\overline{D}\in  N_{S,\Q}(X)$, we say that $\overline{D}$ is \emph{$S$-ample} (resp. \emph{relatively $S$-ample}) if $\OO_X(n\overline{D})$ is induced by an $S$-ample (resp. relatively $S$-ample)  adelic line bundle for some $n\in\N_{>0}$.}
\end{definition}

\begin{remark} \label{existence of S-ample}
Let $\overline L=(L,\metr)$ be a relatively $S$-ample adelic line bundle on the projective variety $X$ over $K$. For an integrable function $c$ on $S$, we define the \emph{twist} $\overline{L}(c)$ by using the metrics  $\metr_{\overline L(c),\omega} \coloneqq e^{-c(\omega)}\metr_{\overline L,\omega}$ for all $\omega \in \Omega$. Obviously, this is again a relatively $S$-ample adelic line bundle. 

Let $Y$ be a closed integral subscheme of $X$. Since $L$ is ample, there is $n \in \N$ such that $L^{\otimes n}$ has a  global section $s_n$ with $s_n|_Y\neq 0$ and the induction formula  yields
\begin{equation} \label{induction formula here}
	h_{\overline L}(Y)=n^{-1}h_{\overline L}(\mathrm{div}(s_n))-n^{-1}\int_\Omega \int_{Y_\omega^\an} \log\|s_n\|_{\overline L^{\otimes n},\omega} c_1(L_\omega,\metr_\omega)^{\wedge\dim(Y)}\,\nu(d\omega),
\end{equation}
see \cref{global intersection number}\ref{global intersection number by induction}. Applying this with $\overline L(c)$ as well, we get
\begin{equation} \label{twist and height}
	h_{\overline L(c)}(Y)=h_{\overline L}(Y)+(\dim(Y)+1)\deg_L(Y)\int_\Omega c \,\nu(d\omega).
\end{equation}
Using that $\metr$ is (locally) $S$-bounded and that $s_n$ is a global section, there is an integrable function $c'$ on $\Omega$ such that $-\log\|s_n\|_{\omega} \geq c'(\omega)$ for all $\omega \in \Omega$. By induction on $\dim(Y)$ and by semipositivity of the metrics, we deduce then from \eqref{induction formula here} that $h_{\overline L}(Y) \geq -C\deg_L(Y)$ for some $C \in \R$ independent of $Y$. 
Choosing $c$ sufficiently large in \eqref{twist and height} and using $\nu(\Omega)\not\in \{0,\infty\}$, we conclude that we always find a twist such that $\overline L(c)$ is $S$-ample. In particular,  any ample line bundle has a twisted Fubini--Study metric which is $S$-ample.
\end{remark}

\begin{lemma} \label{arithmetic ampelness and base change}
Let $K'/K$ be an algebraic field extension, {let $S'\coloneqq S\otimes_K K'$ be the canonical adelic curve on $K'$ as before \cref{def:heightfunction}} and let $X'\coloneqq X_{K'}$ be the base change of the projective variety $X$ to $K'$.  Then  $\overline D\in N_{S,\Q}(X)$  is $S$-ample if and only if the base change $\overline{D'} \coloneqq \overline D_{K'}$ to $X'$ is {$S'$-ample}.
\end{lemma}

\begin{proof} 
	Multiplying $\overline D$ by a suitable non-zero $n \in \N$, we may assume that $\overline D$ is induced by an adelic line bundle $\overline{L}$ and hence $\overline{D'}$ is induced by the base change $\overline{L'}\coloneqq \overline{L}_{K'}$ of $L$ to $X'$. 
First, we assume that $\overline{L'}$ is $S'$-ample. Let $Y$ be an integral closed subscheme of $X$. We note that the base change $Y'\coloneqq Y_{K'}$ is a closed subscheme of $X'$ of pure dimension. Since $\overline{L'}$ is $S$-ample, the inequality \eqref{ample height inequality} holds for every irreducible component of $Y'$. By linearity of the height and the degree, we conclude that \eqref{ample height inequality} holds for $Y'$ and hence for $Y$ by invariance of the height and the degree under base change.

Conversely, we assume that $\overline L$ is $S$-ample. Let $Y'$ be a closed integral subscheme of $X'$. For clarity, we denote by $h_{S',\overline{L'}}(Y')$ the height of $Y'$. 
Then have to show 
	\begin{equation} \label{ample height inequality after base change}
	h_{S',\overline{L'}}(Y') \geq \varepsilon \deg_{L'}(Y') (\dim(Y')+1).
\end{equation}
Since $Y'$ and $X'$ are defined over a finite  subextension of $K'/K$, we may assume that $K'/K$ is finite. Let $Y$ be the image of $Y'$ with respect to the canonical morphism $\pi\colon X_{K'}\to X$. 
The way the heights are set up, and the projection formula show
\begin{equation} \label{heights and base change}
	[K':K]h_{S',\overline{L'}}(Y')= h_{S,\pi^*\overline L}(Y')=[K'(Y'):K(Y)]h_{S,\overline L}(Y)
\end{equation}
and a similar formula holds for the degree. Since $\overline L$ is $S$-ample, we have \eqref{ample height inequality} for $Y$. Plugging in \eqref{heights and base change} and the corresponding identity for the degree, we deduce \eqref{ample height inequality after base change}.
\end{proof}

\begin{art} \label{our notion for arithmetic nef}
	For $X$ projective, we define  
	$$N_{S,\Q}^+(X) \coloneqq \{(D,g_D) \in  M_{S,\Q}(X)\mid \text{$(D,g_D)$ is $S$-ample}\}$$
	and its closure $\overline{N_{S,\Q}^+}(X)$ in $M_{S,\Q}(X)$ in the sense of  \ref{closure for finite subspaces}. 
\end{art}

\begin{lemma} \label{lemma: difference of S-ample}
	On a projective variety $X$, 
	any adelic divisor in $M_{S,\Q}(X)$ is the difference of two effective $S$-ample adelic divisors, in particular $M_{S,\Q}(X)=N_{S,\Q}^+(X)-N_{S,\Q}^+(X)$. 
\end{lemma}

\begin{proof}
Let $\overline D\in M_{S,\Q}(X)$. We have to show that  $\overline D\in N_{S,\Q}^+(X)-N_{S,\Q}^+(X)$.  By Remark \ref{adelic Green functions of Chen-Moriwarki}, we have $M_{S,\Q}(X)=N_{S,\Q}(X)-N_{S,\Q}(X)$, so we may assume that $\overline D\in N_{S,\Q}(X)$ which means that $\overline D$ is a semipositive adelic divisor in the sense of Chen and Moriwaki. Since $X$ is projective, there are ample $\Q$-Cartier divisors $D_1,D_2$ with $D=D_1-D_2$. We endow the ample $\Q$-line bundle 
$\mathcal O_X(D_2)$ with a semipositive adelic metric $\metr_2$ and $\mathcal O_X(D_1)=\mathcal O_X(D) \otimes\mathcal O_X(D_2)$ with the induced tensor metric denoted by $\metr_1$. Then $\metr_1$ is a semipositive adelic metric as well and we get induced relatively $S$-ample adelic divisors $\overline{D_1}, \overline{D_2}$ with $\overline D=\overline{D_1}- \overline{D_2}$. We choose an integrable function $c$ on $\Omega$ such that $\int_\Omega c \, \nu(d\omega) \gg 0$. Then Remark \ref{existence of S-ample} shows that  $\mathcal O_X(\overline{D_i})(c)$ is $S$-ample for both $i=1,2$. Replacing $\metr_1,\metr_2$ by the twist with $c$, we get the claim.
\end{proof}

We conclude from Lemma \ref{lemma: difference of S-ample} that $(M_{S,\Q}(X),N^+_{S,\Q}(X))$ is an abstract divisorial space for any projective variety $X$.
When $K$ is perfect, we will show next that we have a numerical description for $\overline{N_{S,\Q}^+}(X)$ using the  minimal slopes introduced  in \cref{def:minimal slop}.

\begin{art} \label{def: asymptotic minimal slope}
	Let $\overline{L}=(L,\metr)$ be an adelic line bundle on a projective variety $X$ in the sense of Chen and Moriwaki. Let $\pi\colon X \to \Spec(K)$ be the morphism of structure. {Since {either $\mathcal{A}$ is discrete or} $K$ is countable,} it follows from \cite[Theorem~6.1.13, Theorem~6.1.32]{chen2020arakelov} that  $\pi_*(L)$ endowed with the supremum norm of global sections is an adelic vector bundle over $S$ which we denote by $\pi_*(\overline L)$. {In \cref{def:minimal slop}, we defined the minimal slope $\widehat{\mu}_{\min}(\overline{E})$ for a non-zero adelic vector bundle $\overline E$ over $S$.}
	
	Now we assume that $K$ is perfect. For $L$ ample, the \emph{asymptotic minimal slope} $\widehat{\mu}^{\mathrm{asy}}_{\min}(\overline{L})$ is defined in \cite[Definition~6.2.1]{chen2024positivity} by
	$$\widehat{\mu}^{\mathrm{asy}}_{\min}(\overline{L})= \lim_{n \to \infty} \frac{1}{n} \widehat{\mu}_{\min}(\pi_*(\overline L^{\otimes n}))$$
	where the existence of the limit in $\R$ follows from \cite[Proposition~6.1.3]{chen2024positivity}. For $L$ nef,  we choose an adelic line bundle $\overline A$ on $X$ with $A$ ample. Following \cite[Definition~6.4.2]{chen2024positivity}, we set
	$$\widehat{\mu}^{\mathrm{asy}}_{\min}(\overline{L})\coloneqq\lim\limits_{n\to+\infty}\frac{1}{n}\widehat{\mu}^{\mathrm{asy}}_{\min}(\overline{L}^{\otimes n}\otimes\overline{A})$$
	where the limit exists in  $\R\cup\{-\infty\}$, does not depend on the choice of $\overline{A}$ and agrees with the previous definition in the ample case, see \cite[Proposition~6.4.1]{chen2024positivity}.

Assuming $K$ still perfect, the asymptotic minimal slope can be defined on an adelic $\Q$-line bundle $\overline{L}$ {with $L$ nef} by the formula $\widehat{\mu}_{\min}^{\mathrm{asy}}(\overline{L}) \coloneq \frac{1}{m}\widehat{\mu}_{\min}^{\mathrm{asy}}(\overline{L}^{\otimes m})$ for any $m\in\N_{>0}$ such that $L^{\otimes m}$ is induced by a line bundle on $X$. For $\overline{D}\in M_{S,\Q}(X)$ with $D$ nef, we set $\widehat{\mu}_{\min}^{\mathrm{asy}}(\overline{D})\coloneq \widehat{\mu}_{\min}^{\mathrm{asy}}(\OO_X(\overline{D}))$.
\end{art}

\begin{remark} \label{slopes and relatively ample}
If $K$ is perfect and $\overline L$ is a {relatively $S$-ample} adelic line bundle on the projective variety $X$ over $K$, then it is shown in \cite[Theorem~8.8.3]{chen2024positivity} that 
\begin{equation} \label{slope formula with intersection numbers}
	\widehat{\mu}_{\min}^{\mathrm{asy}}(\overline{L}) =\inf_{Y\in\Theta_X}\frac{(\overline{L}^{\dim(Y)+1}\mid Y)_S}{(\dim(Y)+1)\deg_L(Y)} 
\end{equation}
where  $\Theta_X$ is the set of all integral closed subschemes of $X$.

By definition and \eqref{slope formula with intersection numbers}, a relatively $S$-ample $\Q$-line bundle $\overline{L}$ is $S$-ample if and only if  $\widehat{\mu}_{\min}^{\mathrm{asy}}(\overline{L})>0$, see \cite[Proposition~9.1.2]{chen2024positivity}. Since the $S$-ampleness is invariant under algebraic base change, we can use this characterization for $S$-ampleness even if $K$ is not perfect using base changes to a perfect algebraic field extension. In particular, by \cite[Proposition~6.4.4]{chen2024positivity} we know that $N_{S,\Q}^+(X), \overline{N_{S,\Q}^+}(X)$ are cones in any case. 
\end{remark}

\begin{proposition} \label{proposition:arithmetically nef for projective}
	Let $X$ be a projective variety  and $\overline{D}\in N_{S,\Q}(X)$. Then the following are equivalent.
	\begin{enumerate}
		\item \label{arithmetically nef by closure} $\overline{D}\in\overline{N_{S,\Q}^+}(X)$.
		\item \label{arithmetically nef by chen} There is $\overline{A}\in{N_{S,\Q}^+}(X)$ and  $N \in \N$ such that $n\overline{D}+ \overline{A}\in {N_{S,\Q}^+}(X)$ for any $n\in \N_{\geq N}$.
		\item \label{arithmetically nef by an S-ample} For any $\overline{A}\in{N_{S,\Q}^+}(X)$, we have that $\overline{D}+\overline{A}\in {N_{S,\Q}^+}(X)$.
		\item \label{arithmetically nef by minimal slope}   
		There is a perfect algebraic field extension $K'/K$ such that $\widehat{\mu}^{\mathrm{asy}}_{\min}(\overline{D}_{K'})\geq 0$, where $\overline{D}_{K'}$ is the base change of $\overline{D}$ to $K'$.
		\item \label{arithmetically nef by minimal slope variant}   
		For all perfect algebraic  field extensions $K'/K$, we have  $\widehat{\mu}^{\mathrm{asy}}_{\min}(\overline{D}_{K'})\geq 0$.
	\end{enumerate} 
\end{proposition}

Chen and Moriwaki define nef adelic line bundles by \ref{arithmetically nef by chen} (\cite[Definition~9.1.6]{chen2024positivity}) not assuming $\overline D \in N_{S,\Q}(X)$.
%\green{We will call  $\overline D$ \emph{$S$-nef} if $\overline D$ satisfies the above  {four} equivalent conditions.}
\begin{proof} 
It follows from  \cref{lemma: difference of S-ample} that $(M_{S,\Q}(X),N^+_{S,\Q}(X))$ is an abstract divisorial space.  Then \cref{lemma:closure in finite subspace topology} yields that \ref{arithmetically nef by chen} and \ref{arithmetically nef by closure} are equivalent.
	Obviously,  \ref{arithmetically nef by an S-ample} implies \ref{arithmetically nef by chen}.

	\ref{arithmetically nef by chen} $\Longrightarrow$ \ref{arithmetically nef by minimal slope variant}. This follows from \cref{arithmetic ampelness and base change} and \cite[Proposition~9.1.7]{chen2024positivity}. Obviously, 
    \ref{arithmetically nef by minimal slope variant} yields \ref{arithmetically nef by minimal slope}.
	
	\ref{arithmetically nef by minimal slope} $\Longrightarrow$ \ref{arithmetically nef by an S-ample}. Let $\overline{A}\in{N_{S,\Q}^+}(X)$. Since $\overline{D}\in N_{S,\Q}(X)$, then $\overline{D}+\overline{A}$ is relatively $S$-ample.  
	Let $K'/K$ be a field extension as in \ref{arithmetically nef by minimal slope}. {By \cref{arithmetic ampelness and base change},} it suffices to show that $\overline{D}_{K'}+\overline{A}_{K'}$ is $S'$-ample, where $S'$ is the adelic curve for the field extension $K'/K$. By \cite[Proposition~6.4.4]{chen2024positivity}, we have that
	\[\widehat{\mu}_{\min}^{\mathrm{asy}}(\overline{D}_{K'}+\overline{A}_{K'})\geq \widehat{\mu}_{\min}^{\mathrm{asy}}(\overline{D}_{K'})+\widehat{\mu}_{\min}^{\mathrm{asy}}(\overline{A}_{K'})>0.\]
	Hence $\overline{D}_{K'}+\overline{A}_{K'}$ is $S'$-ample using Remark \ref{slopes and relatively ample}.
	\end{proof}

\begin{corollary}  \label{cor:arith nef stable under algebraic field extension}
	Let $X$ be a projective variety over $K$, $\overline{D}\in N_{S,\Q}(X)$, and $K'/K$ an algebraic field extension. Then $\overline{D}\in \overline{N_{S,\Q}^+}(X)$  if and only if $\overline{D}_{K'}\in \overline{N_{S',\Q}^+}(X_{K'})$ where $S'$ is the natural adelic curve for $K'$.
\end{corollary}
\begin{proof} 
	Let $S'$ be the adelic curve for the field extension $K'/K$. 
	The "only if" part is obvious from \cref{arithmetic ampelness and base change}. Conversely, assume that $\overline{D}_{K'}\in \overline{N_{S',\Q}^+}(X_{K'})$. 
	%\green{Using \cref{existence of S-ample}, we can take an} 
	Let us pick any $S$-ample adelic divisor $\overline A$ on $X$. By \cref{proposition:arithmetically nef for projective},  we have that $\overline{D}_{K'}+\overline{A}_{K'}$ is $S'$-ample. By \cref{arithmetic ampelness and base change}, this implies that $\overline{D}+\overline{A}$ is $S$-ample. It follows from \cref{proposition:arithmetically nef for projective} that $\overline{D} \in \overline{N_{S,\Q}^+}(X)$.
\end{proof}

We illustrate these positivity notions in our running example. 
For generalizations to toric varieties, we refer to \cite{bmps-2016}. In this reference, only finitely many local metrics of an adelic toric metric are allowed to be different from the canonical metric, but the characterization remains the same and the arguments are easily adapted.

\begin{example}   \label{ample and nef adelic metrics in running example}
	We continue with the Example \ref{semipositive adelic metrics in running example} where $X=\mathbb P_K^1$ for $K=\mathbb Q$ and $S=(\Q, \Omega, \mathcal A,\nu)$ is the canonical structure of a proper adelic curve on $\Q$. For $D=[\infty]$, we have the corresponding piecewise linear function $\Psi(u) =\min(u,0)$ on $\R$. We have seen that a semipositive toric $S$-Green function $g_D$ is induced by family of corresponding concave functions $(\psi_v)_{v \in \Omega}$ with $\psi_v- \Psi=O(1)$. Then $L=\OO_X(D)$ is a toric line bundle which we endow with the semipositive toric metric $\metr$ induced by  $g_D$. We will also consider the canonical metric $\metr_{\rm can}$ of $L$ which is the semipositive toric metric of $L$ corresponding to the concave function $\Psi$ itself (for all $v \in \Omega$).
			
For the concave function $\psi_v\colon \R \to \R$,  the \emph{Legendre--Fenchel dual} is defined by
$$\psi_v^\vee\colon [0,1] \longrightarrow \R \cup \{-\infty\} \, , \quad    \psi_v^\vee(m) \coloneqq \inf_{u \in \R} \left(mu-\psi_v(u)\right)$$
using that dual polytope is $[0,1]$ as the asymptotic slopes of $\psi_v$ and $\Psi$ are $0$ and $1$. See \cite[Chapter 2]{BPS} for details and properties. 
In \cite[Definition 5.1.1]{BPS}, the toric local height $h_{(L_v,\metr_v)}^{\rm tor}(X)$ of the toric variety $X=\mathbb P^1$ was introduced with respect to the toric line bundle $(L_v,\metr_v)$ as  the local height of $X$ with respect to $(L_v,\metr_v)$ minus  the local height of $X$ with respect to $(L_v,\metr_{{\rm can},v})$. By \cite[Theorem 5.1.6]{BPS}, we have
$$h_{(L_v, \metr_v)}^{\rm tor}(X_v)=2\int_{[0,1]} \psi_v^\vee(m) \, dm.$$

Now we assume that $(D,g_D)\in N_{S,\Q}(X)$. In other words, we assume that the toric metric $\metr$ of $\overline L$ is $S$-measurable and locally $S$-bounded. 
The first condition is always satisfied as the $\sigma$-algebra $\mathcal A$ of $S$ is discrete.

We want also an interpretation of locally $S$-boundedness in terms of $\psi_v^\vee$. This is based on \emph{Fenchel's inequality} which states that for all concave functions $f$ on $\R$ with asymptotic slopes $1$ for $u \to -\infty$ and $0$ for $u\to \infty$,
we have
$$mu \geq f(u) + f^\vee(m)$$
 for all $u \in \R$ and $m \in [0,1]$. This inequality follows immediately from the definition of $f^\vee$, and yields for any other concave function $g$ on $\R$ with the same asymptotic slopes as $f$ that
 $$f^\vee(m)= \inf_{u \in \R} \left(mu - f(u)\right)\geq \inf_{u \in \R}\left(g(u)+ g^\vee(m)-f(u)\right)= g^\vee(m)-\sup_{u \in \R}\left(f(u)-g(u)\right).$$
Now we assume that $f=\Psi+O(1)$ and $g=\Psi+O(1)$. Using the supremum norm and reversing the role of $f$ and $g$, we deduce $\|g^\vee-f^\vee\|_\infty \leq \|g-f\|_\infty$. Applying the same  for $f^\vee$ and $g^\vee$ and using biduality, we get 
\begin{equation} \label{isometry of Legendre-Fenchel}
\|g^\vee-f^\vee\|_\infty = \|g-f\|_\infty	
\end{equation}
Recall from \eqref{integrable bound} that locally $S$-bounded means that there is an integrable function $C$ on $\Omega$ such that
$|\psi_v- \Psi| \leq C(v)$. Then \eqref{isometry of Legendre-Fenchel} gives $|\psi_v^\vee- \Psi^\vee| \leq C(v)$ and using $\Psi^\vee=0$, we get
\begin{equation*} \label{Legendre-Fenchel dual and uniform estimate}
|\psi_v^\vee(m)|\leq C(v)
\end{equation*}
for all $m \in \Delta$. We conclude that $\psi_v^\vee(m)$ is an integrable function in $v \in \Omega$ for all $m \in [0,1]$, so we can define the \emph{roof function}
\begin{equation} \label{roof function}
 \vartheta \colon	[0,1]  \longrightarrow \R \, , \quad m \mapsto \vartheta(m) \coloneqq \int_{\Omega} \psi_v^\vee(m) \,\nu(dv).
\end{equation} 
Since every $\psi_v$ is concave, we conclude that $\vartheta$ is a concave function. Let us first consider the canonical metric $\metr_{\rm can}$ of $L$. Then it is easy to show by using the projection formula and $[m]^*\metr_{\rm can}=\metr^{\otimes m}$ that $h_{(L,\metr_{\rm can})}(X)=0$ and hence Fubini's theorem gives 
\begin{equation} \label{BPS formula}
	h_{(L,\metr)}(X)= \int_{[0,1]} h_{(L_v, \metr_v)}^{\rm tor}(X_v)=2\int_{[0,1]} \vartheta(m) \, dm.
\end{equation}
Note that the closed points $0,\infty$ of $X$ are toric subvarieties. As above, we get $h_{(L,\metr)}(0)=\vartheta(0)=-\int_\Omega \sup_{u \in \R}\psi_v(u) \, \nu(dv)$ and $h_{(L,\metr)}(\infty)=\vartheta(1)$. 
Now let $Y$ be any closed point of $T=X\setminus\{0,\infty\}$ and let $u_v\coloneqq \trop_v(Y)$ for $v \in \Omega$. Then  we get
\begin{equation} \label{global height formula in running example}
\frac{h_{(L,\metr)}(Y)}{\deg_L(Y)}= -\frac{\int_\Omega\log\|s_D(Y)\|_v \, \nu(dv)}{[\Q(Y):\Q]}= - \int_\Omega \psi_v(u_v) \, \nu(dv) \geq \vartheta(0)= h_{(L,\metr)}(0).
\end{equation}
It follows from \eqref{slope formula with intersection numbers} in Remark \ref{slopes and relatively ample} and from \eqref{global height formula in running example} together with   the concavity of $\vartheta$ that 
\begin{equation*} 
	\widehat{\mu}_{\min}^{\mathrm{asy}}(\overline{L}) =\inf_{Y\in\Theta_X}\frac{(\overline{L}^{\dim(Y)+1}\mid Y)_S}{(\dim(Y)+1)\deg_L(Y)}=\min(h_{\overline L}(0),h_{\overline L}(\infty)) = \min(\vartheta(0),\vartheta(1)).
\end{equation*}
Using $\vartheta=\psi^\vee$ concave,  $\overline L$ is $S$-ample (resp.~in  $\overline{N_{S,\Q}^+}(X)$) if and only if $\vartheta > 0$ (resp.~$\geq 0$). Since $h_{\overline L}(\infty)=\vartheta(1)=-\int_\Omega\sup_{u\in \R}(\psi_v(u)-u) \, \nu(dv)$, the asymptotic minimal slopes is given by
\begin{equation} \label{asymptotic slope and psi}
\widehat{\mu}_{\min}^{\mathrm{asy}}(\overline{L}) =\min(\vartheta(0),\vartheta(1))= -\int_\Omega \sup_{u \in \R} \left(\psi_v(u)-\Psi(u) \right) \,\nu(dv).
\end{equation}
\end{example}

The following is a rather quick generalization of a result of Chen and Moriwaki \cite[Proposition~9.1.8~(1)]{chen2024positivity} which was formulated there in case of a perfect field.

\begin{proposition} \label{prop:positivity of arith nef}
	Let $X$ be a projective variety over $K$ of dimension $d$.
		\begin{enumerate1}
			\item \label{positivity of arith nef} For any $\overline{D_0},\dots, \overline{D_d}\in\overline{N_{S,\Q}^+}(X)$, we have that {$(\overline{D_0}\cdots \overline{D_d}\mid X)_S\geq0$.}
			\item \label{positivity of arith nef with positive} If $\overline{D_0}\leq \overline{D_0}'$ in $M_{S,\Q}(X)$, then for any  $\overline{D_1},\dots, \overline{D_d}\in\overline{N_{S,\Q}^+}(X)$, we have
			{$$(\overline{D_0} \overline{D_1}\cdots \overline{D_d}\mid X)_S\leq (\overline{D_0}' \overline{D_1}\cdots \overline{D_d}\mid X)_S.$$}
		\end{enumerate1}
\end{proposition}
\begin{proof}
	Notice that arithmetic intersection numbers are invariant under algebraic field extensions 
	and that $S$-ampleness is stable under such field extension,   
    so we may assume $K$ perfect. 
	Then \ref{positivity of arith nef} is shown in \cite[Proposition~9.1.8~(1)]{chen2024positivity}.
	
	For \ref{positivity of arith nef with positive}, we can assume that $\overline{D_0}=0$, $\overline{D_0}'=(D_0',g')\geq 0$. Then $(\overline{D_0}' \overline{D_1}\cdots \overline{D_d}\mid X)_S$ equals
	\[(\overline{D_1}\cdots \overline{D_d} \mid \mathrm{cyc}(D_0'))_S +\int_{\Omega}\left(\int_{X^\an_\omega}g_{\omega}'c_1(\overline{D_{1,\omega}})\wedge\cdots\wedge c_1(\overline{D_{d,\omega}})\right)\,\nu(d\omega)\geq 0\]
	by using \ref{positivity of arith nef} where $\mathrm{cyc}(D_0')$ is the Weil $\Q$-divisor associated to $D_0'$.
\end{proof}

\begin{lemma} \label{S-nef invariant under birational}
		Let $\varphi\colon X'\to X$ be a birational morphism of projective varieties, and $\overline{D}\in {N_{S,\Q}(X)}$. Then $\overline{D}\in \overline{N_{S,\Q}^+}(X)$ if and only if $\varphi^*\overline{D}\in\overline{N_{S,\Q}^+}(X')$.
\end{lemma}
\begin{proof}
	 By \cref{cor:arith nef stable under algebraic field extension}, we can assume that $K$ is {perfect}. The "only if" part {follows from \cref{proposition:arithmetically nef for projective}} since $\widehat{\mu}_{\min}^{\mathrm{asy}}(\varphi^*\overline{D})\geq \widehat{\mu}_{\min}^{\mathrm{asy}}(\overline{D})\geq 0$ by \cite[Theorem~6.6.6, Proposition~9.1.7]{chen2024positivity}. Conversely, we consider first the case where $\overline{D}$ is relatively $S$-ample with $\varphi^*\overline{D}\in \overline{N_{S,\Q}^+}(X')$. In this case, by {Remark \ref{slopes and relatively ample}}, 
	we have that
	\[\widehat{\mu}_{\min}^{\mathrm{asy}}(\overline{D})=\inf_{Y\in\Theta_X}\frac{(\overline{D}^{\dim(Y)+1}\mid Y)_S}{(\dim(Y)+1)\deg_D(Y)}\]
	where $\Theta_X$ is the set of all integral closed subschemes of $X$. Notice that for any $Y\in \Theta_X$, we can find {an integral closed} subscheme $Z\subset\varphi^{-1}(Y)$ such that $\varphi(Z)=Y$ {set theoretically} and $\dim(Z)=\dim(Y)$. By projection formula and \cref{prop:positivity of arith nef}~\ref{positivity of arith nef}, we have that
	\[\frac{(\overline{D}^{\dim(Y)+1}\mid Y)_S}{(\dim(Y)+1)\deg_D(Y)}=\frac{((\varphi^*\overline{D})^{\dim(Z)+1}\mid Z)_S}{(\dim(Z)+1)\deg_{\varphi^*D}(Z)}\geq 0.\]
	Hence the above shows  $\widehat{\mu}_{\min}^{\mathrm{asy}}(\overline{D})\geq 0$. This implies that $\overline{D}\in \overline{N_{S,\Q}^+}(X)$ by \cref{proposition:arithmetically nef for projective}. 
Let us deal with the general case. 
%\green{By Remark \ref{existence of S-ample}, there is $\overline{A}\in N_{S,\Q}^+(X)$. For any $n\in \N_{>0}$,} 
For any $\overline{A}\in N_{S,\Q}^+(X)$, the adelic divisor $\overline{D}+\frac{1}{2}\overline{A}$ is relatively $S$-ample, and $\varphi^*\overline{D}+\frac{1}{2}\varphi^*\overline{A}\in\overline{N_{S,\Q}^+}(X')$ by the "only if" part.  
	By our discussion above, we have that $\overline{D}+\frac{1}{2}\overline{A}\in\overline{N_{S,\Q}^+}(X)$, so $\overline{D}+\overline{A}=(\overline{D}+\frac{1}{2}\overline{A})+\frac{1}{2}\overline{A}\in N_{S,\Q}^+(X)$ by \cref{proposition:arithmetically nef for projective}.
	Hence $\overline{D}\in \overline{N_{S,\Q}^+}(X)$ by \cref{proposition:arithmetically nef for projective} {again}. This completes the proof of the lemma.
\end{proof}

Recall for a proper variety $X$ over $K$, we denote by $M_{S,\Q}(X)$  the space of integrable adelic divisors in the sense of Chen and Moriwaki, see \cref{classical global Arakelov theory as divisorial space} and \cref{adelic Green functions of Chen-Moriwarki}.

\begin{prop} \label{prop: S-nef cone}
	For any proper variety $X$ over $K$, 
	there is a  cone $N_{S,\Q}'(X)$  contained in the semipositive cone $N_{S,\Q}(X)$ of
	 $M_{S,\Q}(X)$ with the following three properties:
	\begin{enumerate}
		\item \label{first property N'}
		If $X$ is projective, then $N_{S,\Q}'(X)=\overline{N_{S,\Q}^+}(X)\cap N_{S,\Q}(X)$. 
		%\green{the closure of the cone $N_{S,\Q}^+(X)$ of $S$-ample divisors from Definition \ref{our notion for arithmetic nef} in $M_{S,\Q}(X)$.}
		\item \label{second property N'}
		Let $\varphi \colon X' \to X$ be a birational morphism of proper varieties over $K$. Then {$\overline{D} \in N_{S,\Q}(X)$} 
		is in $N_{S,\Q}'(X)$ if and only if $\varphi^*\overline{D} \in N_{S,\Q}'(X')$.
		\item \label{third property N'}
		If $K'/K$ is an algebraic field extension and $\overline{D}  \in N_{S,\Q}(X)$, then  $\overline{D}  \in N_{S,\Q}'(X)$ if and only if base change to $K'$ yields $\overline{D}_{K'} \in N_{S,\Q}'(X_{K'})$.
	\end{enumerate}
	The cones  $N_{S,\Q}'(X)$ are uniquely characterized by \ref{first property N'} and \ref{second property N'}.  
\end{prop}
\begin{proof} 
By Chow's lemma, there is a projective variety $X'$ over $K$ and  a birational morphism $\varphi \colon X' \to X$ over $K$. We define
\begin{equation} 
N_{S,\Q}'(X) \coloneqq \{\overline D \in N_{S,\Q}(X) \mid \varphi^*\overline D \in \overline{N_{S,\Q}^+}(X')\}.
\end{equation}
By \cref{S-nef invariant under birational}, this does not depend on the choice of $\varphi$ and satisfies \ref{first property N'}, \ref{second property N'}. {Uniqueness is clear and property \ref{third property N'} follows from \cref{cor:arith nef stable under algebraic field extension}.}
\end{proof}

\begin{remark} \label{S-nef on proper varieties}
%\green{We  use \cref{prop: S-nef cone} to generalize the notion of $S$-nef adelic divisors to any proper variety $X$ over $K$. The elements of $N_{S,\Q}'(X)$ and the underlying adelic $\Q$-line bundles will be called \emph{$S$-nef}.} 
We claim that $N_{S,\Q}'(X)$  is stable under tensor product and under pull-back. The first claim is immediately clear from {\cref{prop: S-nef cone}}. For projective varieties, the second claim follows from \cite[Proposition 9.1.8~(2)]{chen2024positivity} and hence follows for proper varieties from {\cref{prop: S-nef cone}}.
%\green{It follows easily from \cref{prop: S-nef cone} that the cone $N_{S,\Q}'(X)$ is closed with respect to the finite subspace topology of $M_{S,\Q}(X)$.}
\end{remark}

\begin{remark} \label{twist of relatively nef divisors}
	Let $\overline D=(D,g_D) \in N_{S,\Q}(X)$ on a proper variety $X$. 
	For an integrable function $c$ on $\Omega$, we consider the twisted adelic divisor $\overline D(c)=(D,g_D+c)$. There is such a $c \geq 0$ with $\int_\Omega c \,\nu(d\omega) >0$. %\green{We claim that $\overline D(nc){\in \overline{N_{S,\Q}^+}(X)}$ 
	%\green{is an $S$-nef adelic divisor} for $n \in \N$ sufficiently large.} 
{We investigate the \emph{twist property} whether $\overline D(nc){\in N_{S,\Q}'(X)}$ 
	%\green{is an $S$-nef adelic divisor} 
	for $n \in \N$ sufficiently large.} 
We may assume that $X$ is projective and that $K$ is {perfect}  by {\cref{prop: S-nef cone}}. Using the  criterion in Proposition \ref{proposition:arithmetically nef for projective} with the asymptotic minimal slope, the twist property holds if and only if $\widehat{\mu}^{\mathrm{asy}}_{\min}(\overline{D})>-\infty$. Indeed,
the claim follows from 
	$$\widehat{\mu}_{\min}^{\mathrm{asy}}(\overline E(c))=\widehat{\mu}_{\min}^{\mathrm{asy}}(\overline E)+\int_\Omega c \,\nu(d\omega) >0$$ for any adelic vector bundle $\overline E=(E,\metr)$ on $S$ with twist $\overline E(c) \coloneqq (E,e^{-c}\metr)$. {By \cite[Remark~6.4.3]{chen2024positivity}, if $D$ is semiample, then the twist property for $(\overline{D},c)$ holds (even when $X$ is proper by \cref{prop: S-nef cone}~\ref{second property N'}).}
\end{remark}

In the language of abstract divisorial spaces introduced in Section \ref{section: abstract divisorial spaces}, the above results and constructions can be generalized to proper varieties as follows:

\begin{theorem} \label{global abstract divisorial space}
Let $X$ be a proper variety over $K$ of dimension $d$. 
%\green{Then we have $M_{S,\Q}(X)= N_{S,\Q}'(X)-N_{S,\Q}'(X)$ and hence $(M_{S,\Q}(X),N_{S,\Q}'(X))$} 
{For the subspace $M_{S,\Q}'(X)\coloneqq N_{S,\Q}'(X)-N_{S,\Q}'(X)$ of $M_{S,\Q}(X)$, we get} an abstract divisorial space  over $\Q$ given by $(M_{S,\Q}'(X),N_{S,\Q}'(X))$   for which the arithmetic intersection numbers from \cref{global intersection number} give a $(d+1)$-intersection map to $\R$ in the sense of \cref{def: intersection map}. 
{If $X$ is projective or if there is a trivial valuation in $\Omega$, then $M_{S,\Q}'(X)=M_{S,\Q}(X)$.}
\end{theorem}

{\begin{proof}
%\green{By construction, we have $M_{S,\Q}(X)=N_{S,\Q}(X)-N_{S,\Q}(X)$.} %\green{Using suitable twists as in Remark \ref{twist of relatively nef divisors}, we get $M_{S,\Q}(X)= N_{S,\Q}'(X)-N_{S,\Q}'(X)$.} 
It follows from \cref{global intersection number}~\ref{classical global intersection number factorial} and \cref{prop:positivity of arith nef} that the arithmetic intersection numbers give a $(d+1)$-intersection map on the  abstract divisorial space over $\Q$ defined as {$(M_{S,\Q}'(X),N_{S,\Q}'(X))$}. 
If $X$ is projective, then we $M_{S,\Q}'(X)=M_{S,\Q}(X)$ by {\cref{lemma: difference of S-ample} and} Proposition \ref{prop: S-nef cone}. If there is a trivial valuation in $\Omega$, then we have seen after Remark \ref{adelic Green functions of Chen-Moriwarki} that all line bundles in $N_{S,\Q}(X)$ are semiample. Using suitable twists as in Remark \ref{twist of relatively nef divisors}, we get $M_{S,\Q}(X)= N_{S,\Q}'(X)-N_{S,\Q}'(X)=M_{S,\Q}'(X)$.
\end{proof}}

%%%%%%%%%%%%%%%%%%%%%%%%%%%%%%%%%%%%%%%%%%%%%%%%%%%%%%%%%%%%%%%%%%%%%%%%%%%%%%%%%%%%%%%%%%%%%%%%%%%%%%%%%%%%%%%%%%%%%%%%%%%%%%%%%%%%%%%%%%%%%%%%%%%%%%%%%%%%%%%%%%%%%%%%%%%%%%%%%%%%%%%%%%%%%%%%%%%%%%%%%%%%%%%%%%%%%%%%%%%%%%%%%%%%%%%%%%%%%%%%%%%%%%%%%%%%%%%%%%%%%%%%%%%%%%%%%%%%

\section{The boundary completion in the global theory} \label{section: The global boundary completion}

In this section, we will generalize the boundary completion, introduced by Yuan and Zhang in the geometric and in the number field case, to the case of a proper adelic base curve using the Arakelov theory by Chen and Moriwaki considered in Section \ref{section: global theory}.

We fix a proper adelic curve $S=(K,\Omega,\mathcal{A},\nu)$ satisfying the assumptions in \cref{assumeption on adelic curves} as in \cref{section: global theory}. We also fix a $d$-dimensional algebraic variety $U$ over $K$.

\subsection{Compactified $S$-metrized divisors on an algebraic variety}  \label{subsection: adelic divisors on an algebraic variety}

For a proper variety $X$ over $K$, we recall {the submonoid $N_{S,\Q}(X)$ of $\widehat{\Div}_{S,\Q}(X)$ in \cref{global adelic Green functions for proper varieties}, and $M_{S,\Q}(X)=N_{S,\Q}(X)-N_{S,\Q}(X)$ which is the space of integrable adelic divisors $(D,g_D)$ in the sense of Chen and Moriwaki when $X$ is projective.} 
Then we have seen in \cref{classical global Arakelov theory as divisorial space} that $(M_{S,\Q}(X),N_{S,\Q}(X))$ is an abstract divisorial space over $\Q$. The following generalizes this notion to any algebraic variety $U$ over $K$. 
We will also use  the groups $\widehat{\Div}_{S,\Q}(X)$ from \ref{global Green functions} given by pairs $(D,g_D)$, where $D$ is a $\Q$-Cartier divisor on $X$ and where $g_D$ is the Green function for $D$ associated to a (locally) $S$-bounded $S$-measurable metric on the $\Q$-line bundle $\mathcal{O}_X(D)$. It is clear that $M_{S,\Q}(X)$ is a subspace of $\widehat{\Div}_{S,\Q}(X)$.

\begin{definition} \label{def:CM divisors on non-proper}
	The space  $M_{S,\Q}(U)$ on $U$ and the subcone $N_{S,\Q}(U)$  are defined by the direct limit  
	\[(M_{{S,\Q}}(U),N_{{S,\Q}}(U))\coloneq \varinjlim_X (M_{{S,\Q}}(X),N_{{S,\Q}}(X))\]
	in the category of abstract divisorial spaces over $\Q$, where $X$ runs through  all proper compactifications of $U$. 
	{Recall that we call such $X$ \emph{proper $K$-models of $U$}, see Section \ref{sec: geometric theory}.}
	In particular,  $M_{{S,\Q}}(U)$ is an ordered $\Q$-vector space.
\end{definition}

\begin{art} 	\label{def:boundarytopologyglobal}
	A {\emph{weak boundary divisor}} (resp. a \emph{boundary divisor}) of $U$ is a pair $({X}_0,\overline{B})$ consisting of a proper $K$-model $U\hookrightarrow {X}_0$ over $K$ and 
	$\overline{B}\in M_{S,\Q}(U)_{\geq 0}$  such that $|B|\subset{X}_0\setminus U$ (resp. $|B|={X}_0\setminus U$). 
		 Since the isomorphism classes of proper $K$-models of $U$ form a directed system, it is clear that the {weak} boundary divisors form a directed subset $T$ of $M_{S,\Q}(U)_{\geq 0}$. 
		 {We emphasize that we assume here in the global theory that 	$\overline{B}\in M_{S,\Q}(U)_{\geq 0}$.} 
		 {In contrast to the local theory, there is usually no cofinal boundary divisor, see Example \ref{singular arithmetic nef in running example}.}  
		 We set
	\[\widehat{\Div}_{S,\Q}(U)_\CM\coloneq\varinjlim_{X}\widehat{\Div}_{S,\Q}(X)\supset M_{S,\Q}(U),\] 
	where $X$ runs through all proper $K$-models of $U$. 
	As in \cref{def:geometric boundary topology} and in  \cref{def:boundarytopologylocal}, see also \cref{def: b-metric}, a {weak} boundary divisor $(X_0,\overline{B})$ gives a topology, called \emph{$\overline{B}$-boundary topology}, on $\widehat{\Div}_{S,\Q}(U)_\CM$,  hence on $M_{S,\Q}(U)$, such that a basis of neighborhoods of  $\overline{D}\in \widehat{\Div}_{S,\Q}(U)_\CM$  is given by
	$$B(r,D)\coloneq\{\overline{E}\in \widehat{\Div}_{S,\Q}(U)_\CM\mid -r\overline{B}\leq \overline{E}-\overline{D}\leq r\overline{B}\}, \, r \in \Q_{>0}.$$ 
	The $\overline B$-boundary topology can also be defined using a natural pseudo-metric $d_{\overline B}$ as we have seen in \S~\ref{subsection: b-topology and completion}. 
	We set $\widehat{\Div}_{S,\Q}(U)_\CM^{d_{\overline{B}}}$ as the completion of the $\Q$-vector space $\widehat{\Div}_{S,\Q}(U)_\CM$ with respect to the $\overline{B}$-boundary topology. The space
	\[\widehat{\Div}_{S,\Q}(U)_{\cpt}\coloneq\varinjlim_{\overline{B}\in T}\widehat{\Div}_{S,\Q}(U)_\CM^{d_{\overline{B}}}\]
	is called the space of \emph{compactified $S$-metrized} \emph{divisors of $U$}. 
	{A compactified $S$-metrized divisor $\overline D\in \widehat{\Div}_{S,\Q}(U)_{\cpt}$ is called 
	\begin{enumerate}
			\item \label{strongly relatively nef}
			\emph{strongly relatively nef} if $\overline D$ is in the $T$-completion of $N_{S,\Q}(U)$; 
			\item \label{relatively integrable}
			\emph{relatively integrable} if it is the difference of two strongly relatively nef ones;
			\item  \label{relatively nef}
			\emph{relatively nef} if $\overline D$ is in the closure of the cone of strongly relatively nef compactified $S$-metrized divisors {in the subspace of relatively integrable compactified $S$-metrized divisors} with respect to the finite subspace topology.
		\end{enumerate}
	We denote the subspace of relatively integrable compactified $S$-metrized divisors in the above definition  \ref{relatively integrable} by $\widehat{\Div}_{S,\Q}(U)_\relint$. Together with the cone of strongly relatively nef compactified $S$-metrized divisors from \ref{strongly relatively nef}, {denoted by $\widehat{\Div}_{S,\Q}(U)_\relsnef$}, we get an abstract divisorial space over $\Q$ which is the $T$-completion of $(M_{S,\Q}(U),N_{S,\Q}(U))$ in the category of abstract divisorial spaces over $\Q$ as introduced in \cref{S-completion}. The closure in \ref{relatively nef} is meant with respect to the finite subspace topology introduced in \cref{closure for finite subspaces}. We denote the cone of relatively nef compactified $S$-metrized divisors by $\widehat{\Div}_{S,\Q}(U)_\relnef$, then we have an abstract divisorial space}
	$$(\widehat{\Div}_{S,\Q}(U)_\relint, \widehat{\Div}_{S,\Q}(U)_\relnef)$$
	over $\Q$.
	In more concrete terms, 
	we have the following descriptions:
	
	\begin{enumerate}
		\item \label{concrete relatively nef adelic divisors}
		A \emph{{strongly} relatively nef compactified $S$-metrized divisor of $U$} is given by choosing a {weak} boundary divisor $\overline B$ and then by a Cauchy sequence contained in $N_{{S,\Q}}(U)$ with respect to the $\overline{B}$-boundary topology.
		\item \label{concrete relatively integrable adelic divisors}
		A \emph{relatively integrable compactified $S$-metrized divisor of $U$}  is the difference of two strongly relatively nef compactified $S$-metrized divisors of $U$. Since $T$ is a directed set, both strongly relatively nef $S$-metrized divisors can be represented in \ref{concrete relatively nef adelic divisors} by Cauchy sequences in $N_{{S,\Q}}(U)$ with respect to the same $\overline{B}$-boundary topology.
		\item \label{relatively nef adelic divisors}
		{A \emph{relatively nef compactified $S$-metrized divisor $\overline D$ of $U$} is a relatively integrable compactified $S$-metrized divisor such that there is a strongly relatively nef compactified $S$-metrized divisor $\overline A$ of $U$ with $\overline D + \frac{1}{n}\overline A$ strongly relatively nef for all $n\in\N_{\geq 1}$, see \cref{lemma:closure in finite subspace topology}.} 
		\item \label{identification of Cauchy sequences}
		Two relatively integrable $S$-metrized compactified divisors of $U$ agree if and only if there is a {weak} boundary divisor $\overline B$ of $U$ such that their difference can be represented by a Cauchy sequence in $M_{S,\Q}(U)$ which converges to $0$ with respect to the $\overline B$-boundary topology. 
	\end{enumerate}
\end{art}
\begin{remark} \label{dominated by S-nef boundary divisors}
	Every weak boundary divisor $(X_0,\overline E)$ is dominated by a  boundary divisor in $N^+_{S,\Q}(X_0')$ for a suitable projective $K$-model $X_0'$ dominating $X_0$ after shrinking $U$. 
	%\green{Indeed, let $(X_0,\overline{E})$ be a weak boundary divisor of $U$.} 
	{Indeed,} by Chow's lemma, we have a birational morphism $\pi\colon X_0'\to X_0$ with $X_0'$ a  projective variety. We can take a quasi-projective open subset $U'$ of $U$ such that $\pi$ induces an isomorphism $\pi^{-1}(U')\simeq U'$ which we use for identification. Then $X_0'$ is a projective $K$-model of $U'$. 
	By blowing up the boundary, we may assume that the boundary $X_0' \setminus U'$ is the support of an effective Cartier divisor $E'$. Since $X_0'$ is projective, we may write $E'=B'-A'$ as the difference of two very ample  $\Q$-Cartier divisors $A',B'$ on $X_0'$. Similarly, by \cref{lemma: difference of S-ample}, we write $\pi^*\overline{E}=\overline{B}-\overline{A}$ as the difference of two effective $S$-ample $\Q$-divisors  $\overline{A}, \overline{B}$ on $X_0'$. Using a global twisted Fubini--Study Green function $g_{B'}\geq0$ for $B'$,  
	%\green{and replacing $U'$ by $U' \setminus |B'+B|$,} 
	we get an {effective} $S$-ample divisor $\overline{B'} =(B',g_{B'})$ 
	%\green{$\in N_{\mo,\Q}'(U')$ of $U'$}  
	on $X_0'$. Note that 
	$\pi^*\overline{E}\leq \overline{B}+\overline{B'}$, so $\overline{B}+\overline{B'}$ is $S$-ample on $X_0'$ and a
	%\green{an $S$-nef} 
	boundary divisor of $U''\coloneqq U' \setminus |B'+B|$ which dominates $\pi^*\overline{E}$. 
\end{remark}

The following result connects the global theory with the local theory from \cref{section: local theory}.

\begin{proposition} \label{prop:from global adelic to local adelic}
	For every $\overline D \in \widehat{\Div}_{S,\Q}(U)_\cpt$, there is a unique compactified geometric divisor $D \in \widetilde{\Div}_\Q(U)_\cpt$ and a unique $S$-Green function $g_D=(g_{D,\omega})_{\omega \in \Omega}$ for the $\Q$-Cartier divisor $D|_U$ with the following three  properties:
	\begin{enumerate}
		\item \label{first property for pair}
		If $\overline D$ is induced by a locally $S$-bounded $S$-measurable $S$-metrized divisor on a proper $K$-model $X$ of $U$,  then $D$ is induced by the underlying $\Q$-Cartier divisor on $X$ and $g_D$ is induced by the underlying $S$-Green function  on $X$. 
		\item \label{first' property for pair} The following map is continuous:
		$$\widehat{\Div}_{S,\Q}(U)_\cpt \longrightarrow \widetilde{\Div}_\Q(U)_\cpt \,, \quad \overline{D} \mapsto D.$$
		\item \label{second property for pair}
		For every $\omega \in \Omega$, the following map is continuous:
		$$\widehat{\Div}_{S,\Q}(U)_\cpt \longrightarrow \widehat{\Div}_\Q(U_\omega)_\cpt \,, \quad \overline{D} \mapsto ({D_\omega},g_{D,\omega}),$$
		{where $D_\omega$ is the pull-back of $D$ to $U_\omega$.}
		{Furthermore, it induces a morphism
			$$(\widehat{\Div}_{S,\Q}(U)_\relint, \widehat{\Div}_{S,\Q}(U)_\relnef) \longrightarrow 
			(\widehat{\Div}_{\Q}(U_\omega)_\integrable, \widehat{\Div}_{\Q}(U_\omega)_\nef)$$
			of abstract divisorial spaces over $\Q$.}
		\item \label{third property for pair}
		The $S$-Green function $g_D$ is $S$-measurable and locally S-bounded on $U$, and the map
		\[\widehat{\Div}_{S,\Q}(U)_\cpt\to \widehat{\Div}_{S,\Q}(U)\,, \quad \overline{D}\mapsto (D|_U,g_{D}),\]
		is additive and injective.
	\end{enumerate}
\end{proposition}
\begin{proof}
Uniqueness of $D\in\widetilde{\Div}_{S,\Q}(U)_\cpt$ is clear from \ref{first property for pair} and from continuity in \ref{first' property for pair}. Similarly, the uniqueness of $g_D$ is from \ref{first property for pair} and \ref{second property for pair}. We have to show existence.
	
	Let $(X_0,\overline B)$ be a {weak} boundary divisor of $U$ with $\overline B=(B,g_B)$. Then it is clear that $b_\omega =(B_\omega,g_{B,\omega})$ is a {weak} boundary divisor for $U_\omega$. The map 
	$${\widehat{\Div}_{S,\Q}(U)_\CM} \longrightarrow \widetilde{\Div}_\Q(U)_\cpt \, , \quad (D,g_D) \mapsto D$$
	is continuous with respect to the $\overline B$-boundary topology on ${\widehat{\Div}_{S,\Q}(U)_\CM}$ and with respect to the boundary topology on $\widetilde{\Div}_\Q(U)_\cpt$ {since $B$ is dominated by a geometric boundary divisor of $U$, see \cref{remark:independent of the choice of boundary divisors}}.  Hence the universal properties of direct limits and of completions give a unique continuous extension of the above map  from ${\widehat{\Div}_{S,\Q}(U)_\CM}$ to $\widetilde{\Div}_\Q(U)_\cpt$. For any $\overline D \in \widehat{\Div}_{S,\Q}(U)_\cpt$, we define $D\in \widetilde{\Div}_{\Q}(U)_\cpt$ as the image of $\overline D$ along this map. Then \ref{first' property for pair} is satisfied.
	
	For every $\omega \in \Omega$, 
	we proceed similarly using the local theory from Section \ref{section: local theory}. The map
	$${\widehat{\Div}_{S,\Q}(U)_\CM} \longrightarrow \widehat{\Div}_\Q(U_\omega)_\cpt \, , \quad (D,g_D) \mapsto (D_\omega,g_{D,\omega})$$
	 is continuous with respect to the $\overline B$-boundary topology on ${\widehat{\Div}_{S,\Q}(U)_\CM}$ and with respect to the boundary topology on $\widehat{\Div}_\Q(U_\omega)_\cpt$ since $(B_\omega,g_{B,\omega})$ is dominated by a cofinal boundary divisor of $U_\omega$, see Remark \ref{remarks about boundary topology}. 
	 Hence the universal properties of direct limits and of completions give a unique continuous extension of the above map  from ${\widehat{\Div}_{S,\Q}(U)_\CM}$ to $\widehat{\Div}_{S,\Q}(U)_\cpt$. For any $\overline D \in \widehat{\Div}_{S,\Q}(U)_\cpt$, the image along this map is given by  $\overline{D_\omega}\coloneq(D_\omega, g_{D,\omega})\in \widehat{\Div}_\Q(U_\omega)_\cpt$, where $D_\omega \in \widetilde{\Div}_\Q(U_\omega)_\cpt$ and $g_{D, \omega}$ is a Green function for the $\Q$-Cartier divisor $D_\omega|_{U_\omega}$ on $U_\omega$, see \ref{justification for writing pairs}. This gives a canonical continuous map $\widehat{\Div}_{S,\Q}(U)_\cpt \to \widehat{\Div}_{\Q}(U_\omega)_\cpt$ mapping $\overline{D}$ to $\overline{D_\omega}$. In more concrete terms, we can write $\overline D\in \widehat{\Div}_{S,\Q}(U)_\cpt$
	 as the limit of a sequence $\overline{D_i}=(D_i,g_i)\in {\widehat{\Div}_{S,\Q}(U)_\CM}$ inside $\widehat{\Div}_{S,\Q}(U)_\cpt$. For $\varepsilon \in \Q_{>0}$, there is $i_0 \in \N$ such that we have
	 $-\varepsilon \overline B \leq \overline{D_i} - \overline{D_j} \leq \varepsilon \overline B$ for all $i,j \geq i_0$. 
	 Equivalently, this means
	 \begin{align} \label{ineq:cauchy sequence of divisors}
	 	-\varepsilon B\leq D_{i}-D_j\leq \varepsilon B,
	 \end{align}
	 and 
	 \begin{align} \label{ineq:cauchy sequence of green}
	 	|g_{i,\omega}(x)-g_{j,\omega}(x)|\leq \varepsilon g_{B,\omega}(x)
	 \end{align}
	 for every $i,j \geq i_0$, for every $\omega \in \Omega$ and for every $x \in U_\omega^\an$. Since $B|_U=0$, we conclude from \eqref{ineq:cauchy sequence of divisors} that the sequence $D_i$ converges to  $D \in \widetilde{\Div}_\Q(U)_\cpt$ and that  $D_i|_U$ becomes eventually stationary and so defines a well-defined $\Q$-Cartier divisor denoted by $D|_U$. Then \eqref{ineq:cauchy sequence of green} shows that for every $\omega \in \Omega$ the sequence of Green functions $g_{i,\omega}$ converges locally uniformly on $U_\omega^\an$ to the Green function $g_{D,\omega}$ of $(D|_U)_\omega=D_\omega|_{U_\omega}$. In particular, $g_D \coloneqq (g_{D,\omega})_{\omega \in \Omega}$ is an $S$-Green function for $D|_U$. Since the $S$-Green functions $g_i$ are $S$-measurable, we deduce from the pointwise limit $g_{D,\omega}=\lim\limits_{j\to \infty} g_{j,\omega}$ on $U_\omega^\an$ that  $g_D$ is $S$-measurable on $U$. 
	 Since $g_B$ and all $g_{i}$ are locally $S$-bounded Green functions, it follows from \eqref{ineq:cauchy sequence of green} and from the pointwise limit $g_{D,\omega}=\lim\limits_{j\to\infty} g_{j,\omega}$ on $U_\omega^\an$ that $g_D$ is locally $S$-bounded on $U$. 
	 
	 {We conclude from the above that we can associate to $\overline{D} \in \widehat{\Div}_{S,\Q}(U)_\cpt$ a well-defined pair $(D,g_D)$ with $D \in \widetilde{\Div}_\Q(U)_\cpt$ and $g_D$ an $S$-measurable and locally $S$-bounded Green function for $D|_U$. Clearly, the map $\overline D \mapsto (D, g_D)$ is additive.  By construction, properties  \ref{first property for pair} and \ref{second property for pair} are satisfied.}
	 
	 To prove \ref{third property for pair}, we note first that the map $\overline D \mapsto (D|_U, g_D)$ is also  additive. It remains to show that this map is injective. In the definition $\widehat{\Div}_{S,\Q}(U)_\CM=\varinjlim_{X}\widehat{\Div}_{S,\Q}(X)$, we can replace each $X$ by its normalization in $U$, such that $X$ is integrally closed in $U$. 
For any Cauchy sequence $\{(D_i,g_i)\}_{i\geq 1}$ of $\widehat{\Div}_{S,\Q}(U)_\CM$ with respect to some $\overline{B}$-boundary topology, where $\overline{B}$ is a weak boundary divisor living on a proper $K$-model $X_0$ which is integrally closed in $U$, if its image in $\widehat{\Div}_{S,\Q}(U)$ is $0$, then for $\varepsilon\in \Q_{>0}$, we have $B|_U=D_i|_U=0$ and
	 \[-\varepsilon g_{B,\omega}\leq g_{i} \leq\varepsilon g_{B,\omega}\]
	 on $U_\omega^\an$ for any large $i$ and any $\omega\in\Omega$. By {\cref{lemma: global injectivity of restriction from integrally closed},}  we have $({D}_i,g_i)+\varepsilon\overline{B}\geq 0$ and $\varepsilon \overline{B}-({D}_i,g_i)\geq 0$
	 for any large $i$. Hence $\{(D_i,g_i)\}_{i\geq 1}$ is a Cauchy sequence equivalent to $0$ with respect to the $\overline{B}$-boundary topology. This proves  injectivity and hence we get \ref{third property for pair}.
\end{proof}

\begin{definition} \label{global Green functions for compactified divisors}
	We use \cref{prop:from global adelic to local adelic} to write an element $\overline D \in \widehat{\Div}_{S,\Q}(U)_{\cpt}$ from now on as $(D,g_D)$ with $D\in\widetilde{\Div}_\Q(U)_\cpt$ and $g_D$ an $S$-measurable, locally $S$-bounded $S$-Green function for $D|_U$.
	
	Let $D\in \widetilde{\Div}_{\Q}(U)_{\cpt}$. An (\emph{$S$-})\emph{Green function for $D$} is an $S$-Green function $g$ for $D|_U$ such that $(D,g)\in \widehat{\mathrm{Div}}_{S,\Q}(U)_{\cpt}$. 
\end{definition}

{\begin{remark} \label{remark: determined by Green function}
	A compactified $S$-metrized divisor $\overline D=(D,g_D) \in \widehat{\mathrm{Div}}_{S,\Q}(U)_{\cpt}$ is not always determined by the $S$-Green function $g_D$, but it is so if $U$ is normal. To see this, we assume that $g_D=0$. By Proposition \ref{prop:from global adelic to local adelic}~\ref{third property for pair}, it is enough to show that $D|_U=0$. We pick any $\omega \in \Omega$. Since $g_{D,\omega}=0$ is a Green function for $(D|_U)_\omega$ on $U_\omega$, the Weil $\Q$-divisor associated to $(D|_U)_\omega$ is $0$ as a Weil $\Q$-divisor on $U_\omega$. This is obvious if $U_\omega$ is normal, but follows also in general by passing to the normalization of $U_\omega$ and using the projection formula. Using that base change is injective on Weil $\Q$-divisors, we deduce that the Weil $\Q$-divisor associated to $D|_U$ is zero. Since $U$ is normal, we deduce $D=0$ as desired.
\end{remark}}

The abstract divisorial space $(\widehat{\Div}_{S,\Q}(U)_\relint, \widehat{\Div}_{S,\Q}(U)_\relnef)$ over $\Q$ is too large for extending the arithmetic intersection numbers as we will see later. This is the reason that we have to repeat the above construction replacing $(M_{S,\Q}(X),N_{S,\Q}(X))$ by %\green{$(M_{S,\Q}(X),N_{S,\Q}'(X))$ from \cref{prop: S-nef cone}.}
the abstract divisorial space $(M_{S,\Q}'(X),N_{S,\Q}'(X))$ from \cref{global abstract divisorial space}. %\green{Recall  that for a proper variety $X$ over $K$,  the latter is an abstract divisorial space over $\Q$ by \cref{global abstract divisorial space}.}  %\green{based on $S$-nef adelic divisors.}}

\begin{art} \label{boundary divisors in M'}
	We start the construction by defining
	$$(M_{S,\Q}'(U),N_{{S,\Q}}'(U))\coloneq %\green{\varinjlim_X (M_{{S,\Q}}(X),N_{{S,\Q}}'(X))}
	{\varinjlim_X (M_{{S,\Q}}'(X),N_{{S,\Q}}'(X))}$$
	as a direct limit in the category of abstract divisorial spaces over $\Q$ where $X$ runs over all proper $K$-models of $U$. 
\end{art}
	
\begin{art} \label{divisorial space on U based on arithmetic nef}
	{We denote by $T'$ the set of weak boundary divisors of $U$ contained in  {$M'_{S,\Q}(U)$}.} Similarly as in \cref{def:boundarytopologyglobal}, 
	a compactified $S$-metrized divisor $\overline D\in \widehat{\Div}_{S,\Q}(U)_{\cpt}$ is called 
	\begin{enumerate}
		\item \label{strongly arithmetically nef}
		\emph{strongly {arithmetically} nef} if $\overline D$ is in the $T'$-completion of $N'_{S,\Q}(U)$; 
		\item \label{arithmetically integrable}
		\emph{{arithmetically} integrable} if it is the difference of two strongly {arithmetically} nef ones;
		\item  \label{arithmetically nef}
		\emph{arithmetically nef} if $(D,g)$ is in the closure of the cone of strongly {arithmetically} nef compactified $S$-metrized divisors in the space of arithmetically integrable compactified $S$-metrized divisors with respect to the finite subspace topology.
	\end{enumerate}
	We denote the subspace of {arithmetically} integrable compactified $S$-metrized divisors in the above definition  \ref{arithmetically integrable} by $\widehat{\Div}_{S,\Q}(U)_\arint$. Together with the cone of strongly arithmetically  nef compactified $S$-metrized divisors from \ref{strongly relatively nef}, {denoted by $\widehat{\Div}_{S,\Q}(U)_\arsnef$,} we get an abstract divisorial space 
	which is the $T'$-completion of $(M_{S,\Q}'(U),N'_{S,\Q}(U))$ in the category of abstract divisorial spaces over $\Q$ as introduced in \cref{S-completion}. 
	We denote the cone of {arithmetically} nef compactified $S$-metrized divisors by $\widehat{\Div}_{S,\Q}(U)_\arnef$, then we have an abstract divisorial space		
	$$(\widehat{\Div}_{S,\Q}(U)_\arint, \widehat{\Div}_{S,\Q}(U)_\arnef)$$ 
	over $\Q$.
	Since $N_{{S,\Q}}'(U) \subset N_{{S,\Q}}(U)$, it follows  that we have a canonical injective morphism 
	$$(\widehat{\Div}_{S,\Q}(U)_\arint, \widehat{\Div}_{S,\Q}(U)_\arnef)\longrightarrow (\widehat{\Div}_{S,\Q}(U)_\relint, \widehat{\Div}_{S,\Q}(U)_\relnef)$$
	of abstract divisorial spaces over $\Q$ given as subspaces of $\widehat{\Div}_{S,\Q}(U)_\cpt$. The explicit description in \ref{def:boundarytopologyglobal} holds here as well with {arithmetically} nef/integrable replacing relatively nef/integrable and with $(M',N')$ replacing $(M,N)$. 
\end{art}

\begin{remark} \label{functoriality of rel-int}
	Let $\varphi\colon U'\to U$ be a dominant morphism of algebraic varieties. We have a pull-back map $\varphi^*\colon\widehat{\Div}_{S,\Q}(U)\to \widehat{\Div}_{S,\Q}(U')$ which will induce morphisms of abstract divisorial spaces
		\[(\widehat{\Div}_{S,\Q}(U)_{\relint},\widehat{\Div}_{S,\Q}(U)_{\relnef})\to (\widehat{\Div}_{S,\Q}(U')_{\relint}, \widehat{\Div}_{S,\Q}(U')_{\relnef}),\]
		\[(\widehat{\Div}_{S,\Q}(U)_{\arint},\widehat{\Div}_{S,\Q}(U)_{\arnef})\to (\widehat{\Div}_{S,\Q}(U')_{\arint}, \widehat{\Div}_{S,\Q}(U')_{\arnef}).\]
    so we get in particular a restriction map for open subsets. Indeed, $\varphi^*\colon\widehat{\Div}_{S,\Q}(U)\to \widehat{\Div}_{S,\Q}(U')$ is obviously well-defined. {For any proper $K$-model $X$ of $U$, Nagata's compactification theorem yields an extension of $\varphi$ to a morphism $X'\to X$ over $K$ for a suitable proper $K$-model $X'$ of $U'$. Using the pull-back from \cref{global functoriality}, we get a morphism $\varphi^*\colon (M_{S,\Q}(U), N_{S,\Q}(U))\to (M_{S,\Q}(U'), N_{S,\Q}(U'))$ of abstract divisorial spaces over $\Q$.} 
    For a weak boundary divisors $b=(X_0,\overline{B})$ of $U$ in $M_{S,\Q}(U)$, we can find an extension $\widetilde\varphi\colon X_0'\to X_0$ {of $\varphi$ to a proper $K$-model $X_0'$ of $U'$.} Then $\widetilde\varphi^*\overline{B}$ is a weak boundary divisor of $U'$. 
    By the universal property of $b$-completion in \cref{b-completion}, we have a morphism of abstract divisorial spaces $$(\widehat{\Div}_{S,\Q}(U)_{\relint},\widehat{\Div}_{S,\Q}(U)_{\relnef})\to (\widehat{\Div}_{S,\Q}(U')_{\relint}, \widehat{\Div}_{S,\Q}(U')_{\relnef}).$$ By \cref{S-nef on proper varieties}, the argument is similar in the arithmetically integrable case replacing ($M_{S,\Q},N_{S,\Q})$ by $(M_{S,\Q}',N'_{S,\Q})$.
\end{remark}

Next, we compare $\widehat{\Div}_{S,\Q}(X)_{\relnef}$ (resp. $\widehat{\Div}_{S,\Q}(X)_{\arnef}$) and $N_{S,\Q}(X)$ (resp. $N_{S,\Q}'(X)$) when $X$ is proper.

\begin{proposition} \label{prop:adelic is CM on proper varieties}
	Let $X$ be a proper variety over $K$. Then the following statements hold.
	\begin{enumerate}
		\item \label{boundary on proper 1}
		An adelic divisor of the form $(0,c)$ with $c\in\mathscr{L}^1(\Omega,\mathcal{A},\nu)$ pointwise non-negative is a boundary divisor of $X$. Conversely, for every {weak} boundary divisor $\overline{B}$ on $X$, there is a boundary divisor $\overline{B'}$ of the form above such that $\overline{B}\leq \overline{B'}$.
	\item \label{boundary on proper 2}
	We have that
	$$\widehat{\Div}_{S,\Q}(X)_{\relint}=M_{S,\Q}(X), \quad   \widehat{\Div}_{S,\Q}(X)_{\arint}=M_{S,\Q}'(X).$$
	\item \label{boundary and relatively nef}
		The  strongly relatively nef cone in $\widehat{\Div}_{S,\Q}(X)_{\relint}$ agrees with {$N_{S,\Q}(X)$.}
	\item \label{boundary and arithmetically nef}
	The  strongly arithmetically nef cone in $\widehat{\Div}_{S,\Q}(X)_{\arint}$ agrees with $N'_{S,\Q}(X)$.
	%\green{$$\widehat{\Div}_{S,\Q}(X)_{\arnef}=N_{S,\Q}'(X).$$}
	\item \label{boundary and arithmetically nef for projective}
    If $X$ is projective, then  $\widehat{\Div}_{S,\Q}(X)_{\arnef}$ agrees with the closure of the $S$-ample cone $N^+_{S,\Q}(X)$ in $M_{S,\Q}(X)$.
\end{enumerate}
\end{proposition}
\begin{proof}
	\ref{boundary on proper 1} Obviously, the pair $(0,\{c_\omega\}_{\omega\in\Omega})$ with $c\in\mathscr{L}^1(\Omega,\mathcal{A},\nu)$ pointwise non-negative is a boundary divisor of $X$. Conversely, let $B=(0,g)$ be a {weak} boundary divisor of $X$. 
By \cite[Corollary~6.2.14]{chen2020arakelov}, the function
	$$c\colon \Omega \longrightarrow \R, \,\omega \mapsto \sup\{g(x)\mid x\in X_\omega^\an\}$$
	is $\nu$-integrable. This completes the proof of \ref{boundary on proper 1}.
	
	 \ref{boundary and relatively nef} 
	  By definition, we have that $N_{S,\Q}(X)$ 
	  {is contained in the cone of strongly relatively nef compactified $S$-metrized divisors on $X$.} Conversely, let $\overline{D}=(D,g)$  
	 be a strongly relatively nef compactified $S$-metrized divisor on $X$ represented by a Cauchy sequence of $S$-metrized divisors $\overline{D_n}=(D,g_n)\in N_{S,\Q}(X)$ with respect to some $(0,g_0)$-boundary topology, where $(0,g_0)\in M_{S,\Q}(X)$ with $g_{0,\omega}\geq0$ for any $\omega\in\Omega$. By \ref{boundary on proper 1} there is a divisor $(0,c)\in M_{S,\Q}(X)_{\geq 0}$ with $c\in \mathscr{L}^1(\Omega,\mathcal{A},\nu)$ such that $g_{0,\omega}\leq c(\omega)$. For any $\varepsilon\in \Q_{>0}$, there is $n\in\N$ such that
	\begin{equation}\label{limits of divisors on proper varieties}
\max\limits_{x\in X_\omega^\an}\{|g_{n,\omega}(x)-g_\omega(x)|\}\leq \varepsilon\max\limits_{x\in X_\omega^\an}\{g_0(x)\}\leq \varepsilon c(\omega).
	\end{equation}
	This implies that $(D,g)$ is (locally) $S$-bounded. On the other hand,
	since $\lim\limits_{n\to \infty}g_{n,\omega}(x)=g_\omega(x)$ for any $\omega\in\Omega, x\in X_\omega^\an$, we have that $(D,g)$ is $S$-measurable as the pointwise limit of measurable functions is measurable. Hence $(D,g)\in\widehat{\Div}_{S,\Q}(X)$. Moreover, \eqref{limits of divisors on proper varieties} implies that $(D,g)\in N_{S,\Q}(X)$ by \cref{global adelic Green functions for proper varieties}. 
This proves \ref{boundary and relatively nef}.
	
	\ref{boundary and arithmetically nef}
	Obviously, we have that $N'_{S,\Q}(X)\subset \widehat{\Div}_{S,\Q}(X)_\arsnef$. Conversely, let $\overline{D}=(D,g)$ 
	be a {strongly arithmetically nef compactified $S$-metrized divisor represented by a Cauchy sequence of $S$-metrized divisors $\overline{D_n}=(D,g_n)\in N_{S,\Q}'(X)$} with respect to some $(0,g_0)$-boundary topology. By our discussion above, we have that $\overline{D}\in N_{S,\Q}(X)$. Let $\varphi\colon X'\to X$ be a birational morphism with $X'$ projective. By Proposition \ref{prop: S-nef cone}, it suffices to show that $\varphi^*\overline{D}\in N'_{S,\Q}(X')$. Hence we can assume that $X$ itself is projective and also that $K$ is {perfect}  by \cref{prop: S-nef cone} again. Let $\overline{A}\in N_{S,\Q}^+(X)$. 
	Then $\overline D+\overline A$ is relatively $S$-ample. Set $\Theta_X$ the set of all integral closed subschemes of $X$. {Using \eqref{limits of divisors on proper varieties} and \cref{global intersection number}~\ref{global intersection number by induction} for the change of metrics}, {it follows}  
	\begin{equation} \label{limit formula for intersection numbers}
	\lim\limits_{n\to\infty}((\overline{D_n}+{\overline{A}})^{\dim Y+1}\mid Y)_S=((\overline{D}+{\overline{A}})^{\dim Y+1}\mid Y)_S
    \end{equation} 
    for any $Y\in \Theta_X$.  %\green{Hence $((\overline{D}+\frac{1}{m}\overline{A})^{\dim Y+1}\mid Y)_S\geq 0$ by \cref{prop:positivity of arith nef}~(1) on $\overline{D_n}+\frac{1}{m}\overline{A}$.} 
	By \cref{slopes and relatively ample}, we have that
	\begin{equation}  \label{asympotic slope formula for D_n}
		\widehat{\mu}_{\min}^{\mathrm{asy}}(\overline{D_n}+{\overline{A}}) =\inf_{Y\in\Theta_X}\frac{((\overline{D_n}+{\overline{A}})^{\dim(Y)+1}\mid Y)_S}{(\dim(Y)+1)\deg_{{D_{{n}}+{A}}}(Y)}.
	\end{equation}
{Since $\overline{D_n}+\overline{A} \in N'_{S,\Q}(X)\subset \overline{N_{S,\Q}^+}(X)$, we have $\widehat{\mu}_{\min}^{\mathrm{asy}}(\overline{D_n}+\overline{A}) \geq 0$ by \cref{proposition:arithmetically nef for projective}. It follows from \eqref{asympotic slope formula for D_n} that $((\overline{D_n}+\overline{A})^{\dim Y+1}\mid Y)_S\geq 0$ for any $Y\in \Theta_X$ and hence \eqref{limit formula for intersection numbers} shows that $((\overline{D}+\overline{A})^{\dim Y+1}\mid Y)_S\geq 0$. Using the formula \eqref{asympotic slope formula for D_n} for the relatively $S$-ample $\overline{D}+\overline{A}$ instead of $\overline{D_{{n}}}+\overline{A}$ based  on \cref{slopes and relatively ample}, we conclude that $\widehat{\mu}_{\min}^{\mathrm{asy}}(\overline{D}+\overline{A}) \geq 0$. By \cref{proposition:arithmetically nef for projective}, we get {$\overline{D}\in \overline{N_{S,\Q}^+}(X)$, hence} $\overline D \in N'_{S,\Q}(X){=\overline{N_{S,\Q}^+}(X)\cap N_{S,\Q}(X)}$ proving \ref{boundary and arithmetically nef}.}

Since any arithmetically {integrable} compactified $S$-metrized divisor is the difference of two strongly arithmetically nef ones, we get from \ref{boundary and arithmetically nef} that
	$$\widehat{\Div}_{S,\Q}(X)_{\arint}=M_{S,\Q}'(X).$$
Similarly, we deduce from \ref{boundary and relatively nef} that 
$$\widehat{\Div}_{S,\Q}(X)_{\relint}=M_{S,\Q}(X).$$
which proves \ref{boundary on proper 2}.	
%\green{Moreover, using that $\widehat{\Div}_{S,\Q}(X)_{\arnef}$ is the closure of the cone of strongly arithmetically {nef} compactified $S$-metrized divisors in $\widehat{\Div}_{S,\Q}(X)_{\arint}$, we get \ref{boundary and arithmetically nef}.}

%\green{Finally,} 

{For a projective variety $X$, \cref{prop: S-nef cone}  shows that $N^+_{S,\Q}(X)\subset N'_{S,\Q}(X) \subset \overline{N^+_{S,\Q}}(X)$. Since $\widehat{\Div}_{S,\Q}(X)_{\arnef}$ is the closure of $\widehat{\Div}_{S,\Q}(X)_{\arsnef}$ in $\widehat{\Div}_{S,\Q}(X)$, we deduce \ref{boundary and arithmetically nef for projective} from \ref{boundary on proper 2} and \ref{boundary and arithmetically nef}.}
\end{proof}
\begin{remark} \label{c is S-nef}
	Let $(0,c)$ be an adelic divisor on a proper variety $X$ with $c\in\mathscr{L}^1(\Omega,\mathcal{A},\nu)$ pointwise non-negative. Then $(0,c)\in N_{S,\Q}'(X)$. 
	%\green{Indeed, let $X'\to X$ be a birational morphism with $X'$ projective.} 
	{Using Chow's lemma and \cref{prop: S-nef cone}, we may assume that $X$ is projective.} For any $S$-ample divisor $\overline{A}$ on {$X$}, by \cref{definition: arithmetically ample}  {and \eqref{twist and height}},
	%\green{and \eqref{prop: S-nef cone}}, 
	we  see that $\overline{A}+(0,c)$ is $S$-ample. 
	%\green{Using this for $\frac{1}{n}\overline{A}$ instead of $\overline{A}$, letting $n\to \infty$ and using that the cone $N_{S,\Q}'(X')$ is closed as above} 
	{By \cref{proposition:arithmetically nef for projective}}, we deduce that $(0,c)\in N'_{S,\Q}(X)$. 
	%\green{is $S$-nef on $X'$. By \cref{prop: S-nef cone}~\ref{second property N'}, we have $(0,c)\in N_{S,\Q}'(X)$.}
\end{remark}

\begin{remark}
	{Assume that $K$ is perfect.} If $X$ is projective, since relatively $S$-ample adelic divisors are in $N_{S,\Q}(X)$, and $\widehat{\Div}_{S,\Q}(X)_\relnef$ is the closure of $N_{S,\Q}(X)$ in $\widehat{\Div}_{S,\Q}(X)_\relint=M_{S,\Q}(X)$, then our definition of relative nefness coincide with the one in \cite[Definition~6.4.5]{chen2024positivity} in the projective case.
	{Moreover, \cref{prop:adelic is CM on proper varieties}~\ref{boundary and arithmetically nef for projective} shows that the arithmetically nef compactified divisors on a projective variety are precisely the nef divisors considered by Chen and Moriwaki \cite[Definition 9.1.6]{chen2024positivity}.}
\end{remark}

We illustrate the above notions in our running example.

\begin{example} \label{singular arithmetic nef in running example}
Let $K=\Q$ with canonical adelic curve $S=(\Q,\Omega,\mathcal{A},\nu)$ as in \cref{example:adelic structure of number fields} and let $U=\mathbb A_K^1$ embedded as usual in $X=\mathbb P_K^1$. We consider the  toric divisor $D= [\infty]$ of $X$ which we can view as a geometric compactified divisor in $\widetilde{\Div}_\Q(U)_\cpt$ as in \cref{running example integrable local}.

We consider the toric boundary divisor $\overline H = (H,g_H)$ on $X$ with $H=D=[\infty]$ and $g_H$ the Green function associated to the canonical metric of $L=\OO_X(1)$. This means that $g_{H,v}=\Psi \circ \trop_v$ for all places $v \in \Omega$, where $\Psi(u)=\min(0,u)$ is the concave function on $\R$ corresponding to $D$. Let $\overline B=(B,g_B)$ be any weak boundary divisor of $U$. Replacing the proper $K$-model of $U$ by its normalization, we may assume that $B$ is a Cartier divisor on $X$ with support in $X \setminus U$ which means that $B=k[\infty]$ for some $k \in \Q_{\geq 0}$. Then $kg_H$ and $g_B$ are both Green functions for $kD=k[\infty]$ on $X$, hence locally $S$-boundedness yields a non-negative integrable function $C\colon \Omega \to \R$ such that $kg_H-g_B \geq -C$. We conclude that the set of boundary divisors $\{(H,g_H+C)\mid C \in \mathscr{L}^1(\Omega,\mathcal{A},\nu)_{\geq 0}\}$ is cofinal in the set of weak boundary divisors of $U$. Since there is no $C \in \mathscr{L}^1(\Omega,\mathcal{A},\nu)_{\geq 0}$  such that $\{kC\mid k \in \N\}$ is cofinal in $\mathscr{L}^1(\Omega,\mathcal{A},\nu)$, we conclude that there is no cofinal boundary divisor of $U$. 

In the following, we pick a boundary divisor $\overline B=(B,g_B)$ of $U$ of the form $B=H=[\infty]$ and $g_B=g_H+C$ for any twist  $C \in   \mathscr{L}^1(\Omega,\mathcal{A},\nu)_{\geq 0}$. We have seen above that such boundary divisors are cofinal in the set of weak boundary divisors of $U$.

We describe first the elements $(D,g_D) \in \widehat{\Div}_{S,\Q}(U)_\relsnef$ for $D=[\infty]$ and for a toric Green function $g_D$.
We have seen in \cref{running example integrable local}  that $g_{D,v}$ is a semipositive toric Green function for $D$ if and only if there is a concave function $\psi_v\colon \R \to \R$ with asymptotic slope $1$ for $u \to -\infty$ and $\lim_{u \to \infty}\psi(u)\in \R$ such that $g_{D,v}=-\psi_v \circ \trop_v$. 
 By construction, the element $(D,g_D)\in \widehat{\Div}_{S,\Q}(U)_\relsnef$ is the limit of a sequence $(D_n,g_n)\in \widehat{\Div}_{S,\Q}(X)_\relnef$ as we always can pass to the normalization $X$ of a $K$-model of $U$. The limit is with respect to the $\overline B$-topology for a suitable choice of a boundary divisor $\overline B=(H,g_B=g_H+C)$ as above.   Then $D_n = k_n [\infty]$ for a sequence $k_n \in \Q$ converging to $1$, so we may assume that $D_n=[\infty]=D$ for all $n \in \N$ by scaling. Using torification of the Green functions $g_n$ as in \cref{running example integrable local}, we may assume that $g_{n,v}$ is a semipositive toric Green function induced by a concave function $\psi_{n,v} \colon \R \to \R$. Since $g_{n,v}$ is a Green function for $D=[\infty]$ on $X_v$, 
 we have $\psi_{n,v}-\Psi=O(1)$, see \cref{running example for local Zhang}.    
 Convergence with respect to the $\overline B$-topology can be translated by the tropicalization to $\R$ and means that for all  $\varepsilon \in \Q_{>0}$, we have
\begin{equation} \label{b-convergence for concave functions}
	-\varepsilon \left(\Psi+C(v)\right) \leq \psi_{n,v} - \psi_v \leq \varepsilon \left(\Psi+C(v)\right)
\end{equation}
for $n \gg 0$ (independent of the choice of $v \in \Omega$). We conclude that $(D,g_D) \in \widehat{\Div}_{S,\Q}(U)_\relsnef$ for a toric Green function of $D=[\infty]$ if and only if there is a sequence $\psi_{n,v}$ of concave functions on $\R$ with $\psi_{n,v}-\Psi=O(1)$ and an integrable function $C$ on $\Omega$ such that for all $\varepsilon \in \Q_{>0}$, there is $n_0 \in \N$ such that the inequalities \eqref{b-convergence for concave functions} hold for $n \geq n_0$ and all $v \in \Omega$.

For every $v \in  \Omega$, it follows from \eqref{b-convergence for concave functions} that the concave functions $\psi_{n,v}$  converge pointwise (even local uniformly) to $\psi_v$. Moreover,  as in \cite[Lemma 3.10]{burgos2023on}, it is easy to see that we may assume that the sequence $(\psi_n)_{n \in \N}$ increasingly converges to $\psi$.

We assume now that $(D,g_D) \in \widehat{\Div}_{S,\Q}(U)_\relsnef$ as above. Since the asymptotic slope for the corresponding concave function $\psi_v$ is $1$ for $u \to -\infty$ and since $\lim_{u \to \infty} \psi_v(u) \in \R$, we know that the Legendre--Fenchel dual $\psi_v^\vee$ is a concave function on $[0,1]$ and the same holds for the roof function 
$$ \vartheta(m) \coloneqq \int_{\Omega} \psi_v^\vee(m) \,\nu(dv) \quad (m \in [0,1])$$
from \eqref{roof function} in \cref{ample and nef adelic metrics in running example}.
We claim that $(D,g_D) \in \widehat{\Div}_{S,\Q}(U)_\arsnef$ if and only if $\vartheta \geq 0$. 
We look first at the \emph{non-singular} case which means that $g_D$ is Green function for $D$ on $X$. Equivalently, we have $\psi_v-\Psi=O(1)$ for all $v \in \Omega$. This was the case considered in Example \ref{semipositive adelic metrics in running example} where we have shown that	$(D,g_D) \in \widehat{\Div}_{S,\Q}(X)_\arsnef$ if and only if $\vartheta \geq 0$. In the \emph{singular} case, we describe $(D,g_D) \in \widehat{\Div}_{S,\Q}(U)_\arsnef$  as a limit of $(D_n,g_n) \in \widehat{\Div}_{S,\Q}(X)_\arnef$ with respect to a suitable boundary divisor $\overline B=(H,g_H+C)$. As above, we may assume that $D_n=D$ for all  $n \in \N$ and that all semipositive Green functions $g_n$ are toric. Then for all $\varepsilon \in \Q_{>0}$, there is $n_0 \in \N$ such that \eqref{b-convergence for concave functions} holds for all  $n \geq n_0$ and all $v \in \Omega$. Again, we may assume that for every $v \in \Omega$, the functions $\psi_{n,v}$ increasingly converge to $\psi_v$ pointwise. Since the Legendre-Fenchel dual reverses the partial order between functions and by biduality, we conclude that the sequence $\psi_{n,v}^\vee$ decreasingly converges to $\psi_v^\vee$ pointwise on $[0,1]$. It follows from the non-singular case  and the monotone convergence theorem that
$$\vartheta(m)= \int_\Omega \psi_v^\vee(m) \,\nu(dv) = \lim_{n \to \infty}\int_\Omega \psi_{n,v}^\vee(m) \,\nu(dv) \geq 0$$
for all $m \in [0,1]$ proving one direction of the claim. Conversely if $\vartheta \geq 0$ for $(D,g_D) \in \widehat{\Div}_{S,\Q}(U)_\relsnef$ with a toric Green function $g_D$, we describe $g_D$ as above as a decreasing limit of semipositive toric Green functions $g_n$ for $D$ on $X$. Then the corresponding concave functions $\psi_{n,v}$ converge increasingly pointwise to $\psi_v$, hence the Legendre--Fenchel duals $\psi_{n,v}^\vee$ converge decreasingly to $\psi_v^\vee$ pointwise on $[0,1]$. We conclude that the roof function $\vartheta_n$ of $g_{n}$ satisfies $\vartheta_n \geq \vartheta \geq 0$. We deduce from the non-singular case (see \cref{semipositive adelic metrics in running example}) that $(D,g_n) \in  \widehat{\Div}_{S,\Q}(X)_\arnef$ and hence $(D,g_D) \in \widehat{\Div}_{S,\Q}(U)_\arsnef$ proving the reverse direction.
\end{example}

\begin{remark} \label{rareness of arithmetic nef}
For a description of toric compactified $S$-metrized divisors in higher dimensions, we refer to \cite[Chapter 5]{peralta-thesis24}. {The above example was inspired by \cite[Construction 4.1]{burgos2023on} and we can draw now similar conclusions as in  \cite[Proposition 4.2]{burgos2023on}.  Let $U 	= \mathbb A_K^1 \subset X=\mathbb P_K^1$  over $K=\Q$ and let  $(D,g_D) \in  \widehat{\Div}_{S,\Q}(U)_\relsnef$ such that $D$ is a toric compactified geometric divisor and such that $g_D$ is a toric Green function for $D$. By \cref{example: projective line} and \cref{conclusion of running local example}, we have $D=a[0]+b[\infty]$ for $a \in \Q$ and   $b \in \R_{\geq 0}$ with $a+b \geq 0$. In this case, the asymptotic slopes of the concave functions $\psi_v$ corresponding to the toric Green functions $g_{D,v}$ have asymptotic slopes $b$ for $u \to -\infty$ and $-a$ for $u \to \infty$, hence the roof function $\vartheta$ is a concave function on $[-a,b]$. By a linearity argument similarly as in \cref{conclusion of running local example}, we obtain from Example \ref{singular arithmetic nef in running example} that 
$(D,g_D)$  is strongly arithmetically nef if and only if  $\vartheta \geq 0$. Note that this can only happen if  $(D,g_D) \in  \widehat{\Div}_{S,\Q}(X)_\relsnef$, i.e.~in the non-singular case considered above. Indeed, if $g_D$ is singular, then there is $v \in \Omega$ such that the concave function $\psi_v$ corresponding to the toric Green function $g_{D,v}$ satisfies $\lim_{u \to -\infty} \psi_v(u)-\Psi(u)=-\infty$. This is equivalent to $\lim_{m \to b} \psi_v^\vee(m)=-\infty$ where $m \in [-a,b[$. Then the roof function $\vartheta$ has the same singularity in $m=b$ and hence cannot be $\geq 0$.} 

More generally, we consider a toric $(D,g_D) \in  \widehat{\Div}_{S,\Q}(U)_\arint$. Then we claim that $g_D$ is \emph{non-singular} which means $(D,g_D)\in  \widehat{\Div}_{S,\Q}(X)_\arnef$. By definition, we have $(D,g_D)=(E,g_E)-(F,g_F)$ with $(E,g_E)$ and $(F,g_F)$ in $\widehat{\Div}_{S,\Q}(U)_\arsnef$. Since $D$ is toric, the restrictions of $E,F$ to $T=\mathbb G_{\rm m}^1$ give the same divisor $E_U=F_U$. 
There is $q \in \N$ such that $q[\infty]-E_U$ is ample on $X$. Picking an arithmetically nef Green function $g_q$ for $q[\infty]-E_U$ on $X$ and adding $(q[\infty]-E_U, g_q)$ to $(E,g_E)$ and to $(F,g_F)$, we may assume that both $E$ and $F$ are toric compactified geometric divisors. By torification as in \cref{singular arithmetic nef in running example}, we may assume  that $g_E, g_F$ are toric Green functions. Then the above shows that $(E,g_E),(F,g_F)\in  \widehat{\Div}_{S,\Q}(X)_\relnef$ and hence $(D,g_D)\in  \widehat{\Div}_{S,\Q}(X)_\arnef$.
\end{remark}

{The following result will be applied in later sections.}

	\begin{lemma} \label{lemma:integrable bound for green functions}
		Let $\overline{D}=(D,g) \in \widehat{\Div}_{S,\Q}(U)_{\cpt}$, and $\overline{B}=(B,g_B)\in M_{S,\Q}(U)_{\geq 0}$ a weak boundary divisor. Assume that ${D}\leq0$ and there is a Cauchy sequence in $M_{S,\Q}(U)$ converging to $\overline{D}$ with respect to $\overline{B}$-topology. Then for any $\varepsilon\in\Q_{>0}$, there is a $\nu$-integrable function $C\in\mathscr{L}^1(\Omega,\mathcal{A},\nu)$ such that 
		\[\sup\limits_{x\in U_\omega^\an}\left\{g_\omega(x)-\varepsilon g_{B,\omega}(x)\right\}\leq C(\omega)\]
		for any $\omega\in\Omega$.
	\end{lemma}
	\begin{proof}
		We choose a Cauchy sequence $\overline{D_n}=(D_n,g_n)$ in $M_{S,\Q}(U)$ converging to $\overline{D}$ with respect to $\overline{B}$-topology. Since $D \leq 0$, subtracting from $\overline{D_n}$ a small positive multiple of $\overline{B}$, we may assume that $D_n \leq 0$. For $n \in \N_{\geq 1}$ sufficiently large, the definition of the limit yields
		$$(D,g)-(D_n,g_n) \leq \varepsilon \cdot (B,g_B).$$
		There is a proper $K$-model $X_n$ of $U$ such that $\overline{D_n}\in M_{S,\Q}(X_n)$. 
		Since $D_n \leq 0$, Lemma \ref{Green functions for effective divisors} yields that $g_n \leq C$ for an integrable function $C$ on $\Omega$. This proves the claim.
\end{proof}
	
\subsection{Compactified $S$-metrized line bundles}

\label{subsection:global adelic line bundles}

As in the local case in \cref{subsection:local adelic line bundles}, we will consider compactified $S$-metrized line bundles in the global case. Recall $\widehat \Pic_{S,\Q}(U)$ denotes the group of isometry classes of locally $S$-bounded, $S$-measurable $S$-metrized $\Q$-line bundles on $U_S$, see \cref{group of adelic line bundles}.

\begin{art} \label{boundary topology on S-metrized line bundles}
We fix a weak boundary divisor $(X_0, \overline{B})$ of $U$ with $\overline{B}=(B, g_B)$. The \emph{$\overline{B}$-boundary topology} on $\widehat{\Pic}_{S,\Q}(U)$ is defined such that a {basis of neighborhoods} of an element $\overline{L}=(L, \metr)\in \widehat{\Pic}_{S,\Q}(U)$ of the topology is given by
\[B(r,\overline{L})\coloneq \left\{(L,\metr')\in \widehat\Pic_{S,\Q}(U) \,\middle\vert\, -rg_{B,\omega}\leq \log\frac{\metr_\omega'}{\metr_\omega}\leq rg_{B,\omega} \text{ for any $\omega\in\Omega$}\right\}, \quad r\in \Q_{>0}.\]
{Similarly as in \cref{def: b-metric}, we can define a pseudo-metric $d_{\overline B}$ which defines the boundary topology.} {Using} the local case, it is not hard to show that $\widehat{\Pic}_{S,\Q}(U)$ is complete with respect to the $\overline{B}$-boundary topology for any weak boundary divisor $\overline{B}$ of $U$. In fact, suppose that $(\overline{L_i})_{i\in\N}$ is a Cauchy sequence in $\widehat{\Pic}_{S,\Q}(U)$ with respect to the $\overline{B}$-boundary topology, write $\overline{L_i}=(L_i,\metr_i)$. We may assume that $L_i=L_j$ for any $i,j\in\N$, and denote $L$ the underlying line bundle. Then for any  $\omega\in\Omega$, $(\overline{L_{i,\omega}})$ is a Cauchy sequence in $\widehat{\Pic}_S(U_\omega)$ with respect to the $\overline{B_\omega}$-boundary topology. {Similarly as in \ref{boundary topology on metrized line bundles}, we conclude that the metrics $\metr_{i\omega}$ of $L_\omega$ converge to a metric $\metr_\omega$ of $L_\omega$ with respect to the $\overline{B_\omega}$-boundary topology. This is based on completeness of continuous functions on $U_\omega^\an$ as the argument in \cite[Lemma 3.6.3]{yuan2021adelic} does not use that the boundary divisor $\overline{B_\omega}$ is cofinal. As the $S$-metric $\metr \coloneqq (\metr_\omega)_{\omega \in \Omega}$ is obtained by limits of locally $S$-bounded and $S$-measurable metrics and as $g_B$ also has these properties, it follows  that $\metr$ is a locally $S$-bounded and $S$-measurable metric on $L$. Moreover, it is clear from the construction the $(L,\metr)$ is the limit of $(L_i,\metr_i)$ with respect to the $\overline{B}$-topology. This proves completeness.} 
\end{art}
														
\begin{art} \label{definition of global line bundles}
Let $P_{S,\Q}(U)$ be the subspace of $\widehat{\Pic}_{S,\Q}(U)$ consisting of all $(L,\metr)$ with $\metr$ an $S$-metric induced by an adelic $\Q$-line bundle in the sense of Chen and Moriwaki on some proper $K$-model of $U$. We denote by $Q_\Q(U)$ the subcone of $P_\Q(U)$ induced by the  $(L,\metr)$ as above corresponding {$S$-metrized divisors in $N_{S,\Q}'(U)$}. This means 
\[\text{$P_{S,\Q}(U) = \varinjlim_{X}P_{S,\Q}(X)$  and  $\quad Q_{S,\Q}(U) = \varinjlim_{X}Q_{S,\Q}(X)$}\]
where $X$ ranges over all proper $K$-models of $U$. Indeed, we may restrict the direct limit to the proper $K$-models which are integrally closed and then the map $P_{S,\Q}(X) \to P_{S,\Q}(U)$ becomes injective by \cref{lemma: global injectivity of restriction from integrally closed}, hence the above direct limits become unions. 
\end{art}

\begin{definition} \label{definition: compactified global line bundles}
Using the above terminology, we define $\widehat{\Pic}_{S,\Q}(U)_\cpt$ as the direct limit of the closure of $P_{S,\Q}(U)$ in $\widehat{\Pic}_{S,\Q}(U)$ with respect to the $\overline{B}$-boundary topology where the direct limit is over the weak boundary divisors $\overline{B}$ of $U$.
															
We say that $(L,\metr) \in \widehat{\Pic}_{S,\Q}(U)$ is \emph{strongly arithmetically nef} if $(L,\metr)$ is in the direct limit of the closure of $Q_{\Q}(U)$ with respect to some  $\overline{B}$-boundary topology. The strongly arithmetically nef elements form a cone $\widehat{\Pic}_{S,\Q}(U)_\arsnef$ in $\widehat{\Pic}_{S,\Q}(U)$. We define the subspace
$$\widehat{\Pic}_{S,\Q}(U)_\arint \coloneqq \widehat{\Pic}_{S,\Q}(U)_\arsnef - \widehat{\Pic}_{S,\Q}(U)_\arsnef$$
of $\widehat{\Pic}_{S,\Q}(U)$. Finally, the \emph{arithmetically nef cone} $\widehat{\Pic}_{S,\Q}(U)_\arnef$ is defined as the closure of $\widehat{\Pic}_{S,\Q}(U)_\arsnef$ in $\widehat{\Pic}_{S,\Q}(U)_\arint$ with respect to the finite subspace topology from \cref{finite subspace topology}. 
\end{definition}
														
We have the global versions of \cref{connecting metrized divisors and metrized line bundles} and \cref{pull-back in local case}, the proofs are similar as in the local case, we leave them to the readers. 
																
\begin{prop} \label{global:connecting metrized divisors and metrized line bundles}
	For $(D,g_D) \in \widehat{\Div}_{S,\Q}(U)$, let $\metr_D$ the $S$-metric on $\OO_U(D)$ associated to $g_D$ as in \ref{global Green functions} using \eqref{metric and Green functions}. This gives a surjective map
$$\widehat{\Div}_{S,\Q}(U) \longrightarrow \widehat{\Pic}_{S,\Q}(U) \, , \quad (D,g_D) \mapsto (\OO_U(D),\metr_D)$$ 
which maps the sets $\widehat{\Div}_{S,\Q}(U)_\cpt$,  $\widehat{\Div}_{S,\Q}(U)_\arsnef$, $\widehat{\Div}_{S,\Q}(U)_\arnef$ and $\widehat{\Div}_{S,\Q}(U)_\arint$ onto the sets  $\widehat{\Pic}_{S,\Q}(U)_\cpt$, $\widehat{\Pic}_{S,\Q}(U)_\arsnef$, $\widehat{\Pic}_{S,\Q}(U)_\arnef$ and $\widehat{\Pic}_{S,\Q}(U)_\arint$, respectively.
\end{prop}

\begin{prop} \label{global:pull-back in local case}
Let $\varphi\colon U'\to U$ be a morphism of algebraic varieties over $K$. Then the pull-back $\varphi^*\colon \widehat{\Pic}_{S,\Q}(U)\to \widehat{\Pic}_{S,\Q}(U')$ maps {the sets $\widehat{\Pic}_{S,\Q}(U)_\cpt$, $\widehat{\Pic}_{S,\Q}(U)_\arsnef$, $\widehat{\Pic}_{S,\Q}(U)_\arnef$ and $\widehat{\Pic}_{S,\Q}(U)_\arint$ to the sets $\widehat{\Pic}_{S,\Q}(U')_\cpt$, $\widehat{\Pic}_{S,\Q}(U')_\arsnef$, $\widehat{\Pic}_{S,\Q}(U')_\arnef$ and $\widehat{\Pic}_{S,\Q}(U')_\arint$, respectively.}
\end{prop}

\begin{remark} \label{forgetting map global}
	Similarly as in the local case from \cref{forgetting map local}, we have a canonical linear map \begin{align*}
			\widehat{\Pic}_\Q(U)_\cpt\to \widetilde{\Pic}_\Q(U)_\cpt
		\end{align*}
		which maps the sets 
		$\widehat{\Pic}_{S,\Q}(U)_\arsnef$, $\widehat{\Pic}_{S,\Q}(U)_\arnef$ and $\widehat{\Pic}_{S,\Q}(U)_\arint$ to the corresponding geometric sets $\widetilde{\Pic}_{\gm,\Q}(U')_\relsnef$, $\widehat{\Pic}_{\gm,\Q}(U')_\relnef$ and $\widehat{\Pic}_{\gm,\Q}(U')_\relint$, respectively.
\end{remark}

\subsection{Intersection pairing}

\begin{theorem} \label{global intersection number on algebraic varieties}
	Let $S=(K,\Omega,\mathcal{A},\nu)$ be the given  proper adelic curve. For any algebraic variety $U$ over $K$, for any $\overline{D_0}, \dots \overline{D_k} \in \widehat{\Div}_{S,\Q}(U)_{\arint}$ and any $k$-dimensional cycle $Z$ of $U$, there is a unique $(\overline{D_0} \cdots \overline{D_k} \mid Z)_S \in \R$ with the following properties:
	\begin{enumerate}
		\item \label{global intersection number of compactified line bundles} The number $(\overline{D_0} \cdots \overline{D_k} \mid Z)_S \in \R$ depends only on the isometry classes of the underlying $S$-metrized $\Q$-line bundles $\overline{L_j}=(\OO_U(D_j), \metr_j)$, $j=0,\dots, k$, and on the cycle $Z$, but not on the particular choices of the arithmetically integrable $S$-metrized divisors $\overline{D_0}, \dots, \overline{D_k}$, so we set
		$$(\overline{L_0} \cdots \overline{L_k} \mid Z)_S \coloneqq (\overline{D_0} \cdots \overline{D_k} \mid Z)_S.$$
		\item \label{global intersection number multilinear symmetric} The pairing $(\overline{L_0} \cdots \overline{L_k} \mid Z)_S \in \R$ is multilinear and symmetric in $\overline{L_0}, \dots, \overline{L_k}$ and linear in $Z$. Moreover, if $Z=U$ and $k=d=\dim(U)$, then the pairing induces a 
		$(d+1)$-intersection map from $(\widehat{\Div}_{S,\Q}(U)_\arint, \widehat{\Div}_{S,\Q}(U)_\arnef)$ to $(\R,\R_{\geq0})$.
	    \item \label{global intersection number on proper varieities} If $U=X$ is proper, then $(\overline{L_0},\cdots, \overline{L_k}\mid Z)_S$ are the arithmetic intersection numbers introduced by Chen and Moriwaki, see \cref{global intersection number}. 
		\item \label{global intersection number factorial} If $\varphi \colon U' \to U$ is a morphism of algebraic varieties over $K$ and if $Z'$ is a $k$-dimensional cycle on $U'$, then the projection formula holds:
		$$ (\varphi^*\overline{L_0} \cdots \varphi^*\overline{L_k} \mid Z')_S = (\overline{L_0} \cdots \overline{L_k} \mid \varphi_*Z')_S.$$
		\item \label{global intersection number limits} 
		If $\overline{D_j}$ are the limits of Cauchy sequences $(\overline{D_{j,n_j}})_{n_j\geq 1}$ in {$\widehat{\Div}_{S,\Q}(U)_\arsnef$} with respect to {the $\overline B$-boundary topology for some weak boundary divisor $\overline B$ of $U$,} then
		\[\lim\limits_{(n_0,\dots,n_k)\to\infty}(\overline{D_{0,n_0}}\cdots\overline{D_{k,n_k}}\mid Z)_S = (\overline{D_0}\cdots\overline{D_k}\mid Z)_S.\]

		\item \label{global intersection number field extension} Let $K'/K$ be an algebraic extension of $K$ and let $S'$ be the canonical adelic curve on $K'$ induced by $S$. Let $X',\overline{L_0'}, \dots, \overline{L_k'}, Z'$ be obtained from $X,\overline{L_0}, \dots, \overline{L_k},Z$ by base change. Then we have
		$$(\overline{L_0'} \cdots \overline{L_k'} \mid Z')_{S'}=(\overline{L_0} \cdots \overline{L_k} \mid Z)_S.$$
	\end{enumerate}
\end{theorem}
\begin{proof}
	Uniqueness is from \ref{global intersection number on proper varieities}, \ref{global intersection number factorial}, \ref{global intersection number limits}.

For the existence, we consider first the case where $Z=U$, $k=d=\dim(U)$. A morphism $\varphi:X\to X'$ of proper $K$-models of $U$ extends by definition the identity on $U$ and hence $\varphi$ is surjective. By \cref{global functoriality} and \cref{S-nef on proper varieties}, we have a pull-back $(M_{S,\Q}(X'),N_{S,\Q}'(X'))\to (M_{S,\Q}(X),N_{S,\Q}'(X))$ which is a morphism of abstract divisorial spaces.
Moreover, by \cref{global intersection number}~\ref{classical global intersection number factorial} and \cref{global abstract divisorial space}, the pull-back is a morphism of $(d+1)$-intersection maps where the latter are given by arithmetic intersection numbers. After taking the direct limit of the intersection map $(M_{S,\Q}(X),N_{S,\Q}'(X))^{d+1}\to (\R,\R_{\geq0})$ over proper $K$-models $X$ of $U$, we have a $(d+1)$-intersection map on the abstract divisorial space $(M_{S,\Q}(U),N_{S,\Q}'(U))$ {given by arithmetic intersection numbers.} 

To pass to the completion, let $(X_0, \overline{B})$ be a weak boundary divisor of $U$. By \cref{dominated by S-nef boundary divisors}, there is a dense open subset $V$ of $U$ and a boundary divisor $(X_0', \overline{B'})$  in $N_{S,\Q}'(V)$ such that $\overline{B}\leq \overline{B'}$ in $M_{S,\Q}(V)$. Then we have a morphism $(M_{S,\Q}(U),N_{S,\Q}'(U))\rightarrow (M_{S,\Q}(V),N_{S,\Q}'(V))$ of abstract divisorial spaces over $\Q$. 
{By the universal property of completions in \cref{b-completion},} we have a continuous morphism
	\begin{align} \label{morphism induced by shrinking variety}
(M_{S,\Q}(U)^{\overline B},N_{S,\Q}'(U)^{\overline B})\to (M_{S,\Q}(V)^{\overline{B'}}, N_{S,\Q}'(V)^{\overline{B'}}), 
	\end{align}
	where $(M_{S,\Q}(U)^{\overline B},N_{S,\Q}'(U)^{\overline B})$ (resp. $(M_{S,\Q}(V)^{\overline{B'}}, N_{S,\Q}'(V)^{\overline{B'}})$) is the $\overline B$-completion (resp. $\overline{B'}$-completion) of $(M_{S,\Q}(U)^{\overline B},N_{S,\Q}'(U)^{\overline B})$ (resp. $(M_{S,\Q}(V)^{\overline{B'}}, N_{S,\Q}'(V)^{\overline{B'}})$), see \cref{b-completion}. By \cref{continuous extension of intersection maps in absolute case}, we have a unique $(d+1)$-intersection map
	\begin{align} \label{intersection map on shrinking variety}
(M_{S,\Q}(V)^{\overline{B'}}, N_{S,\Q}'(V)^{\overline{B'}})^{d+1}\to (\R,\R_{\geq 0})
	\end{align} extending the one on $(M_{S,\Q}(V),N'_{S,\Q}(V))$. The composition of \eqref{morphism induced by shrinking variety} and \eqref{intersection map on shrinking variety} gives a $(d+1)$-intersection map on $(M_{S,\Q}(U)^{\overline B},N_{S,\Q}'(U)^{\overline B})$ extending the one on $(M_{S,\Q}(U),N'_{S,\Q}(U))$. Notice that $(d+1)$-intersection maps above are continuous with respect to the corresponding boundary topologies, so the $(d+1)$-intersection map we constructed above is unique. In particular, our construction is independent of the choice of $V$ and $(X_0',\overline{B'})$, and the $(d+1)$-intersection maps are compatible with the partial order on weak boundary divisors of $U$. After taking the direct limit {over all weak boundary divisors $\overline B$ and passing to the closure with respect to the finite subspace topology}, we have a $(d+1)$-intersection map on $(\widehat{\Div}_{S,\Q}(U)_{\arint},\widehat{\Div}_{S,\Q}(U)_{\arnef})$.

The above construction gives for $\overline{D_0}, \dots, \overline{D_d} \in \widehat{\Div}_{S,\Q}(U)_{\arint}$ arithmetic intersection numbers $(\overline{D_0} \cdots \overline{D_d} \mid X)_S \in \R$. We check properties \ref{global intersection number of compactified line bundles}--\ref{global intersection number field extension} in this special case. 
	The arithmetic intersection numbers are multilinear and symmetric in $\overline{D_0}, \dots, \overline{D_d}$. It follows from \cref{global intersection number}~\ref{classical global intersection number line bundles} that they depend only on the isometry class of the underlying $S$-metrized $\Q$-line bundles $\overline{L_j}=(\mathcal O_U(D_j), \metr_j)$, $j=0,\dots, d$. If $U=X$, then by construction we just get the arithmetic intersection numbers introduced by Chen and Moriwaki. Since $(d+1)$-intersection maps are continuous with respect to the boundary topologies, we also get \ref{global intersection number limits} in the special case $k=d$ and $Z=X$. Let $\varphi\colon U' \to U$ be a dominant morphism of $d$-dimensional algebraic varieties over $k$. Then functoriality
	$$(\varphi^*(\overline{D_0}) \cdots \varphi^*(\overline{D_d})\mid U')_S = [U':U] \cdot (\overline{D_0} \cdots \overline{D_d}\mid U)_S$$
	follows from the projection formula in \cref{global intersection number} and from continuity in the special case of \ref{global intersection number limits} shown above. Moreover, we deduce \ref{global intersection number field extension} in the same way from \cref{global intersection number}~\ref{classical global intersection number field extension}.

To define the arithmetic intersection numbers $(\overline{D_0} \cdots \overline{D_k} \mid Z)_S$ for any $k$-dimensional cycle $Z$ of $X$, we proceed by linearity in the components. Since the components are usually not geometrically integral and hence not varieties, we have also to involve base change to an algebraic closure of $K$. Using properties \ref{global intersection number of compactified line bundles}--\ref{global intersection number field extension} in the special case shown above, it is straightforward to show that the procedure leads to well-defined arithmetic intersection numbers $(\overline{D_0} \cdots \overline{D_k} \mid Z)_S$ which satisfy \ref{global intersection number of compactified line bundles}--\ref{global intersection number field extension} in general. We leave the details to the reader.
\end{proof}

%%%%%%%%%%%%%%%%%%%%%%%%%%%%%%%%%%%%%%%%%%%%%%%%%%%%%%%%%%%%%%%%%%%%%%%%%%%%%%%%%%%%%%%%%%%%%%%%%%%%%%%%%%%%%%%%%%%%%%%%%%%%%%%%%%%%%%%%%%%%%%%%%%%%%%%%%%%%%%%%%%%%%%%%%%%%%%%%%%%%%%%%%%%%%%%%%%%%%%%%%%%%%%%%%%%%%%%%%%%%%%%%%%%%%%%%%%%%%%%%%%%%%%%%%%%%%%%%%%%%%%%%%%%%%%%%%%%%%%%%%%%%%%%%%%%%%%%%%%%%%%%%%%%%%%%%%%%%%%%%%%%%%%%%%%%%%%%%%%%%%%%%%%%%%%%%%%%%%%%%%%%%%%%%%%%%%%%%%%%%%%%%%%%%%%%%%%%%%%%%%%%%%%%%%%%%%%%%%%%%%%%%%%%%%%%%%%%%%%%%%%%%%%%%%%%%%%%%%%%%%%%%%%%%%%%%%%%%%%%%%%%%%%%%%%%%%%%%%%%%%%%%%%%%%%%%%%%%%%%%%%%%%%%%%%%%%%%%%%%%%%%%%%%%%%%%%%%%%%%%%%%%%%%%%%%%%%%%%%%%%%%%%%%%%%%%%%%%%%%%%%%%%%%%%%%%%%%%%%%%%%%%%%%%%%%%%%%%%%%%%%%%%%%%%%%%%%%%%%%%%%%%%%%%%%%%%%%%%%%%%%%%%%%%%%%%%%%%%%%%%%%%%%%%%%%%%%%%%%%%%%%%%%%%%%%%%%%%%%%%%%%%%%%%%%%%%%%%%%%%%%%%%%%%%%%%%%%%%%%%%%%%%%%%%%%%%%%%%%%%%%%%%%%%%%%%

\section{Comparing with Yuan-Zhang's theory} \label{sec: Comparing with Yuan-Zhang's theory}

We recall  Yuan-Zhang's theory of adelic divisors in \cite[\S 2.4]{yuan2021adelic} on a quasi-projective variety $U$ over a number field $K$, and compare their definition and ours from \cref{def:boundarytopologyglobal}.

In this section, we fix a number field $K$ and quasi-projective variety $U$ over $K$. We endow $K$ with the structure of a proper adelic curve $S=(K,\Omega,\mathcal{A},\nu)$ as in \cref{example:adelic structure of number fields}.

We start with the following example of an abstract divisorial space which is basic for Arakelov theory. 

\begin{ex} \label{basic divisorial space in Arakelov theory}
	Let $X$ be a projective variety of dimension $d$ over the number field $K$. A \emph{projective $\OO_K$-model} of $X$ is a projective integral scheme $\mathcal{X}$ over $\OO_K$ with generic fiber $X$. We denote by $M_{S}(\mathcal{X})_\mo$ the group of pairs $(\mathcal{D},g_{D,\infty})$, where $\mathcal{D}$ is a Cartier divisor on $\mathcal{X}$ and where $g_{D,\infty}$ is a smooth Green function for $D=\mathcal{D}|_X$ at the archimedean places of $K$. We use the partial order on $M_{S}(\mathcal{X})_\mo$ given by $(\mathcal{D},g_{D,\infty}) \geq 0$ if and only if $\mathcal D$ is an effective Cartier divisor and $g_{D,\infty}\geq 0$. Let  $N_{S}(\mathcal{X})_\mo$ be the submonoid of  $M_{S}(\mathcal{X})_\mo$ given by the pairs $(\mathcal{D},g_{D,\infty})$ with $\mathcal{D}$ a nef divisor on $\mathcal{X}$ and such that the associated smooth hermitian metric of $\mathcal{O}_X(D)$ has positive curvature for all archimedean places. Using that every Cartier divisor on $\mathcal{X}$ is the difference of two very ample Cartier divisors and a similar consideration for the curvature, we see that
	$$\left(M_{S,\Q}(\mathcal{X})_\mo,N_{S,\Q}(\mathcal{X})_\mo \right)
	\coloneqq 
	\left(M_{S}(\mathcal{X})_\mo \otimes_\Z \Q,(N_{S}(\mathcal{X})_\mo)_\Q \right)$$
	is an abstract divisorial space over $\Q$ where $(N_{S}(\mathcal{X})_\mo)_\Q$ is the cone in $M_{S,\Q}(\mathcal{X})_\mo$ generated by the image of $N_{S}(\mathcal{X})_\mo$, see  \ref{base extension from integers to rationals}.

	Let $(M_{S,\Q}(X),N_{S,\Q}(X))$ be the abstract divisorial space over $\Q$ given by the adelic divisors in the sense of Chen and Moriwaki, see Proposition \ref{classical global Arakelov theory as divisorial space} and \cref{adelic Green functions of Chen-Moriwarki}. Then we have a morphism $M_{S,\Q}(\mathcal{X})_\mo \to M_{S,\Q}(X)$ of abstract divisorial spaces over $\Q$, given by $(\mathcal{D},g_{D,\infty})\mapsto (D,g_D)$, where $g_D$ is the Green function for $D=\mathcal{D}|_X$ given at the non-archimedeam places by the Green functions associated to the model metrics of $\mathcal{O}_X(\mathcal{D})$. To see that $g_D$ induces really an $S$-bounded and $S$-measurable metric on $X$, we note that this is the case if $\mathcal{D}$ is induced by a very ample divisor as then the associated metric is a Fubini--Study metric. In general, any model divisor $\mathcal D$ is the difference of two very ample divisors and so the claim follows easily by linearity. 
\end{ex}	

\begin{art} \label{arithmetically nef in Arakelov theory}
	Let $X$ be a projective variety over $K$, and $\mathcal{X}$ a projective model of $X$. We denote by $\overline{\mathcal{D}}$ the elements  $(\mathcal{D},g_{D,\infty})\in M_{S,\Q}(\mathcal{X})_\mo$. We call $\overline{\mathcal{D}}$ \emph{arithmetically nef} if the Green function $g_{\mathcal{D},\omega}$ is semipositive for any $\omega\in\Omega_\infty$ and $\overline{\mathcal{D}}$ has a non-negative arithmetic degree on any $1$-dimension integral closed subscheme of $\mathcal{X}$ (see \cite[A.5.3]{yuan2021arithmetic} for the definition of arithmetic degree). We denote   
	by $N'_{S,\Q}(\mathcal{X})_\mo$ the set of {arithmetically} nef elements in $M_{S,\Q}(\mathcal{X})_\mo$. It follows similarly as in \cref{lemma: difference of S-ample} that 
	$$\left(M_{S,\Q}(\mathcal{X})_\mo,  N'_{S,\Q}(\mathcal{X})_\mo   \right)$$
	is an abstract divisorial space over $\Q$. Recall that $N'_{S,\Q}(X)=\overline{N_{S,\Q}^+}(X)\cap N_{S,\Q}(X)$ from \cref{prop: S-nef cone}~\ref{first property N'}.
	%\green{denotes the cone of $S$-nef divisors in $M_{S,\Q}(X)$ in the sense of Chen and Moriwaki.} 
	By \cref{example:arithmetically nef over number fields} below, the morphism $M_{S,\Q}(\mathcal{X})_\mo \to M_{S,\Q}(X)$ from Example \ref{basic divisorial space in Arakelov theory} maps $N'_{S,\Q}(\mathcal{X})_\mo$ to $N'_{S,\Q}(X)$ and hence induces a morphism of abstract divisorial spaces with respect to the arithmetically nef cones.
\end{art}

The following lemma shows that in the number field case, 
%\green{the $S$-nef notion of Chen and Moriwaki introduced in \cref{S-nef on proper varieties}} 
$N'_{S,\Q}$ introduced in \cref{prop: S-nef cone} agrees with  the set of arithmetically nef divisors used in classical Arakelov theory. 
\begin{lemma} \label{example:arithmetically nef over number fields}
	Let $X$ be a projective variety over $K$ and let $\overline{D}\in {N}_{S,\Q}(X)$ induced by some $\overline{\mathcal{D}}\in N_{S,\Q}(\mathcal{X})_\mo$ on some projective model $\mathcal{X}$ of $X$. Then  $\overline{D}\in N_{S,\Q}'(X)$  if and only if $h_{\overline{D}}(x)\geq 0$ for any $x\in X(\overline{K})$.
\end{lemma}
\begin{proof}
	If $\overline{D}$ is in $N'_{S,\Q}(X)$, by \cite[Proposition~9.1.8~(3)]{chen2024positivity}, we have that $h_{\overline{D}}(x)\geq 0$ for any $x\in X(\overline{K})$. Conversely, assume that $h_{\overline{D}}(x)\geq 0$ for any $x\in X(\overline{K})$. 
	{We choose an ample divisor $A$ of $X$ which we endow with a twisted Fubini--Study metric such that $\overline A$ is $S$-ample, see Remark \ref{existence of S-ample}. We may choose the twist $c$ such that $c(v)\neq 0$  only for archimedean $v \in \Omega$ (one such $v$ would be enough) and hence $\overline A$ is induced by an element in $N_{S,\Q}'(\mathcal{X})_\mo$ for a suitable projective $\OO_K$-model $\mathcal{X}$ of $X$. 
		We conclude that the $S$-metrized adelic line bundles $\OO_X(\overline A)$ and $\OO_X(\overline A + n \overline D)$ are strongly Minkowskian in the sense of \cite[Definition~8.9.1]{chen2024positivity} for any $n \in \N_{>0}$}.  
	By \cite[Proposition~9.1.4]{chen2024positivity}, there is $\varepsilon>0$ such that for any $n\in\N_{>0}$, $x\in X(\overline{K})$, we have that 
	\[h_{n\overline{D}+\overline{A}}(x)=n\cdot h_{\overline{D}}(x)+h_{\overline{A}}(x)>\varepsilon.\]
	Note that $n\overline{D}+\overline{A}$ is relatively $S$-ample. It follows from  the converse direction in \cite[Proposition~9.1.4]{chen2024positivity} that $n\overline{D}+\overline{A}$ is $S$-ample, so $\overline{D}{\in N_{S,\Q}'(X)}$  by \cref{proposition:arithmetically nef for projective}.
\end{proof}

\begin{art} \label{arithmetic intersection numbers in Arakelov theory}
	Let $d \coloneqq \dim(X)=\dim(\mathcal{X})-1$. For $\overline{\mathcal{D}_0}, \dots, \overline{\mathcal{D}_d} \in M_{S,\Q}(\mathcal{X})_\mo$, Arakelov theory gives arithmetic intersection numbers $(\overline{\mathcal{D}_0} \cdots \overline{\mathcal{D}_d} \mid X) \in \R$. 
	They can also be obtained by using the morphism $M_{S,\Q}(\mathcal{X})_\mo \to M_{S,\Q}(X)$ and then using the arithmetic intersection numbers from \cref{global intersection number}. 
	Then the  \emph{height} is a $(d+1)$-intersection map 
	$$h\colon (M_{S,\Q}(\mathcal{X})_\mo,N'_{S,\Q}(\mathcal{X})_\mo)^{d+1} \longrightarrow (\R,\R_{\geq0}), \, (\overline{\mathcal{D}_0},  \dots ,\overline{\mathcal{D}_d}) \longmapsto \left(\overline{D_0} \cdots \overline{D_d} \mid X\right).$$ 
\end{art}

\begin{art} \label{model divisors on quasi-projective model}
	For a quasi-projective variety $U$ over $K$ and a quasi-projective integral scheme $\mathcal{U}$ over $\OO_K$ with generic fiber $U$, we consider the abstract divisorial space
	$$\left(M_{S,\Q}(\mathcal{U})_\mo,N'_{S,\Q}(\mathcal{U})_\mo\right)\coloneqq \varinjlim_{\mathcal{X}} \left(M_{S,\Q}(\mathcal{X})_\mo,N'_{S,\Q}(\mathcal{X})_\mo \right)$$
	where $\mathcal{X}$ ranges over all projective integral schemes over $\OO_K$ containing $\mathcal{U}$ as an open subset, see \cref{direct limits}. These schemes $\mathcal{X}$ over $\OO_K$ are called \emph{projective $\OO_K$-models of $\mathcal U$}. 
\end{art}

\begin{art} \label{original boundary completion of Yuan-Zhang}
	Let $\mathcal{U}$ be a quasi-projective integral scheme over $\OO_K$ with generic fiber $U$. A \emph{cofinal boundary divisor} of $\mathcal{U}$ is a pair $(\mathcal{X}_0, \overline{\mathcal{B}})$ consisting of a projective model $\mathcal{X}_0$ of $\mathcal{U}$ over $\OO_K$ and  $\overline{\mathcal{B}}\in M_{S,\Q}(\mathcal{X}_0)_\mo$ such that $|\mathcal{B}|=\mathcal{X}_0\setminus\mathcal{U}$ and such that $g_{B,\omega}>0$ on $U_\omega^\an$ for all $\omega \in \Omega_\infty$.
	
	Yuan and Zhang showed in \cite[\S 2.4.1]{yuan2021adelic} that a cofinal boundary divisor $\overline{\mathcal{B}}$ exists and that the induced $\overline{\mathcal{B}}$-topology on the abstract divisorial space $M_{S,\Q}(\mathcal{U})_\mo$ does not depend on the choice of $\overline{\mathcal{B}}$. 
	{In fact, a cofinal boundary divisor $\overline{\mathcal{B}}$ determines a pseudo-metric $d_{\overline{\mathcal{B}}}$ on ${M}_{S,\Q}(\mathcal{U})_\mo$ which gives the boundary topology, see \cref{def: b-metric}.} Then Yuan and Zhang considered the completion  
	$$\widehat{\Div}_{S,\Q}(\mathcal{U})_{\cpt}^{\YZ} \coloneqq \widehat{M}_{S,\Q}^{d_{\overline{\mathcal{B}}}}(\mathcal{U})_\mo$$
	of $M_{S,\Q}(\mathcal{U})_\mo$ with respect to this boundary topology, see \cref{b-completion}. Note that this group of compactified divisors on $\mathcal{U}$ is denoted by $\widehat{\Div}(\mathcal{U}/\OO_K)_\Q$ in \cite{yuan2021adelic} and they also define a lattice structure $\widehat{\Div}(\mathcal{U}/\OO_K)$, but the lattice structure does not play a role for our considerations about arithmetic intersection numbers. Moreover, Yuan and Zhang defined $(\widehat{\Div}_{S,\Q}(\mathcal{U})_{\arint}^{\YZ}, \widehat{\Div}_{S,\Q}(\mathcal{U})_{\arsnef}^{\YZ})$ as the completion of the abstract divisorial space over $\Q$ $(M_{S,\Q}(\mathcal{U})_\mo,N'_{S,\Q}(\mathcal{U})_\mo)$ with respect to the $\overline{\mathcal{B}}$-topology as defined in \cref{b-completion}. This means that $\widehat{\Div}_{S,\Q}(\mathcal{U})_{\arsnef}^{\YZ}$ is the closed cone in $\widehat{\Div}_{S,\Q}(\mathcal{U})_{\cpt}^{\YZ}$ generated by the image of $N'_{S,\Q}(\mathcal{U})_\mo$ and that
	$$\widehat{\Div}_{S,\Q}(\mathcal{U})_{\arint}^{\YZ}=\widehat{\Div}_{S,\Q}(\mathcal{U})_{\arsnef}^{\YZ}-\widehat{\Div}_{S,\Q}(\mathcal{U})_{\arsnef}^{\YZ}.$$
\end{art}

\begin{definition} \label{def:adel divisors on essentially quasi-projective}
	The set {of isomorphism classes} of quasi-projective models of $U$ over $\OO_K$ is directed. Yuan and Zhang defined the set of \emph{compactified $S$-metrized divisors} on $U$ by
	\[ \widehat{\Div}_{S,\Q}(U)^{\mathrm{YZ}}_{\cpt}\coloneq\varinjlim\limits_{\mathcal{U}}\widehat{\Div}_{S,\Q}(\mathcal{U})_{\cpt}^{\mathrm{YZ}},\]
	where $\mathcal{U}$ runs through the set of isomorphism classes of quasi-projective models $\mathcal{U}$ of $U$ over $\OO_K$. 
	We say that $\overline D \in \widehat{\Div}_{S,\Q}(U)^{\mathrm{YZ}}_{\cpt}$ is \emph{{strongly} arithmetically nef} if $\overline D$ is
		contained in the image of $\widehat{\Div}_{S,\Q}(\mathcal{U})_{\arsnef}^{\YZ}$ for some quasi-projective model $\mathcal{U}$ of $U$. We call $\overline D$ \emph{arithmetically integrable} if $\overline D$ is the difference of two strongly arithmetically nef elements in $\widehat{\Div}_{S,\Q}(U)^{\mathrm{YZ}}_{\cpt}$. The arithmetically integrable elements form a $\Q$-vector space $\widehat{\Div}_{S,\Q}(U)^{\mathrm{YZ}}_{\arint}$. We use the closure  $\widehat{\Div}_{S,Q}(\mathcal{U})_{\arnef}^{\YZ}$ of the cone $\widehat{\Div}_{S,\Q}(\mathcal{U})_{\arsnef}^{\YZ}$ in $\widehat{\Div}_{S,\Q}(U)^{\mathrm{YZ}}_{\arint}$ with respect to the finite subspace topology to get an abstract divisorial space 
		$$\left(\widehat{\Div}_{S,\Q}(U)^{\mathrm{YZ}}_{\arint},\widehat{\Div}_{S,\Q}(U)^{\mathrm{YZ}}_{\arnef} \right)$$
		over $\Q$, see \cref{closure for finite subspaces} and \cref{finite subspace topology}. 
	Yuan and Zhang show in \cite[Proposition 4.1.1]{yuan2021adelic} that the height from \cref{arithmetic intersection numbers in Arakelov theory} extends uniquely to a $(d+1)$-intersection map
	$$h\colon \left(\widehat{\Div}_{S,\Q}(U)^{\mathrm{YZ}}_{\arint},\widehat{\Div}_{S,\Q}(U)^{\mathrm{YZ}}_{\arnef} \right)^{d+1} \longrightarrow (\R,\R_{\geq0}), \, (\overline{D_0},  \dots ,\overline{D_d}) \longmapsto \left(\overline{D_0} \cdots \overline{D_d} \mid U\right).$$
\end{definition}

Now we relate the compactified $S$-metrized divisors in the sense of Yuan and Zhang from this section to the compactified $S$-metrized divisors introduced in \cref{section: The global boundary completion} based on the adelic divisors of Chen and Moriwaki.

\begin{art} \label{the comparison map iota}
	For any quasi-projective model $\mathcal{U}$ of $U$ over $\OO_K$, it follows from \cref{arithmetic intersection numbers in Arakelov theory} by passing to the direct limit that there is a canonical map $\iota \colon M_{S,\Q}(\mathcal{U})_\mo \to M_{S,\Q}(U)$ which maps the  cone $N'_{S,\Q}(\mathcal{U})_\mo$ to  the cone $N'_{S,\Q}(U)$ from \cref{boundary divisors in M'}. 
	This induces a morphism
	$$\iota \colon ( M_{S,\Q}(\mathcal{U})_\mo, N'_{S,\Q}(\mathcal{U})_\mo) \longrightarrow ( M_{S,\Q}({U}), N'_{S,\Q}({U}))$$
	of abstract divisorial spaces over $\Q$.
	Let $(\mathcal{X}_0,\overline{\mathcal{B}})$ be a cofinal boundary divisor of $\mathcal{U}$ as in \ref{original boundary completion of Yuan-Zhang}. Then $\overline B \coloneqq \iota(\overline{\mathcal{B}}) \in M_{S,\Q}(U)$ is a boundary divisor for $U$ in the sense of \ref{def:boundarytopologyglobal}. It is clear that $\iota$ 
	is continuous with respect to the boundary topologies induced by $\overline{\mathcal{B}}$ and $\overline B$, respectively. By passing to completions and then to the direct limits, we get   a canonical map $$\iota \colon 
	\widehat{\Div}_{S,\Q}(U)^{\mathrm{YZ}}_{\cpt} \longrightarrow \widehat{\Div}_{S,\Q}(U)_{\cpt}$$
	to the space of compactified $S$-metrized divisors introduced in \cref{def:boundarytopologyglobal}. We denote this map also by $\iota$. It induces a morphism
	$$\iota \colon 
	\left(\widehat{\Div}_{S,\Q}(U)^{\mathrm{YZ}}_{\arint},\widehat{\Div}_{S,\Q}(U)^{\mathrm{YZ}}_{\arnef} \right) \longrightarrow \left(\widehat{\Div}_{S,\Q}(U)_{\arint},\widehat{\Div}_{S,\Q}(U)_{\arnef} \right)$$
	of abstract divisorial spaces over $\Q$.
\end{art}

\begin{prop}\label{prop:comparison with YZ-theory}
	The canonical map $\iota \colon \widehat{\Div}_{S,\Q}(U)^{\mathrm{YZ}}_{\cpt} \to \widehat{\Div}_{S,\Q}(U)_{\cpt}$ is injective. Moreover, for $\overline{D_0}, \dots, \overline{D_d} \in \widehat{\Div}_{S,\Q}(U)^{\mathrm{YZ}}_{\arint}$, we have 
	$$\left( \overline{D_0} \cdots \overline{D_d} \mid U \right)= \left( \iota(\overline{D_0}) \cdots \iota(\overline{D_d}) \mid U \right)_S.$$
\end{prop}

\begin{proof}
	Let $\overline{D} \in \widehat{\Div}_{S,\Q}(U)^{\mathrm{YZ}}_{\cpt}$. By \cite[Proposition 3.3.1]{yuan2021adelic}, the compactified $S$-metrized divisor $\overline{D}$ is determined by a $\Q$-Cartier divisor $D_U$ on $U$ and a $S$-Green function $g_{D_U}$ for $D_U$ on $U$. By \cref{prop:from global adelic to local adelic}, we have similarly that $\iota(\overline D)$ is determined by $D_U$ and the $S$-Green function $g_D$ for $D_U$. Clearly, we have  $g_D=g_{D_U}$. This proves injectivity.
	
	Let $\mathcal{U}$ be a quasi-projective $\OO_K$-model of $U$. 
	The restriction of $\iota$ gives a morphism 
	$$\iota \colon ( M_{S,\Q}(\mathcal{U})_\mo, N'_{S,\Q}(\mathcal{U})_\mo) \longrightarrow ( M_{S,\Q}({U}), N'_{S,\Q}({U}))$$
	which preserves arithmetic intersection numbers using \ref{arithmetic intersection numbers in Arakelov theory}. In other words, it is a morphism between the $(d+1)$-intersection maps in the sense of \ref{def: morphism of ADS}. Since this is kept after completion by \cref{b-completion} and after passing to direct limits by \cref{direct limits}, we deduce the last claim.
\end{proof}

\begin{remark} \label{comparison is not surjective}
	Note that the map $\iota$ cannot be surjective in general as the space $M_{S,\Q}(X)$ from the Chen-Moriwaki theory is usually larger than $\widehat{\Div}_{S,\Q}(X)^{\mathrm{YZ}}_{\cpt}$ for a  {projective} variety $X$ over $K$. As an example, we can take $X=\mathbb P_K^1$ for $K=\Q$ and $(D,g_D)$ for the toric divisor $D=[\infty]$ and the toric Green function $g_D$ induced by the family of concave functions $\psi_v=\Psi+\frac{1}{v^2}$ with $v$ ranging over all prime numbers and $\infty$ identifying them with the places of $K=\Q$. {In particular, $\psi_\infty \coloneqq \Psi$.} Then it follows from \cref{semipositive adelic metrics in running example} and \cref{prop:adelic is CM on proper varieties} that $(D,g_D)$ is in the cone $N_{S,\Q}(X) =\widehat{\Div}_{S,\Q}(X)_\relsnef$ of $M_{S,\Q}(X)=\widehat{\Div}_{S,\Q}(X)_\arint$. On the other hand, we have $(D,g_D) \not\in \widehat{\Div}_{S,\Q}(X)^{\mathrm{YZ}}_{\cpt}$ as otherwise there would be a projective model $\mathcal X$ of $X$ over a non-empty open subset of $\Spec(\Z)$ and a $\Q$-Cartier divisor $\mathcal D$ on $\mathcal X$ with generic fiber $D$ such that $g_{D,v}$ corresponds to the model metric induced by $\mathcal D$ up to finitely many primes $v$. Since two models with the same generic fiber agree outside finitely many places, we would have $\psi_v=\Psi$ up to finitely many places as $\Psi$ is the concave function corresponding to the canonical metric of $\OO_X(D)$ which is a model metric for $v \neq \infty$.
\end{remark}

%----------------------------------------------------------------------------------------
%	Mixed relative energy
%----------------------------------------------------------------------------------------	

\section{Mixed relative energy in the local case} \label{section: mixed relative energy in the local case}

In this section, we fix a complete field $K$ with a non-trivial  valuation $v$, and fix an algebraic variety $U$ over $K$. 
We follow our conventions from the local theory given in \cref{section: local theory}. This section is strongly inspired by \cite{burgos2023on} where the complex case is handled. The non-archimedean case is new here, in the archimedean case there is a slight difference to \cite{burgos2023on} as they use pluripotential theory on complex K\"ahler manifolds while we work with semipositive metrics in the algebraic setting which allows us to handle the non-archimedean case simultaneously. At the end of every subsection, we extend our results to the trivially valued case.

\subsection{Singularities of Green functions}

The study of singularities of plurisubharmonic functions is well known in complex pluripotential theory. In this subsection, we apply it to Green functions and handle also the non-archimedean case.

\begin{definition} \label{def: more singular}
	Let $D\in  {\widetilde{\Div}_\Q(U)_\cpt}$, and $g_1, g_2$ Green functions for $D$ {in the sense of \cref{Green functions for compactified divisors}}. We say that $g_1$ is \emph{more singular than $g_2$}, denoted by $[g_1]\leq [g_2]$, if there is a constant $C\in\R$ such that $g_1\leq g_2+C$. We say that $g_1$, $g_2$ \emph{have equivalent singularities}, denoted by $[g_1]=[g_2]$, if $[g_2]\leq [g_1]$ and $[g_1]\leq[g_2]$, i.e. $|g_1-g_2|$ is bounded on $U^\an$.
\end{definition}

\begin{remark}\label{rmk:local D,g}
	A Green function for $D\in\widetilde{\Div}_\Q(U)_\cpt$ in the sense of \cref{Green functions for compactified divisors} is also a Green function for $D|_U$ in sense of \cref{Green functions} . In the following, the notation $(D,g)\in\widehat{\Div}_{\Q}(U)_\cpt$ always means that $D$ is a compactified divisor and $g$ is a Green function for $D$ as in \cref{Green functions for compactified divisors}.
\end{remark}

\begin{lemma} 	\label{lemma:maximum is nef}
	Let $(D,g), (D',g')\in \widehat{\Div}_\Q(U)_{\cpt}$ with $D\geq D'$. Set 
	\[h\coloneq\max\{g,g'\}.\]
	Then $(D,h)\in \widehat{\Div}_\Q(U)_{\cpt}$. Moreover, {if $(D,g), (D',g')$ are both nef (resp.~strongly nef) and} 
	if there is a {cofinal} boundary divisor in $N_\mo(U)$, then  $(D,h)$ is {nef (resp.~strongly nef).}
\end{lemma}
\begin{proof}
	We only consider the case where $\overline{D}\coloneq(D,g)\in \widehat{\Div}_\Q(U)_{\nef}$ and $\overline{D'}\coloneq(D',g')\in \widehat{\Div}_\Q(U)_{\nef}$, the proofs for general compactified metrized divisors are similar. {Moreover, it is clear that it is enough to show that if $\overline{D}, \overline{D'}$ are strongly nef compactified metrized divisors, then $(D,h)$ is also strongly nef.}
	
	We first consider the case where $U=X$ is proper over $K$.
	It is clear that $h$ is a Green function for $D\in\Div_\Q(X)$. Indeed, since $D\geq D'$, we have that $g-g'=\infty$ on the support of the effective $\Q$-Cartier divisor $D-D'$ and so 
	\[h-g = \max\{0,g'-g\}\]
	is a continuous function on $X^\an$.
	Since $g$ is a Green function for $D$, the same is true for $h$. 
	To prove that $(D,h)$ is the limit of some nef model metrized divisors, it is sufficient to consider the case where $(D,g), (D',g') \in N_{\mo,\Q}(X)$. By multiplying with a suitable positive integer, we may also assume that $D,D'$ are Cartier divisors on $X$. 
	
	When $v$ is archimedean, we will use the regularized maximum from \cite[Lemma~I.5.18]{demailly2012complex}. 
	For any $\eta>0$, there is a function $M_\eta\colon\R^2\rightarrow \R$ having the following properties:
	\begin{enumeratea}
		\item \label{Meta smooth} $M_\eta$ is smooth, convex and not decreasing in all variables;
		\item \label{bound Meta} $\max\{t_1,t_2\} \leq M_\eta(t_1, t_2)\leq \max\{t_1,t_2\}+\eta$;
		\item \label{Meta linearity} $M_\eta(t_1+a, t_2+a) = M_\eta(t_1, t_2)+a$ for any $a\in \R$.
	\end{enumeratea}
	For any $\eta>0$, we set 
	\[h_\eta\coloneq M_\eta(g,g').\]
	The lemma holds in the archimedean case if the following claims hold:
	\begin{enumerate}
		\item \label{heta uniformly converge} $h_\eta \to h$ uniformly when $\eta\to 0$; 
		\item \label{heta smooth} for any $\eta>0$,  $h_\eta$ is a smooth semipositive Green function  for $D$.
	\end{enumerate}
	The statement \ref{heta uniformly converge} is from the property \ref{bound Meta} of $M_\eta$. 
	For \ref{heta smooth}, let $V\subset X$ be an open subset such that the Cartier divisor $D$ is given by a local equation $f_{D}$. Using property \ref{Meta linearity}, we have
	$$h_\eta+\log|f_D|=M_\eta(g+\log|f_D|,g'+\log|f_{D}|).$$
	Property \ref{Meta smooth} shows that the right hand side is a smooth psh function on $V^\an$, where we use \cite[Theorem I.5.6]{demailly2012complex} for plurisubharmonicity. This proves \ref{heta smooth}.

	When  $v$ is non-archimedean, this is shown as follows. First, we note that semipositivity is a local property, see \cite[Proposition 3.11]{gubler2019on}. Then around a given point $x \in X^{\rm an}$, we may change $D'$ to $D'-{\rm div}(f')$ for a local equation $f'$ of $D'$ and replace $D$ by $D-{\rm div}(f')$, so we may assume that $D'=0$. Now if $x$ is not in the support of $D$, then $g,g'$ are semipositive functions and hence $\max\{g,g'\}$ is semipositive. This follows from  \cite[Propositions 2.7 and 3.12]{gubler2019on}. If $x$ is in the support of $D$, then $g(x)=\infty$ and hence $\max\{g,g'\}=g$ in a neighborhood of $x$. Then $\max\{g,g'\}$ is semipositive at $x$ as this holds for $g$. This proves the non-archimedean case.

	Next, we consider the general case when $U$ is not necessarily proper over $K$. We fix a {cofinal} boundary divisor $(X_0, \overline{B})$ {in $N_{\mo}(U)$}. 
	{By \cref{interpretation of YZ-elements}, we may represent $\overline D$ and $\overline{D'}$ by Cauchy sequences} 
	$(\overline{D_n})_{n\geq 1}$ and $(\overline{D_n'})_{n\geq 1}$ with respect to the $\overline B$-topology. We assume that $\overline{D_n}=(D_n,g_n), \overline{D_n'}=(D_n',g_n')\in {N_{\mo,\Q}(X_n)}$ 
	for some proper $K$-model $X_n$ of $U$. Since $D\geq D'$, we can assume that $D_n\geq D_n'$ after adding to $\overline{D_n}$ a small positive multiple of $\overline{B}$ if necessary. Set 
	\[h_{n}\coloneq\max\{g_{n},g'_{n}\}.\]
	By our discussion above, $(D_n, h_n)\in \widehat{\Div}_\Q(X_n)_{\nef}\subset \widehat{\Div}_\Q(U)_{\nef}$. 
	Since $(D_n,g_n)$ converges to $(D,g)$ and $(D_n',g_n')$ converges to $(D',g')$ with respect to the boundary topology on $\widehat{\Div}_\Q(U)_{\cpt}$, we conclude that $(D_n,h_n)$ converges to $(D,h)$ proving the claim.
\end{proof}

\begin{remark} \label{nef boundary divisor assumption}
	The existence of a {cofinal} boundary divisor $\overline B \in N_\mo(U)$ is crucial here and in the following. Note that we can always obtain it by passing to a suitable dense open subset $U'$. Indeed, in a first step, we can choose a quasi-projective  open subset $U'$ of $U$. Then there is a projective $K$-model $X'$ of $U'$. By blowing up the boundary, we may assume that the boundary $X' \setminus U'$ is the support of an effective Cartier divisor $E$. Since $X'$ is projective, we may write $E=B-A$ as the difference of two very ample divisors $A,B$ on $X'$. Using a twisted Fubini--Study Green function $g_B>0$ for $B$ and replacing $U'$ by $U' \setminus |B|$, we get a {cofinal} boundary divisor $\overline B =(B,g_B) \in N_{\mo,\Q}(U')$ of $U'$. 
\end{remark}

\begin{proposition}	\label{prop:local convergent sequence}
	Let $(D,g), (D,g')\in \widehat{\Div}_\Q(U)_{\cpt}$. 
	For $n\in\N$, we set
	\[g_{n}'\coloneq\max\{g-n,g'\}.\]
	Then for each $n\in\N$, we have that $(D,g_n')\in\widehat{\Div}_\Q(U)_{\cpt}$ and the sequence $(D,g_n')_{n\geq 1}$ converges decreasingly to $(D,g')$ with respect to the boundary topology. If $(D,g), (D,g')$ {are both nef (resp.~strongly nef)} 
	and if there is a cofinal boundary divisor of $U$ in $N_\mo(U)$, then $(D,g_n')$ is nef (resp.~strongly nef).
\end{proposition}
\begin{proof}
	The proof is similar to the one in \cite[Proposition~3.42]{burgos2023on}.
	
	We fix a {cofinal} boundary divisor $(X_0,\overline{B})$ with $\overline{B}=(B,g_B)$. {By Lemma \ref{lemma:maximum is nef},} 
	$g_n'$ is {a Green function for $D$ and $(D,g_n')\in \widehat{\Div}_\Q(U)_{\cpt}$.}

	By \cref{proposition:green functions for adelic divisors}, there is a continuous function $h$ with $h\in C(X_{0}^\an)$ such that $g-g'=h\cdot g_{B}$, and $h$ vanishes on $|B|^\an$. For every $\varepsilon\in \Q_{>0}$, let
	\[K_\varepsilon\coloneq\{x\in X_{0}^\an\mid h(x)\geq \varepsilon\}\]
	which is a compact subset of $U^\an$. Since $g-g'$ is continuous, we can find an integer $n_\varepsilon\geq 1$ such that $g-g'=hg_{B}\leq n_\varepsilon$ on $K_\varepsilon$.
	For any $n\geq n_\varepsilon$, we claim that
	\[0\leq g_n'-g'\leq \varepsilon g_B,\]
	this will imply that $(D,g_n')_{n\geq 1}$ converges decreasingly to $(D,g')$ with respect to the boundary topology. The left-hand side is clear since $g_n'-g'=\max\{g-g'-n,0\}\geq0$. For the right-hand side, let $x\in U^\an$. If $g(x)-n\leq g'(x)$, then $g_n'(x)=g'(x)$, and the right-hand side inequality follows from positivity of $g_B$.  If $g(x)-n> g'(x)$, then $$h(x)g_B(x)=g(x)-g'(x)>n\geq n_\varepsilon,$$ so $x\not\in K_\varepsilon$, i.e. $h(x)<\varepsilon$. Hence 
	\[g'_{n}(x)-g'(x)=g(x)-g'(x)-n\leq g(x)-g'(x)=h(x)g_{B}(x)\leq \varepsilon g_{B}(x).\]
	This proves our claim.
	
	If $(D,g), (D,g')$ {are both nef (resp.~strongly nef)} 
	{and if $\overline B \in N_\mo(U)$},  then it follows from \cref{lemma:maximum is nef}  that $(D,g_n')$ is {nef (resp.~strongly nef).}  
	This completes the proof.
\end{proof}

\begin{lemma} 	\label{lemma: same divisoral part for local divisors}
	Let $(D,g), (D,g')\in \widehat{\Div}_\Q(U)_{\cpt}$. Then there are sequences of model metrized divisors $(D_n,g_n)_{n\geq 1}, (D_n,g_{n}')_{n\geq 1}$ in $\widehat{\Div}_\Q(U)_{\mathrm{mo}}$ with the same divisorial part on proper $K$-models of $U$, decreasingly converging to $(D,g)$ and to $(D,g')$, respectively. 
	Moreover, 
	\begin{enumerate1}
		\item \label{local:same divisoral part nef} if there is a {cofinal} boundary divisor in $N_\mo(U)$ and if $(D,g), (D,g')$ are  nef (resp. strongly nef), then we can choose $(D_n,g_n)$ and $(D_n,g'_n)$ to be  nef (resp. strongly nef). 
		\item \label{local:same divisoral part bounded} if $|g-g'|\leq C$, then for any sequence of positive real numbers $(\varepsilon_n)_{n\geq 1}$ converging to $0$, we can choose $(D_n,g_n)$ and $(D_n,g_{n}')$ such that $|g_n-g_n'|\leq C+\varepsilon_n$; 
		\item \label{local:same divisoral part nef bounded} if both conditions \ref{local:same divisoral part nef}  and \ref{local:same divisoral part bounded} hold, then we can choose $(D_n,g_n)$ and $(D_n,g_{n}')$ satisfying the properties in \ref{local:same divisoral part nef}  and \ref{local:same divisoral part bounded} at the same time.
	\end{enumerate1}
\end{lemma}
\begin{proof}
	We fix a {cofinal} boundary divisor $(X_0, \overline{B})$ in $N_\mo(U)$. We only consider the case where $\overline{D}\coloneq(D,g), \overline{D'}\coloneq(D,g')$ {are strongly nef,} 
	the proofs for general compactified metrized divisors and nef compactified metrized divisors are similar. So we only prove \ref{local:same divisoral part nef} and \ref{local:same divisoral part nef bounded}. The other cases are similar.
	
	\ref{local:same divisoral part nef} Let $({D}_{n},g_n)_{n\geq 1}$ and $({D}_{n}', g_n')_{n\geq 1}$ be Cauchy sequences of nef model metrized divisors representing $\overline{D}$ and $\overline{D'}$ respectively. Assume that $\overline{D_n}\coloneq (D_n, g_n), \overline{D_n'}\coloneq(D_n',g_n')\in N_{\mo,\Q}(X_n)$. We claim that we can assume that $D_n'=D_n$. Since the underlying divisors of $\overline{D}$ and $\overline{D'}$ coincide in $\widetilde{\Div}_\Q(U)_\cpt$, we can assume that $D_n\geq D_n'$ after adding to $\overline{D_n}$ a small positive multiple of $\overline{B}$. Set 
	\[h'_{n}\coloneq\max\{g_{n}',g_{n}-n\}.\]
	Then $h_{n}'$ is a Green function for $D_n$, and $(D_n,h'_{n})$ is nef by \cref{lemma:maximum is nef}. Moreover, we will show that $(D_n,h_{n}')$ converges to $(D,g')$ following the proof of \cite[Lemma~3.41~(i)]{burgos2023on}. By \cref{proposition:green functions for adelic divisors}, we have that $|g-g'|=o(g_{B})$ when approaching the boundary $|B|^\an \subset X_{0}^\an$. For any $\varepsilon\in \Q_{>0}$, since $X_{0}^\an$ is compact, there is a constant $C>0$ such that 
	\[|g-g'|\leq C+\frac{\varepsilon}{2}g_{B}.\]
	We take $n_0>C$ such that 
	\[|g_{n}-g|\leq \frac{\varepsilon}{2} g_{B} \ \ \text{ and } \ \ |g_{n}'-g'|\leq {\varepsilon}g_{B}\]
	hold for all $n\geq n_0$. Then, for any $n\geq n_0$, we have that
	\[g_{n}-n \leq g-n+\frac{\varepsilon}{2} g_{B}\leq g'+C-n+{\varepsilon}g_{B}\leq g'+\varepsilon g_{B}.\]
	Since $g'-\varepsilon g_{B} \leq g_{n}'\leq g'+\varepsilon g_{B}$, we have that
	\[g'-\varepsilon g_{B}\leq g_{n}' \leq h'_{n} = \max\{g_{n}',g_{n}-n\}\leq g'+\varepsilon g_{B},\]
	which shows that $(D_n,h_{n}')$ converges to $(D,g')$ with respect to the boundary topology. By \cref{lemma:maximum is nef}, $(D_n,h_n')\in \widehat{\Div}_\Q(X_n)_{\nef}$. So we can replace $h_n'$ by another Green function $h_n$ for $D_n$ such that $(D_n,h_n)$ is a nef model metrized divisor and converges to $(D,g')$. So our claim holds.
	
	We refine $\overline{D_n}, \overline{D_n'}$ by adding a small positive multiple of $\overline{B}$.
	After extracting subsequences, we can assume that
	\[\overline{D}-\frac{1}{2^n}\overline{B}\leq \overline{D_n}\leq \overline{D}+\frac{1}{2^n}\overline{B},\]
	\[\overline{D}'-\frac{1}{2^n}\overline{B}\leq \overline{D_n'}\leq \overline{D'}+\frac{1}{2^n}\overline{B}.\]
	Set $\overline{E_n}\coloneq\overline{D_n}+\frac{4}{2^n}\overline{B}$ and $\overline{E_n'}\coloneq\overline{D_n'}+\frac{4}{2^n}\overline{B}$. Then
	\[\overline{E_{n+1}}=\overline{D_{n+1}}+\frac{4}{2^{n+1}}\overline{B}\leq \overline{D}+\frac{5}{2^{n+1}}\overline{B}\leq \overline{D_n}+\frac{7}{2^{n+1}}\overline{B}\leq \overline{D_n}+\frac{4}{2^n}\overline{B}=\overline{E_n}.\]
	Similarly, we have $\overline{E_{n+1}'}\leq \overline{E_{n}'}$. This completes the proof of (1) after replacing $\overline{D_n}, \overline{D_n'}$ by $\overline{E_n}, \overline{E_{n}'}$.

	\ref{local:same divisoral part nef bounded} Assume that $|g-g'|\leq C$. By (1), we take Cauchy sequences $({D}_{n},g_{0,n})_{n\geq 1}$ and $({D}_{n}, g_{0,n}')_{n\geq 1}$ of nef model metrized divisors representing $\overline{D}$ and $\overline{D}'$, respectively. Set
	\[h_{n}= \max\{g_{0,n},g'_{0,n}-C\} \ \  \text{ and } \ \ h_{n}'= \max\{g_{0,n}',g_{0,n}-C\}.\]
	It is not hard to see that $|h_n-h_n'|\leq C$ and $(D_n,h_n)$ (resp. $(D_n,h_n')$) converges to $\overline{D}$ (resp. $\overline{D'}$) with respect to the boundary topology. 
	By  \cref{lemma:maximum is nef}, $(D_n,h_n)$ is a nef compactified metrized divisor {on a proper $K$-model $X_n$} of $U$, so for any $\varepsilon_n>0$, we can take a Green function $g_n$ for $D_n$ such that $(D_n, g_n)$ is a nef model metrized divisor {on $X_n$} and $|g_n-h_n|\leq \frac{\varepsilon_n}{2}$. Similarly, we can take $g_n'$ for $h_n'$. Then $|g_n-g_n'|\leq C+\varepsilon_n$. After extracting subsequences and adding small positive multiples of $\overline{B}$ as above, we can assume that $(D_n,g_n)_{n\geq 1}$ and $(D_n, g_n')_{n\geq 1}$ decrease. This proves (3). 
\end{proof}

\begin{lemma}[Integration by parts] 	\label{lemma:integral by parts}
	Let $\overline{D_{0}}=(D_0,g_{0}), \overline{D_{0}'}=(D_0,g_{0}'), \overline{D_{1}}=(D_1,g_{1}), \overline{D_{1}}'=(D_1,g_{1}'),\overline{D_{2}}, \dots, \overline{D_{d}}\in \widehat{\Div}_\Q(U)_{\nef}$. Assume that $[g_{0}]=[g'_{0}]$, $[g_{1}]=[g'_{1}]$. Then we have that
	\begin{align*}
		&\int_{U^\an}(g_{0}'-g_{0})c_1(0,g_{1}'-g_{1})\wedge c_1(\overline{D_{2}})\wedge\cdots\wedge c_1(\overline{D_{d}})\\
		=& \int_{U^\an}(g_{1}'-g_{1})c_1(0,g_{0}'-g_{0})\wedge c_1(\overline{D_{2}})\wedge\cdots\wedge c_1(\overline{D_{d}}).
	\end{align*}
\end{lemma}

\begin{proof}
	We first assume that all divisors are nef model metrized divisors on a proper $K$-model $X$ of $U$. Then the equality is from \cite[Theorem~5.6]{moriwaki2014arakelov} for the the archimedean case, and from \cite[Proposition~11.5]{gubler2017a} for the non-archimedean case.
	
	In general, we note first that the projection formula in \cref{proposition:measures for nef adelic}\ref{MA measure via dominant morphism} shows that we may replace $U$ by any dense open subset. Then  \cref{nef boundary divisor assumption} shows that we may assume that $U$ has a {cofinal} boundary divisor $\overline B$ in $N_\mo(U)$.
	Moreover, by {\cref{nef become strongly nef after shrinking U} and projection formula in \cref{proposition:measures for nef adelic}\ref{MA measure via dominant morphism}}, we may assume that $\overline{D_0},\overline{D_0'}, \overline{D_1}, \dots, \overline{D_d}$ are strongly nef {after shrinking $U$}.
	Let $(\overline{D_{j,n}})_{n\geq 1}$ (resp. $(\overline{D'_{j,n}})_{n\geq 1}$) be sequences of  nef model divisors converging to $\overline{D_{j}}$ (resp. $\overline{D_{j}}'$) {in the $\overline B$-topology} with $\overline{D_{j,n}}=(D_{j,n},g_{j,n})$ (resp. $\overline{D_{j,n}'}=(D_{j,n}',g_{j,n}')$) $\in N_{\mo,\Q}(X_n)$ for some proper $K$-model $X_n$ of $U$ for $j=0,\dots, d$ (resp. $j=0,1$). By \cref{lemma: same divisoral part for local divisors}\ref{local:same divisoral part nef bounded} , we can assume that $D_{0,n}=D_{0,n}',$ $D_{1,n}=D_{1,n}'$, and $g_{0,n}'-g_{0,n}$, $g_{1,n}'-g_{1,n}$ are uniformly bounded. Then the first case shows
	\begin{align*}
		&\int_{U^\an}(g_{0,n}'-g_{0,n})c_1(0,g_{1,n}'-g_{1,n})\wedge c_1(\overline{D_{2,n}})\wedge\cdots\wedge c_1(\overline{D_{d,n}})\\
		=& \int_{U^\an}(g_{1,n}'-g_{1,n})c_1(0,g_{0,n}'-g_{0,n})\wedge c_1(\overline{D_{2,n}})\wedge\cdots\wedge c_1(\overline{D_{d,n}}).
	\end{align*}
	By \cref{corollary:weakly convergence of full mass}, \cref{lemma:boundary convergence is locally uniformly convergence} and \cref{cor:weakly converges of measures}, the lemma follows after we take the limits of both sides of the equality above.
\end{proof}

\begin{remark} \label{trivially valued case}
	The results of this subsection also hold in the trivially valued case.  
	The crucial Lemma \ref{lemma:maximum is nef} is proved similarly as in the non-trivially valued case by replacing model functions by twisted Fubini--Study Green functions. We have to use that the maximum of two twisted Fubini--Study Green functions is a twisted Fubini--Study Green function which readily follows from the definition of twisted Fubini--Study metrics given in \cref{Fubini-Study metric}. Then \cref{prop:local convergent sequence} is again a direct consequence of \cref{lemma:maximum is nef}.  
	In the trivially valued case, \cref{lemma: same divisoral part for local divisors} and \ref{lemma:integral by parts} follow by the same  arguments as before, the latter can be also deduced by base change to a non-trivially valued non-archimedean field extension. 
\end{remark}

\subsection{Mixed relative energy} \label{subsection: local mixed relative energy}

The mixed relative energy for singular psh functions was studied in the complex case by  Darvas, Di Nezza and Lu \cite{darvas2023relative}. We give here an approach for singular semipositive Green functions which works also in the non-archimedean case. This extends work of \cite{boucksom2015solution}, \cite{boucksom2021spaces}, \cite{boucksom2021non} and \cite{boucksom2022global}.

For any $(D,g)\in \widehat{\Div}_\Q(U)_{\nef}$, we define $\mathcal{E}(g)$ as the set of Green functions $h$ of $D$ with $(D,h)\in \widehat{\Div}_\Q(U)_{\nef}$ and $[h]\leq [g]$. 
Given $(D_0,g_0),\dots, (D_d,g_d) \text{ and } (D_0,h_0), \dots, (D_d, h_d)\in \widehat{\Div}_\Q(U)_{\nef}$, {we set $\mathbf{g}=(g_0,\dots, g_d), \mathbf{h}=(h_0,\dots, h_d)$. Then we define the relation $[\mathbf{h}]\leq [\mathbf{g}]$ (resp. $[\mathbf{h}]= [\mathbf{g}]$, $\mathbf{h}\leq \mathbf{g}$) by requiring $[h_i]\leq [g_i]$ (resp. $[h_i]= [g_i]$, resp. $h_i\leq g_i$) for each $i$.} We further define
\[\mathcal{E}(\mathbf{g})\coloneq\mathcal{E}(g_0)\times\cdots\times \mathcal{E}(g_d).\]
We say that $\mathbf{g}$ is \emph{strongly nef} if $(D_0, g_0),\dots, (D_d, g_d)\in \widehat{\Div}_\Q(U)_\snef$.
\begin{definition} \label{definition: local mixed energy}
	Let $(D_0,g_0),\dots, (D_d,g_d)\in \widehat{\Div}_\Q(U)_{\nef}$. 
	For any $\mathbf{h}=(h_{0},\dots, h_{d})\in \mathcal{E}(\mathbf{g})$,  
	we define the \emph{mixed relative energy} 
	by
	\[E(\mathbf{g},\mathbf{h}) \coloneq \sum\limits_{j=0}^d\int_{U^\an}(h_{j}-g_{j})c_1(D_0,h_{0})\wedge\cdots\wedge c_1(D_{j-1},h_{j-1})\wedge c_1(D_{j+1},g_{j+1})\wedge \cdots\wedge c_1(D_{d},g_{d}).\] 
	{Note that the mixed Monge--Amp\`ere measures were defined in \cref{proposition:measures for nef adelic}. Since $h_j-g_j$ is a continuous function on $U^\an$ which is bounded from above by a constant using $[h_j] \leq [g_j]$, we conclude that the above integral is well-defined with values in $\R \cup \{-\infty\}$. If $[\mathbf{h}]=[\mathbf{g}]$, then the mixed relative energy is finite.}
	We set 
	$\mathcal{E}^1(\mathbf{g})\coloneq\{\mathbf{h}\in \mathcal{E}(\mathbf{g})\mid  E(\mathbf{g},\mathbf{h})>-\infty\}$.
	\label{def:mixed relative energy at v} 
	
{In the \emph{unmixed case} when $g_0= \dots = g_d=g$ and $h_0=\dots=h_d=h$, then we call $\mathcal E(g,h)\coloneqq E(\mathbf{g},\mathbf{h})$ just the \emph{relative energy}.}
\end{definition}

\begin{theorem}	\label{lemma:basic properties of mixed relative energy}
	Let $(D_0,g_0),\dots, (D_d,g_d)\in \widehat{\Div}_\Q(U)_{\nef}$ and $\mathbf{h}\in \mathcal{E}(\mathbf{g})$.
	\begin{enumerate}
		\item \label{energy1} {Let $U'$ be a dense open subset of $U$. For $\mathbf{g'}\coloneq \mathbf{g}|_{U'}$ and $\mathbf{h'}\coloneq \mathbf{h}|_{U'}$, we have $\mathbf{h'} \in \mathcal{E}(\mathbf{g'})$ on $U'$ and $E(\mathbf{g'},\mathbf{h'})=E(\mathbf{g},\mathbf{h})$.}  
		\item \label{energy2}
		For every permutation $\sigma\in S_{d+1}$ of the set $\{0,\dots, d\}$, we have \[E(\mathbf{g},\mathbf{h})=E(\sigma(\mathbf{g}),\sigma(\mathbf{h})).\]
		\item \label{energy3} Let $\mathbf{u}\in \mathcal{E}(\mathbf{g})$ with $\mathbf{h}\leq \mathbf{u}$. Then
		\[E(\mathbf{g},\mathbf{h})\leq E(\mathbf{g},\mathbf{u}).\]
		In particular, if $\mathbf{h}\in \mathcal{E}^1(\mathbf{g})$, then so is $\mathbf{u}$.
		\item \label{energy4}
		Let $\mathbf{c}=(c_0,\dots, c_d)\in \R^{d+1}$. Then we have
		\[E(\mathbf{g},\mathbf{h}+\mathbf{c}) = E(\mathbf{g},\mathbf{h})+\sum\limits_{j=0}^dc_jD_0\cdots D_{j-1}D_{j}\cdots D_d.\] 
		\item \label{energy5}
		Let $\mathbf{h}_{n}$ be a decreasing sequence in $\mathcal{E}(\mathbf{g})$ such that for $j=0,\dots,d$, we have that $(D_j,h_{j,n})$ {is {strongly} nef and} converges to $(D_j,h_j)$ with respect to the boundary topology when $n\to\infty$. Then 
		\[E(\mathbf{g},\mathbf{h})=\lim\limits_{(n_0,\dots,n_d)\to\infty}E(\mathbf{g},(h_{0,n_0},\dots, h_{d,n_d})).\]
		\item \label{energy6}
		{If 	$[\mathbf g]=[\mathbf{h}]$, then we  have
			\begin{equation*} \label{first part of energy 6}
				E(\mathbf{g},\mathbf{h}) =\inf\{E(\mathbf{g},\mathbf{u}) \mid \text{$\mathbf{u}\in \mathcal{E}(\mathbf{g})$  with $[\mathbf{g}]=[\mathbf{u}]$ and $\mathbf{u}\geq \mathbf{h}$}\}.
		\end{equation*}}
		\item \label{energy7}
		If $\mathbf g$ and $\mathbf h$ are strongly nef 
		and if $U$ has a cofinal boundary divisor in $N_{\mo}(U)$, then 
			\[E(\mathbf{g},\mathbf{h}) =\inf\{E(\mathbf{g},\mathbf{u}) \mid \text{$\mathbf{u}\in \mathcal{E}(\mathbf{g})$ strongly nef with $[\mathbf{g}]=[\mathbf{u}]$ and $\mathbf{u}\geq \mathbf{h}$}\}.\]
	\end{enumerate}
\end{theorem}
\begin{proof} 
	The proof proceeds in several steps related to each others. 
	
	\vspace{2mm} \noindent
	Step 1: {\it Properties \ref{energy1} and  \ref{energy4} hold.}
	
	\vspace{2mm} \noindent
	Property \ref{energy1} follows from the projection formula in \cref{proposition:measures for nef adelic}. Property \ref{energy4} follows from Guo's Theorem, see \cref{proposition:measures for nef adelic}~\ref{guo's theorem}. This proves Step 1.

	Most of the remaining properties will be shown first under the assumption $[\mathbf{g}]=[\mathbf{h}]$.
	
	\vspace{2mm} \noindent
	Step 2: {\it Property \ref{energy2} holds if $[\mathbf{g}]=[\mathbf{h}]$.}
	
	\vspace{2mm} \noindent
	It is sufficient to consider the case where $\sigma=\tau_{j,j+1}$ is the transposition which interchanges $j$ and $j+1$. 
	By \cref{lemma:integral by parts}, we have that
	\[\int_{U^\an}(h_{j}-g_{j})c_1(0,h_{j+1}-g_{j+1})\wedge \Theta = \int_{U^\an}(h_{j+1}-g_{j+1})c_1(0,h_{j}-g_{j})\wedge\Theta.\] 
	where $\Theta$ stands for 
	\[\Theta\coloneq c_1({D}_0,h_{0})\wedge\cdots\wedge c_1({D}_{j-1},h_{j-1})\wedge c_1({D}_{j+2},g_{j+2})\wedge\cdots\wedge c_1({D}_d, g_{d}).\]
	Rearranging terms, we have
	\begin{align*}
		&\int_{U^\an}(h_{j}-g_{j})c_1(D_{j+1},g_{j+1})\wedge \Theta + \int_{U^\an}(h_{j+1}-g_{j+1})c_1(D_{j},h_{j})\wedge \Theta\\
		=& \int_{U^\an}(h_{j+1}-g_{j+1})c_1(D_{j},g_{j})\wedge \Theta + \int_{U^\an}(h_{j}-g_{j})c_1(D_{j+1},h_{j+1})\wedge \Theta.
	\end{align*}
	This immediately implies that $E(\mathbf{g},\mathbf{h})=E(\tau_{j,j+1}(\mathbf{g}),\tau_{j,j+1}(\mathbf{h}))$ {in case of $[\mathbf{g}]=[\mathbf{h}]$.}
	
	\vspace{2mm} \noindent
	Step 3: {\it Property \ref{energy3} holds if $[\mathbf{g}]=[\mathbf{h}]=[\mathbf{u}]$.}
	
	\vspace{2mm} \noindent	
	It is enough to show that for any $j=0,\dots, d$, we have that
	\[E(\mathbf{g},(h_{0},\dots, h_{j-1},h_{j},u_{j+1},\dots, u_{d}))\leq E(\mathbf{g},(h_{0},\dots, h_{j-1},u_{j},u_{j+1},\dots, u_{d})).\]
	By the special case of \ref{energy2} shown in Step 2, it is sufficient to show that
	\[E(\mathbf{g},(h_{0},\dots, h_{d-1},h_d))\leq E(\mathbf{g},(h_{0}, \dots,h_{d-1},u_d)).\]
	This follows from
	\begin{align*}
		&E(\mathbf{g},(h_{0},\dots, h_{d-1},h_d))- E(\mathbf{g},(h_{0},\dots,h_{d-1},u_d))\\
		=&\int_{U^\an}(h_{d}-u_{d})c_1(D_0,h_{0})\wedge\cdots\wedge c_1(D_{d-1},h_{d-1})
		\leq 0.
	\end{align*}

	\vspace{2mm} \noindent
	Step 4: {\it Property \ref{energy5} holds if  $[{h}_j]= [{h}_{j,n_j}] = [{g}_j]$ for all $j\in \{0,\dots,d\}$.}
	
	\vspace{2mm} \noindent	
	Using the definition of the relative mixed energy, it is sufficient to show that for any $0\leq j\leq d$, we have that

	$$\int_{U^\an}(h_{j,n_j}-g_{j}) \,c_1(D_0,h_{0,n_0})\wedge\cdots\wedge c_1(D_{j-1},h_{j-1,n_{j-1}})\wedge c_1(D_{j+1},g_{j+1})\wedge \cdots\wedge c_1(D_{d},g_{d})$$
	converges for $(n_0,\dots,n_{j-1})\to\infty$ to 
	$$ \int_{U^\an}(h_{j}-g_{j}) \,c_1(D_0,h_{0})\wedge\cdots\wedge c_1(D_{j-1},h_{j-1})\wedge c_1(D_{j+1},g_{j+1})\wedge \cdots\wedge c_1(D_{d},g_{d}).$$

	Notice that $h_{j,n_j}-g_{j}$ decreasingly converges to $h_{j}-g_j$ with respect to the boundary topology, and that there is $C\in \R_{>0}$ such that $C>h_{j,0}-g_{j}\geq h_{j,1}-g_{j}\geq \cdots\geq h_{j}-g_{j} > -C$. 
	{The above convergence is also locally uniform on $U^\an$ by \cref{lemma:boundary convergence is locally uniformly convergence} and hence} 
	the claimed convergence holds by \cref{corollary:weakly convergence of full mass} and \cref{cor:weakly converges of measures} 
	which proves \ref{energy5} in this special case.
	
	\vspace{2mm} \noindent
	Step 5: {\it Property \ref{energy6} holds if $[\mathbf{g}]=[\mathbf{h}]$.}
	
	\vspace{2mm} \noindent	
	Since $[\mathbf{g}]=[\mathbf{h}]$, it is clear  that $\mathbf h$ is the smallest element of 
	the set $$\{\mathbf{u}\in \mathcal{E}(\mathbf{g}) \mid \text{ $[\mathbf{g}]=[\mathbf{u}]$ and $\mathbf{u}\geq \mathbf{h}$}\}$$ and hence \ref{energy6} follows from monotony of the energy shown in Step 3.
	
	\vspace{2mm}
	
	To prove the remaining properties for $[\mathbf h] \leq [\mathbf g]$, we will first assume that $(D_j,g_j)$ and $(D_j,h_j)$ are strongly nef and that $U$ has a cofinal boundary divisor $\overline B \in N_{\mo,\Q}(U)$. Later, to achieve that, the idea is to shrink $U$ by \cref{nef become strongly nef after shrinking U} and by \cref{nef boundary divisor assumption} using \ref{energy1} shown in Step 1. We introduce the following notion:
	\begin{equation} \label{alternative energy definition}
			E'(\mathbf{g},\mathbf{h}) \coloneqq \inf\{E(\mathbf{g},\mathbf{u}) \mid \mathbf{u}\in \mathcal{E}(\mathbf{g}) \text{ {strongly nef} with $[\mathbf{g}]=[\mathbf{u}]$ and } \mathbf{u}\geq \mathbf{h}\}.
	\end{equation} 
%\green{	where $\mathbf u$ strongly nef means that $(D_j,u_j)$ is strongly nef for $j=0,\dots,d$. (we have defined this before)}
	
	\vspace{2mm} \noindent
	Step 6: {\it Under the above assumptions, let $\mathbf{h}_{n}$ be a decreasing sequence in $\mathcal{E}(\mathbf{g})$ such that for $j=0,\dots,d$,  {we have that $(D_j,h_{j,n})$ is strongly nef and} converges to $(D_j,h_j)$ with respect to the $\overline B$-topology when $n\to\infty$. Then} 
	\[E'(\mathbf{g},\mathbf{h})=\lim\limits_{(n_0,\dots,n_d)\to\infty}E'(\mathbf{g},(h_{0,n_0},\dots, h_{d,n_d})).\]

	\vspace{2mm} \noindent	
	For {$m \in \N$}, we set \[h^{(m)}_{j}\coloneq
	\max\{h_{j},g_{j}-m\}.\]
	{Clearly, we have $[\mathbf{h}^{(m)}]=[\mathbf{g}]$.}
	By  \cref{prop:local convergent sequence} and by using $\overline B \in N_\mo(U)$, we have that $\mathbf{h}^{(m)}\in \mathcal{E}({\mathbf{g}})$ is strongly nef and for $j=0,\dots,d$ that $(D_j,h_j^{(m)})$ decreasingly converges to $(D_j,h_j)$ with respect to the $\overline B$-topology. We claim that 
	\begin{equation} \label{relative energy and cut}
		E'(\mathbf{g},\mathbf{h})=\lim\limits_{m\to\infty}E(\mathbf{g},\mathbf{h}^{(m)}).
	\end{equation}
	For any {strongly nef} $\mathbf{u}\in \mathcal{E}(\mathbf{g})$ satisfying $[\mathbf{u}]=[\mathbf{g}]$ and $\mathbf{u}\geq \mathbf{h}$, we  define $\mathbf{u}^{(m)}$  as above.  
	We may apply the above with $\mathbf u$ instead of $\mathbf h$ to see that $(D_j,u_j^{(m)})$ is strongly nef and decreasingly converges to $(D_j,u_j)$ with respect to the $\overline B$-topology. 
	Since $[\mathbf{u}^{(m)}]=[\mathbf u]=[\mathbf g]$, we can use Step 4 to deduce  
	\[E(\mathbf{g},\mathbf{u})=\lim_{m\to\infty}E(\mathbf{g},\mathbf{u}^{(m)})\geq \limsup\limits_{m\to\infty}E(\mathbf{g},\mathbf{h}^{(m)}).\]
	Here, the inequality follows from Step 3 as we have obviously $\mathbf{h}^{(m)} \leq \mathbf{u}^{(m)}$. By definition of $E'(\mathbf{g},\mathbf{h})$, the above yields $E'(\mathbf{g},\mathbf{h})\geq \limsup_{m\to\infty}E(\mathbf{g},\mathbf{h}^{(m)})$. 
	Conversely, as $\mathbf{h}^{(m)}$ is part of the set $\{\mathbf{u}\in \mathcal{E}(\mathbf{g}) \mid \text{{$\mathbf{u}$ strongly nef,} $[\mathbf{g}]=[\mathbf{u}]$ and  $\mathbf{u}\geq \mathbf{h}$}\}$, we deduce
	\begin{equation*} \label{zeroth identity in energy proof}
		\liminf\limits_{m\to\infty}E(\mathbf{g},\mathbf{h}^{(m)})\geq E'(\mathbf{g},\mathbf{h}).
	\end{equation*}
	This proves \eqref{relative energy and cut}. 
	
	Next, we claim that for  any $m\in \N$, we have that $[\mathbf{h}^{(m)}]=[\mathbf{h}_{n}^{(m)}] = [\mathbf{g}]$ and that for $j=0,\dots,d$ the sequence $(D_j,h_{j,n}^{(m)})_{n \in \N}$ is strongly nef and decreasingly converges to $(D_j,{h}_j^{(m)})$ with respect to the $\overline B$-topology on $\widehat{\Div}_\Q(U)_\snef$. 
	The first property was already shown above for $\mathbf{h}$ and follows in the same way for $\mathbf{h}_n$. Since $\mathbf{h}_n$ is decreasing, it is clear that $\mathbf{h}_{n}^{(m)}$ is decreasing in $n$. It remains to see the convergence with respect to the $\overline B$-topology. Let $\varepsilon \in\Q_{>0}$.  For  $n \gg 0$,
	$$0 \leq (D_j,h_{j,n}) - (D_j,h_j) \leq \varepsilon \overline B$$
	holds. This comes from the assumption that $(D_j,h_{j,n})$ decreasingly converges to $(D_j,h_j)$ with respect to the $\overline B$-topology. Since
	$$0 \leq h_{j,n}^{(m)}-h_j^{(m)}\leq \max\{h_{j}+\varepsilon g_B,g_j-m\} - h_j^{(m)} \leq h_j^{(m)}+ \varepsilon g_B - h_j^{(m)}=\varepsilon g_B,$$
	we deduce that 
	$$0 \leq (D_j,h_{j,n}^{(m)}) - (D_j,h_j^{(m)}) \leq \varepsilon \overline B$$
	proving the claimed convergence with respect to the $\overline B$-topology. 
	
	Using the claim just proved, Step 4 yields \[E(\mathbf{g},\mathbf{h}^{(m)})=\lim\limits_{(n_0,\dots,n_d)\to\infty}E(\mathbf{g},({h}_{0,n_0}^{(m)},\dots,{h}_{d,n_d}^{(m)})).\]
	After taking the limit $m \to \infty$ on the both sides, {it follows from \eqref{relative energy and cut} that}
	\begin{equation} \label{first identity in energy proof}
		E'(\mathbf{g},\mathbf{h})=\lim\limits_{m\to\infty}E(\mathbf{g},\mathbf{h}^{(m)})=\lim\limits_{m \to \infty} \lim\limits_{(n_0,\dots,n_d)\to\infty}E(\mathbf{g},({h}_{0,n_0}^{(m)},\dots,{h}_{d,n_d}^{(m)})).
	\end{equation}
	{For any $n_0,\dots,n_d \in \N$, we have $E'(\mathbf{g},(h_{0,n_0}, \dots, h_{d,n_d})) \geq E'(\mathbf{g},\mathbf{h})$ by simply using $h_{j,n_j} \geq h_j$ and the definition of $E'$ in \eqref{alternative energy definition}.} We conclude that 
	\begin{equation} \label{second identity in energy proof}
		\liminf_{(n_0,\dots,n_d)\to \infty } E'(\mathbf{g},(h_{0,n_0}, \dots, h_{d,n_d})) \geq E'(\mathbf{g}, \mathbf{h}).
	\end{equation}
	As $({h}_{0,n_0}^{(m)},\dots,{h}_{d,n_d}^{(m)}) \in \{\mathbf{u}\in \mathcal{E}(\mathbf{g}) \mid \text{{$\mathbf u$ strongly nef,} $[\mathbf{g}]=[\mathbf{u}]$ and  $\mathbf{u}\geq (h_{0,n_0}, \dots, h_{d,n_d})$}\}$, the definition of $E'$ in \eqref{alternative energy definition} yields
	\begin{equation*}
		E(\mathbf{g},({h}_{0,n_0}^{(m)},\dots,{h}_{d,n_d}^{(m)})) \geq E'(\mathbf{g},(h_{0,n_0}, \dots, h_{d,n_d}))
	\end{equation*}
	and hence
	\begin{equation} \label{third identity in energy proof}
		\lim\limits_{(n_0,\dots,n_d)\to\infty}E(\mathbf{g},({h}_{0, n_0}^{(m)},\dots,{h}_{d,n_d}^{(m)}))
		\geq \limsup_{(n_0,\dots,n_d)\to\infty}  E'(\mathbf{g},(h_{0,n_0}, \dots, h_{d,n_d})).
	\end{equation}
	Using \eqref{third identity in energy proof} in \eqref{first identity in energy proof} and then \eqref{second identity in energy proof}, we conclude
	$$\liminf_{(n_0,\dots,n_d)\to \infty } E'(\mathbf{g},(h_{0,n_0}, \dots, h_{d,n_d})) \geq E'(\mathbf{g}, \mathbf{h}) \geq  \limsup_{(n_0,\dots,n_d)\to\infty}  E'(\mathbf{g},(h_{0,n_0}, \dots, h_{d,n_d}))$$
	proving Step 6.
	
	\vspace{2mm} \noindent
	{Step 7: {\it   If $U$ has a cofinal boundary divisor in $N_{\mo,\Q}(U)$ and if ${\mathbf g, }\mathbf h$ are  strongly nef, 
	then $E'(\mathbf g, \mathbf h)=E(\mathbf g, \mathbf h)$ which means that property \ref{energy7} holds}.}

	\vspace{2mm} \noindent	
	We will apply Step 6 to the sequences $h_{j,n}\coloneq \max\{h_{j}, g_{j}-n\}$ for $j=0,\dots,d$. 
	By \cref{prop:local convergent sequence}, we know that $(D_{j}, h_{j,n})$ {is strongly nef and} decreasingly converges to $(D_{j},h_{j})$  
	with respect to the boundary topology. For $n_0,\dots, n_{j-1} \in \N$, we set
	\[\Theta_{j,n_0,\dots, n_{j-1}}\coloneq c_1({D}_0,h_{0,n_0})\wedge\cdots\wedge c_1({D}_{j-1},h_{j-1,n_{j-1}})\wedge c_1({D}_{j+1},g_{j+1})\wedge\cdots\wedge c_1({D}_d, g_{d}).\]
	and
	\[\Theta_{j}\coloneq c_1({D}_0,h_{0})\wedge\cdots\wedge c_1({D}_{j-1},h_{j-1})\wedge c_1({D}_{j+1},g_{j+1})\wedge\cdots\wedge c_1({D}_d, g_{d}),\]
	By \cref{corollary:weakly convergence of full mass}, we have $\Theta_{j,n_0,\dots, n_{j-1}}\overset{w}{\to} \Theta$ for $(n_0,\dots, n_{j-1})\to \infty$. For every $n_j \in \N$, we have  $[h_{j,n_j}]=[g_j]$ and hence we get
		\begin{align*}
			\lim\limits_{n_{j-1}\to\infty}\cdots\lim\limits_{n_0\to\infty}\int_{U^\an}(h_{j,n_j}-g_{j})\Theta_{j,n_0,\dots, n_{j-1}}
			=\int_{U^\an}(h_{j,n_j}-g_{j})\Theta_{j}.
		\end{align*}
		The monotone convergence theorem yields
		\begin{equation} \label{multiple limit in energy proof}
			\lim_{n_j \to \infty} \lim\limits_{n_{j-1}\to\infty}\cdots\lim\limits_{n_0\to\infty}\int_{U^\an}(h_{j,n_j}-g_{j})\Theta_{j,n_0,\dots, n_{j-1}} = \int_{U^\an}(h_{j}-g_{j})\Theta_{j}.
		\end{equation}
		Using $[h_{j,n_j}]=[g_j]$, we show as in Step 5 that  $$E'(\mathbf{g}, (h_{0,n_0},\dots, h_{d,n_d}))=E(\mathbf{g}, (h_{0,n_0},\dots, h_{d,n_d}))$$ and hence Step 6 gives
	\begin{align*}
		E'(\mathbf{g},\mathbf{h})=&\lim\limits_{n_{d}\to\infty}\cdots\lim\limits_{n_0\to\infty}E(\mathbf{g},(h_{0,n_0},\dots, h_{d,n_d}))\\
		=&\lim\limits_{n_{d}\to\infty}\cdots\lim\limits_{n_0\to\infty}\sum\limits_{j=0}^d\int_{U^\an}(h_{j,n_j}-g_{j})\Theta_{j,n_0,\dots, n_{j-1}}.
	\end{align*}
	By inserting \eqref{multiple limit in energy proof}, we get
	$$E'(\mathbf{g},\mathbf{h})=\sum\limits_{j=0}^d\int_{U^\an}(h_{j}-g_{j})\Theta_{j}=E(\mathbf{g},\mathbf{h})$$
	proving Step 7.
	
	\vspace{2mm} \noindent
	Step 8: {\it   The remaining properties \ref{energy2}, \ref{energy3} and \ref{energy5} 
	hold in general.}
	
	\vspace{2mm} \noindent	
	For the proof, we may assume that $U$ has a cofinal boundary divisor $\overline B \in N_{\mo,\Q}(U)$ and that $\mathbf g, \mathbf h$ are 
	strongly nef as we have already remarked before Step 6. We have shown the remaining properties in the special case of equivalent singularities, see 
	Steps 2--4.  This yields  properties \ref{energy2} and \ref{energy3} 
	in general, but for $E'$ instead of $E$. The property \ref{energy5} for $E'$ was shown already in Step 6. Since $E$ agrees with $E'$ by Step 7, this proves Step 8 and hence the theorem.
\end{proof}

\begin{proposition} \label{prop: transitivity of local mixed energy}
	Let $(D_0,g_0),\dots, (D_d,g_d)\in \widehat{\Div}_\Q(U)_{\nef}$, and $\mathbf{g}'\in \mathcal{E}(\mathbf{g}), \mathbf{h}\in \mathcal{E}(\mathbf{g}')$.  Then we have
	\begin{align} \label{eq:transitivity}
		E(\mathbf{g},\mathbf{h}) =E(\mathbf{g},\mathbf{g}')+ E(\mathbf{g}',\mathbf{h}).
	\end{align}
	In particular, $\mathbf{h}\in \mathcal{E}^1(\mathbf{g})$ if and only if $\mathbf{h}\in \mathcal{E}^1(\mathbf{g}')$ and $\mathbf{g}'\in\mathcal{E}^1(\mathbf{g})$. 
\end{proposition}
\begin{proof}
	{Using induction on $n\coloneq\#\{j\in\{0,\dots, d\}\mid h_j\not=g_j'\}$, we easily reduce to show the claim for $n=1$. Using \cref{lemma:basic properties of mixed relative energy}~\ref{energy2}, we may assume $(h_0,\dots, h_d)=(g_1',\dots, g_{d-1}',h_d)$. In this case, the claimed transitivity property is obvious from the definition of the relative mixed energy in Definition \ref{definition: local mixed energy}.}
\end{proof}

We compute the relative energy in our running example.
\begin{example} \label{relative energy in running example}
	{Let $K=\Q_v$ for a place $v$ of $\Q$ and let $U=\mathbb A_K^1$ embedded in $X=\mathbb P_K^1$ as usual. We fix our attention to the toric divisor $D=[\infty]$ and to a toric Green function $g$ for $D$ such that $(D,g) \in \widehat{\Div}_{S,\Q}(U)_\snef$. By \cref{running example integrable local}, this means that there is a concave function $\psi \colon \R \to \R$ with asymptotic slopes $1$ for $u \to -\infty$ and $0$ for $u\to \infty$ such that $g= -\psi \circ \trop$ on $T^\an = X^\an\setminus \{0,\infty\}$. We pick a toric $h \in \mathcal E(g)$ which means that $h$ is similarly induced by a concave function $\varphi\colon \R \to \R$ with asymptotic slopes  $1$ for $u \to -\infty$ and $0$ for $u\to \infty$ such that $\psi \leq \varphi+C$ for some constant $C \in \R$. It follows from \cref{MA measure in running example} that the relative energy is given by
	$$E(g,h)= \int_\R (\psi-\varphi)(u) \, \left(\mathrm{MA}(\varphi)+\mathrm{MA}(\psi)\right)(du) \in \R \cup \{-\infty\}.$$}
\end{example}
				
\begin{remark} \label{local energy depends only on isometry class}
	Let $(D_0,g_0),\dots, (D_d,g_d)\in \widehat{\Div}_\Q(U)_{\nef}$ and let $\mathbf{h}=(h_{0},\dots, h_{d})\in \mathcal{E}(\mathbf{g})$. Then it is obvious from the definition that the relative mixed energy $E(\mathbf{g},\mathbf{h})$ depends only on the isometry class of the underlying metrized $\Q$-line bundles $\overline{L_j}= (\OO_U(D_j), \metr_j)$ and $\overline{L_j'}=(\OO_U(D_j), \metr_j')$, where $\metr_j$ is the metric corresponding to the Green function $g_j$ and $\metr_j'$ is the metric corresponding to the Green function $h_j$ for $j=0,\dots,d$. We will then use also the notation
	$$E(\mathbf{\overline{L}}, \mathbf{\overline{L'}}) \coloneqq E(\mathbf{g},\mathbf{h})$$
	for the relative mixed energy. In this way, we can use the relative mixed energy on the groups $\widehat{\Pic}_\Q(U)_\nef$ defined in \cref{definition: compactified line bundles} for $(L_j,\metr_j), (L_j,\metr_j')$ in $\widehat{\Pic}_\Q(U)_\nef$ with $\metr_j'$ more singular that $\metr_j$ for $j=0,\dots, d$. The latter means that $-\log(\metr_j'/\metr_j)$ is bounded from above on $U^\an$.
\end{remark}

\begin{prop} \label{functoriality of energy}
	Let $\varphi\colon U' \to U$ be  a morphism of $d$-dimensional algebraic varieties over $K$. For $j=0,\dots, d$, let $(L_j,\metr_j), (L_j,\metr_j')$ in $\widehat{\Pic}_\Q(U)_\nef$ with $\metr_j'$ more singular than $\metr_j$ for $j=0,\dots, d$. Then we have
	$$E(\varphi^*\mathbf{\overline{L}}, \varphi^*\mathbf{\overline{L'}}) = \deg(\varphi)E(\mathbf{\overline{L}},\mathbf{\overline{L'}}).$$
\end{prop}		
\begin{proof}
	This follows from the projection formula in \cref{proposition:measures for nef adelic}.
\end{proof}
				
\begin{remark} \label{relative energy in the trivially vaued case}
	The results of this subsection hold also in the trivially valued case. We can just use that there is a non-archimedean field extension $L/K$ with non-trivial valuation. The definition of relative energy does not depend on the choice of $L$ which follows from \cref{proposition:measures for nef adelic}\ref{MA measure via field extension}.
\end{remark}

%----------------------------------------------------------------------------------------
% Mixed relative energy in the global case
%----------------------------------------------------------------------------------------	

\section{{Mixed relative energy in the global case}} \label{section: mixed relative energy in the global case}
				
In this section, we generalize the notions from the previous section to the global case.  We fix a proper adelic curve $S=(K,\Omega,\mathcal{A},\nu)$ satisfying the normalization assumption for archimedean valuations given in \cref{assumeption on adelic curves}. We also assume that either the $\sigma$-algebra $\mathcal{A}$ is discrete, or that $K$ is countable. We consider any algebraic variety $U$ over $K$.
				
We recall first the following notation introduced in \cref{section: global theory}. We will use the abstract divisorial space $(M_{S,\Q}(U),N_{S,\Q}(U))$  from Definition \ref{def:CM divisors on non-proper} build as a direct limit of the $S$-integrable adelic divisors on proper $K$-models of $U$ in the sense of Chen and Moriwaki. Moreover, we use the abstract divisorial space $(M_{S,\Q}'(U),N'_{S,\Q}(U))$ from \ref{boundary divisors in M'} build similarly.
% \green{upon the $S$-nef divisors on proper $K$-models of $U$}.

\begin{definition} \label{def:global more singular}
	Let $D\in \widetilde{\Div}_\Q(U)_\cpt$, and $g_1, g_2$ $S$-Green functions for $D$, see \cref{global Green functions for compactified divisors}. We say that $g_1$ is \emph{more singular than $g_2$}, 
	denoted by $[g_1]\leq [g_2]$, if there is a $\nu$-integrable function $C\in \mathscr{L}^1(\Omega,\mathcal{A},\nu)$ such that $g_{1,\omega}\leq g_{2,\omega}+C(\omega)$ for any $\omega\in\Omega$. We say that $g_1$, $g_2$ \emph{have equivalent singularities}, 
	denote by $[g_1]=[g_2]$, if $[g_2]\leq [g_1]$ and $[g_1]\leq[g_2]$, i.e. $|g_{1,\omega}-g_{2,\omega}|$ is bounded as a function of $\omega \in \Omega$ by a $\nu$-integrable function.
\end{definition}
\begin{remark}
As in Remark~\ref{rmk:local D,g} in the local case,  an $S$-Green function for $D\in\widetilde{\Div}_\Q(U)_\cpt$ in the sense of \cref{global Green functions for compactified divisors} is also a $S$-Green function for $D|_U$ in sense of \cref{global Green functions}. In the following, the notation $(D,g)\in\widehat{\Div}_{S,\Q}(U)_\cpt$ always means that $D$ is a compactified divisor and $g$ is an $S$-Green function for $D$ as in \cref{global Green functions for compactified divisors}.
\end{remark}

\begin{lemma}	\label{lemma:maximum is nef global}
	Let $(D,g), (D',g')\in \widehat{\Div}_{S,\Q}(U)_\relsnef$
	with $D\geq D'$ and let $h\coloneq\max\{g,g'\}$. Then $(D,g), (D',g')$ are given as limits of sequences $(\overline{D_{n}})_{n\geq 1}$ and $(\overline{D_{n}'})_{n\geq 1}$ in $N_{S,\Q}(U)$ with respect to the $\overline{B}$-boundary topology for some {weak} boundary divisor $\overline B$ of $U$. 
	\begin{enumerate}
		\item If we can choose $\overline B$ in $N_{S,\Q}(U)$, then $h$ is an $S$-Green function for $D\in \widetilde{\Div}_{\Q}(U)_\cpt$, and $(D,h) {\in \widehat{\Div}_{S,\Q}(U)_\relsnef}$. 
		\item If we can choose %\green{$\overline B$ in $N_{S,\Q}(U)$} 
		$\overline B$ in $N_{S,\Q}'(U)$  and if additionally $(D,g)\in \widehat{\Div}_{S,\Q}(U)_\arsnef$ is given as a limit of a sequence $(\overline{D_{n}})_{n\geq 1}$ in $N'_{S,\Q}(U)$, then  $(D,h)\in \widehat{\Div}_{S,\Q}(U)_\arsnef$. 
	\end{enumerate}
\end{lemma}

\begin{proof}
	It follows from \cref{lemma:maximum is nef} that $h$ is an $S$-Green function for $D$. To prove the claims, let us assume first that $U=X$ is a proper variety over $K$. Recall from Definition \ref{global adelic Green functions for proper varieties} that $N_{S,\Q}(X)$ denotes the space of $(D,g)$ with $D$ a $\Q$-Cartier divisor on $X$ and $g$ an $S$-measurable (locally) $S$-bounded Green function for $D$ such that for every non-trivial $\omega\in \Omega$, the Green function $g_{D,\omega}$ is a uniform limit of semipositive model Green functions. If  $\omega$ gives  the trivial valuation, then $g_{D,\omega}$ is supposed to be a uniform limit of twisted Fubini--Study Green functions.

	\vspace{2mm}
	\noindent
	Step 1: \emph{Let $(D,g), (D,g') \in N_{S,\Q}(X)$ with $D \geq D'$ and let $h \coloneqq \max\{g,g'\}$, then $(D,h)\in N_{S,\Q}(X)$.}
								
	\vspace{2mm}
	\noindent
	By definition, this means that $g,g'$ are $S$-measurable locally $S$-bounded functions on $U$ and  $g_\omega$ (resp.~$g_\omega'$)  is a uniform limit of semipositive model/twisted Fubini--Study Green functions $g_{\omega,n}$ (resp.~$g'_{\omega,n}$) for $D$ (resp.~$D'$). Obviously, the maximum of two $S$-measurable (resp.~locally $S$-bounded) functions of $X$ is $S$-measurable (resp.~locally $S$-bounded). We conclude that $h$ is an $S$-measurable, locally $S$-bounded  function on $X$. Moreover, Lemma \ref{lemma:maximum is nef} and Remark \ref{trivially valued case} show that  $\max\{g_{\omega,n}, g'_{\omega,n}\}$ is a nef model/twisted Fubini--Study $S$-Green function for $D_\omega$. Since $\max\{g_{\omega,n}, g'_{\omega,n}\}$ converges uniformly to $h_\omega$ on $X_\omega^\an$, we conclude that $(D,h) \in  N_{S,\Q}(X)$ proving Step 1.
								
	\vspace{2mm}
	
		We use now the cone $N_{S,\Q}'(X)$ 
		%\green{of $S$-nef metrized divisors}  
		from  \cref{prop: S-nef cone}.
	
	\noindent
	Step 2: \emph{Under the assumptions from Step 1 and if $(D,g) \in N_{S,\Q}'(X)$, then $(D,h)  \in N_{S,\Q}'(X)$.}
								
	\vspace{2mm}
	\noindent
	Using Chow's lemma and 
	%\green{\cref{global abstract divisorial space}} 
	\cref{prop: S-nef cone}~\ref{second property N'}, we may assume $X$ projective. By base change to the algebraic closure and using  \cref{prop: S-nef cone}~\ref{third property N'}, we may assume $K$ perfect. Then we have seen in Proposition \ref{proposition:arithmetically nef for projective} that an element of  $N_{S,\Q}(X)$ is in $N'_{S,\Q}(X)$ if and only if its asymptotic minimal slope is non-negative. Since $h \geq g$, Step 2 follows from $0 \leq \widehat{\mu}_{\min}^{\mathrm{asy}}(D,g)\leq \widehat{\mu}_{\min}^{\mathrm{asy}}(D,h)$.
								
	\vspace{2mm}
	
	Now we come back to the general case when $U$ is any algebraic variety over $K$. 
	
	\noindent
	Step 3:  \emph{Let $\overline{D}=(D,g), \overline{D'}=(D',g')\in \widehat{\Div}_{S,\Q}(U)_\relsnef$ given as limits of sequences in $N_{S,\Q}(U)$ with respect to 
	a {weak} boundary divisor $\overline{B}\in N_{S,\Q}(U)$ with $D\geq D'$, then $(D,h)\in \widehat{\Div}_{S,\Q}(U)_\relsnef$.}
	
	\vspace{2mm} \noindent								
	Let $(\overline{D_{n}})_{n\geq 1}$, $(\overline{D_{n}'})_{n\geq 1}$ be sequences  in $N_{S,\Q}(U)$ converging to $\overline{D}$, $\overline{D'}$ respectively. Using $D \geq D'$ and adding small multiples of the {weak} boundary divisor $\overline B$, we may assume that $D_n \geq D_n'$. Passing to a  dominating $K$-model, we may assume that $D_n$ and $D_n'$ are $\Q$-Cartier divisors on the same proper $K$-model $X_n$ of $U$. For $h_{n}\coloneq\max\{g_{n},g'_{n}\}$, we have $(D_n,h_n)\in N_{S,\Q}(X_n)$ by Step 1.  Since $(D_n,g_n)$ converges to $(D,g)$ and $(D_n',g_n')$ converges to $(D',g')$ with respect to $\overline{B}$-boundary topology, we deduce that $(D_n,h_n)$ converges to $(D,h)$ with respect to the $\overline{B}$-topology proving Step 3.
								
	\vspace{2mm}
	\noindent
	{Step 4:}  \emph{Let $(D,g)$ (resp.~$(D',g')$) be strongly arithmetically nef (resp.~strongly relatively nef) given as a limit of a sequence in $N'_{S,\Q}(U)$ (resp.~$N_{S,\Q}(U)$) with respect to
	a weak boundary divisor $(X_0, \overline{B})\in  {N'_{S,\Q}(U)}$. Then $(D,h)\in \widehat{\Div}_{S,\Q}(U)_\arsnef$.}
	
		\vspace{2mm} \noindent								
%	\green{We first note that we may twist $\overline B$ by a suitable non-negative integrable function $c$ on $\Omega$ to get a boundary divisor $\overline B(c) \in N'_{S,\Q}(U)$ of $U$, see Remark \ref{twist of relatively nef divisors}. Replacing $\overline B$ by $\overline B(c)$, we may assume that $\overline B \in N'_{S,\Q}(U)$. Then} 
Step 4 follows from the same proof as in Step 3 relying now on Step 2 instead of Step 1.
\end{proof}

\begin{remark} \label{global nef boundary divisor assumption}
	The assumptions on the {weak} boundary divisor $\overline B$ in \cref{lemma:maximum is nef global} are crucial.  Similarly as in Remark \ref{nef boundary divisor assumption}, we can always obtain it by passing to a suitable dense open subset $U'$, see \cref{dominated by S-nef boundary divisors}. {Moreover, we may assume in this way that $\overline B$ is a boundary divisor.}
											
	By shrinking $U'$ further, we may  assume, in addition to the above assumptions, that finitely many given relatively/arithmetically nef compactified $S$-metrized divisors become strongly relatively/arithmetically nef similarly as in \cref{nef become strongly nef after shrinking U}.
\end{remark}

\begin{lemma} 	\label{lemma: same divisoral part for global divisors}
	Let $(D,g), (D,g')\in \widehat{\Div}_{S,\Q}(U)_{\cpt}$ with the same divisorial part. We assume that $(D,g)$ is given by a Cauchy sequence $({D}_{n},g_n)\in N_{S,\Q}(U)$ 
	%{(resp.  $(D_n,g_n) \in N'_{S,\Q}(U)$)} 
	with respect to a boundary divisor $\overline B \in N_{S,\Q}(U)$  
	and that $(D,g')$ is given by  a Cauchy sequence  $(D'_{n},g'_n)$ in {$N_{S,\Q}(U)$} 
	with respect to a {weak} boundary divisor {$\overline{B'}\in N_{S,\Q}(U)$} 
	of $U$. Replacing the {weak} boundary divisors by a higher {weak} boundary divisor in {$N_{S,\Q}(U)$},  
	we may choose the Cauchy sequences such that the following properties hold:
	\begin{enumerate}
		\item \label{item same divisor} $D_n=D_n'$ for all $n \in \N$.
		
		\item \label{item decreasing}$(D_n,g_n)\geq (D_{n+1},g_{n+1})$ and $(D_n',g_n')\geq (D_{n+1}',g_{n+1}')$ for all $n \in \N$.
		
		\item  \label{item bound} If $|g-g'|\leq C$ for  some $C\in\mathscr{L}^1(\Omega,\mathcal{A},\nu)$, then $|g_n-g_n'|\leq C$.
	\end{enumerate}
If we have the stronger assumptions that $(D,g)$ is given by a Cauchy sequence $(D_n,g_n)\in N_{S,\Q}'(U)$ and that the weak boundary divisors $\overline{B},\overline{B'}$ are in $N_{S,\Q}'(U)$ where $(D,g')$ is still given by a Cauchy sequence in $N_{S,\Q}(U)$ with respect to the $\overline{B'}$-topology, then the same conclusion holds for choosing the Cauchy sequences.
\end{lemma}									
\begin{proof}
	Write $\overline{D}\coloneq(D,g), \overline{D'}\coloneq(D,g')$. Passing to $\overline B + \overline{B'}$, we may assume $\overline{B'}=\overline{B}=(B,g_B) {\in N_{S,\Q}(X_0)}$ {(even in  $ N_{S,\Q}'(X_0)$ under the stronger assumptions)} for a proper $K$-model $X_0$ of $U$. %\green{Replacing $\overline{B}$ by a suitable twist by a non-negative integrable function of $\Omega$, Remark \ref{twist of relatively nef divisors} shows that we may assume $\overline{B} \in N'_{S,\Q}(X_0)$.} 
	Passing to a higher proper $K$-model $X_n$ of $U$ which also dominates $X_0$, we may assume that $(D_n, g_n)\in N_{S,\Q}(X_n)$ {(resp.~$(D_n, g_n)\in N'_{S,\Q}(X_n)$)} and $(D_n',g_n')\in N_{S,\Q}(X_n)$. We claim that we can assume that $D_n'=D_n$. Since the underlying divisors of $\overline{D}$ and $\overline{D'}$ in $\widetilde{\Div}_\Q(U)_{\cpt}$ coincide, we can assume that $D_n\geq D_n'$ after adding to $\overline{D_n}$ a small positive multiple of $\overline{B}$ using that {$\overline{B} \in N_{S,\Q}(X_n)$ (resp.~$\overline{B} \in N'_{S,\Q}(X_n)$)}. For any $n\in\N_{>0}$, by \cref{lemma:integrable bound for green functions}, there is a $\nu$-integrable function $C_n\in\mathscr{L}^1(\Omega,\mathcal{A},\nu)$ such that
	\[\max\left\{\sup\limits_{x\in X^\an_{0,\omega}}\left\{g_\omega-g_\omega'-\frac{1}{2n}g_{B,\omega}\right\}, \sup\limits_{x\in X^\an_{0,\omega}}\left\{g_\omega'-g_\omega-\frac{1}{2n}g_{B,\omega}\right\}\right\}\leq C_n(\omega)\]
	for all $\omega\in \Omega$. We can assume that $(C_n(\omega))_{n\geq 1}$ is an increasing sequence. For any $n\in\N_{>0}$, we take $i_n\in\N$ such that for any $m\geq i_n$, we have that
	\[|g_m-g|\leq \frac{1}{2n}g_B, \ \ |g_m'-g'|\leq \frac{1}{n}g_B.\]
	We can assume that the sequence $(i_n)_{n \geq 1}$ is increasing, and set 
	\[h'_{n}\coloneq\max\{g_{i_n}',g_{i_n}-C_{n}\}.\]
	By \cref{lemma:maximum is nef global},  $h_{n}'$ is a Green function for $D_n$  and $(D_n,h'_{n})\in N_{S,\Q}(X_n)$. Moreover, we will show that $(D_{i_n},h_{n}')$ converges to $(D,g')$ following the proof of \cite[Lemma~3.41~(i)]{burgos2023on}. 
	For any $n\in \N_{>0}$, by our choice of $C_n$, we have that
	\[|g_\omega-g_\omega'|\leq C_{n}(\omega)+\frac{1}{2n}g_{B,\omega}\]
	on $X_{0,\omega}^\an$ for all $\omega\in\Omega$.
	For any $m\geq n$ and $\omega\in\Omega$, we have that $i_m\geq i_n$ and
	\[g_{i_m,\omega}-C_m(\omega) \leq g_{\omega}-C_m(\omega)+\frac{1}{2n} g_{B,\omega}\leq g_\omega'+C_{n}(\omega)-C_m(\omega)+\frac{1}{n}g_{B,\omega}\leq g_\omega'+\frac{1}{n} g_{B,\omega}.\]
	Since $g'-\frac{1}{n} g_{B} \leq g_{i_m}'\leq g'+\frac{1}{n} g_{B}$, we have that
	\[g'-\frac{1}{n} g_{B}\leq g_{{i_m}}' \leq h'_{m} = \max\{g_{i_m}',g_{i_m}-C_m\}\leq g'+\frac{1}{n} g_{B}.\]
	This shows that $(D_{i_n},h_{n}')_{n\geq 1}$ converges to $(D,g')$ with respect to $\overline{B}$-topology. On the other hand, we have that $(D_{i_n},g_{i_n})_{n\geq 1}$ converges to $(D,g)$, hence our claim holds.
	
	From our claim, we can assume that $D_n=D_n'$. {By passing to  subsequences and adding small positive multiples of $\overline B$, we may assume that  $\overline{D_n}\coloneq(D_n,g_n), \overline{D_{n}'}\coloneq(D_n,g_n')$ are} 
	decreasing Cauchy sequences. This proves \ref{item same divisor} and \ref{item decreasing}.
	
	To prove \ref{item bound}, we assume that $|g_\omega-g_\omega'|\leq C(\omega)$ for $C\in\mathscr{L}^1(\Omega,\mathcal{A},\nu)$. We take Cauchy sequences $({D}_{n},g_{0,n})$ and $({D}_{n}, g_{0,n}')$ as above satisfying   \ref{item same divisor}, \ref{item decreasing} and  representing $\overline{D}$ and $\overline{D'}$, respectively. Set
	\[h_{n}= \max\{g_{0,n},g'_{0,n}-C\} \ \  \text{ and } \ \ h_{n}'= \max\{g_{0,n}',g_{0,n}-C\}.\]
	It is not hard to see that $|h_n-h_n'|\leq C$ and $(D_n,h_n)$ (resp. $(D_n,h_n')$) decreasingly converges to $\overline{D}$ (resp. $\overline{D'}$) with respect to $\overline{B}$-topology. By {Step 1 (resp.~Step 2) of the proof of \cref{lemma:maximum is nef global}, we have $(D_n,h_n)\in N_{S,\Q}(X_n)$ (resp.~$(D_n,h_n)\in N'_{S,\Q}(X_n)$) and $(D_n,h_n')\in N_{S,\Q}(X_n)$.} This proves \ref{item bound}. 
\end{proof}

\begin{lemma}		\label{lemma:convergent sequence}
Let $(D,g), (D,g')\in \widehat{\Div}_{S,\Q}(U)_{\cpt}$  
and let $\overline B$ be a {weak} boundary divisor of $U$ contained in $N_{S,\Q}(U)$ {(resp. in $N_{S,\Q}'(U)$)}.
We assume that $(D,g)$ is strongly relatively nef
and that it is the  limit of a sequence in $N_{S,\Q}(U)$ with respect to the $\overline B$-boundary topology. Moreover, we assume that   $(D,g')$ is {strongly relatively nef (resp.~strongly arithmetically nef)}  
and that it is the limit of a sequence in $N_{S,\Q}(U)$ (resp.~in $N'_{S,\Q}(U)$) with respect to the $\overline B$-topology.  
Then  there is an increasing sequence of $\nu$-integrable functions $(C_n)_{n\geq 1}$ in  $\mathscr{L}^1(\Omega,\mathcal{A},\nu)$ such that $(D,g_n')\coloneq(D,\max\{g-C_n,g'\})$ is {a strongly relatively nef (resp.~ strongly arithmetically nef) compactified $S$-metrized divisor of $U$} 
and the sequence $(D,g'_n)_{n\in\N}$ converges decreasingly to $(D,g')$ with respect to the  $\overline B$-boundary topology.
\end{lemma}
												
\begin{proof}
The proof is similar  to the one of \cref{lemma: same divisoral part for global divisors}. 
By \cref{lemma:integrable bound for green functions}, for any $n\in\N_{>0}$, there is a $\nu$-integrable function $C_n\in\mathscr{L}^1(\Omega,\mathcal{A},\nu)$ such that
\[\max\left\{\sup\limits_{x\in X^\an_{0,\omega}}\left\{g_\omega-g_\omega'-\frac{1}{2n}g_{B,\omega}\right\}, \sup\limits_{x\in U^\an_{\omega}}\left\{g_\omega'-g_\omega-\frac{1}{2n}g_{B,\omega}\right\}\right\}\leq C_n(\omega),\]
and hence
\[|g_\omega-g_\omega'|\leq C_n(\omega)+\frac{1}{2n}g_{B,\omega} \ \ \text{on $U_{\omega}^\an$},\]
for any $\omega\in \Omega$. We can assume that $(C_n(\omega))_{n\geq 1}$ is an increasing sequence. We set $g_n'=\max\{g-C_n,g'\}$. By \cref{lemma:maximum is nef global}, we have that $(D,g_n')$ {is strongly relatively nef (resp.~strongly arithmetically nef).} 
Since the sequence is obviously decreasing, it remains to show that $(D,g_n')$ converges to $(D,g')$ with respect to $\overline{B}$-topology. For any $n\in\N_{>0}$ and $\omega\in\Omega$, we have that
\[0\leq g_{n,\omega}'-g_\omega'=\max\{g_\omega-g_\omega'-C_n(\omega),0\}\leq \frac{1}{2n}g_{B,\omega}.\]
This completes the proof.
\end{proof}

{For any $(D,g)\in \widehat{\Div}_{S,\Q}(U)_{\relnef}$, we define $\mathcal{E}(g)$ as the set of Green functions $h$ of $D$ with $(D,h)\in \widehat{\Div}_{S,\Q}(U)_{\relnef}$ and $[h]\leq [g]$. 
For $(D_0,g_0),\dots, (D_d,g_d)\in \widehat{\Div}_{S,\Q}(U)_{\relnef}$, we set
\[\mathcal{E}(\mathbf{g})\coloneq\mathcal{E}(g_0)\times\cdots\times \mathcal{E}(g_d).\]
For  $h\in \mathcal{E}(g)$, we have obviously $h_\omega \in \mathcal{E}(g_\omega)$ for every $\omega \in \Omega$ using the corresponding local definition from \S \ref{subsection: local mixed relative energy}.}

\begin{lemma} \label{lemma:measurability of mixed relative energy}
	Let $\overline{D_j}=(D_j,g_j)\in \widehat{\Div}_{S,\Q}(U)_{\relnef}$ for $j=0,\dots, d$. Then for any $\mathbf{h}\in \mathcal{E}(\mathbf{g})$, the function
	\[\Omega\longrightarrow\R, \ \ \omega\mapsto E(\mathbf{g}_\omega,\mathbf{h}_\omega)\]
	is $\mathcal{A}$-measurable and bounded above by a $\nu$-integrable function. If $[\mathbf g]=[\mathbf h]$, then the above function is $\nu$-integrable.
\end{lemma}
\begin{proof}
	Shrinking $U$ as in \cref{global nef boundary divisor assumption}, we may assume  that all $(D_j,h_j)$ are strongly relatively nef  given by Cauchy sequences in $N_{S,\Q}(U)$ with respect to the same boundary divisor $\overline B$ of $U$. 
Since the limit of a sequence of $\mathcal{A}$-measurable functions is also $\mathcal{A}$-measurable,  \cref{lemma:convergent sequence} and \cref{lemma:basic properties of mixed relative energy}~\ref{energy5} show that we can assume  $[\mathbf{h}]=[\mathbf{g}]$. Moreover, for the same reason and \cref{lemma: same divisoral part for global divisors}~\ref{item same divisor}, we can assume that $(D_j,g_j), (D_j,h_j)\in N_{S,\Q}(U)$, i.e.~there is a proper $K$-model $X$ of $U$ such that $(D_j, g_j), (D_j,h_j)\in N_{S,\Q}(X)$ for any $j=0,\dots, d$. Then measurability of the  mixed relative energy is a consequence of \cite[Theorem~4.2.9]{chen2021arithmetic}. 

Since $[\mathbf h]\leq [\mathbf g]$,  for $j=0, \dots, d$, there is an integrable function $C_j$ on $\Omega$ with $h_j \leq g_j+C_j $ and hence the definition of $E_\omega(\mathbf g_\omega,\mathbf h_\omega)$ in \cref{definition: local mixed energy} and Guo's Theorem  (see \cref{proposition:measures for nef adelic}~\ref{guo's theorem}) yield
$$E_\omega(\mathbf g_\omega,\mathbf h_\omega) \leq \sum_{j=1}^d C_j(\omega)  \, D_0 \cdots D_{j-1}D_{j+1} \cdots D_d .$$
This gives a $\nu$-integrable upper bound. If $[g_j]=[h_j]$, then by symmetry, there is also an integrable lower bound. Using measurability shown above, we conclude that the mixed relative energy is integrable on $\Omega$.
\end{proof}

\begin{definition} \label{definition: global mixed relative energy}
	Let $(D_0,g_0),\dots, (D_d,g_d)\in \widehat{\Div}_{S,\Q}(U)_{\relnef}$. For any $\mathbf{h}=(h_{0},\dots, h_{d})\in \mathcal{E}(\mathbf{g})$, we define the \emph{mixed relative energy}  by
	\[E(\mathbf{g},\mathbf{h}) \coloneq \int_{\Omega} E(\mathbf g_\omega,\mathbf h_\omega) \,\nu(d\omega) \in \R \cup \{-\infty\}\]
     by using \cref{lemma:measurability of mixed relative energy}. If $[\mathbf{h}]=[\mathbf{g}]$, then it follows also from \cref{lemma:measurability of mixed relative energy} that the mixed relative energy is finite.
     We set $\mathcal{E}^1(\mathbf{g})\coloneq\{\mathbf{h}\in \mathcal{E}(\mathbf{g})\mid  E(\mathbf{g},\mathbf{h})>-\infty\}$.
\end{definition}

\begin{theorem}	\label{lemma:basic properties of global mixed relative energy}
	Let $(D_0,g_0),\dots, (D_d,g_d)\in \widehat{\Div}_{S,\Q}(U)_{\relnef}$ and $\mathbf{h}\in \mathcal{E}(\mathbf{g})$.
	\begin{enumerate}
		\item \label{global energy1} Let $U'$ be a dense open subset of $U$. For $\mathbf{g'}\coloneq \mathbf{g}|_{U'}$ and $\mathbf{h'}\coloneq \mathbf{h}|_{U'}$, we have $\mathbf{h'} \in \mathcal{E}(\mathbf{g'})$ on $U'$ and $E(\mathbf{g'},\mathbf{h'})=E(\mathbf{g},\mathbf{h})$. 
       \item \label{global energy2} For every permutation $\sigma\in S_{d+1}$ of the set $\{0,\dots, d\}$, we have \[E(\mathbf{g},\mathbf{h})=E(\sigma(\mathbf{g}),\sigma(\mathbf{h})).\]
       \item \label{global energy3} Let $\mathbf{u}\in \mathcal{E}(\mathbf{g})$ with $\mathbf{h}\leq \mathbf{u}$. Then
       \[E(\mathbf{g},\mathbf{h})\leq E(\mathbf{g},\mathbf{u}).\]
       In particular, if $\mathbf{h}\in \mathcal{E}^1(\mathbf{g})$, then so is $\mathbf{u}$.
       \item \label{global energy4}
       Let $\mathbf{c}=(c_0,\dots, c_d)$  for integrable functions $c_j$ on $\Omega$. Then we have
       \[E(\mathbf{g},\mathbf{h}+\mathbf{c}) = E(\mathbf{g},\mathbf{h})+\sum\limits_{j=0}^d \int_{\Omega}c_j(\omega) \,\nu(d\omega) \, D_0\cdots D_{j-1}D_{j}\cdots D_d.\] 
       \item \label{global energy5} Let $\mathbf{h}_{n}$ be a decreasing sequence in $\mathcal{E}(\mathbf{g})$ such that for $j=0,\dots,d$, we have that $(D_j,h_{j,n})$ is strongly relatively nef and converges to $(D_j,h_j)$ with respect to {some weak boundary divisor} when $n\to\infty$. Then
       \[E(\mathbf{g},\mathbf{h})=\lim\limits_{(n_0,\dots,n_d)\to\infty}E(\mathbf{g},(h_{0,n_0},\dots, h_{d,n_d})).\]
       \item \label{global energy6} If $[\mathbf g]=\mathbf[\mathbf{h}]$, then \[E(\mathbf{g},\mathbf{h}) =\inf\{E(\mathbf{g},\mathbf{u}) \mid \mathbf{u}\in \mathcal{E}(\mathbf{g}) \text{ with $[\mathbf{g}]=[\mathbf{u}]$ and } \mathbf{u}\geq \mathbf{h}\}.\]
       \item \label{global energy7} If  {$(D_j,g_j)$ and} $(D_j,h_j)$ {are} given as {limits} of sequences  in $N_{S,\Q}(U)$ with respect to a boundary divisor $\overline B \in N_{S,\Q}(U)$ of $U$, then 
       \[E(\mathbf{g},\mathbf{h}) =\inf\{E(\mathbf{g},\mathbf{u}) \mid \mathbf{u}\in \mathcal{E}(\mathbf{g}) \text{ strongly relatively nef with $[\mathbf{g}]=[\mathbf{u}]$ and } \mathbf{u}\geq \mathbf{h}\}.\]
   \end{enumerate}
\end{theorem}
\begin{proof}
	Properties \ref{global energy1}--\ref{global energy4} {and \ref{global energy6}} follow immediately from the corresponding properties in \cref{lemma:basic properties of mixed relative energy} by integrating. 
	
	For \ref{global energy5}, there is a {weak} boundary divisor $\overline B=(B,g_B)$ of $U$ such that for $j=0,\dots, d$, the sequence $(D_j,h_{j,n})$ decreasingly converges to $(D_j,h_j)$ with respect to the $\overline B$-boundary topology. Since $(B_\omega,g_{B,\omega})$ is dominated by a 
	cofinal boundary divisor of $U_\omega$, the sequence $(D_{j,\omega},h_{j,n,\omega})$ decreasingly converges to $(D_{j,\omega},h_{j,\omega})$ with respect to the boundary topology of $\widehat{\Div}_{\Q}(U_\omega)_\cpt$ for every $\omega \in \Omega$. Then \cref{lemma:basic properties of mixed relative energy}~\ref{energy5} and  the monotone convergence theorem give \ref{global energy5}.
	
	To prove \ref{global energy7}, we get from \cref{lemma:convergent sequence} an increasing sequence $(C_{j,n})_{n \geq 1}$ of $\nu$-integrable functions on $\Omega$ such that the cutting functions $h^{(n)}_{j}\coloneq\max\{h_{j},g_{j}-C_{j,n}\}$ lead to a decreasing sequence $(D_j,h^{(n)}_j)_{n \in \N}$ of strongly relatively nef compactified $S$-metrized divisors of $U$ 
	which converges to $(D_j,h_j)$ with respect to the boundary divisor $\overline{B} \in N_{S,\Q}(U)$ for every $j=0,\dots,d$. We set
	\[E'(\mathbf{g},\mathbf{h}) =\inf\{E(\mathbf{g},\mathbf{u}) \mid \mathbf{u}\in \mathcal{E}(\mathbf{g}) \text{{ strongly relatively nef} with $[\mathbf{g}]=[\mathbf{u}]$ and } \mathbf{u}\geq \mathbf{h}\}.\]
	Note that $[h^{(n)}_j]=[g_j]$ and hence $E'(\mathbf{g},\mathbf{h}) 
	\leq E(\mathbf{g},\mathbf{h}_j^{(n)})$. It follows from \ref{global energy5} that the right hand side converges to $E(\mathbf{g},\mathbf{h})$ for $n \to \infty$ and hence $E'(\mathbf{g},\mathbf{h}) \leq E(\mathbf{g},\mathbf{h})$. On the other hand, property \ref{global energy3} yields $E(\mathbf{g},\mathbf{h}) \leq E'(\mathbf{g},\mathbf{h})$ and so we get \ref{global energy7}.
\end{proof}

\begin{proposition} \label{prop: transitivity of global mixed energy}
	Let $(D_0,g_0),\dots, (D_d,g_d)\in \widehat{\Div}_{S,\Q}(U)_{\relnef}$, and $\mathbf{g}'\in \mathcal{E}(\mathbf{g}), \mathbf{h}\in \mathcal{E}(\mathbf{g}')$.  Then we have
	\begin{align} 
		E(\mathbf{g},\mathbf{h}) = E(\mathbf{g}',\mathbf{h})+E(\mathbf{g},\mathbf{g}').
	\end{align} 
\end{proposition}
\begin{proof}
	This follows from the definition of the global mixed relative energy and the corresponding formula for the local mixed relative energy given in \cref{prop: transitivity of local mixed energy}.
\end{proof}
\begin{remark} \label{global energy depends only on isometry class}
	It follows from \cref{local energy depends only on isometry class} that the global mixed relative energy depends only on the isometry classes of the underlying $S$-metrized line bundles and that we have a projection formula for the global mixed relative energy as in \cref{functoriality of energy}.
\end{remark}

%%%%%%%%%%%%%%%%%%%%%%
% An extension of intersection numbers
%%%%%%%%%%%%%%%%%%%%%%

\section{An extension of the arithmetic intersection pairing}  \label{section: extension of arithmetic intersection pairing}

We fix a proper adelic curve $S=(K, \Omega,\mathcal{A},\nu)$ satisfying the normalization assumption for archimedean valuations given in \cref{assumeption on adelic curves}. We also assume that either the $\sigma$-algebra $\mathcal{A}$ is discrete, or that $K$ is countable.

We consider any algebraic variety $U$ over $K$ of dimension $d$. Recall from Section \ref{section: The global boundary completion} that we have defined the arithmetic intersection pairing as a $(d+1)$-multilinear map on $\widehat{\Div}_{S,\Q}(U)_\arnef$. The goal of this section is to extend the pairing to all relatively nef compactified {$S$-metrized} divisors on $U$ which allow an arithmetically nef $S$-metric.  The backslash is that we have to allow that the arithmetic intersection numbers take the value $-\infty$. This is a generalization to the adelic setting of a construction of Burgos and Kramer \cite{burgos2023on} where the number field case was considered and only variations of archimedean places were allowed.

Recall from \cref{def:CM divisors on non-proper} that the abstract divisorial space $(M_{S,\Q}(U),N_{S,\Q}(U))$  is  built upon the $S$-semipositive adelic divisors on proper $K$-models of $U$ in the sense of Chen and Moriwaki and that the abstract divisorial space $(M_{S,\Q}(U),N'_{S,\Q}(U))$ from \ref{boundary divisors in M'} is built upon the  {$S$-metrized divisors in $N'_{S,\Q}(X)$ on proper $K$-models $X$} of $U$. From the boundary completions of these two spaces, we obtained $\widehat{\Div}_{S,\Q}(U)_{\relsnef}$ and $\widehat{\Div}_{S,\Q}(U)_{\arsnef}$. {Passing to closures with respect to the finite subspace topologies, we obtained $\widehat{\Div}_{S,\Q}(U)_{\relnef}$ and $\widehat{\Div}_{S,\Q}(U)_{\arnef}$.}

\begin{lemma} 		\label{lemma:difference of intersection of arith nef}
	Let $\overline{D_0'}=(D_0,g_0'), \overline{D_0}=(D_0,g_0),\overline{D_1},\dots, \overline{D_d} \in\widehat{\Div}_{S,\Q}(U)_{\arsnef}$ 
	which are given by Cauchy sequences in $N'_{S,\Q}(U)$ with respect to the $\overline{B}$-boundary topology for some {weak} boundary divisor  $\overline B \in {N_{S,\Q}'(U)}$. Assume that $[g_0]=[g_0']$. 
	Then the function
	\[\Omega\longrightarrow\R, \ \ \omega\mapsto \int_{U_\omega^\an}(g_{0,\omega}-g_{0,\omega}') c_1(\overline{D_1})_\omega\wedge\dots\wedge c_1(\overline{D_d})_\omega\]
	is $\nu$-integrable, and
	\begin{align*}
(\overline{D_0}\cdot\overline{D_1}\cdots\overline{D_d}\mid U)_S-&(\overline{D_0'}\cdot\overline{D_1}\cdots\overline{D_d}|U)_S \\&= \int_{\Omega}\left(\int_{U_\omega^\an}(g_{0,\omega}-g_{0,\omega}') c_1(\overline{D_1})_\omega\wedge\dots\wedge c_1(\overline{D_d})_\omega\right)\,\nu(d\omega).
	\end{align*}
\end{lemma}

\begin{proof}
	Replacing the weak boundary divisor $\overline B$ by a higher {weak} boundary divisor in ${N_{S,\Q}'(U)}$, it follows from  \cref{lemma: same divisoral part for global divisors} that we can take Cauchy sequences $(D_{0,n},g_{0,n}')_{n\geq 1}$, $(D_{0,n},g_{0,n})_{n\geq 1 }$ and $(D_{j,n},g_{j,n})_{n\geq 1}$ in $N_{S,\Q}(U)$ (decreasingly)  converging with respect to the $\overline B$-topology to $\overline{D_0'}, \overline{D_0},  \overline{D_j}$ for $j=1,\dots, d$, respectively. Moreover, again by \cref{lemma: same divisoral part for global divisors}, we can assume that all 	$(D_{j,n},g_{j,n})$ for $j=0,\dots, d$ are contained in $N'_{S,\Q}(U)$ and that for any $\omega\in\Omega$, we have $|g_{0,n}-g_{0,n}'|$  bounded by a $\nu$-integrable function $C$ on $\Omega$. 
	Since $(D_{0,n},g_{0,n}'+C)\in N_{S,\Q}(U)$ and  $(D_{0,n},g_{0,n})\in N'_{S,\Q}(U)$, we use $g_{0,n}'+C \geq g_{0,n}$ to deduce from Lemma \ref{lemma:maximum is nef global} that $(D_{0,n},g_{0,n}'+C)\in N'_{S,\Q}(U)$. Since this is a Cauchy sequence with respect to the $\overline B$-topology on $N'_{S,\Q}(U)$ converging to $(D,g'_0+C)$, we conclude that
	\begin{equation} \label{first arithmetic convergence}
		\lim_{n \to \infty} ((D_{0,n},g_{0,n}'+C) \cdot (D_{1,n},g_{1,n}) \cdots (D_{d,n},g_{d,n})\mid U)_S= ((D_0,g'_0+C)\cdot \overline{D_1}\cdots\overline{D_d}\mid U)_S.
	\end{equation}
	Since $C \geq 0$, by \cref{c is S-nef}, we have $(0,C) \in N'_{S,\Q}(U)$ 
	and hence the same argument gives 
	\begin{equation} \label{second arithmetic convergence}
		\lim_{n \to \infty} ((0,C) \cdot (D_{1,n},g_{1,n}) \cdots (D_{d,n},g_{d,n})\mid U)_S= ((0,C)\cdot \overline{D_1}\cdots\overline{D_d}\mid U)_S.
	\end{equation}
	There is a proper $K$-model $X_n$ of $U$ such that 
	$(D_{j,n},g_{j,n})$ for $j=0,\dots,d$ and $(D_{0,n},g'_{0,n})$ are all in $N'_{S,\Q}(X_n)$. By Guo's theorem (see \cref{proposition:measures for nef adelic}~\ref{guo's theorem}), we have that
	\[\left|\int_{U_\omega^\an}f_\omega c_1(\overline{D_1})_\omega\wedge\dots\wedge c_1(\overline{D_d})_\omega\right|\leq C(\omega)\left|\int_{U_\omega^\an} c_1(\overline{D_1})_\omega\wedge\dots\wedge c_1(\overline{D_d})_\omega\right|=C(\omega)D_1\cdots D_d\]
	for $f_\omega \coloneqq g_{0,\omega}-g'_{0,\omega}$.
	By \cref{lemma:measurability of mixed relative energy}, the integral $ \int_{U_\omega^\an}f_\omega c_1(\overline{D_1})_\omega\wedge\dots\wedge c_1(\overline{D_d})_\omega$ is a $\nu$-integrable function in $\omega \in \Omega$. Similarly, by our choice of $g_{0,n}$ and $g_{0,n}'$, we have that
	\begin{align*}
		&\left|\int_{X_{n,\omega}^\an}(g_{0,n,\omega}-g_{0,n,\omega}') c_1(D_{1,n},g_{1,n})_\omega\wedge\dots\wedge c_1(D_{d,n},g_{d,n})_\omega\right|\\
		&\leq C(\omega)\left|\int_{X_{n,\omega}^\an} c_1(D_{1,n},g_{1,n})_\omega\wedge\dots\wedge c_1(D_{d,n},g_{d,n})_\omega\right|\\
		& = C(\omega)(D_{1,n}\cdots D_{d,n})\\
		& \leq C(\omega)(D_1\cdots D_d+1)
	\end{align*}
	for sufficiently large $n$. We deduce from weak convergence in  \cref{corollary:weakly convergence of full mass} that
	\begin{align*}
		&\lim_{n \to \infty} \int_{X_{n,\omega}^\an}(g_{0,n,\omega}-g_{0,n,\omega}') c_1(D_{1,n},g_{1,n})_\omega\wedge\dots\wedge c_1(D_{d,n},g_{d,n})_\omega\\
		=& \int_{U_{\omega}^\an} f_\omega \, c_1(D_{1},g_{1})_\omega\wedge\dots\wedge c_1(D_{d},g_{d})_\omega
	\end{align*} 
	and from the dominated convergence theorem, we get
	\begin{align} \label{limit formula and integration}
		\begin{split}
			&\lim_{n \to \infty} \int_{\Omega} \left(\int_{X_{n,\omega}^\an}(g_{0,n,\omega}-g_{0,n,\omega}') c_1(D_{1,n},g_{1,n})_\omega\wedge\dots\wedge c_1(D_{d,n},g_{d,n})_\omega \right) \, \,\nu(d\omega)\\
			=& \int_{\Omega} \left(\int_{U_{\omega}^\an} f_\omega \, c_1(D_{1,n},g_{1,n})_\omega\wedge\dots\wedge c_1(D_{d,n},g_{d,n})_\omega \right) \, \nu(d\omega).
		\end{split}
	\end{align} 	 
	In the obvious formula
	\begin{align} \label{obvious formula}
		\begin{split}
			&(\overline{D_0}\overline{D_1}\cdots\overline{D_d}\mid U)_S-(\overline{D_0'}\overline{D_1}\cdots\overline{D_d}\mid U)_S\\
			=& (\overline{D_0}\overline{D_1}\cdots\overline{D_d}\mid U)_S-((\overline{D_0'}+(0,C))\overline{D_1}\cdots\overline{D_d}\mid U)_S+((0,C)\cdot\overline{D_1}\cdots\overline{D_d}\mid U)_S,
		\end{split}
	\end{align}
	we use that the definition of the arithmetic intersection numbers gives
	$$\lim_{n \to \infty} ((D_{0,n},g_{0,n}) (D_{1,n},g_{1,n}) \cdots (D_{d,n},g_{d,n})\mid U)_S= (\overline{D_0}\overline{D_1}\cdots\overline{D_d}\mid U)_S$$
	for the first summand. For the second summand on the right of \eqref{obvious formula}, we use \eqref{first arithmetic convergence} to get 
	\begin{align} \label{first two summands from obvious formula}
		\begin{split}
			&(\overline{D_0}\overline{D_1}\cdots\overline{D_d}\mid U)_S
			-((\overline{D_0'}+(0,C))\overline{D_1}\cdots\overline{D_d}\mid U)_S \\
			=&\lim_{n \to \infty} ((0,g_{0,n}-g'_{0,n})\cdot(D_{1,n},g_{1,n})\cdots (D_{d,n},g_{d,n})\mid U)_S\\
			&-((0,C)\cdot(D_{1,n},g_{1,n})\cdots (D_{d,n},g_{d,n})\mid U)_S\\
			=&\lim_{n \to \infty} \int_{\omega \in \Omega} \left( \int_{X_{n,\omega}^\an}(g_{0,n,\omega}-g_{0,n,\omega}') c_1(D_{1,n},g_{1,n})_\omega\wedge\cdots\wedge c_1(D_{d,n},g_{d,n})_\omega \right) \, \nu(d\omega)\\
			&-((0,C) \cdot (D_{1,n},g_{1,n}) \cdots (D_{d,n},g_{d,n})\mid U)_S.
		\end{split}
	\end{align}
	Inserting \eqref{first two summands from obvious formula} in \eqref{obvious formula}, the claim follows from additivity of limits and from \eqref{second arithmetic convergence}, \eqref{limit formula and integration}.
\end{proof}

In the following main result, we will use the space $\mathcal{E}(\mathbf{g})$ introduced before Lemma \ref{lemma:measurability of mixed relative energy}. We will also use the mixed relative energy $E(\mathbf g, \mathbf h)$ introduced in the local case in Section \ref{section: mixed relative energy in the local case} and in the global case in Section \ref{section: mixed relative energy in the global case}.
															
\begin{theorem} 	\label{thm:difference of intersection of arithmetically nef}
	Let $\overline{D_j}=(D_j,g_j), \overline{D_j'}=(D_j,g_j')\in \widehat{\Div}_{S,\Q}(U)_{\arnef}$ for $j=0,\dots, d$ {and assume that $[\mathbf{g'}]\leq [\mathbf{g}]$.}   
   Then $E(\mathbf{g}_\omega,\mathbf{g}_\omega')$ is a $\nu$-integrable function of $\omega \in \Omega$ and 
	\[(\overline{D_0'}\cdots\overline{D_d'}\mid U)_S=(\overline{D_0}\cdots\overline{D_d}\mid U)_S+\int_{\Omega}E(\mathbf{g}_\omega,\mathbf{g}_\omega')\,\nu(d\omega)= (\overline{D_0}\cdots\overline{D_d}\mid U)_S+E(\mathbf{g},\mathbf{g}').\]
\end{theorem}
															
This was shown in \cite[Theorem~4.4]{burgos2023on} for the number field case where the metrics are only allowed to change at archimedean places and where $U$ was supposed to be normal. We will use similar arguments to prove the claim over proper adelic curves.
\begin{proof}
Note that the function $\omega\mapsto E(\mathbf{g}_\omega,\mathbf{g}_\omega')$ is measurable by \cref{lemma:measurability of mixed relative energy}. {By \cref{global nef boundary divisor assumption}, we may replace $U$ by a suitable dense open subset of $U$ to assume that all $\overline{D_j}=(D_j,g_j), \overline{D_j'}=(D_j,g_j')$ are given by Cauchy sequences in $N'_{S,\Q}(U)$ with respect to a boundary divisor $\overline B \in {N_{S,\Q}'(U)}$. Here, we use that the arithmetic intersection numbers do not change by passing to an open dense subset, see the projection formula in Theorem \ref{global intersection number on algebraic varieties}.}
																
First, we assume that, for any $j=0,\dots, d$, there is $C_j\in\mathscr{L}^1(\Omega,\mathcal{A},\nu)$ such that $|{g}_{j,\omega}-{g}_{j,\omega}'|\leq C_j(\omega)$ on $U_\omega^\an$ for any $\omega\in\Omega$. By \cref{lemma:difference of intersection of arith nef} {and the symmetry of arithmetic intersection numbers}, we have that
\begin{align*}
	&(\overline{D_0'}\cdots\overline{D_d'}\mid U)_S-(\overline{D_0}\cdots\overline{D_d}\mid U)_S\\
	=&\sum\limits_{j=0}^d((\overline{D_0'}\cdots\overline{D_j'}\cdot\overline{D_{j+1}}\cdots\overline{D_d}\mid U)_S-(\overline{D_0'}\cdots\overline{D_{j-1}'}\cdot\overline{D_j}\cdots\overline{D_d}\mid U)_S)\\
	=& \sum\limits_{j=0}^d\int_{\Omega}\left(\int_{U_\omega^\an}(g_{j,\omega}'-g_{j,\omega})c_1(\overline{D_0'})_\omega\wedge\cdots\wedge c_1(\overline{D_{j-1}'})_\omega\wedge c_1(\overline{D_{j+1}})_\omega\cdots\wedge c_1(\overline{D_d})_\omega\right) \,\nu(d\omega)\\
	=&
	\int_{\Omega}E(\mathbf{g}_\omega,\mathbf{g}_\omega')\,\nu(d\omega) = E(\mathbf{g},\mathbf{g'}).
\end{align*}
																
Next, we consider the general case when $|g_{j,\omega}-g_{j,\omega}'|$ is not necessarily bounded by an integrable function. 
For any $j=0,\dots, d$, by \cref{lemma:convergent sequence},   there is an increasing sequence of $\nu$-integrable functions $(C_{j,n})_{n\geq 1}\in \mathscr{L}^1(\Omega,\mathcal{A},\nu)$ such that 
\[\overline{D_{j,n}'}\coloneq(D_j,g_{j,n}')\coloneq(D_j,\max\{g_j-C_{j,n},g_j'\}) \in\widehat{\Div}_{S,\Q}(U)_{\arsnef}\]
converges decreasingly to $(D_j,g_j')$ with respect to the $\overline B$-topology.
Since ${g}_{j}-C_{j,n} \leq g_{j,n}'$ and $[g_{j}']\leq [g_{j}]$, we have that $[g'_{j,n}]=[g_j]$.
From the preceding, we have that
\[(\overline{D_{0,n}'}\cdots\overline{D_{d,n}'}\mid U)_S-(\overline{D_0}\cdots\overline{D_d}\mid U)_S =E(\mathbf{g},\mathbf{g}_{n}').\]
By the continuity of arithmetic intersection in Theorem \ref{global intersection number on algebraic varieties}, we have that
\[\lim\limits_{n\to\infty}(\overline{D_{0,n}'}\cdots\overline{D_{d,n}'}\mid U)_S = (\overline{D_{0}'}\cdots\overline{D_{d}'}\mid U)_S\]
and hence the claim follows from \cref{lemma:basic properties of global mixed relative energy}~\ref{global energy5}.
\end{proof}

In the following, we extend the arithmetic intersection numbers from \cref{global intersection number on algebraic varieties} to all {$S$-metrized} divisors $(D,g_D) \in \widehat{\Div}_{S,\Q}(U)_{\relnef}$ such that there is an $S$-Green function $h_D$ of $D$ with $(D,h_D) \in \widehat{\Div}_{S,\Q}(U)_{\arnef}$. In other words, we assume that the underlying $S$-metrized $\Q$-line bundle $\overline L=(\mathcal O_U(D), \metr)$ has another $S$-metric $\metr'$ such that $(L,\metr') \in \widehat{\Pic}_{S,\Q}(U)_{\arnef}$. We denote the monoid of all such relatively nef compactified $S$-metrized divisors $(D,g_D)$ by $\widehat{\Div}_{S,\Q}(U)_{\relnef}^\arnef$.

\begin{theorem} \label{extension of global intersection number on algebraic varieties}
  Let $S=(K, \Omega,\mathcal{A},\nu)$ be the given  proper adelic curve. For any  algebraic variety $U$ over $K$ of dimension $d$, for any $\overline{D_0}, \dots, \overline{D_k} \in \widehat{\Div}_{S,\Q}(U)_{\relnef}^\arnef$ and any effective $k$-dimensional cycle $Z$ of $U$, there is a unique $(\overline{D_0} \cdots \overline{D_k} \mid Z)_S \in \R \cup \{-\infty\}$ with the following properties:
  \begin{enumerate}
	\item \label{main1} The number $(\overline{D_0} \cdots \overline{D_k} \mid Z)_S \in \R \cup \{-\infty\}$ depends only on the isometry classes of the underlying $S$-metrized $\Q$-line bundles $\overline{L_j}=(\mathcal O_U(D_j), \metr_j)$, $j=0,\dots, k$, and on the cycle $Z$, but not on the particular choice of the {compactified $S$-metrized} divisors $\overline{D_0}, \dots, \overline{D_k}$, so we set
	$$(\overline{L_0} \cdots \overline{L_k} \mid Z)_S \coloneqq (\overline{D_0} \cdots \overline{D_k} \mid Z)_S.$$
	If $k=d$ and $Z=U$, then we just write $\overline{L_0} \cdots \overline{L_d}\coloneqq \overline{D_0} \cdots \overline{D_d} \coloneqq (\overline{D_0} \cdots \overline{D_d} \mid U)_S$.
	\item \label{main2} The pairing $(\overline{L_0} \cdots \overline{L_k} \mid Z)_S \in \R\cup \{-\infty\}$ is multilinear and symmetric in $\overline{L_0}, \dots, \overline{L_k}$ and linear in $Z$. 
	\item \label{main3} If $\overline{D_0}, \dots, \overline{D_k} \in \widehat{\Div}_{S,\Q}(U)_{\arnef}$, then $(\overline{L_0} \cdots \overline{L_k} \mid Z)_S$ is the arithmetic intersection pairing from Theorem \ref{global intersection number on algebraic varieties}. 
    \item \label{main4} For $\overline{D_0}=(D_0,g_0), \dots, \overline{D_d}=(D_d,g_d)\in  \widehat{\Div}_{S,\Q}(U)_\relnef^\arnef$ and $\mathbf h =(h_0,\dots,h_d)\in \mathcal E(\mathbf g)$, we have
    $$(D_0,h_0) \cdots (D_d,h_d) =(D_0,g_0) \cdots (D_d,g_d)+E(\mathbf g,\mathbf h).$$
    \item \label{main5} Let $K'/K$ be an algebraic extension of $K$ and let $S'$ be the canonical adelic curve on $K'$ induced by $S$. Let $U',\overline{L_0'}, \dots, \overline{L_k'},Z'$ be obtained from $U,\overline{L_0}, \dots, \overline{L_k},Z$ by base change. Then we have
    $$(\overline{L_0'} \cdots \overline{L_k'} \mid Z')_{S'}=(\overline{L_0} \cdots \overline{L_k} \mid Z)_S.$$
    \item \label{main6} If $\varphi \colon U' \to U$ is a morphism of algebraic varieties over $K$ and if $Z'$ is an effective $k$-dimensional cycle on $U'$, then the projection formula holds:
    $$(\varphi^*\overline{L_0} \cdots \varphi^*\overline{L_k} \mid Z')_S = (\overline{L_0} \cdots \overline{L_k} \mid \varphi_*Z')_S.$$
    \end{enumerate}
\end{theorem}
\begin{proof}
	We first show that the uniqueness follows from \ref{main1}--\ref{main5}. By base change and \ref{main5}, we may assume $K$ algebraically closed. By multilinearity in \ref{main2}, we may assume $k=d$ and $Z=U$. We use \ref{main4} to reduce to the arithmetically nef case and then uniqueness follows from \ref{main3}. 
	
	To define the pairing, we use the same procedure. We easily reduce to the case $k=d$ and $Z=U$  by using base change to the algebraic closure of $K$ and then by proceeding by linearity in the components of $Z$. Then we have to define the arithmetic intersection pairing for $(D_0,g_0), \dots, (D_d,g_d)\in \widehat{\Div}_{S,\Q}(U)_{\relnef}^\arnef$. By definition, for any $j=0,\dots, d$, there is an $S$-Green function $h_j$ for $D_j$ such that $(D_j,h_j) \in \widehat{\Div}_{S,\Q}(U)_{\arnef}$. By Remark \ref{global nef boundary divisor assumption}, there is a dense open subset $U'$ of $U$ with a boundary divisor $\overline B \in N_{S,\Q}(U')$ such that $(D_j,g_j)$ is given as a limit of a sequence in $N_{S,\Q}(U')$ with respect to the $\overline B$-boundary topology and such that $(D_j,h_j)|_{U'}$ is given  as a limit of a sequence in $N'_{S,\Q}(U')$ with respect to the $\overline B$-boundary topology for $j=0,\dots,d$. For $D_j'\coloneqq D_j|_{U_j'}$ and $h_j' \coloneqq \max\{g_j|_{U'},h_j|_{U'}\}$,    \cref{lemma:maximum is nef global}  shows that  $(D_j',h_j')$ {is a strongly arithmetically nef compactified $S$-metrized divisor of $U'$.} 
	As we have $[g_j] \leq [h_j]$, we get $\mathbf{g}|_{U'}\in \mathcal{E}(\mathbf h')$. Then we define 
	$$((D_0,g_0) \cdots (D_d,g_d)\mid U)_S \coloneqq ((D_0|_{U'},h_0') \cdots (D_d|_{U'},h_d')\mid U')_S + E(\mathbf h', \mathbf g|_{U'}).$$
	The arithmetic intersection product on the right and the global mixed relative energy are computed on $U'$. The projection formulas for the arithmetic intersection product in \cref{global intersection number on algebraic varieties} and for the global relative mixed energy in \cref{global energy depends only on isometry class} show that the above definition does not depend on the choice of $U'$. To see that the definition is independent of the choice of $h_0,\dots, h_d$, let $f_0, \dots, f_d$ be another choice. Using the above, we may assume that the same $U'$ works for both choices and even that $U=U'$. Replacing $f_j$ by $\max\{f_j,h_j\}$ and using \cref{lemma:maximum is nef global}, we may assume $h_j \leq f_j$. Then independence follows from \cref{prop: transitivity of global mixed energy} and \cref{thm:difference of intersection of arithmetically nef}.  This shows that the extension of the arithmetic intersection pairing is well-defined. 
	
	The properties \ref{main1}--\ref{main6} then follow from the corresponding properties of the pairing in the arithmetically nef case given in \cref{global intersection number on algebraic varieties} and from the properties of the global mixed relative energy given in \cref{lemma:basic properties of global mixed relative energy}, \cref{prop: transitivity of global mixed energy} and \cref{global energy depends only on isometry class}. The proof of \ref{main4} follows in fact from the same arguments as above.
\end{proof}

We illustrate the above extension of arithmetic intersection numbers in our running example.

\begin{example} \label{running example: extension of arithmetic intersection numbers}
We view $K=\Q$ in the standard way as an adelic curve $S=(\Q,\Omega,\mathcal{A},\nu)$, see \cref{example:adelic structure of number fields}. We  embed $U=\mathbb A_K^1$  as usual in $X=\mathbb P_K^1$. We consider the toric divisor $D= [\infty]$ of $X$  as a geometric compactified divisor in $\widetilde{\Div}_\Q(U)_\cpt$, see \cref{example: projective line}. Let $\Psi(u)\coloneqq \min(u,0)$ be the associated concave piecewise linear function on $\R$. We pick an $S$-metrized divisor $(D,g_D)\in \widehat{\Div}_{S,\Q}(U)_\relsnef$ with toric Green function $g_D$. We have seen in \cref{singular arithmetic nef in running example} that for every $v \in \Omega$, there is a concave function $\psi_v\colon \R \to \R$ with asymptotic slopes $1$ for $u \to -\infty$ and $0$ for $u \to \infty$ such that $g_{D,v}= \psi_v \circ \trop_v$ on $T^\an=X^\an\setminus \{0,\infty\}$. Moreover, the function $\psi_v$ is an increasing limit of concave functions $\psi_{n,v}$ with the same asymptotic slopes and with $\psi_{n,v}-\Psi=O(1)$ for each $v \in \Omega$ such that \eqref{b-convergence for concave functions} is satisfied. Equivalently, this means that the corresponding Green functions $g_n$ for $D$ give rise to $(D,g_n)\in \widehat{\Div}_{S,\Q}(X)_\relsnef$ decreasingly converging to $(D,g_D)$ with respect to the boundary topology with respect to a suitable choice of a boundary divisor $\overline B$ as in \cref{singular arithmetic nef in running example}.

We recall from \cref{ample and nef adelic metrics in running example} that the Green function $g$ for $D$ associated to $\Psi$ is the Green function for $D$ corresponding to the canonical metric of $L=\OO_X(D)$. We have seen there that $(D,g) \in \widehat{\Div}_{S,\Q}(X)_\arnef$.  By definition, we have $h\coloneqq g_D \in \mathcal E(g)$ and $(D,h) \in \widehat{\Div}_{S,\Q}(U)_{\relnef}^\arnef$. Our goal is to compute the \emph{global height}
$$h_{\overline L}(U) \coloneqq  \overline L \cdot \overline L = (D,g_D) \cdot (D,g_D).$$
We have seen in \cref{ample and nef adelic metrics in running example} that the global height of $X$ (and hence of $U$) with respect to the canonical metric is $0$ and  that for non-singular metrics (i.e.~in the case $U=X$) the global height is given by the formula of Burgos--Philippon--Sombra \eqref{BPS formula} involving the roof function. We apply this to the Green functions $g_n$ and their roof functions $\vartheta_n$ to deduce
$$E(g,g_n)=(D,g_n)\cdot (D,g_n) = 2 \int_{[0.1]} \vartheta_n(m) \, dm$$
where the first equality uses \cref{extension of global intersection number on algebraic varieties}\ref{main4}. Since the concave functions $\psi_{n,v}$ increasingly converge to the concave function $\psi_v$ with the same asymptotic slopes, it follows from biduality that the Legendre--Fenchel duals $\psi_{n,v}^\vee$ decreasingly converge to the concave function $\psi_v^\vee$, hence the roof functions $\vartheta_n$ decreasingly converge to the roof function $\vartheta$ of $g$. The continuity of the relative energy in Theorem \ref{lemma:basic properties of global mixed relative energy}\ref{global energy5} and the monotone convergence theorem give the desired formula for the global height
\begin{equation} \label{roof formula for singular}
	h_{\overline L}(X)=(D,g) \cdot (D,g) =E(g,h)=  2 \int_{[0.1]} \vartheta(m) \, dm
\end{equation}
where the second equality uses again \cref{extension of global intersection number on algebraic varieties}\ref{main4} and that the global height of $X$ with respect to the canonical metric is zero. For a generalization of this formula to arbitrary toric varieties over a number field, we refer to \cite[Theorem 5.5.3]{peralta-thesis24}.
\end{example}

We specify the height formula \eqref{roof formula for singular} as in \S \ref{subsec: a running example}.

\begin{example} \label{specific example for height function}
	We use the same setting as in \cref{running example: extension of arithmetic intersection numbers}. We fix $\alpha \in ]0,1[$ and a place $v$ of $K=\Q$ where the singular Green function $g_{D,v}$ is given by the concave piecewise smooth function $\psi_v \coloneqq \Psi + \rho$ with $\rho(u)\coloneqq \frac{1}{\alpha}$ for $u \geq 0$ and $\rho(u)\coloneqq \frac{1}{\alpha}(1-u)^\alpha$ for $u \leq 0$. For all other places $\omega \in \Omega$, we take $\psi_\omega=\Psi$ which means that the corresponding Green function $g_{D,\omega}$ is associated to the canonical metric of $\OO_{X_\omega}(D_\omega)$. Since the asymptotic slopes of $\psi_v$ are $1$ for $u\to -\infty$ and $0$ for $u \to \infty$, this choice gives rise to $(D,g_D)\in \widehat{\Div}_{S,\Q}(U)_\relsnef$ and hence fits to \cref{running example: extension of arithmetic intersection numbers}. We have seen in \cref{rareness of arithmetic nef} that the metric $(D,g_D)$ is not arithmetically integrable. We give two approaches to compute the global height $h_{\overline L}(U)$ of $U$ with respect to $\overline L=\OO_U(D,g_D)$.
	
	First, we use the formula for the relative energy from \cref{relative energy in running example} applied with $g$ the Green function of $D$ associated to the canonical metric of $L$ and $h=g_D$ leading to
	$$E(g,h)= \int_\R (\Psi-\psi_v)(u) \, \left(\mathrm{MA}(\psi_v)+\mathrm{MA}(\Psi)\right)(du) .$$
	A priori, this is the local relative energy at the place $v$, but as $\psi_\omega=\Psi$ for all other places $\omega$, this is also the global relative energy. By \cref{MA measure in running example}, we have that $\mathrm{MA}(\Psi)$ is the Dirac measure in the origin $0$ of $\R$.  Since the slope of $\psi_v$ in $0$ is $0$ from both sides, we get
	$$\mathrm{MA}(\psi_v)=-\psi_v''(u) \chi_{[-\infty,0]} du= -(\alpha-1)(1-u)^{\alpha-2}\chi_{[-\infty,0]}du,$$
	where $\chi_{[-\infty,0]}$ is the characteristic function of ${[-\infty,0]}$.  Then we get 
	\begin{align*}
		E(g,h)=&\int_{\R}-\rho(u) \, \mathrm{MA}(\psi_v)+\int_{\R}-\rho(u)\,\mathrm{MA}(\Psi) \\
		=&\frac{\alpha-1}{\alpha}\int_{[-\infty,0]}(1-u)^{2\alpha-2}du- \frac{1}{\alpha}\\
		=&\begin{cases}
			\frac{2-3\alpha}{\alpha(2\alpha-1)} & \text{if $\alpha <1/2$,}\\
			-\infty & \text{ if $\alpha\geq 1/2$.}
		  \end{cases}
	  \end{align*} Hence $g_D=h$ has finite relative energy with respect to $g$  if and only if $\alpha < 1/2$. Since the height of $U$ with respect to the canonical metric is $0$, we deduce from  \cref{extension of global intersection number on algebraic varieties}\ref{main4}  that $h_{\overline L}(U)=E(g,h)$ and hence the above is our desired formula for the global height.
	  
	  In our second approach to compute the global height, we use the roof function $\vartheta$ of $g_D$ introduced in \cref{ample and nef adelic metrics in running example} and the formula \eqref{roof formula for singular}. In this specific case, the roof function $\vartheta$ is the Legendre--Fenchel dual $\psi_v^\vee$. Note that $\psi_v'(u)=1-(1-u)^{\alpha-1}$ is a diffeomorphism from $]-\infty,0]$ onto $[0,1[$, hence the inverse $f$ can be easily computed as $$f(m)=1-(1-m)^{\frac{1}{\alpha-1}}$$ for $m \in [0,1[$. It follows from the computation of the Legendre--Fenchel dual in the smooth case \cite[Theorem 2.4.2]{BPS} that 
	  \begin{align*}
	  	\vartheta(m) &= \psi_v^\vee(m)= m f(m)- \psi(f(m))=(m-1)f(m){-\frac{1}{\alpha}(1-f(m))^\alpha}\\&= (m-1) \left(1-(1-m)^{\frac{1}{\alpha-1}} \right){-\frac{1}{\alpha}(1-m)^{\frac{\alpha}{\alpha-1}}}\\
	  	&{=\left(1-\frac{1}{\alpha}\right)(1-m)^{\frac{\alpha}{\alpha-1}}+m-1}.
	  \end{align*}
	  This is singular in $m=1$. Now we deduce from \eqref{roof formula for singular} in \cref{running example: extension of arithmetic intersection numbers} that
	  $$h_{\overline L}(U)= 2 \int_{[0.1]} \vartheta(m) \, dm= -1 +2\cdot\frac{\alpha-1}{\alpha}\int_0^1(1-m)^{\frac{\alpha}{\alpha-1}} \,dm$$
	  which leads to the same result as above.
\end{example}

\appendix

\section{Radon measures and convergence results}

In measure theory, we have different definitions of convergence of measures. To show the main results, i.e. \cref{cor:weakly converges of measures}, we recall two kinds of convergence. {Before that, we fix our notation used in this appendix. For a topological space $X$, we denote $C_b(X)$ (resp. $C_c(X)$) the set of bounded continuous real functions (resp. of continuous real functions with compact support) on $X$.}

\begin{definition} \label{def: vague convergence}
	Let $X$ be a topological space. A net $(\mu_n)_{n \in I}$ of measures on $X$ \emph{weakly} (resp. \emph{vaguely}) \emph{converges} to a measure $\mu$, denoted by $\mu_n\overset{w}{\rightarrow}\mu$ (resp. $\mu_n\overset{v}{\rightarrow}\mu$) if for each $f\in C_b(X)$ (resp. $f\in C_c(X)$), we have the equality
	\[\lim\limits_{n\to\infty}\int_Xf\mu_n=\int_Xf\mu.\]
\end{definition}

These two definitions of convergence are different in general, but they are equivalent in some cases. {The following proposition is formulated in terms of nets. Strictly speaking, we don't have to use nets in this paper, but it is very convenient as we see in} {\cref{proposition:measures for nef adelic}~\ref{MA measure via convergence in boundary topology} and~\ref{MA measure via convergence in finite subspace topology} where one can take the net $\N^d$.}

 \begin{prop}[\cite{bogachev2018weak}, Proposition~4.5.11]	\label{thm:equivalence of vague convergence and weak convergence}
	Let $X$ be a locally compact (Hausdorff) space, and $\mu$ a positive Radon measure, $(\mu_n)_{n \in I}$ a net of positive Radon measures on $X$ with $\mu(X),\mu_n(X)<\infty$. Then $\mu_n\overset{w}{\rightarrow}\mu$ if and only if $\mu_n\overset{v}{\rightarrow}\mu$ and $\lim\limits_{n\to\infty}\mu_n(X)=\mu(X)$. 
\end{prop}
\begin{proof}
	\cite[Proposition~4.5.11]{bogachev2018weak} states the result for a sequence $(\mu_n)_{n}$, but it also works when  $(\mu_n)_{n}$ is a net. For convenience of the readers, we give a direct proof here. It suffices to show that the vague convergence and  $\lim\limits_{n\to\infty}\mu_n(X)=\mu(X)$ implies the weak convergence. We can assume that $\mu_n(X), \mu(X)\leq 1$. We claim that for any $\varepsilon>0$, there is a compact subset $K$ and $N$ such that for any $n\geq N$, we have that 
	\[\mu_n(X\setminus K), \mu(X\setminus K)\leq \varepsilon.\]
	Indeed, this can be proved similarly as \cite[Proposition~4.5.11]{bogachev2018weak}. Let $\varepsilon>0$, by regularity of Radon measures, we find a compact subset $K'$ and $N'$ such that $\mu(K')>\mu(X)-\frac{\varepsilon}{3}$ and $\mu_n(X)<\mu(X)+\frac{\varepsilon}{3}$ for any $n\geq N'$. 
	By Urysohn's lemma for locally compact spaces, there is $\varphi\in C_c(X)$ with $\varphi|_{K'}=1$ and $0\leq \varphi\leq 1$, hence $\int_X\varphi\mu\geq \mu(K')$.
	We take $N\geq N'$ such that\[\int_X\varphi\mu_n\geq \int_X\varphi\mu-\frac{\varepsilon}{3} \ \ \text{ for all $n\geq N$}.\]
	Note that the compact set $K\coloneq\Supp(\varphi)$ contains $K'$. Then for any $n\geq N$, we have that
	\[\mu_n(K)\geq \int_X\varphi\mu_n\geq \int_X\varphi\mu-\frac{\varepsilon}{3}\geq\mu(K')-\frac{\varepsilon}{3}\geq \mu(X)-\frac{2\varepsilon}{3}\geq \mu_n(X)-\varepsilon,\]
	\[\mu(K)\geq\int_X\varphi\mu \geq \mu(K')\geq \mu(X)-\frac{\varepsilon}{3}.\]
	This proves our claim.
	
	For any $\varepsilon>0$, we take $K$ and $N$ as above. Let $f\in C_b(X)$. We may assume that $|f|\leq 1$. By Urysohn's lemma for locally compact spaces, there is a compactly supported continuous function $\rho\colon X\rightarrow [0,1]$ such that $\rho|_K\equiv 1$. By vague convergence, there is $N'\geq N$ such that for any $n\geq N'$, we have that
	\[\left|\int_X\rho f\mu_n-\int_X\rho f\mu\right|\leq \varepsilon.\]
	Then for any $n\geq N'$, we have that
	\begin{align*}
		\left|\int_Xf\mu_n-\int_Xf\mu\right| &\leq \left|\int_Kf\mu_n-\int_Kf\mu\right|+\left|\int_{X\setminus K}f\mu_n-\int_{X\setminus K}f\mu\right|\\
		&\leq \left|\int_K\rho f\mu_n-\int_K\rho f\mu\right|+2\varepsilon\\
		&\leq \left|\int_X\rho f\mu_n-\int_X\rho f\mu\right|+ \left|\int_{X\setminus K}\rho f\mu_n-\int_{X\setminus K}\rho f\mu\right|+2\varepsilon\\
		&\leq 5\varepsilon
	\end{align*}
	where we used $\mu_n(X \setminus K)\leq \varepsilon$ for the second inequality and $\mu(X \setminus K) \leq \varepsilon$ for the last inequality. This completes the proof.
\end{proof}

This theorem allows us to translate the results for weak convergence to the ones for vague convergence which is appears frequently in Arakelov geometry.

\begin{lemma} \label{lemma:weakly converges of measures}
Let $X$ be a locally compact (Hausdorff) space, $(f_m)_{m \in J}$ a net of continuous functions on $X$ converging uniformly to a bounded (continuous) function $f$, and $(\mu_n)_{n \in I}$ a net of positive Radon measures on $U$ converging weakly to a positive Radon measure $\mu$ with $\mu(X)<\infty$. % and $\lim\limits_{n\to\infty}\mu_n(X)=\mu(X)$. 
Then $f_m\mu_n$ weakly converges to $f\mu$ when $(m,n)\to\infty$.
\end{lemma}

\begin{proof}
Let $\varphi\in C_b(X)$. There is $C>0$ such that $|\varphi|\leq C$. For any $\varepsilon>0$, uniform convergence yields $m_0\in J$ such that for any $m\geq m_0$, we have 
$$|f_m-f| \leq  \frac{\varepsilon}{2C(\mu(X)+1)}.$$
Weak convergence of measures gives $n_0 \in I$ such that for any $n \geq n_0$, we have 
$$\left|\int_X \varphi f \mu_n - \int_X \varphi f \mu \right| \leq \frac{\varepsilon}{2}.$$
Using $\mu_n(X) \to \mu(X)$, we may also assume that for all $n \geq n_0$, we have
$$\mu_n(X) \leq \mu(X)+1.$$
For $m \geq m_0$ and $n \geq n_0$, the three displayed inequalities above yield
\begin{align*}
	\left|\int_{X}\varphi f_m\mu_n-\int_{X}\varphi f\mu\right| &\leq \int_{X}\left|\varphi (f_m-f)\right|\mu_n+\left|\int_{X}\varphi f(\mu_n-\mu)\right|\\
	& \leq C \cdot \frac{\varepsilon}{2C(\mu(X)+1)}\cdot  \mu_n(X) + \frac{\varepsilon}{2} \leq 
	\varepsilon. 
\end{align*}
This proves the lemma.
\end{proof}

\begin{corollary} 	\label{cor:weakly converges of measures}
Let $X$ be a locally compact (Hausdorff) space, $(f_m)_{m \in J}$ a net of  continuous functions on $X$ converging locally uniformly to a (continuous) function $f$, and $(\mu_n)_{n \in I}$ a net of positive Radon measures on $X$ converging weakly to a positive Radon measure $\mu$ with $\mu(X)<\infty$. Assume that there is a constant $C$ such that for all $m \in J$, we have $|f_m| \leq C$. 
Then $f_m\mu_n$ weakly converges to $f\mu$ when $(m,n)\to\infty$. In particular,
\[\lim\limits_{(m,n)\to\infty}\int_Xf_m\mu_n=\int_Xf\mu.\]
\end{corollary}
\begin{proof}
The proof is similar to the one of \cite[Proposition~3.22~(iv)]{burgos2023on}. 
Let $\varphi\in C_b(X)$. By the regularity of Radon measures {and the assumption that $X$ is locally compact}, for any $\varepsilon>0$, there are relatively compact open subsets {$W\subset V\subset X$} such that $\mu(X\setminus W) \leq \varepsilon$ {and $\overline W\subset V$, where $\overline W$ is the closure of $W$ in $X$}. By Urysohn's lemma for locally compact spaces, there is a continuous function $\rho\colon X\rightarrow [0,1]$ with $\rho|_W\equiv1$ and $\rho|_{X\setminus V}\equiv 0$. Since $f_m$ locally uniformly converges to $f$ on $X$, we have that $\rho f_m$ is a net of continuous functions converging uniformly to $\rho f$ on $V$. 
By \cref{lemma:weakly converges of measures}, we have that $\rho f_m\mu_n$ weakly converges to $\rho f\mu$ on $V$. Hence
\[\lim\limits_{(m,n)\to\infty}\int_X\rho\varphi f_m\mu_n = \lim\limits_{(m,n)\to\infty}\int_V\rho\varphi f_m\mu_n=\int_V\rho\varphi f\mu = \int_X\rho\varphi f\mu.\]
On the other hand, we have that
\[\liminf\limits_{n\to\infty}\mu_n(W)\geq \mu(W),\]
see \cite[Corollary~4.3.9]{bogachev2018weak}, thus  
\[\limsup\limits_{n\to\infty}\mu_n(X\setminus W)\leq \mu(X\setminus W)\leq \varepsilon.\]
Since $\varphi, \rho, f_m, f$ are uniformly bounded on $X$, there is a constant $C'>0$ independent of $\varepsilon$ such that the quantities
\[\lim\limits_{(m,n)\to\infty}\int_{X\setminus W}|\varphi f_m|\mu_n, \ \  \lim\limits_{(m,n)\to\infty}\int_{X\setminus W}\rho|\varphi f_m|\mu_n, \ \ \int_{X\setminus W}|\varphi f|\mu, \ \  \int_{X\setminus W}\rho|\varphi f|\mu\]
are all bounded by $\varepsilon C'$. Hence, we have that
\begin{align*}
&\limsup\limits_{(m,n)\to\infty}\left|\int_X\varphi f_m\mu_n-\int_X\varphi f\mu\right|\\
 & \leq \limsup\limits_{(m,n)\to\infty}\left|\int_X \rho\varphi f_m\mu_n-\int_X\rho\varphi f\mu\right|+\limsup\limits_{(m,n)\to\infty}\left|\int_{X\setminus W} (1-\rho)\varphi f_m\mu_n-\int_{X\setminus W}(1-\rho)\varphi f\mu\right|\\
 & \leq 4\varepsilon C
\end{align*}
proving the claim.
\end{proof}

\begin{remark}
If a net of continuous functions $f_m$ converges pointwise and decreasingly to a continuous function $f$ on a locally compact space, then it converges locally uniformly to $f$ by Dini's theorem.
\end{remark}

\bibliography{references}
\bibliographystyle{alpha.bst}
\end{document}